\documentclass[12pt]{article}

\RequirePackage{amsmath,amsfonts ,mathrsfs,amsthm,setspace}
\RequirePackage{xcolor}
\RequirePackage[sort]{natbib}
\usepackage{caption}
\usepackage{natbib}
\setcitestyle{authoryear,open={(},close={)}} %
\usepackage{float}
\usepackage{graphicx}

\usepackage{algorithm}
\usepackage{algorithmic}

\usepackage{tikz}
\usetikzlibrary{shapes.geometric, arrows, positioning}

\usepackage{enumitem}

\newcommand{\innerpoduct}[2]{\left\langle #1, #2\right\rangle }
\newcommand{\biginnerpoduct}[2]{\big\langle #1, #2\big\rangle }

\newcommand{\norm}[1]{\left\lVert#1\right\rVert}
\newcommand{\tr}[0]{ \operatorname{tr} }

\theoremstyle{plain}
\newtheorem{theorem}{Theorem}[section]
\newtheorem{proposition}[theorem]{Proposition}
\newtheorem{lemma}[theorem]{Lemma}
\newtheorem{corollary}[theorem]{Corollary}
\theoremstyle{definition}
\newtheorem{definition}[theorem]{Definition}
\newtheorem{assumption}{Assumption}
\let\origtheassumption\theassumption

\usepackage{chngcntr}

\theoremstyle{remark}
\newtheorem{remark}[theorem]{Remark}
\newtheorem{example}{Example}

\theoremstyle{challenge}
\newtheorem{challenge}[theorem]{Challenge}
\theoremstyle{definition} %

\newcommand{\inD}{ \stackrel{\mathrm{d}}{\rightarrow}}
\newcommand{\inP}{ \stackrel{\mathbb{P}_{*}}{\rightarrow}}

\usepackage{verbatim}
\usepackage{bm} %

\newcommand{\ShatA}{\mathcal{S}^{\mathcal{A}}}
\def\va{{\bm{a}}}
\newcommand{\blA}{\bm{A}}
\newcommand{\blB}{\bm{B}}
\newcommand{\blC}{\bm{C}}

\newcommand{\blE}{\bm{E}}

\newcommand{\blH}{\bm{H}}

\newcommand{\blJ}{\bm{J}}

\newcommand{\blM}{\bm{M}}

\newcommand{\blP}{\bm{P}}

\newcommand{\blS}{\bm{S}}
\newcommand{\blT}{\bm{T}}
\newcommand{\blU}{\bm{U}}
\newcommand{\blV}{\bm{V}}
\newcommand{\blW}{\bm{W}}
\newcommand{\blX}{\bm{X}}
\newcommand{\blY}{\bm{Y}}
\newcommand{\blZ}{\bm{Z}}
\newcommand{\bla}{\bm{a}}
\newcommand{\blb}{\bm{b}}

\newcommand{\bld}{\bm{d}}
\newcommand{\ble}{\bm{e}}

\newcommand{\blh}{\bm{h}}

\newcommand{\bll}{\bm{l}}

\newcommand{\blr}{\bm{r}}

\newcommand{\blu}{\bm{u}}
\newcommand{\blv}{\bm{v}}

\newcommand{\blx}{\bm{x}}
\newcommand{\bly}{\bm{y}}
\newcommand{\blz}{\bm{z}}

\newcommand{\bgbeta}{\bm{\beta}}

\newcommand{\bgdelta}{\bm{\delta}}

\newcommand{\bgtheta}{\bm{\theta}}

\newcommand{\bgmu}{\bm{\mu}}

\newcommand{\bgxi}{\bm{\xi}}

\newcommand{\bgpi}{\bm{\pi}}

\newcommand{\bgTheta}{\bm{\Theta}}

\newcommand{\bgSigma}{\bm{\Sigma}}

\newcommand{\bgOmega}{\bm{\Omega}}

\newcommand{\cA}{\mathcal{A}}
\newcommand{\cR}{\mathcal{R }}

\usepackage{array}
\newcolumntype{L}[1]{>{\raggedright\let\newline\\\arraybackslash\hspace{0pt}}p{#1}}
\newcolumntype{C}[1]{>{\centering\let\newline\\\arraybackslash\hspace{0pt}}p{#1}}
\newcolumntype{R}[1]{>{\raggedleft\let\newline\\\arraybackslash\hspace{0pt}}p{#1}}

\title{Globally-Optimal Greedy Experiment Selection for Active Sequential Estimation}

\author{
  Xiaoou Li\\
  University of Minnesota\vspace{0.5cm}\\  
  Hongru Zhao\\
  University of Minnesota}
\date{}

\begin{document}
\maketitle

\onehalfspacing
\begin{abstract}
Motivated by modern applications such as computerized adaptive testing, sequential rank aggregation, and heterogeneous data source selection, we study the problem of active sequential estimation, which involves adaptively selecting experiments for sequentially collected data. %
The goal is to design experiment selection rules for more accurate model estimation. Greedy information-based experiment selection methods, optimizing the information gain for one-step ahead, have been employed in practice thanks to their computational convenience, flexibility to context or task changes, and broad applicability. However, statistical analysis is restricted to one-dimensional cases due to the problem's combinatorial nature and the seemingly limited capacity of greedy algorithms, leaving the multidimensional problem open.

In this study, we close the gap for multidimensional problems. In particular, we propose adopting a class of greedy experiment selection methods and provide statistical analysis for the maximum likelihood estimator following these selection rules. This class encompasses both existing methods and introduces new methods with improved numerical efficiency. We prove that these methods produce consistent and asymptotically normal estimators. Additionally, within a decision theory framework, we establish that the proposed methods achieve asymptotic optimality when the risk measure aligns with the selection rule.  We also conduct extensive numerical studies on both simulated and real data to illustrate the efficacy of the proposed methods.

From a technical perspective, we devise new analytical tools to address theoretical challenges. For instance, we demonstrate that functions of inverted Fisher information have a regularization effect when used in selection rules, thereby automatically exploring necessary experiments. Additionally, we show that a class of greedy and stochastic optimization methods converges to the minimum of a convex function over a simplex almost surely. These analytical tools are of independent theoretical interest and may be reused in related problems involving stochastic approximation and sequential designs.
\end{abstract}

\noindent
Keywords: 
{Active sequential estimation},
{optimality theory},
{sequential analysis},
{computerized adaptive testing}
\section{Introduction}\label{sec:intro}
In many modern applications, data are collected sequentially and adaptively through varied experiments, with the distribution being influenced by both unknown model parameters and the experiments. Active sequential estimation, which involves the adaptive selection of the experiments, enables more efficient model estimation. It has received considerable attention across various disciplines recently. A few examples are provided below.

\paragraph*{Computerized Adaptive Testing (CAT)} CAT refers to a form of educational assessment where test items are administered adaptively and sequentially based on the test taker's responses to previous items. For instance, if a test taker answers questions correctly, they may receive a more challenging item subsequently. Over the past decades, CAT has gained popularity due to its ability to achieve a more accurate assessment with fewer test items compared to traditional non-adaptive tests.
To implement CAT, Item Response Theory (IRT) models are typically employed \citep{chen2021item,reckase200618}. IRT models assume that a test-taker's responses, whether correct or incorrect, are influenced by both their latent trait parameter and the selected item. A crucial aspect of CAT design involves developing effective item selection rules to estimate the latent trait parameter as accurately as possible. For a comprehensive review on this topic, see \cite{wang2017computerized, bartroff2008modern, chang2009nonlinear}, and the references therein.

\paragraph*{Sequential rank aggregation} 
The rank aggregation problem involves inferring a global rank for a set of items by aggregating noisy pairwise comparison results. This problem finds applications across various domains such as social choice (\cite{saaty2012possibility}), sports (\cite{elo1978rating}), and search rankings (\cite{page1999pagerank}). Statistical models such as the Bradley-Terry model \citep{bradley1952rank}, which assigns a latent score parameter to each object, are often utilized to model the noisy pairwise comparison results. Subsequently, the global rank can be inferred from the estimated latent score parameters. 
Recently, the sequential rank aggregation problem has attracted increased interest. This approach involves sequentially and adaptively selecting the next pair to compare based on the comparison results of previously selected pairs (see, e.g.,\cite{chen2022asymptotically,chen2013pairwise,chen2016bayesian}). A key question of interest is the design of pair selection rules to enhance the efficiency of the rank aggregation process.

Besides the aforementioned applications, additional areas of application include active sampling in signal processing \citep{mukherjee2022active}, active contextual search \citep{chen2023active}, and dynamic pricing \citep{chen2023robust}, among others.

In all of the above applications, the problem can be formulated as a sequential design-and-estimation problem where data $X_1, \cdots , X_n, \cdots $ are collected sequentially. Each $X_n$ has a density function $f_{\boldsymbol{\theta}, a_n} (\cdot)$ relative to a baseline measure, with ${\bgtheta}\in\mathbb{R}^p$ representing the underlying model parameter,  $a_n\in\mathcal{A}$ denoting the experiment selected at time $n$, and $\mathcal{A}$ being a finite set encompassing all possible 
 experiment choices. For example, in the context of CAT, $\bgtheta$ corresponds to the latent proficiency level of a test-taker on $p$ subjects or skills, $a_n$ indicates the $n$-th test item, $\mathcal{A}$ indicates the item bank which collects all the potential test items, and $X_n\in\{0,1\}$ indicates that whether the test-taker answers the $n$-th question correctly or not. 
 At each time step $n$, a decision maker needs to select an experiment $a_n$ based on the past observations $X_1, a_1, X_2, a_2, \cdots, X_{n-1}, a_{n-1}$, sample $X_n$ accordingly, and construct an estimator $\widehat{\boldsymbol{\theta}}_n$ for estimating ${\bgtheta}$. The goal is to find a good adaptive experiment selection rule and an estimator $\widehat{\boldsymbol{\theta}}_N$ so that $\widehat{\boldsymbol{\theta}}_N$ is as accurate as possible, where $N$ could be a fixed sample size or a random stopping time depending on the application.

Greedy information-based experiment selection rules that maximize one-step-ahead information gain have been commonly adopted for item selection in CAT (see, e.g.,  \cite{chang1996global,wang2011item,van1999multidimensional,cheng2009cognitive}).
For example, \cite{wang2011item}  and \cite{tu2018item} describe the following experiment selection rule:
\begin{equation}\label{eq:trace-selection}
    a_{n+1} = \arg\min_{a\in\mathcal{A}} \operatorname{tr} \Big[\big\{\mathcal{I} (\widehat{\bgtheta}_{n}^{\text{ML}};\va_n,a)  \big\}^{-1}\Big],
\end{equation}
where $\va_n= (a_1,\cdots,a_n)$ denotes the experiments selected up to time $n$,  $$\widehat{\bgtheta}^{\text{ML}}_n=\arg\max_{{\bgtheta}}\sum_{i=1}^n \log f_{{\bgtheta}, a_i }(X_i)$$
denotes the maximum likelihood estimator (MLE) with $n$ observations,  $$\mathcal{I} (\widehat{\bgtheta}_{n}^{\text{ML}};\va_n,a) = \frac{1}{n+1}\Big\{\sum_{i=1}^{n}\mathcal{I}_{a_i}(\widehat{\bgtheta}^{\text{ML}}_{n})+\mathcal{I}_{a }(\widehat{\bgtheta}^{\text{ML}}_{n})\Big\}$$ represents the rescaled Fisher information matrix associated with the first $n$ experiments and one extra experiment $a$, while $\mathcal{I}_a ({\bgtheta}) = \mathbb{E}_{X\sim f_{{\boldsymbol\theta},a}} [\nabla \log f_{\bgtheta,a}(X) \{\nabla \log f_{\bgtheta,a}(X)\}^T ]$ denotes the Fisher information matrix associated with the experiment $a$ at the parameter ${\bgtheta}$. Other experiment selection rules in a similar form (e.g., substituting the trace function with other functions like $\log\det(\cdot)$) are also explored in \cite{wang2011kullback}.

These information-based experiment selection rules offer several benefits. First, the selection processes only require the calculation of the Fisher information and are easy to implement. %
Moreover, they are inherently parallelizable, offering scalability when \( |\mathcal{A}| \) is large. %
Second, they quantify the information gain associated with each experiment, thereby providing priority scores for them. This feature enables extension of these rules to various contexts and tasks (e.g.,  $\mathcal{A}$ varies over time). Additionally, given a parametric model, these rules can readily address problems in other applications.

Despite the computational advantages and wide applicability,  the statistical analysis of greedy information-based experiment selection methods is limited to the one-dimensional case ($p=1$) in existing research. In this context,  \cite{chang2009nonlinear} established the consistency, asymptotic normality and optimality results for the MLE, and discussed the application in CAT. However, the multidimensional ($p>1$) case remains an open problem, partly due to the challenges regarding the combinatorial nature of the multidimensional problem and the seemingly limited capacity of greedy methods. 
The following example, which mimics the settings of an educational test measuring two latent traits, illustrates that one has to combine experiments carefully in order to obtain a consistent and/or risk-optimal estimator.
\begin{example}\label{eg:toy-eg}
Let $\boldsymbol{\theta} = (\theta_1,\theta_2)^T$ and $\mathcal{A}=\{1,2,3\}$. Let $f_{\boldsymbol{\theta}, a}$ be the probability mass function for Bernoulli variables with the probability parameter $(1+\exp(-\theta_1+0.1))^{-1}$, $(1+\exp(-\theta_2))^{-1}$, and $(1+\exp(-\theta_1/2 -\theta_2))^{-1}$, for $a=1,2,3$, respectively. Let $n_k$ be the number of times that experiment $k$ is selected and $\pi_k = n_k/n$ be its frequency ($k=1,2,3$) with $n=\sum_{k}n_k$. Then, a necessary condition for the existence of a consistent estimator $\widehat{\boldsymbol{\theta}}_n$ is $\max\{\min(n_1,n_2),\min(n_1,n_3), \min(n_2,n_3)\}\to\infty$.
Moreover, in order to minimize the mean squared error $\mathbb{E}\|\widehat{{\bgtheta}}_n - {\bgtheta}\|^2$ asymptotically, a necessary condition is $(\pi_1,\pi_2,\pi_3)\to {\boldsymbol \pi}^*({\boldsymbol{\theta}})$ as the total sample size grows, where ${\boldsymbol \pi}^*({\boldsymbol{\theta}})$ is a vector-valued optimal proportion function depending on $\bgtheta$.
See Figure~\ref{fig:opt_prop} for an illustration of the function ${ {\bgpi}}^*({\boldsymbol{\theta}})$ and additional details in Section~\ref{sec:app-cat}.

\end{example}
In this example, achieving consistent or asymptotically optimal estimators requires experiments to be combined carefully with a parameter-dependent frequency.
However, information-based selection methods, being one-step-ahead greedy, do not consider the benefits of combining experiments or multi-step planning. Thus, it remains an open question whether these selection methods lead to consistent, asymptotically normal, or risk-optimal estimators.

In this study, we provide a definitive answer to the above question for a class of greedy-information-based experiment selection rules. In particular, we introduce two experiment selection rules based on a pre-specified criterion function $\mathbb{G}_{\bgtheta}:\mathbb{R}^{p\times p}\to \mathbb{R}$, 
\begin{equation}\label{eq:GI0}
 \textrm{GI0}: \quad   a_{n+1} = \arg\min_{a\in\mathcal{A}}\mathbb{G}_{\widehat{\bgtheta}_{n}^{\text{ML}}}\Big[\big\{   \mathcal{I}(\widehat{\bgtheta}_{n}^{\text{ML}};\va_n,a)\big\}^{-1} \Big], \text{ and}
\end{equation}
\begin{equation}\label{eq:GI1}
      \textrm{GI1}:  a_{n+1} = \arg\max_{a\in\mathcal{A}}  \operatorname{tr} \Big[  \nabla\mathbb{G}_{\widehat{\bgtheta}_{n}^{\text{ML}}}\big(   \widehat\bgSigma_n   \big) 
  \widehat\bgSigma_n \mathcal{I}_a(\widehat{\bgtheta}_{n}^{\text{ML}})\widehat\bgSigma_n \Big], 
\end{equation}

where $\widehat\bgSigma_n=\{\mathcal{I}(\widehat{\bgtheta}_{n}^{\text{ML}}; \va_n ) \}^{-1}$, $\nabla \mathbb{G}_{\bgtheta}(\bgSigma)=\Big(\frac{\partial \mathbb{G}_{\bgtheta}(\bgSigma)}{\partial \bgSigma_{ij}}\Big)_{1\leq i,j\leq p}$ denotes the gradient of $\mathbb{G}_{\bgtheta}$ with respect to its matrix input and recall $
\mathcal{I} (\widehat{\bgtheta}_{n}^{\text{ML}};\va_n) = \frac{1}{n} \sum_{i=1}^{n}\mathcal{I}_{a_i}(\widehat{\bgtheta}^{\text{ML}}_{n}) 
$.
We refer to the selection rule in \eqref{eq:GI0} as the zero-order greedy information-based selection rule (\textrm{GI0}), and that in \eqref{eq:GI1} as the first-order greedy information-based selection rule (\textrm{GI1}), because GI0 is minimizing a certain function of the Fisher information at the next time point, while GI1 is derived based on a first-order Taylor expansion of \textrm{GI0}; see Section~\ref{sec:method} for more details. 
\text{GI0} generalizes the selection rule in \eqref{eq:trace-selection}, accommodating more diverse settings. New methods can be obtained by specifying an appropriate function $\mathbb{G}_{\bgtheta}$. \textrm{GI1} offers a class of new experiment selection rules which share similar asymptotic properties as \textrm{GI0} but are computationally more efficient when both $p$ and $|\mathcal{A}|$ are large.

Our main theoretical contributions are as follows. First, we show that MLE is strongly consistent and asymptotically normal when using GI0 or GI1 as the experiment selection rule, under mild conditions. Second, we derive the asymptotic covariance matrix of the MLE as a function involving $\mathbb{G}_{\bgtheta}(\cdot)$ and the Fisher information. Third, we prove that the empirical frequency of selected experiments converges to a limiting frequency. Fourth, we show that the experiment selection rule GI0 (or GI1) combined with the MLE is asymptotically optimal in minimizing certain risk measures related to $\mathbb{G}_{\bgtheta}(\cdot)$. In particular, if $\mathbb{G}_{\bgtheta}(\cdot) = \operatorname{tr}(\cdot)$, then the MLE has the smallest asymptotic mean squared error (MSE), when compared with other experiment selection rules and estimators. Moreover, these results are valid not only for fixed sample sizes, but also for random stopping times, which is beneficial for applications that use early stopping criteria.

Beyond the methodological and theoretical contributions, we have developed new analytical tools for addressing technical challenges. For example, we show that the inverted Fisher information, through its directional derivatives in experiment selection rules, acts as a regularizer. This facilitates automatic exploration of necessary experiments, removing the need for additional exploration steps traditionally employed in stochastic control methods for related problems (e.g., two-stage design in sequential design for hypothesis tests \citep{chernoff1959,naghshvar2013active}). 
Furthermore, we show that a class of greedy and stochastic optimization methods converges to the minimum of a convex function over a simplex almost surely.
In addition, we refine and extend several classic results in stochastic analysis, such as Anscombe's theorem \citep{anscombe1952large} and the Robbins-Siegmund theorem \citep{robbins1971convergence}.
These theoretical results and technical tools are important in their own right and may be reused in other related problems. See Section~\ref{sec:proof-sketch} for more details of the technical challenges and our new analytical tools.

The rest of the paper is organized as follows. Section~\ref{sec:problem} formalizes the active sequential estimation problem. Section~\ref{sec:method} introduces the greedy information-based experiment selection rules GI0 and GI1, elaborating on their implementation. Section~\ref{sec:theory} offers the main theoretical results regarding the MLE and the experiment selection rules. Section~\ref{sec:application} details the methods and theory in applications including the item selection in CAT and sequential rank aggregation. Section~\ref{sec:proof-sketch} gives new analytical tools and a proof sketch. Section~\ref{sec:numerical} presents two simulation studies, which illustrate the finite sample performance and the computational efficiency of the proposed methods. Section~\ref{sec:real_data_example} showcases the performance of the proposed method on a real-data example. Section~\ref{sec:discussion} summarizes the main results and provides discussions on future directions.
All the technical proofs for the theoretical results and additional simulation results are given in the supplementary material.
\subsection{Notations}
In this paper, we use the following notations and mathematical conventions. Let \(\overline{C}\) and \(\underline{C}\) represent generic constants that are bounded from above and below, respectively. 
These generic constants are independent of \(\boldsymbol{\theta}\) and \(a \in \mathcal{A}\), and their values may vary from place to place. Let $|\mathcal{A}|$ denote the cardinality of a set $\mathcal{A}$. Let $I(\cdot)$ denote the indicator function. Let $I_p$ denote the $p\times p$ identity matrix. The inner product between real matrices (or vectors) $\blA$ and $\blB$ of the same size is defined by $\innerpoduct{ \blA }{\blB}=\mathrm{tr} (\blA^T\blB)$. For a real matrix $\blA$, define the operator norm $\norm{\blA}_{op}$ as the maximum singular value of $\blA$. For a vector $\blx$, denote its Euclidean norm by $\norm{\blx}=\sqrt{\innerpoduct{ \blx }{\blx}}$. For a symmetric matrix $\blA$, $\lambda_{max}(\blA)$, $\lambda_{min}(\blA)$, and $\kappa(\blA)$ denote its maximum eigenvalue, minimum eigenvalue, and condition number, respectively. If $\blA$ is a positive definite matrix, then $\kappa(\blA) = \frac{\lambda_{max}(\blA)}{\lambda_{min}(\blA)}$. For a differentiable matrix function $\mathbb{G}(\bgSigma)$, its gradient is denoted by $\nabla \mathbb{G}(\bgSigma)$, and is defined as the matrix such that
\(\mathbb{G}(\bgSigma+\Delta 
\bgSigma)- \mathbb{G}(\bgSigma )= \innerpoduct{\nabla \mathbb{G}(\bgSigma)}{\Delta 
\bgSigma}+o(\norm{\Delta 
\bgSigma})\). For symmetric matrices $\blA$ and $\blB$, define the partial order $\blA\preceq \blB$ %
if and only if $\blB-\blA$ is a positive semidefinite matrix. Throughout the paper, all the vectors are column vectors, unless otherwise specified.

\section{Problem Statement}\label{sec:problem}
Let $X_1,\cdots,X_n,\cdots$ be data collected sequentially, $\mathcal{A}$ be a finite set with cardinality $k$, and $a_1,\cdots,a_n,\cdots\in\mathcal{A}$ be the experiments selected at different time points. Denote by $\mathcal{F}_n = \sigma(a_1,X_1,\cdots, a_n,X_n)$, the sigma field that contains information of the observations and the selected experiments up to time $n$. At each time $n$, a decision maker needs to select the experiment $a_{n+1}$ adaptively based on past information. That is, $a_{n+1}$ is measurable with respect to $\mathcal{F}_n$. Throughout the study, we assume that the distribution of $X_{n+1}$ satisfies
$$
X_{n+1}|\mathcal{F}_n \sim f_{\bgtheta,a_{n+1}}(\cdot) \text{ for } \bgtheta\in\bgTheta \subset \mathbb{R}^p
$$
where $\bgtheta$ is a $p$-dimensional model parameter, $\bgTheta$ is a compact parameter space and $f_{\bgtheta,a_{n+1}}(\cdot)$ denotes the probability density of $X_{n+1}$ with respect to a baseline measure. That is,  $X_{n+1}$ is assumed to follow a parametric model, and its distribution is determined by both the underlying model parameter $\bgtheta$ and the selected experiment $a_{n+1}$.

In an active sequential estimation problem, the goal is to design an experiment selection rule for $\{a_n\}_{n\geq 1}$ and find an estimator $\widehat{\bgtheta}_n$ that is measurable with respect to $\mathcal{F}_n$, so that  $\widehat{\bgtheta}_n$ is close to the true underlying parameter $\bgtheta^*$ with high probability. In some applications, the data collection process may be stopped early to save for the sampling cost. In these cases, we are also interested in $\widehat{\bgtheta}_N$, where $N$ is a random stopping time.

\section{Methods}\label{sec:method}
For the estimation method, we focus on the MLE, although some of the methods and theoretical results may be extended to other estimators. The definition of MLE is given as follows. Let the selected experiments up to time $n$ be $\bla_n=(a_1,\cdots,a_n)$. Then, the rescaled log-likelihood and the corresponding MLE are
\begin{equation}\label{equ:log-like}
    l_n(\bgtheta)=l_n(\bgtheta;\va_n)=\frac{1}{n}\sum_{i=1}^n \log f_{{\bgtheta}, a_i }(X_i),\text{ and}
\end{equation}
\begin{equation}\label{def:MLE}
\widehat{\bgtheta}_{n}^{\text{ML}}\in\arg\max_{\bgtheta \in \bgTheta} l_n(\bgtheta;\va_n).
\end{equation}
We propose adopting two experimental selection rules, including the zero-order greedy information-based selection rule \textrm{GI0} and the first-order greedy information-based selection rule \textrm{GI1}. The precise description of these methods are given in Algorithms~\ref{alg:gi0} and \ref{alg:gi1}.

\begin{algorithm}[ht]
\caption{\textrm{GI0} Algorithm}
\label{alg:gi0}
\begin{algorithmic}[1]
\STATE {\textbf{Input:} $\widehat{\bgtheta}_{0}$, $a_1^0, \cdots, a_{n_0}^0$.}
\STATE {\textbf{Require:} $\widehat{\bgtheta}_{0}\in \bgTheta$, $a_1^0, \cdots, a_{n_0}^0\in \mathcal{A}$ such that $ \sum_{i=1}^{n_0}\mathcal{I}_{a^0_{i}}(\widehat{\bgtheta}_{0})$ is nonsingular. }
\STATE{\textbf{Initialization}: $a_1=a_1^0, \cdots, a_{n_0}=a_{n_0}^0$, collecting responses $X_1,X_2,\cdots,X_{n_0}$ correspondingly.}

\FOR{$n = n_0$ to $N$} %
    \STATE{calculating the MLE $\widehat{\bgtheta}_{n }^{\text{ML}}$ according to equation \eqref{def:MLE}}
    \STATE{selecting experiment $a_{n+1}$ according to equation \eqref{eq:GI0}}
    \STATE{collecting response $X_{n+1}$ corresponding to the selected experiment $a_{n+1}$} 
\ENDFOR %
\STATE{\textbf{Output:} $\widehat{\bgtheta}_{N }^{\text{ML}}$}
\end{algorithmic}
\end{algorithm}

\begin{algorithm}[ht]
\caption{\textrm{GI1} Algorithm}\label{alg:gi1}
\begin{algorithmic}[1]
\STATE {\textbf{Input:} $\widehat{\bgtheta}_{0}$, $a_1^0, \cdots, a_{n_0}^0$.}
\STATE {\textbf{Require:} $\widehat{\bgtheta}_{0}\in \bgTheta$, $a_1^0, \cdots, a_{n_0}^0\in \mathcal{A}$ such that $ \sum_{i=1}^{n_0}\mathcal{I}_{a^0_{i}}(\widehat{\bgtheta}_{0})$ is nonsingular. }
\STATE{\textbf{Initialization}: $a_1=a_1^0, \cdots, a_{n_0}=a_{n_0}^0$, collecting responses $X_1,X_2,\cdots,X_{n_0}$ correspondingly.}

\FOR{$n = n_0$ to $N$} %
    \STATE{calculating the MLE $\widehat{\bgtheta}_{n }^{\text{ML}}$ according to equation \eqref{def:MLE}}
    \STATE{selecting experiment $a_{n+1}$ according to equation \eqref{eq:GI1}}
    \STATE{collecting response $X_{n+1}$ corresponding to the selected experiment $a_{n+1}$} 
\ENDFOR %
\STATE \textbf{Output}:$\widehat{\bgtheta}_{N }^{\text{ML}}$  
\end{algorithmic}
\end{algorithm}

We explain steps in Algorithms~\ref{alg:gi0} and \ref{alg:gi1}. First, we note that both algorithms require a pre-specified criterion function $\mathbb{G}_{\bgtheta}:\mathbb{R}^{p\times p}\to\mathbb{R}$. Motivated by \cite{kiefer1974general} on the design of experiments, a reasonable choice is 
\begin{equation}\label{def:Phi_q}
\mathbb{G}_{\bgtheta}(\bgSigma)=\Phi_q( \bgSigma )= \left\{\begin{array}{ll}
\log\det(\bgSigma), & \text { if } q=0; \\
\tr ( \bgSigma^{q}), & \text { if } 0<q<1;\\
(\tr ( \bgSigma^{q}) )^{1/q}, & \text { if } q\geq 1,
\end{array}\right.
\end{equation}
for a prespecified $q\geq 0$. %
In the context of adaptive item selection in CAT, \textrm{GI0} with $\mathbb{G}_{\bgtheta}(\bgSigma)=\Phi_q( \bgSigma )$ has been adopted in
\cite{van1999multidimensional} and \cite{wang2011item}. In particular, the selection rule in \eqref{eq:trace-selection} corresponds to \textrm{GI0} with $\mathbb{G}_{\bgtheta}(\cdot)=\Phi_1(\cdot)$. Both \textrm{GI0} and \textrm{GI1} are relatively new in other applications described in Section~\ref{sec:intro}. Note that for $\mathbb{G}_{\bgtheta}(\bgSigma)=\Phi_q( \bgSigma )$, it is a function independent with the input ${\bgtheta}$. Another option is \(\mathbb{G}_{\bgtheta}(\bgSigma) = \operatorname{tr}(\blH_{\bgtheta} \bgSigma )\), where $\blH_{\bgtheta}$ is a positive definite matrix depending on $\bgtheta$. This criterion function is useful in the cases where we would like to assign different weights to different values of ${\bgtheta}$. For more details, please refer to Theorem~\ref{thm:ultimate}. 

Second, both algorithms require an initialization step where $n_0$ experiments are selected so that the Fisher information matrix $\mathcal{I}(\widehat{\bgtheta}_{0};\va_{n_0})$ is nonsingular. This initialization step  ensures that  $\mathcal{I}(\widehat{\bgtheta}_{n};\va_{n})$ is nonsingular and the experiment selection rules in \eqref{eq:GI0} and \eqref{eq:GI1} are well-defined for all $n\geq n_0$. In practice, it is usually straightforward to find such $a_1^0, \cdots, a_{n_0}^0$. For instance, in Example~\ref{eg:toy-eg}, we could choose $\widehat{\bgtheta}_0= (0,0)$,  $(a_1^0,  a_{2}^0)=(1,2)$, and $n_0=2$. Then, at each time point, the algorithm first calculates the MLE based on the available information, selects a new experiment according to \eqref{eq:GI0} for \textrm{GI0} (or \eqref{eq:GI1} for \textrm{GI1}), and then samples a new observation according to the selected experiment. 

We refer to the selection rule in Algorithm~\ref{alg:gi0} as \textrm{GI0} and  that in Algorithm~\ref{alg:gi1} as \textrm{GI1}, because \textrm{GI0} tries to minimize the criterion function $\mathbb{G}_{\widehat{\bgtheta}_{n}}\Big[\big\{   \mathcal{I}(\widehat{\bgtheta}_{n}^{\text{ML}};\va_n,a)\}^{-1} \Big]$ for one-step ahead, while \textrm{GI1} tries to minimize its first-order approximation, i.e.,
{ 
\begin{equation}\label{eq:first-approximation}
\begin{split}  &\mathbb{G}_{\widehat{\bgtheta}_{n}^{\text{ML}}}\Big[\big\{   \mathcal{I}( \widehat{\bgtheta}_{n}^{\text{ML}};\va_n,a) \}^{-1} \Big] - \mathbb{G}_{\widehat{\bgtheta}_{n}^{\text{ML}}}\Big[\big\{   \mathcal{I}(\widehat{\bgtheta}_{n}^{\text{ML}};\va_n)\}^{-1}\Big]\\
    \approx & { \biginnerpoduct{\overline{{\bgpi}}^a_{n+1}-\overline{{\bgpi}}_{n}}{\frac{\partial}{\partial \bgpi}   \mathbb{G}_{\widehat{\bgtheta}_{n}^{\text{ML}}}\big[\big\{\sum_{a \in \mathcal{A} } \pi(a) \mathcal{I}_a( {\widehat{\bgtheta}_{n}^{\text{ML}}}  ) \big\}^{-1}\big]\Big|_{\bgpi = \overline{{\bgpi}}_{n}}} } 
\\
    = &-\frac{1}{n+1}\operatorname{tr} \left[  \nabla\mathbb{G}_{\widehat{\bgtheta}_{n}^{\text{ML}}} \Big(\{   \mathcal{I}(\widehat{\bgtheta}_{n}^{\text{ML}};  \va_n ) \}^{-1}   \Big) 
  \{\mathcal{I}(\widehat{\bgtheta}_{n}^{\text{ML}}; \va_n)\}^{-1} \mathcal{I}_a(\widehat{\bgtheta}_{n}^{\text{ML}})\{\mathcal{I} (\widehat{\bgtheta}_{n}^{\text{ML}}; \va_n)\}^{-1}  \right]\\
  &+\frac{1}{n+1}\operatorname{tr} \left[  \nabla\mathbb{G}_{\widehat{\bgtheta}_{n}^{\text{ML}}} \Big(\{   \mathcal{I}(\widehat{\bgtheta}_{n}^{\text{ML}};  \va_n ) \}^{-1}   \Big) 
  \{\mathcal{I}(\widehat{\bgtheta}_{n}^{\text{ML}}; \va_n)\}^{-1}   \right],
\end{split}
\end{equation}
where %
the empirical frequency %
vector $\overline{{\bgpi}}_n $ is defined as
\begin{equation}\label{def:pi_n}
    \overline{{\bgpi}}_n = \overline{{\bgpi}}_n[\va_n] = \Big(\frac{1}{n}|\{ i; a_i=a,1\leq i\leq n\} |\Big)_{a\in\mathcal{A} },
\end{equation}
with $\va_n=(a_1,\cdots,a_n)$ collects experiments selected up to time $n$, 
and $\overline{{\bgpi}}^a_{n+1}(a')=\frac{n}{n+1} \overline{{\bgpi}}_{n}(a')+\frac{1}{n+1} I(a=a')$ is the empirical frequency at the time $n+1$ if $a'$ is selected at that time.
}
Note that \textrm{GI0} minimizes the first line of \eqref{eq:first-approximation}, \textrm{GI1} minimizes the first term on the last equation of \eqref{eq:first-approximation}, and the second term  on the last equation of \eqref{eq:first-approximation} does not depend on the choice of experiment $a$. This suggests that \textrm{GI0} and \textrm{GI1} are asymptotically equivalent, although the rigorous theoretical justification is much more involved.
\subsection{Improving Computational Efficiency}

If $k,p$ are large, and $\mathcal{I}_a(\bgtheta)$ has some low-dimensional representation, \textrm{GI1} can be implemented with improved numerical efficiency. In particular, we consider two specific cases which are commonly seen in applications, including (1) low-rank information: $    \mathcal{I}_a(\bgtheta) = L_a(\bgtheta)L^T_a(\bgtheta)$ where $L_a(\bgtheta)\in \mathbb{R}^{p\times s}$ for all $a$ and $\bgtheta$ and $s<p$;  (2) sparse and low-rank information: $L_a(\bgtheta)$ has no more than $s$ non-zero rows. For these cases,   Algorithm~\ref{alg:gi1} can be implemented using the following accelerated version.
\begin{algorithm}[H]
\caption{Accelerated \textrm{GI1} Algorithm}
\label{alg:gi1-Accelerated}
\begin{algorithmic}
\STATE \quad We modify line 6 in Algorithms~\ref{alg:gi1}, while keeping the other lines of the algorithms unchanged.
    \STATE{6:\quad selecting experiment $a_{n+1}$ according to }
\begin{equation*}
\begin{split}
    &\blM= \{\mathcal{I} (\widehat{\bgtheta}_{n}^{\text{ML}}; \va_n)\}^{-1} \nabla\mathbb{G}_{\widehat{\bgtheta}_{n}^{\text{ML}}}\big(\{   \mathcal{I}(\widehat{\bgtheta}_{n}^{\text{ML}}; \va_n ) \}^{-1}   \big) 
    \{\mathcal{I} (\widehat{\bgtheta}_{n}^{\text{ML}}; \va_n)\}^{-1}, \\
    &a_{n+1}     = \arg\max_{a\in \mathcal{A}}\tr\left[ L^T_a(\widehat{\bgtheta}_{n}^{\text{ML}}) \blM
   L_a(\widehat{\bgtheta}_{n}^{\text{ML}})\right].
\end{split}
\end{equation*}   
\end{algorithmic}
\end{algorithm}
\begin{lemma}\label{lem:complexity}
Assume the computational complexity of evaluating $\mathbb{G}_{ \bgtheta }(\bgSigma)$ and $\nabla \mathbb{G}_{ \bgtheta }(\bgSigma)$ is no more than $O(p^3)$. Given the MLE \(\widehat{\bgtheta}_n^{\text{ML}}\) and \(\mathcal{I}(\widehat{\bgtheta}_n^{\text{ML}};\va_n)\), we have
\begin{enumerate}
    \item the computational complexity for each iteration in GI0 is of the order \(O(kp^3)\);
\item the computational complexity for each iteration in  the accelerated \textrm{GI1} Algorithm~\ref{alg:gi1-Accelerated} is $O(ksp^2 + p^3)$, assuming that the information matrices are low-rank matrices with given $L_a(\bgtheta)\in\mathbb{R}^{p\times s}$. Moreover, the computational cost for the accelerated \textrm{GI1} Algorithm~\ref{alg:gi1-Accelerated} becomes \(O(ks^2p + p^3)\) if $L_a(\bgtheta)$ has no more than $s$ non-zero rows.
\end{enumerate}

\end{lemma}
According to the above lemma, the accelerated \textrm{GI1} algorithm is computationally much more efficient than \textrm{GI0}, when  $k$ and $p$ are large and $s$ is small. Numerical results supporting these findings can be found in Section~\ref{sec:simulation_ra}.

\subsection{Early Stopping}\label{sec:method-stopping}
In many applications, the data collection process is stopped early when sufficient observations have been gathered to make accurate statistical inference. For instance, in the context of CAT, educational tests often have variable lengths determined by specific early stopping rules. These rules generally lead to less fatigue and a better experience for examinees. In this section, we introduce two early stopping rules suitable for active sequential estimation. 

The first stopping rule $\tau^{(1)}_{c}$ is concerned with the estimation of a differentiable function of the parameter $h(\bgtheta) \in\mathbb{R}$, and it is defined as
\begin{equation}\label{equ:early stopping g}
\begin{split}
    &\tau^{(1)}_{c}=\min \big\{m\geq n_0;  \widehat{\text{SE}}(h( \widehat{\bgtheta}_{ m}^{\text{ML}} )) \leq c \big\},
\end{split}
\end{equation}
where $\widehat{\text{SE}}^2(h( \widehat{\bgtheta}_{ m}^{\text{ML}} ))=  \frac{1}{m} \{\nabla h( \widehat{\bgtheta}_{ m}^{\text{ML}} )\}^T \{\mathcal{I} (\widehat{\bgtheta}_{m}^{\text{ML}}; \va_m)\}^{-1} \nabla h( \widehat{\bgtheta}_{ m}^{\text{ML}} )$.
The second stopping rule $ \tau^{(2)}_{c}$ is concerned with the estimation of the vector $\bgtheta$, and is defined as
\begin{equation}\label{equ:early stopping TV}
    \tau^{(2)}_{c}=\min \Big\{m\geq n_0; \widehat{\text{MSE}}(\widehat{\bgtheta}_{m}^{\text{ML}}) \leq c  \Big\},\ \text{ where }\widehat{\text{MSE}}(\widehat{\bgtheta}_{m}^{\text{ML}})=\frac{1}{m} \tr \Big(  \{\mathcal{I} (\widehat{\bgtheta}_{m}^{\text{ML}}; \va_m)\}^{-1}   \Big).
\end{equation}
Here, $\widehat{\text{SE}}$ serves as an approximation of $\text{sd}(h( \widehat{\bgtheta}_{m}^{\text{ML}}))$ and $\widehat{\text{MSE}}(\widehat{\bgtheta}_{m}^{\text{ML}})$ serves as an approximation of $\text{MSE}( \widehat{\bgtheta}_{m}^{\text{ML}})=\mathbb{E}_{\bgtheta^*}\|\widehat{\bgtheta}_{m}^{\text{ML}}-\bgtheta^*\|^2$. Both rules terminate the data collection process once a certain error estimator falls below a predetermined threshold $c$.

\section{Theoretical Results}\label{sec:theory}
In this section, we first introduce the regularity conditions, and then present the main theoretical results regarding the consistency, asymptotic normality, and the optimality of the proposed method.

\subsection{Regularity Conditions}\label{sec:assumptions}
Throughout Section~\ref{sec:theory}, we
make the following Assumptions \ref{ass:1}--\ref{ass:5}, along with Assumptions \ref{ass:6A} and \ref{ass:7A}, and we will refer to this set of assumptions as the `regularity conditions'. All the theoretical results still hold when \ref{ass:6A} and \ref{ass:7A} are replaced with the more relaxed Assumptions \ref{ass:6B} and \ref{ass:7B}.

\begin{assumption}\label{ass:1}
The parameter space $\bgTheta$ is a non-empty compact and convex subset of $\mathbb{R}^p$. The true parameter $\bgtheta^*$ is an interior point of $\bgTheta$.
\end{assumption}
\begin{assumption}\label{ass:2}
    The support of the probability density $f_{\bgtheta,a}$, denoted as $\operatorname{supp}(f_{\bgtheta,a})$, depends only on $a$ and does not depend on $\bgtheta$,  where the support of a function is defined as
\begin{equation*}
\operatorname{supp}(f_{\bgtheta,a})=\operatorname{cl}\{ x^a; f_{\bgtheta,a}(x^a)>0 \}, %
\end{equation*}
and $\operatorname{cl}(S)$ denotes the closure of a set $S$. Moreover, for all $a\in \mathcal{A}$ and $X^a\in \operatorname{supp}(f_{\bgtheta,a})$, the gradient $\nabla_{\bgtheta}\log f_{\bgtheta,a}(X^a)=(\frac{\partial \log f_{\bgtheta,a}(X^a)}{\partial \theta_i})_{1\leq i \leq p}$ and the Hessian matrix $\nabla^2_{\bgtheta}\log f_{\bgtheta,a}(X^a)=(\frac{\partial^2 \log f_{\bgtheta,a}(X^a)}{\partial \theta_i\partial \theta_j})_{1\leq i,j \leq p}$ exist, where $\bgtheta=(\theta_1,\cdots,\theta_p)^T.$ Assume that there exist functions $\Psi_1^a$ and $\Psi_2^a$ satisfying $\sup_{\bgtheta\in \bgTheta }\mathbb{E}_{X^a\sim f_{\bgtheta,a}}\{\Psi_1^a(X^a)\}^2<\infty$, $\sup_{\bgtheta\in \bgTheta }\mathbb{E}_{X^a\sim f_{\bgtheta,a}}\Psi_2^a(X^a)<\infty, $
\begin{equation}
\norm{\nabla_{\bgtheta}\log f_{\bgtheta_1,a}(X^a)-\nabla_{\bgtheta}\log f_{\bgtheta_2,a}(X^a)  }\leq \Psi_1^a(X^a)  \norm{\bgtheta_1-\bgtheta_2},\text{ and}
\end{equation}
\begin{equation}\label{equ:hessian_lips}
\norm{\nabla^2_{\bgtheta}\log f_{\bgtheta_1,a}(X^a)-\nabla^2_{\bgtheta}\log f_{\bgtheta_2,a}(X^a)  }_{op}\leq \Psi_2^a(X^a)  \norm{\bgtheta_1-\bgtheta_2},
\end{equation}
for all $\bgtheta_1,\bgtheta_2\in \bgTheta$ and $a\in\mathcal{A}$. Furthermore, for all $a\in \mathcal{A}$,
\begin{equation*}
\sup_{\bgtheta\in \bgTheta} \mathbb{E}_{X\sim f_{\bgtheta^*, a}} \{ \norm{\nabla_{\bgtheta} \log f_{\bgtheta, a} (X)}^2 \}<\infty \text{ and } \sup_{\bgtheta\in \bgTheta} \mathbb{E}_{X\sim f_{\bgtheta^*, a}} \{ \norm{\nabla^2_{\bgtheta}\log f_{\bgtheta, a} (X)}_{op} \}<\infty.
\end{equation*}
\end{assumption}
\begin{assumption}\label{ass:3}
The Fisher information matrices satisfy the following conditions:
\begin{equation*}
\mathcal{I}_a( {\bgtheta} )= \mathbb{E}_{X\sim f_{\bgtheta, a}} \left[ \nabla_{\bgtheta} \log f_{\bgtheta, a}(X) \{\nabla_{\bgtheta} \log f_{\bgtheta, a}(X)\}^T \right] = -\mathbb{E}_{X\sim f_{\bgtheta, a}} \left\{ \nabla^2_{\bgtheta} \log f_{\bgtheta, a}(X) \right\},
\end{equation*}
and those Fisher information matrices are continuously differentiable with respect to $\bgtheta$ for all $a\in \mathcal{A}$. Furthermore, $\sum_{a\in \mathcal{A}}\mathcal{I}_a(\bgtheta)$ is positive definite for every $\bgtheta\in \bgTheta$.
\end{assumption}
\begin{assumption}\label{ass:4}
Let $M(\bgtheta;{\bgpi})=\sum_{a\in \mathcal{A}} \pi(a)\mathbb{E}_{X\sim f_{\bgtheta^*, a}} \{\log f_{\bgtheta,a}(X)\}$ for $\bgpi=(\pi(a))_{a\in\mathcal{A}}$. Assume the following uniform law of large numbers holds {for all sequence $\va_n=(a_1,\cdots,a_n)$} such that $a_i$ is measurable with respect to $\mathcal{F}_{i-1}$, for all $1\leq i\leq n$:
\begin{equation}\label{eq:uslln}
\mathbb{P} \left\{
\lim_{n\to \infty} \sup_{\bgtheta \in \bgTheta} 
|l_n(\bgtheta;\va_n)-{ M(\bgtheta;\overline{\bgpi}_n[\va_n])}|  = 0
\right\} = 1,  
\end{equation}
where {$\overline{{\bgpi}}_n[\va_n]=(\overline{\pi}_n(a; \va_n))_{a\in \mathcal{A}}$, and $\overline{\pi}_n(a;\va_n)=\frac{1}{n}|\{ i; a_i=a,1\leq i\leq n\} |$ denotes the empirical frequency that the experiment $a$ is selected up to time $n$.}
\end{assumption}

\begin{assumption}\label{ass:5}
    The criterion function $\mathbb{G}_{\bgtheta}$ takes one of the following forms:
\begin{enumerate}
\item $\mathbb{G}_{\bgtheta}(\cdot)=\Phi_q(\cdot)$ for some $q\geq 0$ where $\Phi_q(\cdot)$ is defined in \eqref{def:Phi_q}, or
\item the function $\mathbb{G}_{\bgtheta}(\cdot): \mathbb{R}^{p\times p}\mapsto \mathbb{R}$ is convex, and it satisfies: for all positive definite matrix $\bgSigma$, $\nabla_{\bgtheta }\nabla\mathbb{G}_{\bgtheta}(\bgSigma)$ and $\nabla^2 \mathbb{G}_{\bgtheta}(\bgSigma)$ are continuous in $(\bgtheta,\bgSigma)$; and for all positive definite matrices satisfying $\blA \succeq \blB$, we have $\mathbb{G}_{\bgtheta}(\blA)\geq \mathbb{G}_{\bgtheta}(\blB)$. Additionally, $\sup_{\blA}\kappa(\nabla \mathbb{G}_{\bgtheta}(\blA))<\infty$ and 
$\lim_{ \lambda_{\max}(\blA) \to \infty }\inf_{\bgtheta \in \bgTheta} \mathbb{G}_{\bgtheta}(\blA)=\infty$.

\end{enumerate}
\end{assumption}
\edef\oldassumptionOne{\the\numexpr\value{assumption}+1}
\setcounter{assumption}{0}
\renewcommand{\theassumption}{\oldassumptionOne\Alph{assumption}}
\begin{assumption}[Reparametrization]\label{ass:6A}
There exist matrices $\{\bm{Z}_a \}_{a\in \mathcal{A}}$ and probability density functions $\{h_{\bm{Z}_a\bgtheta,a}(\cdot) \}_{a\in \mathcal{A}}$ satisfying the following requirements 
\begin{enumerate}
    \item $\bm{Z}_a$ is a matrix of dimension $p_a\times p$ with rank $p_a$ and $f_{\bgtheta,a}(\cdot)=h_{\bm{Z}_a \bgtheta,a}(\cdot)$ for all $a\in\mathcal{A}$. 
    \item Let $\bgxi_a=\bm{Z}_a \bgtheta$ be a reparametrization of $\bgtheta$. Assume that the Fisher information matrix of each experiment $a$ is nonsingular with respect to $\bgxi_a$.
    That is, the compressed Fisher information matrix
\begin{equation*}
\begin{split}
    \mathcal{I}_{\bgxi_a,a}( {\bgxi_a} )&= \mathbb{E}_{X\sim h_{\bgxi_a, a}} \left[ \nabla_{\bgxi_a} \log h_{\bgxi_a, a}(X) \{\nabla_{\bgxi_a} \log h_{\bgxi_a, a}(X)\}^T \right]\\
    &= -\mathbb{E}_{X\sim h_{\bgxi_a, a}} \left\{ \nabla^2_{\bgxi_a } \log h_{\bgxi_a, a}(X) \right\}
\end{split}
\end{equation*}
is nonsingular for all $\bgtheta \in \bgTheta$. 
\end{enumerate}

\end{assumption}
\let\theassumption\origtheassumption
\setcounter{assumption}{6}
\edef\oldassumptionTwo{\the\numexpr\value{assumption}+1}
\setcounter{assumption}{0}
\renewcommand{\theassumption}{\oldassumptionTwo\Alph{assumption}}
\begin{assumption}[Identifiability]\label{ass:7A}
There exists a constant $C>0$ such that for all $\bgtheta\in \bgTheta$, 
\begin{equation}
    D_{\mathrm{KL}}(h_{\bgxi^*_{a},a}\| h_{\bgxi_{a},a} ) \geq  C \norm{ \bgxi^*_{a} - \bgxi_{a} }^2,
\end{equation}
where $\bgxi_a^*=\bm{Z}_a\bgtheta^*$ is the compressed parameter after reparametrization, and $D_{\mathrm{KL}}(h_{\bgxi^*_{a},a}\|h_{\bgxi_{a},a})$ denotes the Kullback–Leibler divergence between the density functions $h_{\bgxi^*_{a},a}$ and $h_{\bgxi_{a},a}$, and is defined as \(D_{\mathrm{KL}}(h_{\bgxi^*_{a},a} \| h_{\bgxi_{a},a}) = \mathbb{E}_{X\sim h_{\bgxi^*_{a},a}} \log\left(\frac{h_{\bgxi^*_{a},a}(X)}{h_{\bgxi_{a},a}(X)}\right).\)

\end{assumption}
\let\theassumption\origtheassumption

We comment on the above regularity conditions. Assumptions \ref{ass:1}, \ref{ass:2}, \ref{ass:3} and \ref{ass:4} are extensions of standard regularity conditions for the consistency of the MLE based on independent and identically distributed (i.i.d.) observations (see, e.g., Chapter 5 of  \cite{van2000asymptotic}). In particular, Assumption \ref{ass:1} ensures the existence of MLE. Assumption \ref{ass:2} requires that the gradient of log-density function associated with each experiment is stochastic Lipschitz and has a bounded second moment. Condition \eqref{equ:hessian_lips} can be replaced by a more relaxed condition:
\begin{equation}\label{cond:relaxed}
\norm{\nabla^2_{\bgtheta}\log f_{\bgtheta_1,a}(X^a)-\nabla^2_{\bgtheta}\log f_{\bgtheta_2,a}(X^a)  }_{op}\leq \Psi_2^a(X^a)  \psi(\norm{\bgtheta_1-\bgtheta_2}),
\end{equation}
where $\psi:[0,\infty)\to [0,\infty)$ is a strictly increasing continuous function such that $\psi(0)=0$.  Assumption \ref{ass:3} requires that the Fisher information matrices are well-behaved. Under this assumption,  each Fisher information matrix $\mathcal{I}_a(\bgtheta)$ may be singular, but their sum is nonsingular. In other words, if we combine all the experiments together, the Fisher information matrix is nonsingular. Assumption \ref{ass:4} requires that the log-likelihood follows the uniform law of large numbers. This assumption can be verified by uniform martingale laws of large numbers (see \cite{rakhlin2015sequential}) in most applications.
Assumption \ref{ass:5} describes the requirement on the criterion function $\mathbb{G}_{\bgtheta}(\cdot)$. Assumptions \ref{ass:6A} and \ref{ass:7A} require that for each experiment $a$, we can reparameterize the model with a new parameter $\bgxi_a$ with possibly lower dimension $p_a\leq p$ such that $\bgxi_a$ is locally identifiable around the true model parameter, and the Fisher information matrix with respect to $\bgxi_a$ is nonsingular. Note that Fisher information with respect to $\bgtheta$ may be singular in this case. 

All the regularity assumptions are easily satisfied in practical problems, including the item selection in CAT and the sequential rank aggregation problem described in Section~\ref{sec:intro}. See Section~\ref{sec:application} for detailed justifications of the assumptions in these applications. Note that \ref{ass:6A} and \ref{ass:7A} can be relaxed to a more general condition, allowing for non-linear model reparameterization. These relaxed conditions are provided below.

\edef\oldassumptionThree{\the\numexpr\value{assumption}+5}
\setcounter{assumption}{1}
\renewcommand{\theassumption}{\oldassumptionThree\Alph{assumption}}
\begin{assumption}\label{ass:6B}
For $Q\subset \mathcal{A}$, define a vector space $V_Q=V_Q(\bgtheta)=\sum_{a\in Q} \mathcal{R}(\mathcal{I}_a(\bgtheta) )$, where $\mathcal{R}(\blA)$ represents the column space of a matrix $\blA$. Assume that the dimension $\operatorname{dim}(V_Q(\bgtheta))$ does not depend on $\bgtheta$, and there exist constants $0<\underline{c}\leq \overline{c}<\infty$, which do not depend on $Q$ and $\bgtheta$, such that for all $Q \subset \mathcal{A}$ and $\bgtheta \in \bgTheta$
\begin{equation}\label{ineq:double c bound}
    \underline{c} \cdot \blP_{V_Q(\bgtheta)} \preceq \sum_{a\in Q} \mathcal{I}_a(\bgtheta) \preceq \overline{c} \cdot \blP_{V_Q(\bgtheta)},
\end{equation}
    where $\blP_{V_Q(\bgtheta)}$ denotes the orthogonal projection matrix onto vector space $V_Q(\bgtheta)$.
\end{assumption}

\let\theassumption\origtheassumption

\edef\oldassumption{\the\numexpr\value{assumption}+5}

\setcounter{assumption}{1}

\renewcommand{\theassumption}{\oldassumption\Alph{assumption}}
\begin{assumption}\label{ass:7B}
Let $\ShatA=\{ {\bgpi}=(\pi(a))_{a\in\mathcal{A}}:  \sum_{a\in \mathcal{A}}\pi(a)=1 \text{ and } \pi(a)\geq 0 \text{ for all } a\in \mathcal{A}\}$ denote the simplex in $\mathbb{R}^{\mathcal{A}}$. Assume that there exists a positive constant $C$ such that for all ${\bgpi} \in \ShatA $ and $\bgtheta\in \bgTheta$, 
\begin{equation}\label{ineq:piD_KL}
     \sum_{a \in \mathcal{A}}\pi(a)D_{\mathrm{KL}}(f_{\bgtheta^*,a}\|f_{\bgtheta,a} )\geq  {C}   \sum_{a\in \mathcal{A}}\pi(a)(\bgtheta-\bgtheta^*)^T \mathcal{I}_a(\bgtheta^*)(\bgtheta-\bgtheta^*),
\end{equation}
where $D_{\mathrm{KL}}(f_{\bgtheta^*,a}\|f_{\bgtheta,a} )$ is the Kullback–Leibler divergence between the density functions $f_{\bgtheta^*,a}$ and $f_{\bgtheta,a}$.
\end{assumption}

\subsection{Main Theoretical Results}\label{sec:main-theory}
In this section, we present the main theoretical results, including the consistency, asymptotic normality and the optimality of the proposed method. Recall that the regularity conditions (Assumptions \ref{ass:1} -- \ref{ass:5}, along with Assumptions \ref{ass:6A} -- \ref{ass:7A} or \ref{ass:6B} -- \ref{ass:7B}) are assumed  throughout the section. 
\subsubsection{Strong Consistency}
We start with the strong consistency of the MLE following \textrm{GI0} or \textrm{GI1}.
\begin{theorem}[Strong consistency]\label{thm:consistency_final}
 Let $\widehat{\bgtheta}_{n}^{\text{ML}}$ be the MLE following the experiment selection rule \textrm{GI0} or \textrm{GI1}, as described in Algorithm~\ref{alg:gi0} and Algorithm~\ref{alg:gi1}. Then,
$$
\lim_{n\to\infty} \widehat{\bgtheta}_{n}^{\text{ML}} = \bgtheta^* \text{ a.s. } \mathbb{P}_*,
$$
where $\mathbb{P}_*$ denotes the  data-generating probability distribution under the true model parameter  $\bgtheta^*$.

\end{theorem}
Theorem \ref{thm:consistency_final} suggests that the MLE will be close to the true model parameter with a large sample size following \textrm{GI0} or \textrm{GI1}.

\subsubsection{Limiting Selection Frequency and Asymptotic Normality of MLE}
Let
$$
\mathcal{I}^{{\bgpi}}(\bgtheta)=\sum_{a\in \mathcal{A}} \pi(a) \mathcal{I}_a( {\bgtheta}),
$$
be the weighted Fisher information associated with a proportion vector ${\bgpi}\in\ShatA$. {The  distribution of MLE depends on the empirical frequency vector $\overline{{\bgpi}}_n$, which is defined by \eqref{def:pi_n}.
}

We first present an auxiliary asymptotic normality result for the MLE following a general active experiment selection rule that is not necessarily \textrm{GI0} or \textrm{GI1}.
\begin{theorem}[Asymptotic normality following general experiment selection rules]\label{thm:GI0-GI1AN}
Let $\widehat\bgtheta^{ML}_n$ be the MLE calculated according to \eqref{def:MLE} following an active experiment selection rule that is not necessarily \textrm{GI0} or \textrm{GI1}. {Let $\overline{{\bgpi}}_n$ be the corresponding empirical frequency vector.}

Assume that there exists ${\bgpi}\in \ShatA$ such that $\overline{{\bgpi}}_n$ converges to ${\bgpi}$ in probability $\mathbb{P}_*$ as $n\to\infty$, and $\mathcal{I}^{{\bgpi}}(\bgtheta^*)$ is nonsingular. Then, 
     \begin{equation}
         \sqrt{n} (\widehat{\bgtheta}_n^{\text{ML}} - \bgtheta^*) \inD N_p\Big(\bm{0}_p,\big\{\mathcal{I}^{{\bgpi}}(\bgtheta^*)\big\}^{-1} \Big) \text{ as } n\to\infty,
     \end{equation}
     where `$\inD$' denotes the convergence in distribution. 
\end{theorem}

The above Theorem~\ref{thm:GI0-GI1AN} extends the classic asymptotic normality results for MLE to the sequential setting with active experiment selection. It roughly states that if the frequency of the selected experiment approximates a limiting proportion as the sample size grows, and the Fisher information weighted by the limiting proportion is nonsingular, then the MLE is asymptotically normal and the asymptotic covariance matrix is the inverted weighted Fisher information.
Next, we will show that if we follow the experiment selection rule \textrm{GI0} or \textrm{GI1}, then the frequency for the selected experiments is approaching a limiting proportion that is determined by the criterion function $\mathbb{G}_{\bgtheta}$. For this purpose, we first define a function $\mathbb{F}_{\bgtheta}: \ShatA\to\mathbb{R}$,
\begin{equation}\label{def:F_theta}
    \mathbb{F}_{\bgtheta} ({\bgpi})=\mathbb{G}_{\bgtheta}\Big[\Big\{\sum_{a \in \mathcal{A} } \pi(a) \mathcal{I}_a( {\bgtheta}  ) \Big\}^{-1}\Big].
\end{equation}

\begin{theorem}[Limiting experiment selection frequency following \textrm{GI0} or \textrm{GI1}]\label{thm:empirical pi as converge}
Assume that $\mathbb{F}_{\bgtheta^*} ({\bgpi})$ has a unique minimizer, denoted by ${\bgpi}^*$. That is, ${\bgpi}^*=\arg\min_{{\bgpi}\in\ShatA}  \mathbb{F}_{\bgtheta^*} ({\bgpi})$. 
Then,  \textrm{GI0} and \textrm{GI1} both satisfy
\begin{equation}
    \lim_{n\to \infty}\overline{{\bgpi}}_n = {\bgpi}^*\ a.s.\ \mathbb{P}_{*},
\end{equation}
{where $\overline{{\bgpi}}_n$ is the corresponding empirical frequency vector.} Moreover, for a general function $\mathbb{F}_{\bgtheta^*}(\cdot)$ whose minimizer is not necessarily unique, we have
\begin{equation}
\lim_{n\to\infty}n^{\beta}\{\mathbb{F}_{\bgtheta^*}(\overline{{\bgpi}}_n )- \min_{{\bgpi}\in \ShatA }\mathbb{F}_{\bgtheta^*}( {\bgpi} )\}=0\ a.s.\ \mathbb{P}_{*}.
\end{equation}
for all $0\leq \beta<1/2$, given that \textrm{GI0} or \textrm{GI1} is used as the experiment selection rule.

\end{theorem}
The asymptotic normality of the MLE following \textrm{GI0} or \textrm{GI1} is proved by combining the above two theorems. We summarize this result in the next theorem.
\begin{theorem}[Asymptotic normality following \textrm{GI0} or \textrm{GI1}]\label{thm:Asy_Normal_final}
Let $\widehat{\bgtheta}_{n}^{\text{ML}}$ be the MLE following the experiment selection rule \textrm{GI0} or \textrm{GI1}, as described in Algorithm~\ref{alg:gi0} and Algorithm~\ref{alg:gi1}. Assume $\mathbb{F}_{\bgtheta^*} ({\bgpi})$ has a unique minimizer ${\bgpi}^*$. Then, 
     \begin{equation}
         \sqrt{n} (\widehat{\bgtheta}_{n}^{\text{ML}} -\bgtheta^*) \inD N_p\Big(\bm{0}_p,\big\{  \mathcal{I}^{{\bgpi}^*}(\bgtheta^*)\big\}^{-1} \Big).
     \end{equation}
\end{theorem}
The covariance of the MLE can be approximated by the plug-in estimator $ n^{-1}\{\mathcal{I}^{ \overline{{\bgpi}}_{ n} }(\widehat{\bgtheta}_{ n}^{\text{ML}})\}^{-1}$. This is justified by the next theorem.
 
\begin{theorem}[Asymptotic covariance matrix of the MLE]\label{thm:Asy_cov_mle}
    Under the settings of Theorem~\ref{thm:Asy_Normal_final},
    \begin{equation}\label{lim:AN}
          \sqrt{n}\big\{ \mathcal{I}^{ \overline{{\bgpi}}_{ n} }(\widehat{\bgtheta}_{ n}^{\text{ML}}) \big\}^{1/2}(\widehat{\bgtheta}_{n}^{\text{ML}} -\bgtheta^*)\inD N_p(\mathbf{0}_p, I_p).
    \end{equation}
    In addition, 
    for any continuously differentiable function $g: \bgTheta\to \mathbb{R}$ such that $\nabla g(\bgtheta^*)\neq \mathbf{0}_p$,
\begin{equation}\label{lim:lim_g}
     \frac{\sqrt{ n}(g(\widehat{\bgtheta}_{ n}^{\text{ML}}) -g(\bgtheta^*))}{\norm{\big\{  \mathcal{I}^{ \overline{{\bgpi}}_{ n} }(\widehat{\bgtheta}_{ n}^{\text{ML}})\big\}^{ -1/2}\nabla g( \widehat{\bgtheta}_{ n}^{\text{ML}} )}}\inD N\big(0,1\big).
\end{equation}
\end{theorem}

The first part of the above theorem justifies  the use of the plug-in estimator for the covariance matrix of the MLE.  The second part of the theorem suggests that the approximate $1-\alpha$ confidence interval for $g(\bgtheta)$ can be constructed as $g(\widehat{\bgtheta}_{n}^{\text{ML}})\pm z_{\alpha/2} \norm{\big\{  \mathcal{I}^{ \overline{{\bgpi}}_{ n} }(\widehat{\bgtheta}_{ n}^{\text{ML}})\big\}^{ -1/2}\nabla g( \widehat{\bgtheta}_{ n}^{\text{ML}} )}$ where $z_{\alpha/2}$ is the $1-\alpha/2$ quantile of the standard normal distribution.
\subsubsection{Asymptotic Optimality}
In this section, we present results regarding the optimality of the proposed methods. We consider two notions of optimality, including the optimal design and asymptotic efficiency of the estimators under a decision theory framework. The former extends a similar concept in the literature on the design of experiments, and the latter builds upon the classic asymptotic efficiency results for MLE with i.i.d. observations.
We start with the notion of optimality in terms of the optimal design.  
\begin{definition}[$\mathbb{G}_{\bgtheta^*}$- optimality]
    A selection rule is said to be $\mathbb{G}_{\bgtheta^*}$ a.s. optimal design if its corresponding selection frequency $\{\overline{{\bgpi}}_n\}_{n\in\mathbb{Z}_+}$  satisfies
\begin{equation}
    \lim_{n\to \infty}\mathbb{G}_{\bgtheta^*}( \{\mathcal{I}^{\overline{{\bgpi}}_n}(\bgtheta^*)\}^{-1} ) =  \min_{{\bgpi}\in \ShatA} \mathbb{G}_{\bgtheta^*}( \{\mathcal{I}^{{\bgpi}}(\bgtheta^*)\}^{-1} )\ a.s.\ \mathbb{P}_*.
\end{equation}
\end{definition}
 The above notion of $\mathbb{G}_{\bgtheta^*}$- optimal selection rules extends the classic concept of optimal designs adopted in the literature on the design of experiments (see, e.g., \cite{yang2013optimal,kiefer1974general}). It allows for general criteria functions and adaptive experiment selection rules. 
If an adaptive experiment selection rule is $\mathbb{G}_{\bgtheta^*}$- optimal, it approximately minimizes the criterion function when the sample size is large. Theorem~\ref{thm:empirical pi as converge} implies the following result.
\begin{theorem}[$\mathbb{G}_{\bgtheta^*}$- optimal selection]\label{thm:opt_selection}
    Both \textrm{GI0} and \textrm{GI1} are $\mathbb{G}_{\bgtheta^*}$ a.s. optimal.
\end{theorem}
The above theorem indicates that the proposed experiment selection rules have the best performance in some sense when compared with other experiment selection rules. Next, we consider the optimality property of the MLE when combined with \textrm{GI0} or \textrm{GI1}  under the lens of a sequential decision theory framework for the design-and-estimation problem.

{
Consider a loss function \(L(\bgtheta^*,\widehat\bgtheta) \) for an estimator $\widehat{\bgtheta}$ following an active experiment selection rule, and the corresponding risk $\mathbb{E}_{\bgtheta^*}L(\bgtheta^*,\widehat\bgtheta)$. The next theorem first establishes a lower bound for the asymptotic risk for unbiased estimators and then shows that the MLE combined with the selection rule \textrm{GI0} (or  \textrm{GI1}) achieves this lower bound when the criterion function matches the loss function.
}
\begin{theorem}[Minimum risk for unbiased estimators]\label{thm:ultimate}
Let \(L(\bgtheta, \widehat\bgtheta)\) be a loss function {twice continuously differentiable in $\widehat\bgtheta$} satisfying 
that $L(\bgtheta, \widehat\bgtheta)\geq 0$, \(L(\bgtheta, \widehat\bgtheta) = 0\) if and only if \(\widehat\bgtheta = \bgtheta\), and $\eta I_p\preceq \frac{1}{2}\nabla^2_{\widehat\bgtheta} L(\bgtheta^*,\widehat\bgtheta) \preceq \eta' I_p$ for some positive constants $\eta$ and $\eta'$, and all $\widehat\bgtheta\in\bgTheta$. Let $H_{\bgtheta}=\frac{1}{2}\nabla^2_{\widehat \bgtheta} L(\bgtheta,\widehat\bgtheta)\Big\vert_{\widehat\bgtheta=\bgtheta}$. 
Then, the following results hold.

\begin{enumerate}
    \item Assume regularity conditions (but without Assumption \ref{ass:5}) hold.
  {Consider an unbiased estimator  \(\boldsymbol{T}_n\) of $\bgtheta$ following an arbitrary adaptive experiment selection rule. If the loss function does not satisfy \(L(\boldsymbol{\theta^*},\widehat{\boldsymbol{\theta}})  \equiv \langle H_{\boldsymbol{\theta^*}} (\boldsymbol{\theta^*}-\widehat{\boldsymbol{\theta}}), \boldsymbol{\theta^*}-\widehat{\boldsymbol{\theta}} \rangle\), we further assume for any \(\varepsilon>0\), \(\limsup_{n\to \infty} \mathbb{E}_{\boldsymbol{\theta^*}} n \|\boldsymbol{T}_n-\boldsymbol{\theta^*}\|^2 {I}(\|\boldsymbol{T}_n-\boldsymbol{\theta^*}\| >\varepsilon)= 0.\) 
}
Then, 
\begin{equation}\label{lim:efficiency}
    \liminf_{n\to \infty} { \mathbb{E}_{\bgtheta^*}\Big[ n \cdot L(\bgtheta^*, \blT_n )\Big]  } \geq {\inf_{\pi\in \ShatA}  \tr( H_{\bgtheta^*}   \{\mathcal{I}^{ \pi } (\bgtheta^*)\}^{-1})}.
\end{equation}
In particular, if the squared error loss $L(\bgtheta,\widehat\bgtheta)=\|\bgtheta-\widehat\bgtheta\|^2$ is used, then for any unbiased estimator $\blT_n$,
$\liminf_{n\to \infty} \Big[ n \cdot \text{MSE}(\blT_n)\Big]   \geq {\inf_{\pi\in \ShatA}  \tr(    \{\mathcal{I}^{ \pi } (\bgtheta^*)\}^{-1})}.$
    \item Under Assumptions~\ref{ass:1}-\ref{ass:4},  \ref{ass:6A} and \ref{ass:7A}, 
and further assume that there exists $\alpha>0$, such that for any $\bgxi_a=\blZ_a\bgtheta, \bgtheta\in \bgTheta$ and $x^a\in \operatorname{supp}(f_{\bgtheta,a})$,
    \begin{equation}\label{ass:ult_strong convex}
        \lambda_{min}(-\nabla^2_{\bgxi_a} \log h_{\bgxi_a,a}(x^a))\geq \alpha>0.
    \end{equation}
Assume there exists $\delta>0$ such that $\mathbb{E}_{X^a\sim f_{\bgtheta^*,a}} \norm{\nabla_{\bgtheta} \log f_{\bgtheta,a}(X^a) }^{2+\delta}<\infty$. Assume that $\tr(H_{\bgtheta^*}\{\mathcal{I}^{{\bgpi}}(\bgtheta^*)\}^{-1} )$ has a unique minimizer, denoted by ${\bgpi}^*$. If we choose $\mathbb{G}_{\bgtheta}(\bgSigma)=\tr (H_{\bgtheta} \bgSigma)$ and use the experiment selection rule \textrm{GI0} (or \textrm{GI1}) described in Algorithm~\ref{alg:gi0} (or Algorithm~\ref{alg:gi1}), then the MLE achieves the lower bound in \eqref{lim:efficiency}. That is,
    \begin{equation}\label{equ:Loss_H}
        \lim_{n\to \infty}\mathbb{E}_{\bgtheta^*}\{ n \cdot L(\bgtheta^*,\widehat{\bgtheta}_{n}^{\text{ML}} )\} =\min_{{\bgpi}\in \ShatA}\tr (H_{\bgtheta^*} \{ \mathcal{I}^{{\bgpi}}(\bgtheta^*) \}^{-1} ).
    \end{equation}
{In particular, if $L(\bgtheta,\widehat\bgtheta)=\|\bgtheta-\widehat\bgtheta\|^2$, the corresponding criterion function is $\mathbb{G}_{\bgtheta }(\cdot)=\Phi_1(\cdot)=\tr (\cdot)$. MLE combined with \textrm{GI0} (or \textrm{GI1}) achieves the asymptotic lower bound for $n\cdot \text{MSE}(\blT_n)$ for unbiased estimator $\blT_n$.}
\end{enumerate}
\end{theorem}

{
The first part of the above theorem provides a lower bound for the risk of any unbiased estimator combined with an arbitrary experiment selection rule. In particular, when $p=|A|=1$, it aligns with the classic Cram\'er - Rao lower bound for the variance of unbiased estimators with independent observations. The second part of the theorem suggests that the asymptotic risk of the MLE combined with the proposed \textrm{GI0} (or \textrm{GI1}) matches the lower bound, if the criterion function aligns with the loss function. When  $p = |\mathcal{A}|=1$, this matching risk gives an extension of the classic asymptotic efficiency result for MLE with i.i.d. data.

We note that Theorem~\ref{thm:ultimate} does not directly imply that the proposed method minimizes  risk within a class of decision rules, since the MLE is not necessarily unbiased. This scenario is analogous to the classic asymptotic efficiency result for MLE with i.i.d. observations, where the MLE is shown to have the asymptotic variance matching the Cram\'er - Rao bound for unbiased estimators but the MLE itself is not unbiased. On the other hand, the asymptotic optimality of the MLE within a decision theory framework can be formalized using concepts such as local asymptotically normal (LAN) estimators and asymptotic concentration (see Chapter 8 of~\cite{van2000asymptotic}) in classic asymptotic statistics. The next theorem suggests that MLE combined with the proposed experiment selection method is also asymptotically optimal in a similar sense. Here, we omit the definitions of notations and terminology such as ``$\rightsquigarrow$", ``$\star$", %
and ``bowl-shaped functions", and refer readers to Theorem~8.8 and 8.11 in  Chapter 8 of~\cite{van2000asymptotic}, as the formal definitions of these notations are lengthy.   
}

\begin{theorem}[{Local asymptotic minimax risk}]\label{thm:convolution_theorem}
\sloppy Assume $a_1,\cdots,a_n,\cdots$ are experiments selected following an active experiment selection rule such that $a_{n+1}$ is measurable with respect to $\mathcal{F}_n$ for all $n$. Assume that the sequence $(\blT_n(a_1,X_1,\cdots,a_n,X_n), \overline{\bgpi}_n)$ is regular at $(\bgtheta,\bgpi)\in  {\bgTheta}\times {\ShatA}$ for estimating parameter $\bgtheta$, which means that for every $\blh\in \mathbb{R}^p$, 
\begin{equation}\label{equ:lim_T_n}
    \sqrt{n}\Big( \blT_n-(\bgtheta+\frac{\blh}{\sqrt{n}}) \Big) \stackrel{\bgtheta+\frac{\blh}{ \sqrt{n}} }{\rightsquigarrow} L^{\bgpi}_{\bgtheta} \text{ and } \overline{\bgpi}_n \stackrel{\bgtheta+\frac{\blh}{ \sqrt{n}} }{\rightsquigarrow}   \bgpi,
\end{equation}
for some distribution $L^{\bgpi}_{\bgtheta}$, and $\mathcal{I}^{\bgpi}(\bgtheta)=\sum_{a\in \mathcal{A} }\pi(a)\mathcal{I}_a(\bgtheta)$ is nonsingular. %
Then, the following statements hold.
\begin{enumerate}
    \item (Convolution theorem) There exists a probability measure $M^{\bgpi}_{\bgtheta}$ such that
    \begin{equation}\label{equ:Convolution_Theorem}
        L^{\bgpi}_{\bgtheta}=N_p( \bm{0}_p, \{\mathcal{I}^{ \bgpi }(\bgtheta)\}^{-1} ) * M^{\bgpi}_{\bgtheta}.
    \end{equation}
In particular, if $L^{\bgpi}_{\bgtheta}$ has the covariance matrix $\bgSigma^{\bgpi}_{\bgtheta}$, then $\bgSigma^{\bgpi}_{\bgtheta} \succeq \{\mathcal{I}^{ \bgpi }(\bgtheta)\}^{-1}$.
    \item (Local asymptotic minimax theorem) For any bowl-shaped loss function $\ell$,
\begin{equation}\label{conj:local_minimax}
    \sup_{|F|<\infty, F\subset \mathbb{R}^p} \liminf_{n\to \infty} \sup_{h\in F} \mathbb{E}_{\bgtheta+\frac{\blh}{\sqrt{n}}} \ell \Big( \sqrt{n}\big(  \blT_n-(\bgtheta+\frac{\blh}{\sqrt{n}})  \big) \Big) \geq \mathbb{E} \ell(V^{\bgpi}) \geq \min_{\bgpi} \mathbb{E} \ell(V^{\bgpi}),
\end{equation}
where the first supremum is taken over all finite subsets $F$ of $\mathbb{R}^p$, and $V^{\bgpi}\sim N_p(\bm{0}_p, \{\mathcal{I}^{\bgpi}(\bgtheta) \}^{-1} )$.
\end{enumerate}
\end{theorem}
In the case where $\ell(\bgtheta)= \norm{\bgtheta}^2$ and $\mathbb{G}_{\bgtheta}(\cdot) = \Phi_1(\cdot) = \operatorname{tr}(\cdot)$, the second part of Theorem~\ref{thm:ultimate} together with Theorem~\ref{thm:convolution_theorem} imply that the MLE combined with both \textrm{GI0} and \textrm{GI1} selection achieves the local asymptotic minimax lower bound on the MSE of estimators.

\subsection{Theoretical Results Regarding Early Stopping Rules}\label{sec:theory stopping}
As discussed in Section~\ref{sec:method-stopping},
early stopping rules are adopted in many applications to reduce the expected sample size. In this section, we provide consistency and asymptotic normality results  for the MLE obtained at a large random stopping time.

\begin{theorem}[Strong consistency at a random stopping time]\label{thm:a.s. stopping time}
{Let $\widehat{\bgtheta}_{n}^{\text{ML}}$ be the MLE following the experiment selection rule \textrm{GI0} or \textrm{GI1}, as described in Algorithm~\ref{alg:gi0} and Algorithm~\ref{alg:gi1}},
and let $\tau_n\in \mathbb{N}$ be a sequence of stopping time with respect to the filtration $\{\mathcal{F}_n\}_{n\in\mathbb{Z}_+}$ such that $\lim_{n\to\infty}\tau_n= \infty$ a.s. and $\tau_n<\infty$ a.s. for each $n$. Then,
$$
\lim_{n\to\infty} \widehat{\bgtheta}_{\tau_n}^{\text{ML}} = \bgtheta^* \text{ a.s. } \mathbb{P}_*.
$$
\end{theorem}
The above theorem extends Theorem~\ref{thm:consistency_final} to allow for random stopping times.
  It suggests that the MLE is close to the true model parameter at a large random sample size. 
  Next, we present the result on asymptotic normality, which enables statistical inference at large stopping times.
\begin{theorem}[Asymptotic normality following \textrm{GI0} or \textrm{GI1} with an early stopping rule]\label{thm:Asy_Normal_final_stopping_time_sp}
Let $\widehat{\bgtheta}_{n}^{\text{ML}}$ be the MLE following the experiment selection rule \textrm{GI0} or \textrm{GI1}, as described in Algorithm~\ref{alg:gi0} and Algorithm~\ref{alg:gi1}. Assume $\mathbb{F}_{\bgtheta^*} (\bgpi)$ has a unique minimizer $\bgpi^*$. Let $\{c_n\}_{n\geq 0}$ be a positive and decreasing sequence such that $c_n\to 0$ as $n\to \infty$. 
Let $h: \bgTheta\to \mathbb{R}$ be a continuously differentiable function such that $\nabla h(\bgtheta)\neq \bm{0}_p$ for all $\bgtheta\in \bgTheta$.
Consider stopping times $\tau^{(1)}_{c_n}$ and $\tau^{(2)}_{c_n}$ defined in \eqref{equ:early stopping g} and \eqref{equ:early stopping TV}, respectively.
Then, for both stopping time \(\tau_n = \tau^{(1)}_{c_n}\) and \(\tau_n = \tau^{(2)}_{c_n}\), we have
\begin{equation}\label{lim:stopping_time_lim_sp}
    \sqrt{\tau_n} \big\{  \mathcal{I}^{ \overline{\pi}_{\tau_n} }(\widehat{\bgtheta}_{\tau_n}^{\text{ML}})\big\}^{ 1/2}(\widehat{\bgtheta}_{\tau_n}^{\text{ML}} -\bgtheta^*) \inD N_p\big(\bm{0}_p,I_p\big).
\end{equation}
Furthermore, for any continuously differentiable function $g: \bgTheta\to \mathbb{R}$ such that $\nabla g(\bgtheta^*)\neq \bm{0}_p$,
\begin{equation}\label{lim:stopping_time_lim_g_sp}
     \frac{\sqrt{\tau_n}(g(\widehat{\bgtheta}_{\tau_n}^{\text{ML}}) -g(\bgtheta^*))}{\norm{\big\{  \mathcal{I}^{ \overline{\pi}_{\tau_n} }(\widehat{\bgtheta}_{\tau_n}^{\text{ML}})\big\}^{ -1/2}\nabla g( \widehat{\bgtheta}_{\tau_n}^{\text{ML}} )}}\inD N(0,1).
\end{equation}
\end{theorem}

\section{Applications}\label{sec:application}
In this section, we provide details on the methods and theoretical results to applications discussed in Section~\ref{sec:intro}, including  item selection in CAT and adaptive pairs selection in sequential rank aggregation problems. We also provide results regarding  active estimation for generalized linear models (GLM), which encompass many useful models as its special cases.
\subsection{Active Estimation for GLM}\label{sec:app-glm}
Consider the case where the distribution of the observations falls into an exponential family (see, e.g.,  \cite{mccullagh2019generalized}). Following the setting in Section 5 of \cite{chaudhuri2015convergence}, we consider the density functions 
\begin{equation}\label{eq:glm}
f_{\bgtheta,a}(x_a) = \zeta^a(x_a)\exp\left\{x_a \blz_a^T\bgtheta - B_a(\blz_a^T\bgtheta) \right\},
\end{equation}
where $x_a\in \mathbb{R}$, $\blz_a\in \mathbb{R}^p$, and $B_a(\cdot),a\in \mathcal{A}$. {Assume that the support of $B_a$ is $\mathbb{R}$.} 
Under this model, $\bgtheta$ serves as the unknown linear coefficient in a GLM and we are interested in estimating it using the proposed Algorithm~\ref{alg:gi0} and Algorithm~\ref{alg:gi1}. The Fisher information is given by
\begin{equation}\label{equ:easy_B}
 \mathcal{I}_a(\bgtheta)= B''_a(\blz_a^T\bgtheta) \blz_a \blz_a^T \text{ and } \mathcal{I}(\bgtheta;\va_n ) = \sum_{i=1}^n B''_{a_i}(\blz_{a_i}^T\bgtheta) \blz_{a_i} \blz_{a_i}^T.%
\end{equation}
Based on the above equations, Algorithms~\ref{alg:gi0} and \ref{alg:gi1} are simplified as follows.
\begin{algorithm}[H]
\caption{Simplified \textrm{GI0/GI1} Algorithm for GLM}
\label{alg:gi-glm}
\begin{algorithmic}
\STATE \quad We modify the following lines in Algorithms~\ref{alg:gi0} and \ref{alg:gi1}, while keeping the other lines of the algorithms unchanged.
\STATE 2:   {\textbf{Require:} 
    choose $a_1^0,\cdots, a_{n_0}^0$ such that $\dim (\operatorname{span}\{ \blz_{a_n^0}; n=1,2,\cdots, n_0 \})=p$. }
\STATE 6: The Fisher information matrices used in line 6 of Algorithms~\ref{alg:gi0} and \ref{alg:gi1}  are calculated using the formula \eqref{equ:easy_B}.
\end{algorithmic}
\end{algorithm}
\begin{corollary}\label{cor:app-glm}
Assume {the function $B_a$ has the support $\mathbb{R}$, $\dim (\operatorname{span}\{ \blz_{a}; a\in \mathcal{A} \})=p$, and} Assumptions \ref{ass:1} and \ref{ass:5} hold. 
If the above Algorithm~\ref{alg:gi-glm} for \textrm{GI0} or \textrm{GI1} is used,
then {all the theorems presented in Section~\ref{sec:main-theory} hold.}
\end{corollary}
Note that in the above corollary, the assumptions are greatly simplified compared to the regularity conditions described in Section~\ref{sec:assumptions}, thanks to the nice form of GLMs. It only requires that the parameter space is compact, the true parameter is an interior point of the parameter space, and the parameter is identifiable when using all the experiments together. In practice, the parameter space $\bgTheta$ may not be given in advance. In these cases, we may specify $\bgTheta$ as a box (i.e., $\bgTheta = [-r,r]^p$) or ball (i.e., $\bgTheta = \{\bgtheta:\|\bgtheta\|\leq r\}$) for some large $r$. The theoretical results still apply, if the true parameter is an interior point of the parameter space.

\subsection{Computerized Adaptive Testing (CAT)}\label{sec:app-cat}
CAT has gained prominence in recent decades as an innovative approach to educational assessment \citep{wainer2000computerized,bartroff2008modern,chang2009nonlinear}. 
In CAT, test items are sequentially and adaptively chosen from an item pool based on the test-taker's previous responses. This approach enhances test precision and shortens test length by selecting items tailored to the test-taker's individual latent traits. Item Response Theory (IRT) and Multidimensional Item Response Theory (MIRT) models are commonly used to model a test-taker's responses (See, e.g., \cite{chen2021item}, \cite{embretson2013item}, and \cite{reckase200618} for reviews on IRT and MIRT models). In a binary MIRT model, a response to an item is coded as $0$ or $1$, where $1$ indicates that the item was answered correctly and $0$ indicates the it was answered incorrectly.

Let $k$ be the total number of items in the item pool for an educational test, and let $\mathcal{A}=\{1,\cdots, k\}$ represent the indices of these items. Under a MIRT model, each item $j\in\mathcal{A}$ is associated with a multidimensional item parameter $(\blz_j,b_j)$%
, which quantifies item properties such as the item's difficulty and the skills it measures. The test taker is associated with a latent trait parameter $\bgtheta\in\mathbb{R}^p$, typically interpreted as proficiency in $p$ different skills. 
Given the selected items and the test-taker's latent trait parameter, responses are assumed to be conditionally independent. The correct response probability
$P(\bgtheta; \blz_j,b_j)$, also known as the item response function (IRF) of item $j$, is a function of $\bgtheta$ and depends on $(\blz_j,b_j)$. 
For example, the commonly adopted multidimensional two-parameter logistic model (M2PL) 
assumes that the IRF takes the form
\begin{equation}\label{eq:m2pl}
P({\bgtheta};\boldsymbol{z}_j,b_j) = \{1+ \exp(-\boldsymbol{z}_j^T{\bgtheta}-b_j )\}^{-1},
\end{equation}
where $\blz_j$ is the discrimination parameter, indicating the strength of each latent trait's influence on the response, and $-b_j$ is the difficulty parameter of item $j$.

Item selection is critical for efficient CAT design. The objective is to accurately estimate the latent trait parameter $\bgtheta\in\mathbb{R}^p$ by selecting the next item $a_{n+1}\in\mathcal{A}$ based on previously selected items and responses $a_1, X_1, \cdots, a_n, X_n$. Note that item parameters are typically pre-calibrated based on historical data and are assumed to be known in CAT.
In the rest of the section, we provide details on applying the item selection rules \textrm{GI0} and \textrm{GI1} under the M2PL model. First, the density function is $f_{\bgtheta, a}(x) = P({\bgtheta};\boldsymbol{z}_a,b_a)^x(1- P({\bgtheta};\boldsymbol{z}_a,b_a))^{1-x}$, and the corresponding Fisher information is 
\begin{equation}\label{equ:cat_m2pl}
\mathcal{I}_a(\bgtheta) = P({\bgtheta};\boldsymbol{z}_a,b_a)(1-P({\bgtheta};\boldsymbol{z_a},b_a)) \blz_a\blz^T_a \text{ and }
\mathcal{I} (\bgtheta;\va_n)=\sum_{a\in \mathcal{A}} \overline{{\bgpi}}_n(a) \mathcal{I}_a(\bgtheta).
\end{equation}
If the criterion function $\mathbb{G}_{\bgtheta}(\cdot)=\Phi_q(\cdot)$, then  \textrm{GI1} can be simplified as:
\begin{equation}\label{equ:GI1_m2pl}
    a_{n+1}=\arg\max_{a\in \mathcal{A}} P(\widehat{\bgtheta}^{\text{ML}}_n;\boldsymbol{z}_{a},b_{a})(1-P(\widehat{\bgtheta}^{\text{ML}}_n;\boldsymbol{z_{a}},b_{a})) \cdot  \blz^T_{a}\Big( \mathcal{I} (\widehat{\bgtheta}^{\text{ML}}_n;\va_n)\Big)^{-q-1}\blz_{a}  .
\end{equation}
\begin{algorithm}[H]
\caption{Simplified \textrm{GI0}/\textrm{GI1} Algorithm for M2PL model}
\label{alg:gi-cat}
\begin{algorithmic}
\STATE \quad We modify the following lines in Algorithms~\ref{alg:gi0} and \ref{alg:gi1}, while keeping the other lines of the algorithms unchanged.
\STATE 2:   {\textbf{Require:} $\dim (\operatorname{span}\{ \blz_{a}; a\in \mathcal{A} \})=p$ and
    choose $a_1^0,\cdots, a_{n_0}^0$ such that $\dim (\operatorname{span}\{ \blz_{a_n^0}; n=1,2,\cdots, n_0 \})=p$. }
\STATE 6:\ The Fisher information matrices used in line 6 of Algorithms~\ref{alg:gi0} for \textrm{GI0} are calculated using the formula \eqref{equ:cat_m2pl}. Selection in line 6 of Algorithms~\ref{alg:gi1} for \textrm{GI1} is replaced by \eqref{equ:GI1_m2pl}. 
\end{algorithmic}
\end{algorithm}

\begin{corollary}\label{cor:app-cat}
Assume Assumption \ref{ass:1} holds, $\dim (\operatorname{span}\{ \blz_{a}; a\in \mathcal{A} \})=p$, and criterion function $\mathbb{G}_{\bgtheta}(\cdot)=\Phi_q(\cdot)$.
If we consider simplified Algorithm \ref{alg:gi-cat} for \textrm{GI0} and \textrm{GI1} with M2PL model, then %
the conclusions for \textrm{GI0} and \textrm{GI1} from {all the theorems presented in Section~\ref{sec:main-theory} hold.}
\end{corollary}

\subsection{Sequential Rank Aggregation from Noisy Pairwise Comparison}\label{sec:app-RA}
Consider the problem of determining the global rank over $p+1$ objects. Let \(\mathcal{A} \subset \{(j, l);   j, l \in \{0, 1, 2, \ldots, p\}\} \) be a subset of all possible pairs for comparison. At each time $n$, a pair $a_n = (a_{n,1}, a_{n,2})\in  \mathcal{A}$ is chosen for comparison, yielding a random pairwise comparison outcome $X_n\in\{0,1\}$. Here,  \(X_n = 1\) indicates that the object \(a_{n,1}\) is preferred over \(a_{n,2}\) in the comparison, and \(X_n = 0\) indicates the opposite. 
To infer the global rank of objects, ranking models (e.g.,  Bradley-Terry-Luce (BTL) model \citep{bradley1952rank,duncan1959individual} and the Thurstone model \citep{thurstone1927law}) are usually assumed for the noisy pairwise comparison results. These models assume that each object $i$ is associated with a latent score parameter $\theta_i$,  the pairwise comparison result between object $i$ and object $j$ is depending on $\theta_i$ and $\theta_j$, and the true global rank is the rank of the latent score parameters. For example, the BTL model assumes 
\begin{equation}\label{eq:btl}
    f_{\bgtheta,a}(x)=\left(\frac{e^{\theta_i}}{e^{\theta_i}+e^{\theta_j}} \right)^{x}\left(\frac{e^{\theta_j}}{e^{\theta_i}+e^{\theta_j}} \right)^{1-x}
\end{equation}
for the pair $a=(i,j)$. For sequential rank aggregation, the goal is to design an active pair selection rule that determines the next pair $a_{n+1}$ for comparison based on the prior pair comparison results $(a_1,X_1,\cdots, a_n, X_n)$, so that the global rank can be inferred accurately. This problem boils down to the active sequential estimation of the latent score parameters.

In the rest of the section, we elaborate on the implementation and theoretical results for \textrm{GI0} and \textrm{GI1} for the sequential rank aggregation problem under a BTL model. Note that the distribution of the comparison results only depends on the differences $\theta_i-\theta_j$ for $0\leq i,j\leq p$. Thus, we fix ${\theta }_0=0$ to ensure the identifiability of ${\bgtheta} = ( \theta_1, \ldots, \theta_p)^T$. 

When $a=(i,j)$, we set $\blz_a=\ble_j-\ble_i$, where $\{\ble_1,\cdots,\ble_p\}$ is the standard basis of $\mathbb{R}^p$ and $\ble_0=\bm{0}_p$.

For $a=(i,j)$, the Fisher information and the weighted Fisher information are given by 
\begin{equation}\label{equ:BTL_info}
    \mathcal{I}_a(\bgtheta)=\frac{e^{\theta_i-\theta_j}}{(1+e^{\theta_i-\theta_j})^2} \blz_a  \blz_a^T, \text{and }\mathcal{I}(\bgtheta;\va_n )=\sum_{a=(i,j)\in \mathcal{A}} \overline{\pi}_n(a) \frac{e^{\theta_i-\theta_j}}{(1+e^{\theta_i-\theta_j})^2} \blz_a  \blz_a^T.
\end{equation}
If we take the criterion function $\mathbb{G}_{\bgtheta}(\cdot)=\Phi_q(\cdot)$,  \textrm{GI1} can be simplified as 
\begin{equation}\label{equ:gi1_BTL}
a_{n+1}=\operatornamewithlimits{argmax}_{a \in \mathcal{A}} \frac{e^{\blz_a^T \widehat{\bgtheta}^{\text{ML}}_n}}{(1+e^{\blz_a^T \widehat{\bgtheta}^{\text{ML}}_n })^2}    \blz^T_{a}\Big( \mathcal{I}(\widehat{\bgtheta}^{\text{ML}}_n;\va_n )  \Big)^{-q-1}\blz_{a}  .
\end{equation}
We  treat $V=\{0,1,\cdots,p\}$ as vertices and $\mathcal{A}$ as the set of edges. Then, $G=(V, \mathcal{A})$ is an undirected graph. Assume that $G$ is a connected graph. This condition ensures that ${\bgtheta}$ is identifiable when all the pairs in $\mathcal{A}$ are compared. Under this condition, it is possible to select $\mathcal{A}_0=\{a_1^{0},a_2^{0},\cdots, a_{n_0}^{0}\} \subset \mathcal{A}$ so that $(V, \mathcal{A}_0)$ is a connected subgraph of $G$. 

\begin{algorithm}[H]
\caption{Simplified \textrm{GI0}/\textrm{GI1} Algorithm for BTL model}
\label{alg:gi-BTL}
\begin{algorithmic}
\STATE \quad We modify the following lines in Algorithms~\ref{alg:gi0} and \ref{alg:gi1}, while keeping the other lines of the algorithms unchanged.
\STATE 2:\quad   \textbf{Require:} %
The subgraph $(V, \{a_1^0,\cdots,a^0_{n_0}\})$ is a  connected graph.
\STATE 6:\quad The Fisher information matrices used in line 6 of Algorithms~\ref{alg:gi0} for \textrm{GI0} are calculated using the formula \eqref{equ:BTL_info}. Selection in line 6 of Algorithms~\ref{alg:gi1} for \textrm{GI1} is replaced by \eqref{equ:gi1_BTL}.
\end{algorithmic}
\end{algorithm}
\begin{corollary}\label{cor:app-RA}
Assume that Assumption \ref{ass:1} holds, $G$ is a connected graph,
and the criterion function $\mathbb{G}_{\bgtheta}(\cdot)=\Phi_q(\cdot)$. If the above Algorithm~\ref{alg:gi-BTL} for \textrm{GI0} or \textrm{GI1} is used, then %
the conclusions for \textrm{GI0} and \textrm{GI1} from {all the theorems presented in Section~\ref{sec:main-theory} hold.}
\end{corollary}

\section{Technical Challenges, New Analytical Tools and a Proof Sketch for Theorem~\ref{thm:empirical pi as converge}}\label{sec:proof-sketch}
In this section, we highlight the key technical challenges in proving Theorem~\ref{thm:empirical pi as converge} and introduce new analytical tools to address these challenges.  
The primary challenge lies in demonstrating that GI0/GI1 effectively balances the trade-off between exploration and exploitation, a well-known concept in the literature on sequential decision making involving unknown parameters. Exploration means sufficient sampling of all relevant experiments to ensure consistent parameter estimation. Exploitation means optimally sampling experiments once the parameter has been accurately estimated. Below, we discuss these two facets—exploration and exploitation—in the context of active sequential estimation.
\subsection{Exploration}
In order to have a consistent estimator, the selection rule needs to sample relevant experiments sufficiently often. This is formalized by the following condition, 
\begin{equation}\label{eq:exploration-condition}
n_I := \max_{S\subset\mathcal{A}: S \text{ is relevant} }\min_{a\in S} n_{a} \to\infty \text{ as $n\to\infty$,}
\end{equation}
where $n_a= |\{  i; a_i=a , 1\leq i\leq n\}|$ for $a\in\mathcal{A}$. Here, we say that a set of experiments $S$ is relevant if $\sum_{a\in S} \mathcal{I}_a(\bgtheta)$ is nonsingular for any $\bgtheta\in \bgTheta$. If $S$ is relevant, then the model parameter is identifiable when all the experiments in $S$ are sampled. Equation~\eqref{eq:exploration-condition} says that at least one of the relevant sets of experiments needs to be sampled infinitely often, in order to have a consistent estimator. 

\begin{challenge}\label{challenge:exporation}
    Show that $n_I\to\infty$ as $n\to\infty$ following GI0/GI1. 
\end{challenge}
We note that in related sequential design problems, exploration is usually achieved by incorporating an extra exploration step in the experiment selection rule. For example, in active sequential hypothesis testing problems (see, e.g., \cite{chernoff1959,naghshvar2013active}), a two-stage algorithm is often utilized, where the first stage is designed for exploration and the second stage is designed for exploitation. 
Another prevalent method for ensuring sufficient exploration is the use of the epsilon-greedy algorithm in reinforcement learning and multi-armed bandit (MAB) problems, where  
all available experiments are sampled with a minimum probability of $\varepsilon$. For methods that incorporate an explicit exploration component, verifying \eqref{eq:exploration-condition} is usually straightforward. However, for algorithms like GI0/GI1, which are greedy and lack an additional exploration component,  proving (or disproving) Equation~\eqref{eq:exploration-condition} is much more challenging. 

Nevertheless, we tackle Challenge~\ref{challenge:exporation} and establish the following proposition concerning sufficient exploration for \textrm{GI0} and \textrm{GI1}.
\begin{proposition}\label{prop:exploration}Under regularity conditions described in Section~\ref{sec:assumptions}, both \textrm{GI0} and  \textrm{GI1}  satisfy that \(\liminf_{n \to \infty}\frac{n_I}{n} >0.\)
\end{proposition}
Below, we discuss the heuristic ideas for justifying the above proposition, while clarifying the rigorous proof is much more involved. 
{Let $\cA_{\max} = \arg\max_{a\in \mathcal{A}}n_a$ be the set of experiments that are most frequently selected. A key observation is that the {\em inverted Fisher information, through its directional derivatives in experiment selection rules, acts as a regularizer,} which means that if $n_{\max}/n_I$ is large enough, then we can show that
\begin{equation}\label{eq:partial-regularizer}
\partial_{\bll_{a^m}}\mathbb{F}_{\widehat\bgtheta_n}(\overline{\pi}_n) >  \partial_{\bll_{a'}} \mathbb{F}_{\widehat\bgtheta_n}(\overline{\pi}_n )
\end{equation}
\sloppy for all $a^m\in \cA_{\max}$ and some $a'\notin \cA_{\max}$, { where 
$$\partial_{\bll_{a}}\mathbb{F}_{\widehat\bgtheta_n} = \biginnerpoduct{\overline{{\bgpi}}^a_{n+1}-\overline{{\bgpi}}_{n}}{\frac{\partial}{\partial \bgpi}  \mathbb{G}_{\widehat{\bgtheta}_{n}}\big[\big\{\sum_{a \in \mathcal{A} } \pi(a) \mathcal{I}_a( {\widehat{\bgtheta}_{n}}  ) \big\}^{-1}\big]\big|_{\bgpi = \overline{{\bgpi}}_{n}}}  $$
denotes the directional derivative along the direction $\bll_a=\overline{{\bgpi}}^a_{n+1}-\overline{{\bgpi}}_{n}$.}
This implies that no action from $ \cA_{\max}$ will be selected and the ratio \(n_I/n\) is bounded from below, if we follow the experiment selection rule $a_{n+1} \in\arg\min_{a\in\cA}\partial_{\bll_{a}} \mathbb{F}_{\widehat\bgtheta_n}(\overline{\pi}_n)$. According to Equation~\eqref{eq:first-approximation} and additional asymptotic analysis, this experiment selection rule based on directional derivatives is asymptotically equivalent to \textrm{GI0} and \textrm{GI1}, and,  consequently, Proposition~\ref{prop:exploration} holds.

Note that \eqref{eq:partial-regularizer} itself is challenging to prove, for which we first prove that the following decomposition of the information holds:
$\widehat\bgSigma_n=\mathcal{I}(\widehat\bgtheta_n;\va_n)=\blA+\blE$, and $\blA,\blE$ satisfy: 
\begin{enumerate}
    \item 
    $
    \blA=\sum_{a\in \mathcal{A},\frac{n_a}{n}\geq U} \frac{n_a}{n}\mathcal{I}_{a}(\widehat\bgtheta_n ) \text{ and }\blE=\sum_{a\in \mathcal{A},\frac{n_a}{n} < U} \frac{n_a}{n}\mathcal{I}_{a}(\widehat\bgtheta_n),
    $
    for some $U>0$. %
    \item $\blA$ is a {singular} and positive semidefinite matrix.
    \item $\blE$ is a positive semidefinite matrix and the maximum eigenvalue of $\blE$ is much smaller than the  smallest non-zero eigenvalue of $\blA$.
    \item There exists $a'\in  \mathcal{A}$ such that $n_{a'}\leq n_I$ and {$\mathcal{I}_{a'}(\widehat\bgtheta_n)\notin \cR(\blA)$, where $\cR(\blA)$ denotes the column space of $\blA$. This implies 
\[
    \liminf_{\norm{\blE}\to 0}\ \operatorname{tr} \left[ \nabla \mathbb{G}_{\widehat{\boldsymbol{\theta}}_{n}} \big(\widehat\bgSigma_n \big) (\blA+\blE)^{-1}\mathcal{I}_{a'}(\widehat{\boldsymbol{\theta}}_{n}) (\blA+\blE)^{-1} \right] = \infty.
\]    
}
    \item { For all $a^{m}\in \cA_{\max}$, $\mathcal{I}_{a^{m}}(\widehat\bgtheta_n)\in \cR(\blA)$. This implies 
\[
\limsup_{ \norm{\blE}\to 0}\  \operatorname{tr} \left[ \nabla \mathbb{G}_{\widehat{\boldsymbol{\theta}}_{n}} \big( \widehat\bgSigma_n \big) (\blA+\blE)^{-1} \mathcal{I}_{a^{m}}(\widehat{\boldsymbol{\theta}}_{n}) (\blA+\blE)^{-1} \right]<\infty.
\]    
}
\end{enumerate} 
}
{We treat $\blA$ as the dominating term and $\blE$ as a small perturbation when using the above matrix decomposition. With additional matrix perturbation analysis of $\widehat\bgSigma_n^{-1} = (\blA+\blE)^{-1}$ around its non-continuous point $\blA$, a careful use of the Davis-Kahan \(\sin \Theta\) theorem \citep{yu2015useful}, and additional iterative analysis, we can show that $\operatorname{tr} \big[ \nabla \mathbb{G}_{\widehat{\boldsymbol{\theta}} _{n}} (\widehat\bgSigma_n ) \widehat\bgSigma_n^{-1} \mathcal{I}_{a'}(\widehat{\boldsymbol{\theta}}_{n}) \widehat\bgSigma_n^{-1} \big] > \operatorname{tr} \big[ \nabla \mathbb{G}_{\widehat{\boldsymbol{\theta}}_{n}} ( \widehat\bgSigma_n ) \widehat\bgSigma_n^{-1} \mathcal{I}_{a^{m}}(\widehat{\boldsymbol{\theta}}_{n}) \widehat\bgSigma_n^{-1} \big]$ for all $a^m\in \cA_{\max}$.  This, along with Equation~\eqref{eq:first-approximation} implies \eqref{eq:partial-regularizer}.}

\subsection{Exploitation}
\sloppy In active sequential estimation, optimal exploitation requires frequency of the selected experiments to approximate the optimal proportion ${\bgpi}^* = \arg\min_{{\bgpi}\in \ShatA} \mathbb{G}_{\bgtheta^*}( \{ \mathcal{I}^{{\bgpi}}(\bgtheta^*) \}^{-1} )$, when the estimator is  accurate enough. In related sequential decision problems, such as active sequential hypothesis testing, optimal exploitation is commonly attained through a `plug-in' method. This method assumes the estimator is accurate and replaces ${\bgtheta^*}$ with the estimator for calculating the proportion for the subsequent sampling. The `plug-in' method's theoretical analysis usually combines the consistency result with the optimization problem's continuity.
 However, this approach does not work for \textrm{GI0/GI1 algorithms}, which  optimize one-step-ahead information gain over the {\em discrete set} $\mathcal{A}$, rather than  the {\em probability simplex}  $\ShatA$. Consequently, it is challenging to determine whether \textrm{GI0/GI1} approximately solve the long-term optimization problem $\arg\min \mathbb{F}_{\bgtheta^*}(\bgpi)$ over the probability simplex. This issue is divided into two specific challenges:
\begin{challenge}[Noiseless case]\label{challenge:exploitation-noiseless}
For \textrm{GI0} selection \eqref{eq:GI0} and \textrm{GI1} selection \eqref{eq:GI1} with $\widehat{\bgtheta}_1=\widehat{\bgtheta}_2=\cdots=\bgtheta^*$, {do we have the convergence $\lim_{n\to\infty}\mathbb{F}_{\bgtheta^*} (\overline{\bgpi}_n)=\mathbb{F}_{\bgtheta^*} ({\bgpi}^*$)?}
\end{challenge}

\begin{challenge}[Noisy case]\label{challenge:exploitation-noise}
    How does the difference between $\widehat{\bgtheta}_n$ and $\bgtheta^*$ affect the convergence of the algorithms?
\end{challenge} 
{Challenge~\ref{challenge:exploitation-noiseless} is roughly addressed using the following arguments. First, we can show that $\mathbb{F}_{\boldsymbol{\theta}^*}(\cdot)$ is convex. By Jensen's inequality, we obtain that
\begin{equation*}
    \mathbb{F}_{ {\bgtheta}^* }\big(\frac{n-1}{n}\overline{{\bgpi}}_{n-1}+\frac{1}{n}{\bgpi}^*\big)-\mathbb{F}_{ {\bgtheta}^* }({\bgpi}^*)\leq (1-\frac{1}{n})\big\{\mathbb{F}_{ {\bgtheta}^*}(\overline{{\bgpi}}_{n-1})-\mathbb{F}_{ {\bgtheta}^*} ({\bgpi}^*)\big\}. %
\end{equation*} 
Notice that by Taylor expansion, for any ${\bgpi}\in \ShatA$,
\begin{equation*}
   \mathbb{F}_{ {\bgtheta}^* }\big(\frac{n-1}{n}\overline{{\bgpi}}_{n-1}+\frac{1}{n}{\bgpi}\big)-\mathbb{F}_{ {\bgtheta}^*}(\overline{{\bgpi}}_{n-1})=\biginnerpoduct{\nabla\mathbb{F}_{ {\bgtheta}^*}(\overline{{\bgpi}}_{n-1}) }{\frac{1}{n}{\bgpi}-\frac{1}{n}\overline{{\bgpi}}_{n-1}}+O({1}/{n^2}).
\end{equation*}
The first term on the right-hand side of the above equation is linear in ${\bgpi}$ over the simplex $\ShatA$. Thus, its minimum is achieved at a point mass at $a'$ for some $a'\in\cA$, i.e.,  $\bgpi = \bgdelta_{a'}:=(I(a=a'))_{a\in\cA}$. It can be shown that the solution to the optimization $\arg\min_{a'\in \cA}\biginnerpoduct{\nabla\mathbb{F}_{ {\bgtheta}^*}(\overline{{\bgpi}}_{n-1}) }{\frac{1}{n}\bgdelta_{a'}-\frac{1}{n}\overline{{\bgpi}}_{n-1}}$
coincides with the selection rule \textrm{GI1} if we replace the MLE with \( \bgtheta^* \) (see Equation~\eqref{eq:GI1}). Let $a^{0}_{n}$ and \( a_n^{1} \)  be the experiments selected by  \textrm{GI0} and \(\text{GI1}\) (with the MLE replaced by the true parameter), respectively. Combining the above analysis with the definition of \textrm{GI0}, we obtain
\begin{equation*}
\begin{split}
    &\mathbb{F}_{\bgtheta^*} ( \overline{{\bgpi}}^{a^{0}_n}_n)-\mathbb{F}_{\bgtheta^*} (  {{\bgpi}}^*)\leq      \mathbb{F}_{\bgtheta^*} ( \overline{{\bgpi}}^{a^{1}_n}_n)-\mathbb{F}_{\bgtheta^*} ( {{\bgpi}}^* ) = \min_{a'\in \cA}\biginnerpoduct{\nabla\mathbb{F}_{ {\bgtheta}^*}(\overline{{\bgpi}}_{n-1}) }{\frac{1}{n}\bgdelta_{a'}-\frac{1}{n}\overline{{\bgpi}}_{n-1}}+O(1/n^2)\\
    \leq & \mathbb{F}_{ {\bgtheta}^* }(\frac{n-1}{n}\overline{{\bgpi}}_{n-1}+\frac{1}{n}{\bgpi}^*)-\mathbb{F}_{ {\bgtheta}^*}({{\bgpi}}^*)+O({1}/{n^2})\leq (1-\frac{1}{n})(\mathbb{F}_{ {\bgtheta}^*}(\overline{{\bgpi}}_{n-1})-\mathbb{F}_{ {\bgtheta}^*} ({\bgpi}^*))+O({1}/{n^2}).
\end{split}
\end{equation*}
The above display suggests that the distance $\mathbb{F}_{ {\bgtheta}^*}(\overline{{\bgpi}}_{n-1})-\mathbb{F}_{ {\bgtheta}^*} ({\bgpi}^*)$ is reduced at the factor $1-1/n$ for \textrm{GI0} and \textrm{GI1} under the noiseless case. With  additional iterative analysis, we can further show that $\mathbb{F}_{\boldsymbol{\theta}^*}(\overline{\boldsymbol{\pi}}_{n})-\mathbb{F}_{\boldsymbol{\theta}^*}( \bgpi^*)\leq O({ \log n }/{n})$. Consequently, the frequency of the selected experiment converges to the optimal proportion.
}

On the other hand, the above heuristic analysis does not justify the convergence of the algorithm in the noisy case (Challenge~\ref{challenge:exploitation-noise}), nor does it provide the convergence rate. 
We address these challenges by establishing and using a modified Robbins-Siegmund theorem, which extends the classic result by \cite{robbins1971convergence}, %
to the stochastic process $Z_n = \mathbb{F}_{\bgtheta^*}( \overline{{\bgpi}}_n ) - \mathbb{F}_{\bgtheta^*}(  {{\bgpi}^*}  )$. %

\section{Simulation}\label{sec:numerical}
In this section, we present two simulation studies. The first  assesses the finite sample performance of the proposed methods under the setting of Example~\ref{eg:toy-eg}. The second is concerned with situations where $p$ or $|\cA|$ is large. {Throughout the section, we choose the criterion function $\mathbb{G}_{\bgtheta}(\bgSigma)=\Phi_1(\bgSigma) =\operatorname{tr}(\bgSigma)$ for \textrm{GI0} and \textrm{GI1}.}
Due to the page limit, we leave some detailed specifications and additional simulation results in the supplementary material.

\subsection{Simulation Study 1}\label{sec:simulation_study1}
We first evaluate the performance of the proposed methods under the settings of Example~\ref{eg:toy-eg}. Specifically, let \( p = 2 \), \( \mathcal{A} = \{1,2,3\} \), and $f_{\boldsymbol{\theta},1}(1)  =1/(1+ e^{-(-0.1 + \theta_1)})$, $f_{\boldsymbol{\theta},2}(1) = 1/(1+e^{-\theta_2})$ and $f_{\boldsymbol{\theta},3}(1)=1/(1+e^{-({\theta_1}/{2} + \theta_2)})$. Also, let $\bgTheta=[-3,3]^2$ in this section. 

We start with illustrating the optimal proportion $\bgpi^*$. According to Theorem~\ref{thm:empirical pi as converge}, the optimal proportion for experiment selection is 
\begin{equation}\label{equ:lim_cat_eg}
    \bgpi^*=(\pi^*(1),\pi^*(2),\pi^*(3)) = \arg\min_{\pi \in \ShatA}\mathbb{F}_{\bgtheta^*}(\bgpi)
=    \arg\min_{\pi \in \ShatA} \operatorname{tr} \big\{\mathcal{I}^{{\bgpi}}(\boldsymbol{\theta^*})^{-1}\big\}.
\end{equation}
Note that $\bgpi^*$ is dependent on the true model parameter $\bgtheta^*$. %
Figure~\ref{fig:opt_prop} illustrates the dependency between \( \boldsymbol{\pi}^* \) and \( \theta_2 \) while fixing $\theta_1^*=1 $. %
From the figure, we see that  the optimal proportion varies as  \(\theta_2\) changes. Additionally, for relatively small \(\theta_2\), all three experiments have non-zero optimal proportions. However, for large \(\theta_2\), $\pi^*(3)$ stays at zero, meaning that experiment $a=3$ is unnecessary in this case. %

Next, we investigate the empirical proportion of selected experiments following the proposed methods. Recall that the empirical proportion is defined as $
\overline{\pi}_n(a) = \frac{1}{n}  \big|\{i; a_i=a,1\leq i\leq n\}\big|$, for $a\in\{1,2,3\}.$ We generate data from the model with the true parameter \(\bgtheta^* =( 1 , 0)^T\) and plot the sample path of $\overline{\pi}_n(a)$ against different sample size $n$ following \textrm{GI0} and \textrm{GI1} in Figure~\ref{fig:emp_prop_GI}.
We clarify that the values of the empirical proportion in the figure are obtained without averaging. That is, they are based on a Monte Carlo simulation with only one replication. From Figure~\ref{fig:emp_prop_GI}, we can see that the empirical proportions are approximating their respective optimal values as $n$ increases, for both \textrm{GI0} and \textrm{GI1}. This is consistent with Theorem~\ref{thm:empirical pi as converge}, which states that the empirical proportion almost surely converges to the optimal proportion. We also observe that the selections made by \textrm{GI0} and \textrm{GI1} are almost identical. This may be due to the fact that they are asymptotically equivalent (see Equation \eqref{eq:first-approximation}), and they are initialized with the same random seed.

Now, we evaluate the estimation accuracy of MLE following the proposed \textrm{GI0} and \textrm{GI1}, and compare it with other experiment selection methods. The estimation accuracy is quantified using the estimated MSE, defined as
\( \widehat{\text{MSE}}_n = \frac{1}{N} \sum_{j=1}^N \| \widehat{\bgtheta}_{n,j}^{\text{ML}} - \bgtheta^* \|^2,\)
where $N=20000$ is the number of Monte Carlo replications and  $\widehat{\bgtheta}_{n,j}^{\text{ML}}$ is the MLE from the $j$-th Monte Carlo experiment with the sample size $n$. We compare \textrm{GI0} and  \textrm{GI1} with the following experiment selection rules:
\begin{enumerate}
    \item Uniform selection ({Unif}): $a_{n+1}$ is uniformly sampled from $\mathcal{A}$.
    \item Random optimal proportion selection (Opt\_random): $a_{n+1}$ is sampled randomly from $\cA$ according to the optimal proportion $\bgpi^*$. {That is, $\mathbb{P}(a_{n+1} = a|\mathcal{F}_n) = \pi^*(a)$ for $a \in \mathcal{A}.$} %
    \item Deterministic optimal proportion selection (Opt\_deterministic):\\
    $a_{n+1}=\arg\min_{a\in \mathcal{A}} \{\overline{\pi}_n(a) - \pi^*(a) \}.
    $
\end{enumerate}
We clarify that both Opt\_random and Opt\_deterministic  methods require knowledge of the unknown parameter $\bgtheta^*$, so they are not implementable in practice. These methods serve as `oracle' benchmarks allowing comparison with the proposed methods. Figure \ref{fig:estimation_accuracy} depicts the estimated MSE as a function of the sample size \( n \) for different experiment selection rules. According to the figure, \(\textrm{GI0}\), \(\textrm{GI1}\), and Opt\_deterministic perform very similarly and outperform both Unif and Opt\_random. {These findings are consistent with Theorem~\ref{thm:ultimate}}.

Finally, we check the finite sample validity of the normal approximation of the MLE. According to Theorem~\ref{thm:Asy_cov_mle} and Theorem~\ref{thm:Asy_Normal_final_stopping_time_sp}, for large $n$ and small $c$,
\begin{equation}\label{equ:CIs}
\bld^T\widehat{\bgtheta}_{n}^{\text{ML}} \pm \frac{1}{\sqrt{n}}Z_{\alpha/2}  {\norm{\Big\{  \mathcal{I}^{ \overline{\pi}_{ n} }(\widehat{\bgtheta}_{ n}^{\text{ML}})\Big\}^{ -1/2}\bld } }\text{ and }\bld^T\widehat{\bgtheta}_{\tau_c}^{\text{ML}} \pm Z_{\alpha/2}\cdot c
\end{equation}
give approximate $1-\alpha$ confidence intervals (CIs) for \(\bld^T\bgtheta\) where $\bld\in \mathbb{R}^2$ is nonzero and the stopping time  $\tau_c$ is defined in \eqref{equ:early stopping g} with $h(\bgtheta)=\bld^T\bgtheta$. 
Table~\ref{table:CI} shows the coverage probability of the above CIs at different sample sizes, following \textrm{GI0} or \textrm{GI1}, where we set $\alpha = 0.05$, $\bld=(-0.5454216, -0.8381619)^T$ {and $\bgtheta^* = (1,0)^T$}, based on a Monte Carlo simulation. From the table, we see that the coverage probability is  close to the confidence level $1-\alpha$ for reasonably large $n$ and the random stopping time $\tau_c$.

\begin{table}[ht]
\centering
\begin{tabular}{|c|c|c|c|c|}
\hline
$n$ & 25 & 50 & 100 & $\tau_{0.1}$ \\ \hline
\text{GI0} & 0.981  & 0.954  & 0.951 & 0.955 \\ 
\hline
\text{GI1} & 0.977 & 0.958 & 0.959 & 0.938 \\ 
\hline
\end{tabular}
\caption{Coverage probability for CIs based on~\eqref{equ:CIs}, where the number of Monte Carlo replications is $1000$. The Monte Carlo standard error for the values presented in the table is upper bounded by $0.007626$.  }
\label{table:CI}
\end{table}
We have also performed additional simulation studies and produced histograms of the estimators. These additional simulation results are given in the supplementary material.

\begin{figure}[!htb]
    \centering
    \begin{minipage}{0.45\textwidth}
        \centering
\includegraphics[width=1\textwidth,height=5cm]{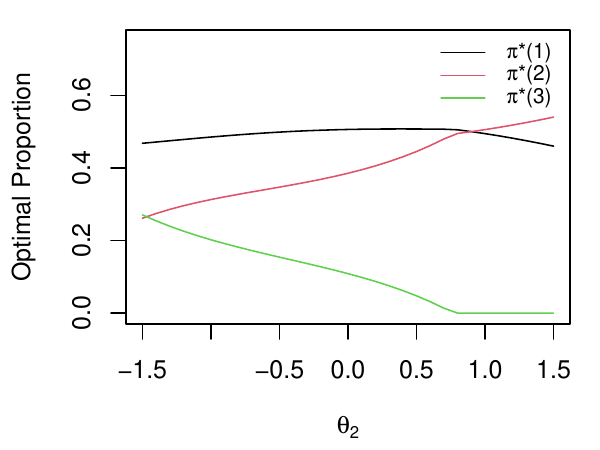}
    \caption{Optimal proportion $\bgpi^*$ as a function of \( \theta_2 \), where the true parameter satisfies \( \boldsymbol{\theta^*} =(1,\theta_2)^T \).}
    \label{fig:opt_prop}
    \end{minipage}
    \begin{minipage}{0.45\textwidth}
        \centering
\includegraphics[width=1\textwidth,height=5cm]{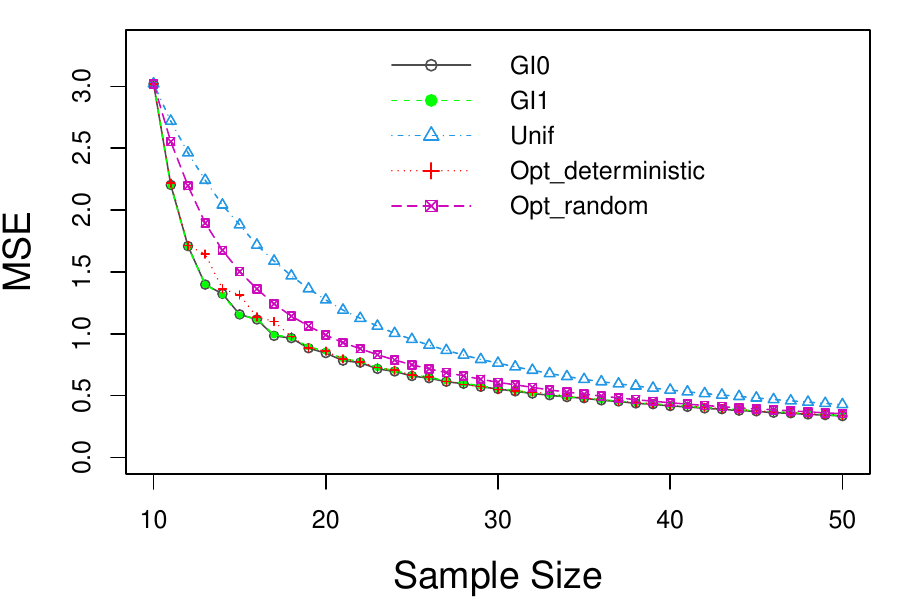}
    \caption{MSE of the MLE as sample size \(n\) varies.}
    \label{fig:estimation_accuracy}
    \end{minipage}
\end{figure}

\begin{figure}[!htb]
    \centering
    \begin{minipage}{0.45\textwidth}
        \centering
        \includegraphics[width=1\textwidth]{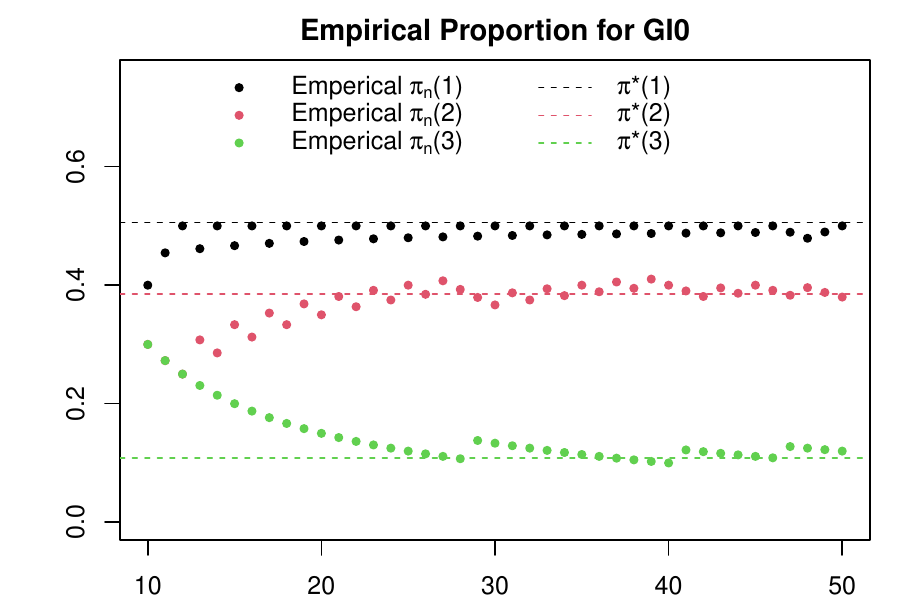}
    \end{minipage}
    \begin{minipage}{0.45\textwidth}
        \centering
    \includegraphics[width=\textwidth]{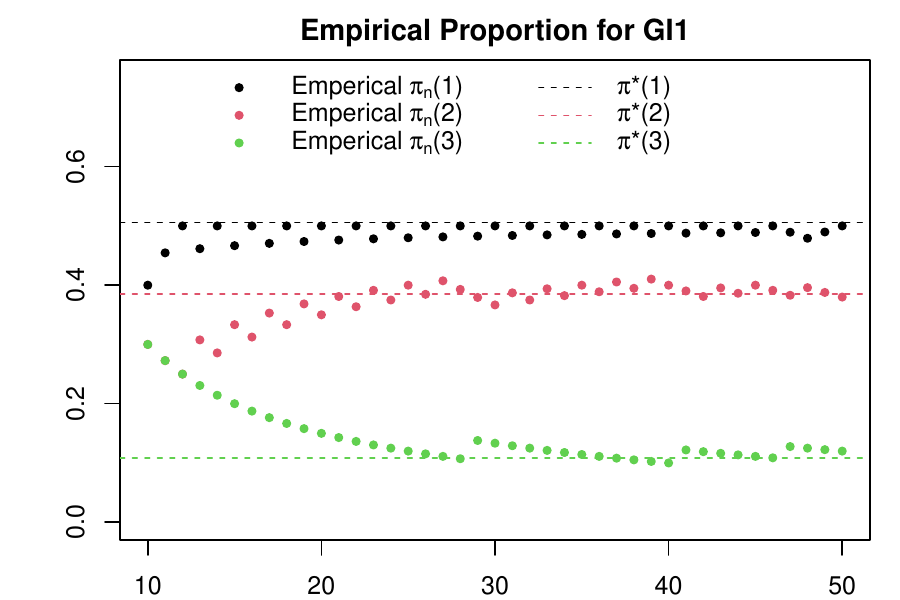}
    \end{minipage}
    \caption{Empirical proportion $\overline{\pi}_n(a)$ and the optimal proportion $\pi(a)$ for $a=1,2,3$. }
    \label{fig:emp_prop_GI}
\end{figure}

\subsection{Simulation Study 2}\label{sec:simulation_ra}
In our theoretical results, we assumed that $|\cA|$ and $p$ are fixed and $n$ grows to infinity. In this simulation study, we investigate the impact of large $|\cA|$ and $p$ on the computational time and the performance of the proposed methods. Consider the sequential rank aggregation problem described in Section~\ref{sec:app-RA}. We simulate the pairwise comparison results from a BTL model (see Equation \eqref{eq:btl}). Each coordinate of the true value of $\boldsymbol{\theta}\in \mathbb{R}^p$ are sampled independently from a uniform distribution $\mathcal{U}(-2,2)$. We vary the value of $p$ and $|\cA|$, with $p$ and $|\cA|$ ranging from $25$ to $500$ and from $p$ to $\frac{p(p+1)}{2}$, respectively. 
The computation time is given by Table~\ref{table:cal time}. Based on Table~\ref{table:cal time}, the non-paralleled \textrm{GI1} is much faster than both the non-paralleled and paralleled \textrm{GI0} when both $p$ and $|\cA|$ are large. This is consistent with Lemma~\ref{lem:complexity}.%
\begin{table*} 
\begin{tabular} 
{ 
 |c|c|c|c|c|c|c|
 }
  \hline
  & \multicolumn{3}{c|}{$p=25$} & \multicolumn{3}{c|}{$p=50$}\\
  \hline
 & $k=25$ & $k=52$ & $k=325$ & $k=50$ & $k=102$ & $k=1275$ \\
    \hline
 $\textrm{GI1}$ non-parallel& 0.676  sec & 0.430  sec & 0.416  sec & 0.963 sec &0.489    sec & 1.046 secs \\
 $\textrm{GI0}$ parallel& 1.328 secs&1.070 secs&1.784 secs&1.639 secs &1.963 secs&10.296 secs\\
 $\textrm{GI0}$ non-parallel& 1.079 secs & 1.325secs&4.790 secs& 3.202 secs & 5.030 secs   &33.430 secs\\
     \hline
  & \multicolumn{3}{c|}{$p=100$} & \multicolumn{3}{c|}{$p=500$}\\
  \hline
 & $k=100$ & $k=202$ & $k=5050$ & $k=500$ & $k=1002$ & $k=125250$ \\
     \hline
 $\textrm{GI1}$  non-parallel& 2.028   secs & 1.358  secs& 10.758  secs & 1.387 mins &57.900 secs & 2.03 hours \\
 $\textrm{GI0}$ parallel& 6.736 secs &9.594 secs& 3.240 mins & 34.322 mins &1.168 hours& 6 days \\
 $\textrm{GI0}$ non-parallel& 19.45 secs &35.065 secs &12.992 mins&  1.925 hours & 3.843 hours   &  about 20 days   \\
  \hline
 \end{tabular}
 \caption{
{The computation time for solving the MLE and selecting a new experiment at a single time point}, 
based on $100$ Monte Carlo replications, is recorded for the non-paralleled \textrm{GI1} algorithm as well as for the non-paralleled and paralleled versions of the \textrm{GI0} algorithm. For each value of \( p \), \( k=|\mathcal{A}| \) takes values in \( p \), \( 2(p+1) \), and \( \frac{p(p+1)}{2} \).
All computations are carried out on a MacBook Pro (13-inch, 2019) equipped with a 1.4 GHz Quad-Core Intel Core i5 processor.
} 
    \label{table:cal time}
\end{table*}

We also perform additional Monte Carlo simulations to assess the estimation accuracy of the proposed methods, and to study how the 
choice of $r$ in \(\bgTheta = [-r, r]^p\) affects the accuracy. Due to the page limit, details of these additional simulation studies are postponed to the supplementary material.

\section{Real Data Example}\label{sec:real_data_example}
We apply the proposed methods to a sushi preference dataset \citep{maystre2017just}. 
This dataset contains feedback from 5,000 participants who ranked 10 different types of sushi, %
selected from a total of $100$ types of sushi. Similar to the data pre-processing steps in \cite{maystre2017just}%
, we first transform each 10-item ranking into pairwise comparison results, {yielding $\binom{10}{2}\times 5000 = 225000$ pairwise comparison results}. {We fit the BTL model described in Equation~\eqref{eq:btl} with $\bgTheta = [-3,3]^{p}$ %
using all pairwise comparison data and treat the MLE of $\bgtheta$ as the ground truth. Under this setting, $p=99$, and $|\mathcal{A}|=4809$. %
We note that \(|\mathcal{A}| < \binom{100}{2}\) due to the absence of comparisons for some pairs in the dataset.}

We vary the sample size $n$ and compare the performance of the proposed \textrm{GI0} and \textrm{GI1} with two other experiment selection methods: uniform sampling and uncertainty sampling. Uncertainty sampling is a popular approach for active learning. In the context of sequential rank aggregation (see \cite{maystre2017just}), uncertainty sampling refers to sampling the pair that is most difficult to distinguish. That is, 
{
\begin{equation}\label{equ:def_uncertainty_sampling}
a_{n+1}=\arg\max_{a\in \mathcal{A}} \big[ \min\{1-f_{\widehat\bgtheta_n,a}(1), f_{\widehat\bgtheta_n,a}(0)\}\big]
= \arg\min_{a=(i,j)\in \mathcal{A}} \{ |\widehat\bgtheta_{n,i}-\widehat\bgtheta_{n,j}| \}.
\end{equation}
}

The performance of the experiment selection rules is measured through the Kendall's \(\tau\) correlation, which is often used to measure the accuracy of ranking algorithms. Specifically, define Kendall's \(\tau\) correlation as
\begin{equation*}
\tau(\widehat\bgtheta_{n},\bgtheta^*)= {\binom{100}{2}}^{-1} \sum_{1\leq i<j\leq 100} \operatorname{sign}( \widehat\bgtheta_{n,i}-\widehat\bgtheta_{n,j} )\cdot \operatorname{sign}( \bgtheta^*_{ i}- \bgtheta^*_{ j}),
\end{equation*}
where $\operatorname{sign}$ denotes the sign function, $\widehat\bgtheta_{n}$ denotes the MLE based on $n$ observations, and $\bgtheta^*$ is the ground truth obtained using the MLE based on all $225000$ comparisons.

Figure \ref{fig:kendall_tau} illustrates the Kendall's \(\tau\) coefficient for \textrm{GI0}, \textrm{GI1}, uniform selection and uncertainty sampling for different number of comparisons $n$, based on a Monte Carlo simulation with $100$ replications. From Figure~\ref{fig:kendall_tau},  \textrm{GI0} and \textrm{GI1} behave similarly, and both outperform uniform selection and uncertainty sampling. Additional details of the Monte Carlo simulation are provided in the supplementary material.

\begin{figure}[ht]
    \centering
\includegraphics[width=0.6\textwidth]{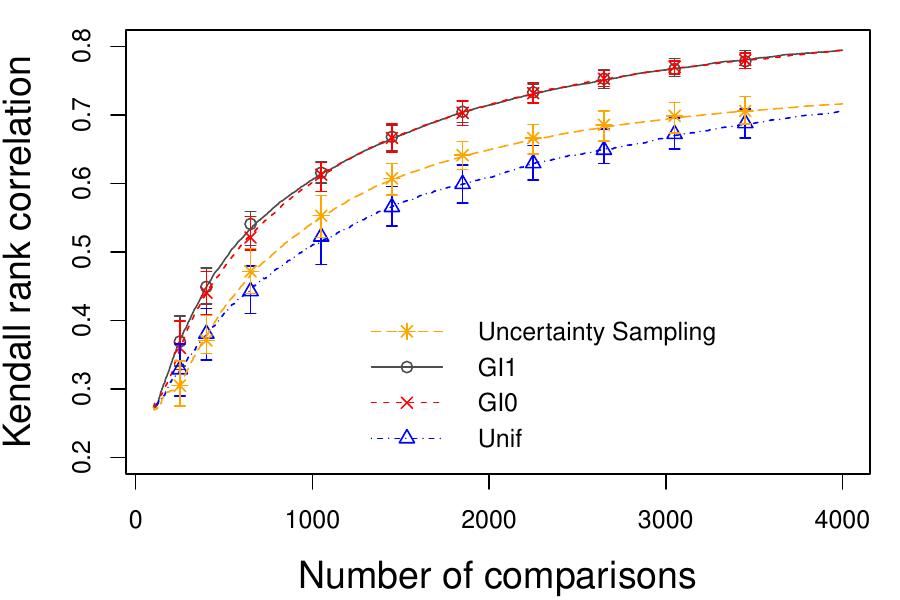}
    \caption{Comparison of different selection methods through Kendall's $\tau$ coefficient.
    The averaged Kendall's \(\tau\) correlation between \(\bgtheta^*\) and \(\widehat{\bgtheta}_{n}^{\text{ML}}\) versus the number of comparisons is plotted, along with the first and third quartiles, following different active experiment selection methods. %
    }
    \label{fig:kendall_tau}
\end{figure}

\section{Conclusion and Further Discussion}\label{sec:discussion}
In this study, we consider the problem of efficient sequential design for active sequential estimation. This problem has widespread applications across different fields; however, a systematic statistical analysis is lacking for the multidimensional case. We introduce a class of experiment selection rules that not only covers existing methods but also presents new approaches with improved numerical efficiency. Furthermore, we provide theoretical analysis including the consistency, asymptotic normality, and asymptotic optimality of the MLE following the proposed selection rule. These findings are also extended to scenarios involving early stopping rules, which are commonly used in practice. The theoretical results are highly non-trivial, and standard techniques in the literature of sequential decision making and stochastic control are not applicable. We have developed new analytical tools to tackle the theoretical challenges, which are important on their own and may be reused for other  related problems.

The current study can be extended in several directions. First, in some applications, different experiments are associated with varying sampling cost. The current method may be extended to incorporate the sampling cost in the experiment selection rules. We expect similar analytical tools can be used in the theoretical analysis. Second, theoretical results can be extended to the case where $p$ and $k$ slowly grow to infinity as $n$ grows. On the other hand, the consistency results do not hold under the high-dimensional setting where $p\geq n$. Some modifications to the estimation and experiment selection methods are necessary to ensure valid statistical inference in this case. Third, nuisance parameters may be present in some applications, where we are only interested in estimating part of the parameter efficiently. In this case, the proposed \textrm{GI0} and \textrm{GI1} still lead to a consistent and asymptotically normal MLE. However, the estimator may be asymptotically inefficient when there are redundant experiments measuring nuisance parameters. Of interest is how to design an experiment selection rule and an estimator to achieve asymptotic optimality. This is worth further investigation.

\newpage
\section*{Supplement to “Globally-Optimal Greedy Experiment Selection for Active Sequential Estimation”}\ \\
This supplement contains additional simulation results, specifications for simulation and real data analysis, and technical proof for all the theoretical results.

\section{Detailed Specifications for Simulation Studies}
In this section, we provide detailed specifications for the simulation studies in  Section \ref{sec:numerical}. %
Recall that $\mathbb{G}_{\bgtheta}(\bgSigma)=\operatorname{tr}(\bgSigma)$ throughout the simulation study.
\subsection{Detailed Specifications for Section \ref{sec:simulation_study1}} 
\subsubsection{Algorithm for Solving the Optimal Optimal Selection Proportion}\label{sec:detail-projected-gradient}
To solve the optimal selection proportion $\bgpi^*$ numerically, we apply the projected gradient descent algorithm over the simplex $\ShatA$ (see, e.g., \cite{chen2011projection}). Let $\blP_{\ShatA}$ denote the projection operator onto the simplex $\ShatA$. Initializing $\bgpi_0=(\frac{1}{3},\frac{1}{3},\frac{1}{3})$, the iterative algorithm is given by
\begin{equation*}
    \bgpi_{n+1}=\blP_{\ShatA}(\bgpi_{n}-\eta \nabla \mathbb{F}_{\bgtheta^*}(\bgpi_n)),
\end{equation*}
where the learning rate \( \eta \) is set to \( 0.001 \) and the maximum number of iterations is set to \(10000\).
{
\subsubsection{
\textrm{GI0} and \textrm{GI1} Implementation}\label{sec:S1.1.2}
When implementing \textrm{GI0} and \textrm{GI1}, the first two lines of the algorithm requires the input $\widehat\bgtheta_0$ and $a_1^0,\cdots, a_{n_0}^0$. Here we specify {$\widehat\bgtheta_0 = (0,0)^T$}, $n_0=9$, and \( a_1 = a_4 = a_7 = 1 \), \( a_2 = a_5 = a_8 = 2 \), and \( a_3 = a_6 = a_9 = 3 \).

Additionally, the MLE is solved using the \texttt{R} function \texttt{glmnet} function from the \texttt{R} package \texttt{glmnet} with the constraint  $\bgTheta = [-3,3]^2$. 
}

\subsubsection{Coverage Probability for CIs}

The coverage probability of confidence intervals described \eqref{equ:CIs} is estimated as follows:
\begin{equation*}
\frac{1}{N}\sum_{j=1}^N I\Big( |\bld^T \widehat{\bgtheta}_n^j - \bld^T {\bgtheta}^*| \leq \frac{Z_{\alpha/2}}{\sqrt{n}} {\norm{\Big\{  \mathcal{I}^{ \overline{\pi}_{ n} }(\widehat{\bgtheta}_{ n}^{\text{ML}})\Big\}^{ -1/2}\bld } }   \Big), \text{ for } n\in\{25,50,100\},
\end{equation*}
and 
\begin{equation*}
    \frac{1}{N}\sum_{j=1}^N I\Big( |\bld^T \widehat{\bgtheta}_{\tau_c}^j - \bld^T {\bgtheta}^*| \leq   Z_{\alpha/2}\cdot c   \Big),
\end{equation*}
where {$N=1000$} is the number of Monte Carlo simulation, $\widehat{\bgtheta}_{}^j$ and $\widehat{\bgtheta}_{\tau_c}^j$ are the MLE obtained in the $j$-th Monte Carlo replication with  the sample size $n$ and $\tau_c$, respectively.
\subsection{Detailed Specifications for Section \ref{sec:simulation_ra}} 
Let $\bgTheta=[-3,3]^p$.
To solve for the MLE (constrained in $\bgTheta$), we use the \texttt{glmnet} function from the \texttt{R} package \texttt{glmnet}. The computation time shown in Table \ref{table:cal time} is determined using RStudio. 

{
\subsubsection{Sampling of $\mathcal{A}$}\label{sec:graph-generation}
For a sequential rank aggregation problem under a BTL model, the graph $G= ( \{0,\cdots, p\}, \mathcal{A} )$ needs to be a connected graph for the identifiability of the model parameter (see Corollary~\ref{cor:app-RA}%
). This implies that $\mathcal{A}$ needs to satisfy some condition rather than being an arbitrary set of pairs to ensure the identifiability of the problem. Below we describe the random sampling scheme of $\mathcal{A}$ used in the Monte Carlo simulation which ensures the connectivity of $G$. Note that $|\mathcal{A}|\in\{p, 2(p+1), \frac{p(p+1)}{2}\}$ in the simulation study.
\begin{itemize}
    \item If \( |\mathcal{A}| = \frac{p(p+1)}{2} \), $G$ is a fully connected graph, 
 meaning that \( \mathcal{A} \) collects all the pairs among the $p+1$ objects. In this case, $\mathcal{A}$ is fixed throughout the Monte Carlo simulation.
    \item If \( |\mathcal{A}| = p \), a connected $G$ is equivalent to that $G$ is a minimal spanning tree for a fully connected graph. In this case, we sample $G$ uniformly from all minimal spanning trees in the Monte Carlo simulation. This is implemented using the function \texttt{sample\_spanning\_tree} from the \texttt{R} package \texttt{igraph}.
    \item If $|\mathcal{A}| = 2(p+1)$, we restrict $G$ to be $4$-regular, which means that each node from $\{0,1,\cdots,p\}$ has exactly $4$ neighbors. In this case, we sample $G$ uniformly from all $4$-regular graphs. This is implemented using the function \texttt{sample\_k\_regular} from the \texttt{R} package \texttt{igraph}.
\end{itemize}
\subsubsection{Initial estimator $\widehat\bgtheta_0$ and experiments $a_1^0,\cdots, a_{n_0}^0$}\label{sec:details-initial}
For implementing \textrm{GI0} and \textrm{GI1}, we set $\widehat\bgtheta_0=\bm{0}$ in Algorithms~\ref{alg:gi-BTL}. %

According to Corollary~\ref{cor:app-RA} in Section~\ref{sec:app-RA}, %
the initial experiments needs to be selected so that $G_0 =(\{0,\cdots,p\}, \{a_1^0,\cdots,a_{n_0}^0\} )$ is a connected subgrpah of $G$. Here, elements of $\{a_1^0,\cdots,a_{n_0}^0\}$ are not necessary to be distinct. Throughout the Monte Carlo simulation, we set $n_0=4p$, and sample $G_0$ randomly using the following steps and collect the edges in $G_0$ to form the set of initial experiments $\{a_1^0,\cdots,a_{n_0}^0\}$.
\begin{enumerate}
    \item[Step 1:] Sample uniformly from all the minimal spanning trees from $G$, which is implemented using the \texttt{R} function \texttt{sample\_spanning\_tree}. Let $(\{0,\cdots,p\},\{a_1^{\text{tree}},\cdots,a_p^{\text{tree}} \})$ denote the sampled tree.
    \item[Step 2:] Randomly sample $3p$ pairs from {$\mathcal{A} $} without replacement. Let $\{a'_{p+1},\cdots, a'_{4p}\}$ denote all sets of pairs (possibly repeated) sampled from this step.
    \item[Step 3:] $\{a_1^0,\cdots,a_{n_0}^0\}=\{a_1^{\text{tree}},\cdots,a_p^{\text{tree}} , a'_{p+1},\cdots, a'_{4p} \}$ collects all the edges generated in the first and second steps.
\end{enumerate}
Among the steps mentioned above, the first step yields a connected subgraph of $G$ with $p$ edges, and the second step expands this subgraph into another connected subgraph to have at most $p+3p=4p$ edges. 

}

\subsubsection{Algorithm Acceleration}
The accelerated \textrm{GI1} algorithm (as described in Algorithm~\ref{alg:gi1-Accelerated}) is employed for \textrm{GI1} selection, because in sequential rank aggregation problem the information matrix can be decomposed into a structure that is both sparse and of low rank with $s=2$ (see Lemma~\ref{lem:complexity}). To accelerate \textrm{GI0} Algorithm~\ref{alg:gi-BTL}, we parallel the calculation of \eqref{eq:GI0} when $|\mathcal{A}|$ is large.

\section{Detailed Specifications for the Real Data Analysis in Section~\ref{sec:real_data_example}}
\subsection{Data Structure} 
{The transformed dataset contains 225000 pairwise comparison results. 
We list these comparison results as the dataset $\mathcal{D}=\{( a^{(i)}, X^{(i)} )\}_{i=1}^{T}$, where $T=225000$, $a^{(i)}$ indicates the pairs to compare and $X^{(i)}$ is binary, indicating the corresponding pairwise comparison result.  We note that $\mathcal{D}$ is a multiset, meaning that it may have repeated elements.
\subsection{Sequential Sampling of the Pairwise Comparison Data}\label{sec:real-data-sampling-spec}
We note that, for the real data analysis, each element in the data set $\mathcal{D}$ is  sampled at most once, to prevent the redundancy of using the same data points multiple times. As a result, when we implementing an active sampling scheme for the real data analysis, we will always sample elements from $\mathcal{D}$ without replacement. 

Specifically, for all the experiment selection methods compared in this section, we set $n_0=p=99$, %
and generate the initial experiments $\{a_1^0,\cdots, a_{n_0}^0\}$ randomly following the Step 1 procedure in Section~\ref{sec:details-initial}. %
For each $j\in \{1,2,\cdots,n_0\} $, we sample the initial pairwise comparison results as follows: we sample $s_j$ uniformly from 
\(\{i: a^{(i)} =  a_j^0 , 1\leq i\leq T\}
\). Then, the initial pairs and comparison results are given by  $ ( a^{(s_1)}, X^{(s_1)} ),\cdots ( a^{(s_{n_0})}, X^{(s_{n_0})} ) $. This gives the initial data $(a_1^0, X_1),\cdots, (a_{n_0}^0, X_{n_0})$.

Let $S_{n_0}=\{s_1,\cdots,s_{n_0}\}$, $[T]=\{1,2,\cdots,T\}$.
For each $S\subset [T]$, define
\[
\mathcal{A}_{S}=\{a^{(i)} \in \mathcal{A}:   i\in [T] \backslash S\}.
\]
Next, we provide details of the implementation of different adaptive pair selection rules for $n>n_0$.
\begin{itemize}
    \item [Uniform sampling:] For $n= n_{0}, \cdots, T-1$,   sample $s_{n+1}$ uniformly from $[T] \setminus S_{n}$. Let $ S_{n+1}= S_{n}\cup \{s_{n+1}\}$. The $(n+1)-$th pair and comparison result $(a_{n+1}, X_{n+1})$ is given by $(a^{(s_{n+1} )} , X^{(s_{n+1} )})$.
    \item[\textrm{GI0} and \textrm{GI1}:]For $n=n_0, \cdots, T-1$,  calculate $a_{n+1}$ according to \eqref{eq:GI0} and \eqref{eq:GI1} with $\mathcal{A}$ replaced by $\mathcal{A}_{S_n}$ for \textrm{GI0} and \textrm{GI1}, respectively. Next, we uniformly sample the index $s_{n+1}$ from $\{ i\in [T] \backslash S_{n} : a^{(i)}=a_{n+1} \}$. Let $ S_{n+1}= S_{n}\cup \{s_{n+1}\}$. The $(n+1)-$th pair and comparison result $(a_{n+1}, X_{n+1})$ is given by $(a^{(s_{n+1} )} , X^{(s_{n+1} )})$.
    \item[Uncertainty sampling:] For $n=n_0, \cdots, T-1$,  calculate $a_{n+1}$ according to \eqref{equ:def_uncertainty_sampling} with $\mathcal{A}$ replaced by $\mathcal{A}_{S_n}$. Next, we uniformly sample the index $s_{n+1}$ from $\{ i\in [T] \backslash S_{n} : a^{(i)}=a_{n+1} \}$. Let $ S_{n+1}= S_{n}\cup \{s_{n+1}\}$. The $(n+1)-$th pair and comparison result $(a_{n+1}, X_{n+1})$ is given by $(a^{(s_{n+1} )} , X^{(s_{n+1} )})$.
\end{itemize}
}
{
\subsection{Specifications for the Monte Carlo Experiments} For Figure~\ref{fig:kendall_tau},
we perform a Monte Carlo simulation with $100$ replications. For each replication, we sample \( 99 \) pairs of comparisons at random for initialization, and then perform sequential sampling for following different methods using the initialization and sampling method described in  Section~\ref{sec:real-data-sampling-spec}. We specify \( \mathbb{G}_{\bgtheta}(\bgSigma) = \operatorname{tr}(\bgSigma) \) for implementing \textrm{GI0} and \textrm{GI1} and $\bgTheta=[-3,3]^{p}$ to solve the MLE. 
}

\section{Additional Simulation Results}
In this section, we present additional simulation results.

\subsection{Additional Simulation Results for Simulation Study 1}
Let the true value $\bgtheta^*=(1,0)^T$. 
{The initial estimator $\widehat\bgtheta_0$, $n_0$, and experiments $\{a_1^0,\cdots a_{n_0}^0\}$ are selected according to Section \ref{sec:S1.1.2}. Let the sample size $n=50$.}
Define
\[
{Z}^j_1=\frac{ \sqrt{N}  (\widehat\theta_1^j -\theta^*_1) }{ \big(\ble_1^T \{\mathcal{I}^{\overline{\bgpi}_n }(\widehat\bgtheta^j_{n} ) \}^{-1} \ble_1 \big)^{1/2} } \text{ and } {Z}^j_2=\frac{ \sqrt{N}  (\widehat\theta_2^j -\theta^*_2) }{ \big(\ble_2^T \{\mathcal{I}^{\overline{\bgpi}_n }(\widehat\bgtheta^j_{n} )\}^{-1}  \ble_2 \big)^{1/2} },
\]
where $\widehat\bgtheta^j_{n}=(\widehat\theta_1^j,\widehat\theta_2^j)^T$ represents the MLE of $\bgtheta$ based on $j-$th Monte Carlo replication, and \(N = 1000\) is the number of Monte Carlo replications.
That is, $Z_1^j$ and $Z_2^j$ are i.i.d. copies of $Z_1=\frac{ \sqrt{N}  (\widehat\theta_1 -\theta^*_1) }{ \big(\ble_1^T \{\mathcal{I}^{\overline{\bgpi}_n }(\widehat\bgtheta_{n} ) \}^{-1} \ble_1 \big)^{1/2} }$ and $Z_2=\frac{ \sqrt{N}  (\widehat\theta_2 -\theta^*_2) }{ \big(\ble_2^T \{\mathcal{I}^{\overline{\bgpi}_n }(\widehat\bgtheta_{n} )\}^{-1}  \ble_2 \big)^{1/2} }$. 
In Figure~\ref{fig:CLT_verify}, we plot the histogram for $\{Z^j_1\}^N_{j=1}$ and $\{Z^j_2\}^N_{j=1}$ following \textrm{GI0} and \textrm{GI1}. 

\begin{figure}[ht]
    \centering
    \includegraphics[width=1\textwidth,height=6cm]{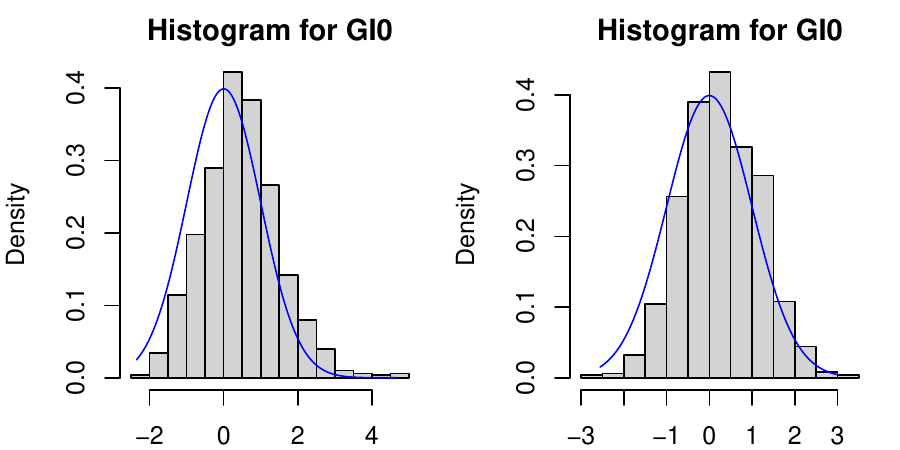}
    \includegraphics[width=1\textwidth,height=6cm]{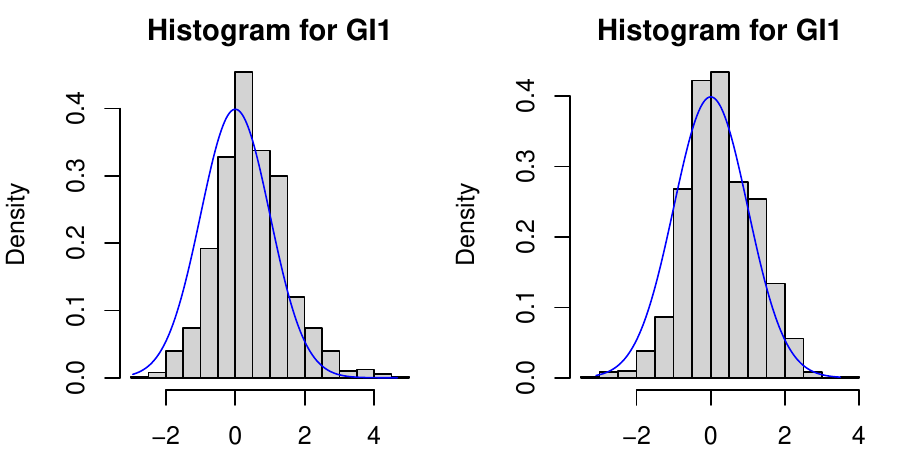}
    \caption{Histograms for $\{Z^j_1\}_{j=1}^N$ and $\{Z^j_2\}_{j=1}^N$ following \textrm{GI0} and \textrm{GI1},  and the density curve for the standard normal distribution. The upper left and bottom left panels show the histogram of $\{Z^j_1\}_{j=1}^N$ following \textrm{GI0} and \textrm{GI1}, respectively. The upper right and bottom right panels show the histogram of $\{Z^j_2\}_{j=1}^N$ following \textrm{GI0} and \textrm{GI1}, respectively. 
    }
    \label{fig:CLT_verify}
\end{figure}
In Figure~\ref{fig:CLT_verify}, the histogram closely approximates the standard normal density curve. This is consistent with Theorem~~\ref{thm:Asy_Normal_final}.
\subsection{Additional Simulation Results for Simulation Study 2}

The theoretical results in the manuscript assume that $p$ and $|\mathcal{A}|$ are fixed while the sample size $n$ grows large. In this section, we investigate the performance of the proposed method when $p$ and $|\mathcal{A}|$ are comparable with $n$, and this condition is violated. We investigate the performance of the proposed methods under a sequential rank aggregation problem assuming a BTL model.

 We consider the following simulation settings. Set $p=10$ or $50$. Entries of $\bgtheta^*$ are i.i.d. and generated from \(\mathcal{U}(-2, 2)\). $|\mathcal{A}|= 2(p+1)$, and $\mathcal{A}$ is sampled uniformly from all $4$-regular graphs (see Section~\ref{sec:graph-generation}). {The initial estimator and experiments are selected in the same way as those in Section~\ref{sec:details-initial} except that $n_0$ is set as $2p$ instead of $4p$.} 
\subsubsection{Empirical and Optimal Frequency}\label{sec:RA_SEA}
We plot the expected value of \(F_n = \mathbb{F}_{\bgtheta^*}(\overline{\bgpi}_n )-  \mathbb{F}_{\bgtheta^*}( {\bgpi}^*  )\) for different methods in Figure \ref{fig:Fn_fig} based on $N=1000$ Monte Carlo replications. {Here, the optimal proportion $\bgpi^*$ is computed according to Section~\ref{sec:detail-projected-gradient}. }
{From the Figure~\ref{fig:Fn_fig}, we see that $F_n$ is approaching zero when $n$ is large. This is consistent with Theorem~\ref{thm:empirical pi as converge}. However, it is far from zero when $p$ is comparable with $n$ (e.g., $p=50$ and $n=150$). This is expected as it becomes a high-dimensional problem under this setting.}

\begin{figure}[ht]
    \centering
    \begin{minipage}[c]{0.48\textwidth}
        \includegraphics[width=\textwidth]{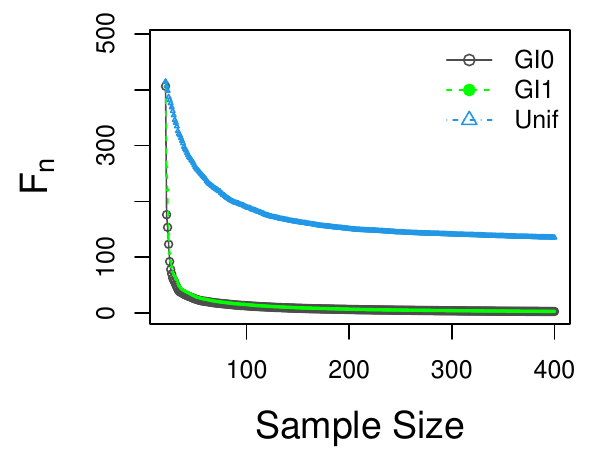}
    \end{minipage}
    \hfill %
    \begin{minipage}[c]{0.48\textwidth}
        \includegraphics[width=\textwidth]{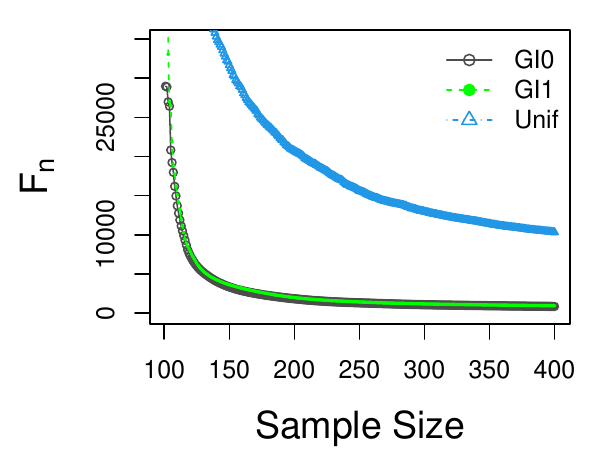}
    \end{minipage}
    \caption{ Comparison of $F_n$ for different selection methods (\textrm{GI0}, \textrm{GI1}, \text{Unif} selections) and different sample size $n$. The left panel and the right panel show $F_n$ with $p=10$ and $p=50$, respectively.
}
    \label{fig:Fn_fig}
\end{figure}

\subsubsection{Estimation Accuracy}\label{sec:RA_EA}

\begin{figure}[ht]
    \centering
    \begin{minipage}[c]{0.48\textwidth}
        \includegraphics[width=\textwidth]{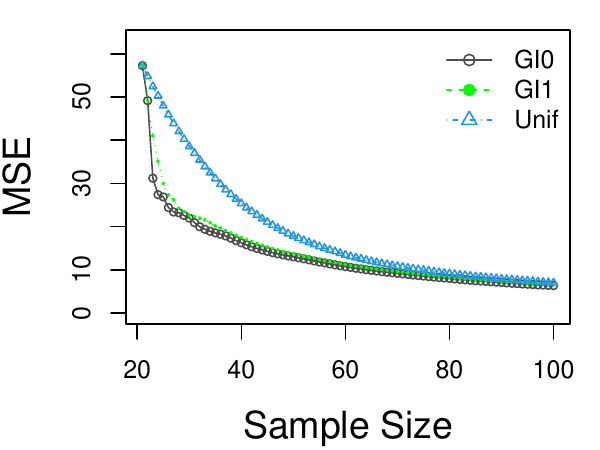}
    \end{minipage}
    \hfill %
    \begin{minipage}[c]{0.48\textwidth}
        \includegraphics[width=\textwidth]{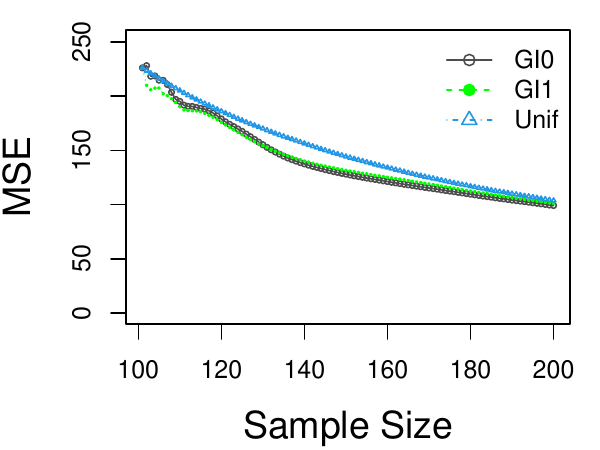}
    \end{minipage}
    \caption{ Comparison of performance of MSE for different selection methods (\textrm{GI0}, \textrm{GI1}, uniform selections) for the rank aggregation problem.\quad The left panel and the right panel show MSE with $p=10$ and $p=50$, respectively.
}
    \label{fig:MSE_fig}
\end{figure}

In Figure \ref{fig:MSE_fig}, we plot the MSE for the MLE at different sample size $n$ following different experiment selection methods based on $N=10000$ Monte Carlo replications.
The MSE is not close to zero when $p$ is relatively large compared to $n$, which is expected. However, \textrm{GI0} and \textrm{GI1} still perform much better when compared with \text{Unif}. %

\subsubsection{Impact of the Choice of $\bgTheta$}
{In our theoretical results, we assume the true parameter $\bgtheta^*$ is an inner point of $\bgTheta$. In this section, we study the impact of the choice of $\bgTheta$ on the estimation accuracy. We consider the following simulation setting: $p=50$, each element of $\bgtheta^*$ is sampled i.i.d. from $\mathcal{U}(-2,2)$,  $\mathcal{A}$ is randomly sampled from all $4-$regular graphs with $N=100$ Monte Carlo simulations. As a result, $|\mathcal{A}|= 102$.
We consider 4 choices of $\bgTheta$ when solving for the MLE: $\bgTheta=[-1,1]^p$, $\bgTheta=[-2,2]^p$, $\bgTheta=[-3,3]^p$ and $\bgTheta=[-5,5]^p$. The initial sample size is set to $n_0=p$. 

In Figure~\ref{fig:cube_RA}, we compare the Kendall's correlation of the MLE following \textrm{GI1} for different choices of $\bgTheta$, and obtain the following findings.}
\begin{enumerate}
    \item For cube 2 ($\bgTheta = [-2,2]^p$), it coincides with the data generation distribution $\mathcal{U}(-2,2)$. The Kendall's $\tau$ correlation is the largest among all cubes and sample sizes.
    \item For cube 1 ($\bgTheta = [-1,1]^p$), it does not satisfy the condition $\bgtheta^*\in \bgTheta$ for the theoretical results. For small sample size ($n<500$), it performs similarly as cube 2. However, for larger $n$, it's performance becomes worse.
    \item For cube 3 and cube 5 ($\bgTheta = [-3,3]^p$ and $\bgTheta = [-5,5]^p$), they cover the true model parameter, but are larger than the support of sampling distribution of $\bgtheta^*$. For small sample size, the larger the cube, the poorer the performance is. However, as the sample size increases, the performance becomes better than cube 1.
    \item Overall, the choice of $r$ in \(\bgTheta=[-r,r]^p\) does not seem affect the overall trend between Kendall's correlation and sample size.
\end{enumerate}

\begin{figure}[H]
\centering
\includegraphics[width=0.9\textwidth]{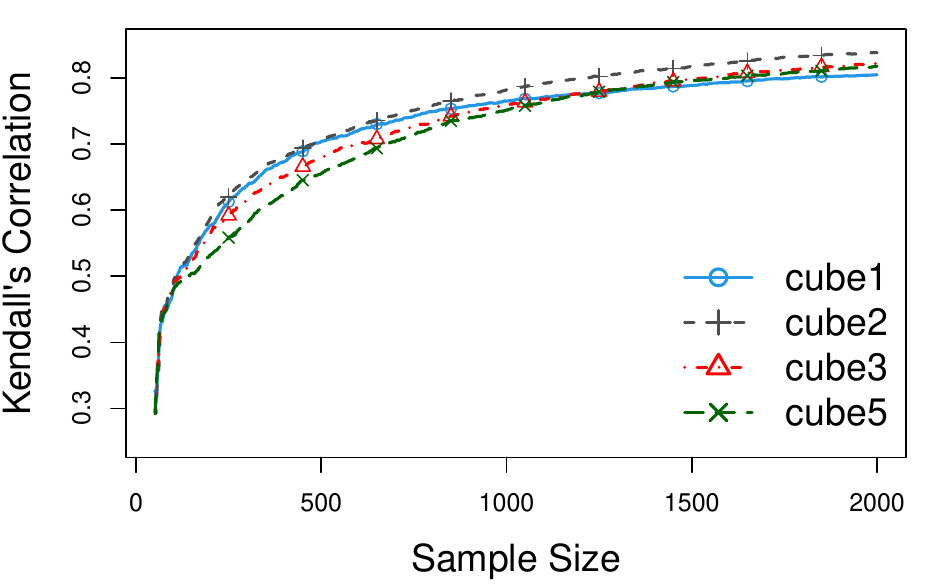}
\centering
\caption{Comparison of Performance of Different Compact Cubes with \textrm{GI1} Selection for the Rank Aggregation Problem.\quad  The curves for Cube1, Cube2, Cube3, and Cube5 represent the plot of Kendall's $\tau$ correlation versus sample size over compact cubes $\bgTheta = [-1,1]^p$, $\bgTheta = [-2,2]^p$, $\bgTheta = [-3,3]^p$, and $\bgTheta = [-5,5]^p$, respectively. %
}
\label{fig:cube_RA}
\end{figure}

\section{Preliminary Theoretical Results and Supporting Lemmas}
In this section, we present  preliminary theoretical results and supporting lemmas which are useful for the rest of the theoretical analysis. Moreover, they may be useful for other problems involving the analysis of stochastic processes, functions of matrices, and linear algebra for spaces indexed by a parameter.

\subsection{Useful Results for the Convergence of Stochastic Processes}
The next lemma extends the classic Kolmogorov's three-series theorem with relaxed moments and independence conditions. It is useful for proving almost sure convergence results for dependent stochastic processes.
\begin{lemma}[Modified Kolmogorov's three-series theorem]\label{lem:K two-series theorem}
Consider nested $\sigma-$fields $\mathcal{F}_{n}\subset \mathcal{F}_{n+1}, n\geq 0$. Let $\{X_n\}_{n=1}^\infty$ and $\{\varepsilon_n\}_{n=1}^\infty$ be two sequences of random variables, adaptive to $\{\mathcal{F}_{n}\}_{n=1}^\infty$, respectively. Consider a sequence of events $E_n$ such that
\begin{equation*}
    \mathbb{P}(\liminf_{n}E_n)=\mathbb{P}\left(\bigcup_{n=1}^\infty\bigcap^\infty_{m=n}E_m\right)=1.
\end{equation*}
If there exists $0<\gamma\leq 1$ such that, 
$$
\mathbb{E}[|X_n|^\gamma I_{E_n} \mid \mathcal{F}_{n-1}]\leq \varepsilon_{n-1}\ a.s.,\text{ and, }\sum_{n=0}^\infty\mathbb{E} \varepsilon_n<\infty,
$$
then $\sum_{n=1}^\infty X_n$ converges almost surely.
\end{lemma}

\begin{proof}[Proof of Lemma \ref{lem:K two-series theorem}]    
Let $S_N=\sum_{n=1}^NX_n$. It is sufficient to show that with probability 1,
$$
\lim_{m\to \infty} \sup_{n,l\geq m}|S_n - S_l|=0.
$$
Applying $C_r$ inequality (see 9.1.a in \cite{lin2010probability}), for any $0<\gamma\leq 1$, $k\geq 1$, we have
$$
\left|\sum_{i=1}^k X_{m+i}\right|^\gamma\leq \sum_{i=1}^k\left| X_{m+i}\right|^\gamma.
$$
For any $m\in \mathbb{N}$, and $\varepsilon>0$, applying $C_r$ inequality (see 9.1.a in \cite{lin2010probability}) and Markov inequality, we have
\begin{equation*}
\begin{aligned}
&\mathbb{P}\left(\sup_{n,l\geq m}|S_n - S_l| \geq 2{\varepsilon}  \right)\\
\leq &\mathbb{P}\left(  2\sup _{k \in \mathbb{N}}\left|\sum_{i=1}^k X_{m+i}\right| \geq 2{\varepsilon}   \right) \\
=&\mathbb{P}\left(\sup_{k \in \mathbb{N}}\left|\sum_{i=1}^k X_{m+i}\right|^\gamma \geq \varepsilon^\gamma \right) \\
\leq&\mathbb{P}\left( \sup_{k \in \mathbb{N}}\sum_{i=1}^k\left| X_{m+i}\right|^\gamma \geq \varepsilon^\gamma \right) \\
\leq&\mathbb{P}\left(\left\{\sup_{k \in \mathbb{N}}\sum_{i=1}^k\left| X_{m+i}\right|^\gamma \geq \varepsilon^\gamma\right\}\bigcap\left( \bigcap^\infty_{n=m+1}E_n \right)  \right)+\mathbb{P}\left(\overline{\bigcap^\infty_{n=m+1}{E_n}} \right) \\
\leq &\limsup_{k\to \infty}\mathbb{P}\left( \sum_{i=1}^k\left| X_{m+i}\right|^\gamma I{\big( \bigcap^\infty_{n=m+1}E_n \big) } \geq \varepsilon^\gamma \right)+\mathbb{P}\left(\overline{\bigcap^\infty_{n=m+1}{E_n}} \right) \\
\leq &\limsup_{k\to \infty}\mathbb{P}\left( \sum_{i=1}^k\left| X_{m+i}\right|^\gamma I{\big(  E_{m+i} \big) } \geq \varepsilon^\gamma \right)+\mathbb{P}\left(\overline{\bigcap^\infty_{n=m+1}{E_n}} \right) \\
\leq& \limsup _{k \rightarrow \infty} \frac{1}{\varepsilon^{\gamma} }  \sum_{i=1}^k \mathbb{E}[\mathbb{E}\left\{|X_{m+i}|^\gamma I_{E_{m+i}} \mid \mathcal{F}_{m+i-1}\right\}]+\mathbb{P}\left(\overline{\bigcap^\infty_{n=m+1}{E_n}} \right)\\
\leq& \frac{1}{\varepsilon^{\gamma} }  \sum_{i=1}^\infty \mathbb{E}\varepsilon_{m+i-1}+\mathbb{P}\left(\overline{\bigcap^\infty_{n=m+1}{E_n}} \right),
\end{aligned}
\end{equation*}
where we used the assumption $\mathbb{E}[|X_n|^\gamma I_{E_n} \mid \mathcal{F}_{n-1}]\leq \varepsilon_{n-1}$ for all $n$ for obtaining the last inequality.
Notice that
\begin{equation*}
    \lim_{m\to \infty }\mathbb{P}\left(\overline{\bigcap^\infty_{n=m+1}{E_n}} \right)=1-\lim_{m\to \infty}\mathbb{P}\left({\bigcap^\infty_{n=m+1}{E_n}} \right)=1-  \mathbb{P}\left(\bigcup_{m=1}^\infty \bigcap^\infty_{n=m+1}{E_n}  \right)=0.
\end{equation*}
Let $m\to \infty$, we obtain that for all $\varepsilon>0$,
\begin{equation*}
\begin{split}
& \mathbb{P}\left( \bigcap_{m\geq 1} \left\{\sup_{n,l\geq m}|S_n - S_l| \geq 2{\varepsilon} \right\} \right)\\
= & 
    \lim_{m\to \infty}\mathbb{P}\left(   \left\{\sup_{n,l\geq m}|S_n - S_l| \geq 2{\varepsilon} \right\} \right)\\
    \leq & \frac{1}{\varepsilon^{\gamma} }  \lim_{m\to \infty} \sum_{i=1}^\infty \mathbb{E}\varepsilon_{m+i-1}+\lim_{m\to \infty}\mathbb{P}\left(\overline{\bigcap^\infty_{n=m+1}{E_n}} \right)\\
    = & 0
\end{split}
\end{equation*}
This implies $ \mathbb{P}\left( \bigcap_{m\geq 1} \left\{\sup_{n,l\geq m}|S_n - S_l| \geq 2{\varepsilon} \right\} \right)=0$ and completes the proof.
\end{proof}
Next, we extends Theorem 2.19 in \cite{hall2014martingale} obtain a law of large number result for martingale differences which allows for adaptive experiment selection.
\begin{lemma}[Modified Theorem 2.19 in \cite{hall2014martingale}]\label{lem:SLLN}
    Let $\{X_n\}_{n=1}^\infty$ be a sequence of random variables and $\{\mathcal{F}_{n}\}_{n=1}^\infty$ be an increasing sequence of $\sigma-$fields with $X_n$ measurable with respect to $\mathcal{F}_n$ for all $n$. Let $\{a_n\}_{n=1}^{\infty}$ denote a sequence of discrete random variables, where each variable takes values from the set $\{1, 2, \ldots, k\}$. Let $X^1,\cdots,X^k$ be a sequence of random variables such that $\max_{1\leq a\leq k }\mathbb{E}|X^a|<\infty$. If the conditional distribution function of $X_n|\mathcal{F}_{n-1},a_{n}=a$ is the same as the distribution function $X^a$ with probability 1, then

\begin{equation}\label{equ:45}
n^{-1} \sum_{i=1}^n\left\{X_i-\mathbb{E}\left(X_i \mid \mathcal{F}_{i-1}\right)\right\}  \stackrel{\text { a.s. }}{\longrightarrow} 0.
\end{equation}
\end{lemma}
\begin{proof}[Proof of Lemma~\ref{lem:SLLN}]
Let $Y_n=X_nI_{\{|X_n|\leq n\}}$, $n\geq 1$. 

Note that $\mathbb{E}|X^a|<\infty$ for any $1\leq a\leq k$, and for any $x>0$, 
\begin{equation*}
\begin{split}
    &\mathbb{P}(|X_n| > x)=\mathbb{E}\ \mathbb{P}(|X_n| > x|\mathcal{F}_{n-1})=\mathbb{E}\sum_{a=1}^k \mathbb{P}(|X_n| > x|\mathcal{F}_{n-1},a_n=a) \mathbb{P}(a_n=a|\mathcal{F}_{n-1} )\\
    =& \mathbb{E}\sum_{a=1}^k \mathbb{P}(|X^a| > x ) \mathbb{P}(a_n=a|\mathcal{F}_{n-1} )\leq \sum_{a=1}^k\mathbb{P}(|X^a| > x )<\infty.
\end{split}
\end{equation*}
Similar to the proof of Theorem 2.19 in \cite{hall2014martingale}, %
we obtain that
\begin{equation*}
\begin{split}
    &\sum_{n=1}^\infty \frac{1}{n^2} \mathbb{E}[ \{Y_n-\mathbb{E}(Y_n|\mathcal{F}_{n-1}) \}^2 ]\leq 2 \sum_{n=1}^\infty \frac{1}{n^2} \int_{0<x\leq n} x \mathbb{P}(|X_n|>x)dx\\
    \leq&2 \sum_{a=1}^k \sum_{n=1}^\infty \frac{1}{n^2} \int_{0<x\leq n} x \mathbb{P}(|X^a|>x)dx\leq 4\sum_{a=1}^k \sum_{i=1}^\infty  \mathbb{P}(|X^a|>i-1) <\infty,
\end{split}
\end{equation*}

\begin{equation*}
n^{-1} \sum_{i=1}^n\left\{Y_i-\mathbb{E}\left(Y_i \mid \mathcal{F}_{i-1}\right)\right\} \stackrel{\text { a.s. }}{\longrightarrow} 0,
\end{equation*}
\begin{equation*}
    \sum_{n=1}^\infty \mathbb{P}(X_n\neq Y_n)=\sum_{n=1}^\infty \mathbb{P}(|X_n|>n)\leq \sum_{a=1}^k\sum_{n=1}^\infty\mathbb{P}(|X^a| > n )<\infty
\end{equation*}
and
\begin{equation}\label{equ:as1}
n^{-1} \sum_{i=1}^n\left\{X_i-\mathbb{E}\left(Y_i \mid \mathcal{F}_{i-1}\right)\right\} \stackrel{\text { a.s. }}{\longrightarrow} 0 .
\end{equation}
Notice that with probability 1, as $n\to \infty$,
\begin{equation*}
\begin{split}
    &\mathbb{E}(\left.|X_n| I(|X_n|>n)\right| \mathcal{F}_{n-1})\\
    = &\int_{n}^\infty \mathbb{P}(|X_n|>x\mid \mathcal{F}_{n-1})dx\\
    =&\int_{n}^\infty \sum_{a=1}^k\mathbb{P}(|X_n|>x\mid \mathcal{F}_{n-1}, a_n=a )\mathbb{P}(a_n=a|\mathcal{F}_{n-1})dx\\
    \leq& 
    \int_n^\infty\sum_{a=1}^k \mathbb{P}(|X^a|>x)dx\\
    = &\sum_{a=1}^k \mathbb{E}( |X^a| I(|X^a|>n) )\\
    \to & 0.
\end{split}
\end{equation*}
Thus, with probability 1,
\begin{equation}\label{eq:x-y-difference}
\begin{split}
    &n^{-1}\sum_{i=1}^n|\mathbb{E}(  X_i-Y_i \mid \mathcal{F}_{i-1})|\\
    \leq & n^{-1}\sum_{i=1}^n \mathbb{E}(\left.|X_i| I(|X_i|>i)\right| \mathcal{F}_{i-1})\\
    \leq & n^{-1}\sum_{i=1}^n\sum_{a=1}^k \mathbb{E}[\{ \left.|X_i| I(|X_i|>i)\right| \mathcal{F}_{i-1},a_i=a\} \mathbb{P}(a_i=a|\mathcal{F}_{i-1})]\\  
    \leq &\sum_{a=1}^k\frac{1}{n} \sum_{i=1}^n \mathbb{E}( |X^a| I(|X^a|>i) ).
\end{split}
\end{equation}
Because $\mathbb{E}|X^a|<\infty$, we know that  $\lim_{n\to\infty}\mathbb{E}( |X^a| I(|X^a|>n) )= 0$. Because the arithmetic mean of a sequence converges to the same limit as the sequence itself, we obtain that for all $a\in\{1,\cdots,k\}$
\begin{equation*}
    \lim_{n\to \infty}\frac{1}{n} \sum_{i=1}^n \mathbb{E}( |X^a| I(|X^a|>i) ) = 0.
\end{equation*}
In conclusion, we obtain, with probability $1$, that
\begin{equation*}
    \Big|n^{-1} \sum_{i=1}^n\mathbb{E}( X_i-  Y_i \mid \mathcal{F}_{i-1} ) \Big|\leq \sum_{a=1}^k\frac{1}{n} \sum_{i=1}^n \mathbb{E}( |X^a| I(|X^a|>i) ),
\end{equation*}
and the expression on the right-hand side is a deterministic sequence converging to $0$, which implies that
\begin{equation}\label{equ:47}
n^{-1} \sum_{i=1}^n\mathbb{E}( X_i-  Y_i \mid \mathcal{F}_{i-1} ) \stackrel{\text { a.s. }}{\longrightarrow} 0.
\end{equation}
Combining \eqref{equ:as1} and \eqref{equ:47}, we obtain \eqref{equ:45}.
\end{proof}

{Anscombe's theorem \citep{anscombe1952large} is a classic limit theorem for randomly indexed processes. We prove a multivariate version of Anscombe's theorem as follows, which generalizes the univariate Anscombe's theorem with Gaussian limit (see \cite{mukhopadhyay2012tribute}). }
\begin{theorem}[Multivariate Anscombe's theorem]\label{thm:mult_random_clt}
    Let $\{T_n\}_{n\geq 1}$ be a sequence of column random vectors and $\{W_n\}_{n\geq 1}$ be a sequence of positive definite matrices satisfying multivariate Anscombe's condition, namely, for every $\varepsilon>0$, $0<\gamma<1$ there exists some $\delta>0$ such that
\begin{equation*}
\limsup_{n \to \infty} \mathbb{P}\left(\max_{\left|n^{\prime}-n\right| \leq \delta n}\norm{T_{n^{\prime}}-T_n} \geq \varepsilon \lambda_{min}(W_n)\right)<\gamma
\end{equation*}
hold. Moreover, we assume that
\begin{equation*}
    \sup_{n\geq 1}\frac{\lambda_{max}(W_n)}{\lambda_{min}(W_n)} <\infty,
\end{equation*}
and there exists positive sequence $\rho_n\to \infty$ such that
\begin{equation}\label{lim:rho_W}
     \rho_nW_n \inP W,
\end{equation}
where $W$ is a real positive definite matrix.

Assume that there exists a real column vector $\theta\in \mathbb{R}^p$ and as $n\to \infty$
\begin{equation*}
    W_n^{-1}(T_n-\theta) \inD N_p(\bm{0}_p,I_p).
\end{equation*}
Consider $\{N_n\}_{n\geq 1}$, a sequence of positive integer-valued stopping times defined on the same probability space where $\{T_n\}_{n\geq 1}$ is defined. Let $\{r_n\}_{n\geq 1}$ be an increasing sequence of positive integers such that $\lim_{n\to \infty }r_n= \infty$. If $N_n/r_n\to 1$ in probability as $n\to \infty$, then as $n\to \infty$
\begin{equation*}
    W_{r_n}^{-1}(T_{N_n}-\theta) \inD N_p(\bm{0}_p,I_p).
\end{equation*}
\end{theorem}

\begin{proof}[Proof of Theorem \ref{thm:mult_random_clt}]
For any $b\in \mathbb{R}^p$ such that $\norm{b}=1$, we have
\begin{equation*}
    \frac{b^T(T_n-\theta)}{\norm{W_nb}}=\frac{(W_nb)^T}{\norm{W_nb}}W_n^{-1}(T_n-\theta).
\end{equation*}
Let $h_n=\frac{W_nb }{\norm{W_nb}} = \frac{\rho_n W_nb }{\norm{\rho_n W_nb}}$ and $h=\frac{Wb }{\norm{Wb}}$. By the continuous mapping theorem, we know that $h_n\to h$ in probability.

Recall that $W_n^{-1}(T_n-\theta) \inD N_p(\bm{0}_p,I_p).$ Hence, by Slutsky theorem, as $n\to \infty$
\begin{equation*}
\begin{split}
    &\frac{(W_{n}b)^T}{\norm{W_{n}b}}W_{n}^{-1}(T_{n}-\theta)
 = h^TW_{n}^{-1}(T_{n}-\theta)+\Big(h_{n}-h\Big)^T W_{n}^{-1}(T_{n}-\theta)
    \inD  N(0,1).
\end{split}
\end{equation*}
In conclusion, we know that
\begin{equation*}
    \frac{b^TT_n-b^T\theta}{\norm{W_nb}} \inD N(0,1).
\end{equation*}
Note that $N_n/r_n\to 1$ in probability and
\begin{equation*}
    \mathbb{P}\Big(\max_{|n'-n|\leq \delta n} |b^TT_n-b^T\theta|\geq \varepsilon \norm{W_n b} \Big) \leq \mathbb{P} \Big(\max_{|n'-n|\leq \delta n} \norm{T_n-\theta}\geq \varepsilon \cdot  \lambda_{min}(W_n) \Big),
\end{equation*}
and
\begin{equation*}
    \limsup_{n\to \infty}\mathbb{P} \Big(\max_{|n'-n|\leq \delta n} \norm{T_n-\theta}\geq \varepsilon \cdot  \lambda_{min}(W_n) \Big)\leq \gamma.
\end{equation*}
Applying Theorem 3.1 in \cite{mukhopadhyay2012tribute}, we obtain that for any $b\neq 0$ and as $n\to \infty$
\begin{equation}\label{lim:random_clt}
    \frac{b^TT_{N_n}-b^T\theta}{\norm{W_{r_n}b}} \inD N(0,1).
\end{equation}
Furthermore, by \eqref{lim:rho_W}, we know that
\begin{equation}\label{equ:S6.5}
     \rho_{r_n}\norm{W_{r_n}b} \to \norm{Wb}.
\end{equation}
Thus, we know that $\rho_{r_n} b^T(T_{N_n}- \theta)=O_p(1)$ for any $b\in \mathbb{R}^p$, which implies 
\begin{equation}\label{equ:S7}
    \rho_{r_n} (T_{N_n}- \theta)=O_p(1).
\end{equation}
Note that
\begin{equation*}
\begin{split}
&\norm{W^{-1}_{r_n}(T_{N_n}- \theta)}\\
\leq & \frac{\norm{ T_{N_n}-\theta }}{\lambda_{min}(W_{r_n}) }\\
\leq&\sum_{i=1}^p\frac{|\ble_i^T(T_{N_n}-\theta)|}{\norm{W_{r_n}\ble_i}}\sup_{n\geq 1} \kappa(W_n)\\
= & O_p(1),
\end{split}    
\end{equation*}
where
\begin{equation*}
\kappa(W_n)=\frac{\lambda_{max}(W_n)}{\lambda_{min}(W_n)}.
\end{equation*}
By Cramér–Wold theorem (see \cite{billingsley1999convergence} p383), it is sufficient to show that for all $h\in \mathbb{R}^p$ such that $\norm{h}=1$, we have
\begin{equation*}
    h^TW^{-1}_{r_n}(T_{N_n}-\theta)\inD N(0,1).
\end{equation*}
Set $b_n=\frac{W^{-1}_{r_n}h}{\norm{W^{-1}_{r_n}h}}$, and $b=\frac{W^{-1} h}{\norm{W^{-1} h}}$. We have $h=\frac{W_{r_n} b_n}{\norm{W_{r_n} b_n}}$. By the continuous mapping theorem, we know that 
\(b_n\to b\) in probability. Notice that
\begin{equation*}
    \left| \frac{\norm{W_{r_n}b_n }}{\norm{W_{r_n}b }} -1 \right|\leq \kappa(W_{r_n})\norm{b_n-b}\to 0,
\end{equation*}
which implies $\frac{\norm{W_{r_n}b_n }}{\norm{W_{r_n}b }} \to 1$ in probability. Combine this with \eqref{lim:random_clt}, we obtain that as $n\to \infty$
\begin{equation*}
\begin{split}
    &h^TW^{-1}_{r_{n}}(T_{N_{n}}- \theta)\\
    = &\frac{ b_n^T (T_{N_{n}}- \theta)}{\norm{W_{r_{n}}b  } }  \frac{\norm{W_{r_n}b }}{\norm{W_{r_n}b_n }} \\
    =&\frac{ b_n^T (T_{N_{n}}- \theta)}{\norm{W_{r_{n}}b  } } (1+o_p(1))\\
    = &\Big\{\frac{ b^T  (T_{N_{n}}- \theta)}{ \norm{W_{r_{n}}b  } }+\frac{(b_n- b)^T \rho_{r_n}(T_{N_{n}}- \theta)}{\rho_{r_n}\norm{W_{r_{n}}b  } }\Big\}(1 +o_p(1)).
\end{split}
\end{equation*}
Combining \eqref{equ:S6.5}, \eqref{equ:S7} and $b_n\to b$ in probability, we know that
\begin{equation*}
    \frac{(b_n- b)^T \rho_{r_n}(T_{N_{n}}- \theta)}{\rho_{r_n}\norm{W_{r_{n}}b  } }=o_{p}(1),
\end{equation*}
which implies that
\begin{equation*}
    h^TW^{-1}_{r_{n}}(T_{N_{n}}- \theta)=\Big\{\frac{ b^T (T_{N_{n}}- \theta)}{\norm{W_{r_{n}}b  } } +o_p(1)\Big\}(1+o_p(1))\inD N(0,1).
\end{equation*}
Thus, we know that for any $h\neq 0$, as $n\to \infty$
\begin{equation*}
    h^TW^{-1}_{r_n}(T_{N_n}-\theta)\inD N(0,\norm{h}^2).
\end{equation*}
By Cramér–Wold theorem (see \cite{billingsley1999convergence} p383), we complete the proof of Theorem \ref{thm:mult_random_clt}. 
\end{proof}

\subsection{Results regarding Functions of Matrices}
In this section, we provide results on derivatives of functions of matrices, and properties on functions of a convex combination of matrices.
\begin{lemma}\label{lem:Differentiation inverse Matrix}
Let $\mathcal{I}_a, a \in \mathcal{A}$ be a sequence of positive semidefinite matrix. Assume $\pi_0(a)\geq 0, a\in \mathcal{A}$ (not necessary that $\bgpi_0\in \ShatA$) such that $\sum_{a'\in \mathcal{A} }  \pi_0(a') \mathcal{I}_{a'}$ is a real positive definite matrix. 
Then, for all $a\in \mathcal{A}$  we have
\begin{equation}\label{equ:derivative_pi(a)}
    \left.\frac{\partial (\sum_{a'\in \mathcal{A}}  \pi(a') \mathcal{I}_{a'} )^{-1} }{\partial \pi(a)}\right|_{\bgpi=\bgpi_0} =-\Big(\sum_{a'\in  \mathcal{A}} \pi_0(a') \mathcal{I}_{a'} \Big)^{-1}\mathcal{I}_a \Big(\sum_{a'\in \mathcal{A}}  \pi_0(a') \mathcal{I}_{a'} \Big)^{-1}.
\end{equation}
\end{lemma}
\begin{proof}[Proof of Lemma \ref{lem:Differentiation inverse Matrix}]
By definition, 
\begin{equation*}
    \left.\frac{\partial (\sum_{a'\in \mathcal{A}}  \pi(a') \mathcal{I}_{a'} )^{-1} }{\partial \pi(a)}\right|_{\bgpi=\bgpi_0} = \lim_{\varepsilon\to 0}  \frac{   (\sum_{a'\in \mathcal{A}}  \pi_0(a') \mathcal{I}_{a'} + \varepsilon \mathcal{I}_{a})^{-1} -  (\sum_{a'\in \mathcal{A}} \pi_0(a') \mathcal{I}_{a'} )^{-1}   }{ \varepsilon }. 
\end{equation*}
Because $\sum_{a'\in \mathcal{A}} \pi_0(a') \mathcal{I}_{a'}$ is a positive definite matrix, then for small enough $\varepsilon$, the inverse of $\sum_{a'\in \mathcal{A}} \pi_0(a') \mathcal{I}_{a'}+\varepsilon \mathcal{I}_a$ exists. Furthermore,
\begin{equation*}
    (\sum_{a'\in \mathcal{A}} \pi_0(a') \mathcal{I}_{a'} + \varepsilon \mathcal{I}_{a})^{-1} - (\sum_{a'\in \mathcal{A}} \pi_0(a') \mathcal{I}_{a'}  )^{-1} = -\varepsilon (\sum_{a'\in \mathcal{A}} \pi_0(a') \mathcal{I}_{a'} + \varepsilon \mathcal{I}_{a})^{-1} \mathcal{I}_a (\sum_{a'\in \mathcal{A}} \pi_0(a') \mathcal{I}_{a'} )^{-1},
\end{equation*}
and 
\begin{equation*}
    \lim_{\varepsilon\to 0}(\sum_{a'\in \mathcal{A}} \pi_0(a') \mathcal{I}_{a'} + \varepsilon \mathcal{I}_{a})^{-1}  =  (\sum_{a'\in \mathcal{A}} \pi_0(a') \mathcal{I}_{a'} )^{-1},
\end{equation*}
which implies \eqref{equ:derivative_pi(a)}. 
\end{proof}

Next, we will define and derive the Gateaux derivative of the criteria function $\Phi_q$. For a real positive definite matrix $\bgSigma$, recall that
\begin{equation*}
    \Phi_0( \bgSigma )= \log|\bgSigma|,\ \Phi_q(\bgSigma)=\operatorname{tr}\bgSigma^q, 0<q < 1,
\end{equation*}
and
\begin{equation*}
    \Phi_q( \bgSigma )= (\operatorname{tr} ( \bgSigma^{q}) )^{1/q},q\geq 1. 
\end{equation*}
The Gateaux derivative $\nabla_{\blH} \Phi_q(\bgSigma)$ of $\Phi_q$ at $\bgSigma$ in direction $\blH$, which is a symmetric matrix, is defined as 
\begin{equation}\label{equ:Gateaux derivative}
\nabla_{\blH}\Phi_q(\bgSigma)=\lim_{\varepsilon\to 0} \frac{ \Phi_q(\bgSigma+\varepsilon \blH)-\Phi_q(\bgSigma ) }{\varepsilon} = \left.\frac{d}{d\varepsilon }   \Phi_q(\bgSigma+\varepsilon \blH)  \right|_{\varepsilon=0}.
\end{equation}
If the limit specified in \eqref{equ:Gateaux derivative} exists for all symmetric matrices $\blH$, we says that $\Phi_q$ is Gateaux differentiable at $\bgSigma$.

{
The next lemma provides the Gateaux derivative of $\Phi_q$. This lemma allows non-integer values for $q$, and is thus more general than a similar result in \cite{yang2013optimal}.
}

\begin{lemma}\label{lem:phi_q}
$\Phi_q$ is Gateaux differentiable at any real positive definite matrix $\bgSigma$ for any $q\geq 0$. Moreover, we have
\begin{equation}\label{equ:Phi_q_Gateaux}
    \nabla_{\blH} \Phi_q(\bgSigma) = 
    \begin{cases}
    \operatorname{tr} (\blH \bgSigma^{-1} ), &\text{if } q=0,  \\ 
    q \cdot \tr \big( \bgSigma^{q-1} \blH \big), &\text{if } 0<q<1,\\
    (\operatorname{tr}\ \bgSigma^q  )^{1/q-1}\cdot \operatorname{tr}(\bgSigma^{q-1}\blH), &\text{if } q\geq 1,
    \end{cases}
\end{equation}
and
\begin{equation}\label{equ:Phi_q_Gateaux_pi}
    \frac{\partial\Phi_q( \mathcal{I}^{-{\bgpi}})   }{\partial \pi(a)}  = 
    \begin{cases}
    -\operatorname{tr} (  \mathcal{I}^{-{\bgpi}} \mathcal{I}_a), &\text{if } q=0,  \\
    -q\cdot \operatorname{tr}\Big( (\mathcal{I}^{-{\bgpi}})^{q+1}\mathcal{I}_a \Big), &\text{if } 0<q<1,\\
    -\Big[ \operatorname{tr}\Big(    \big( \mathcal{I}^{-{\bgpi}}\big)^{ q} \Big) \Big] ^{1/q-1}\cdot \operatorname{tr}\Big( (\mathcal{I}^{-{\bgpi}})^{q+1}\mathcal{I}_a \Big), &\text{if } q\geq 1,
    \end{cases}
\end{equation}
where $\mathcal{I}^{\bgpi}=\sum_{a\in \mathcal{A}} \pi(a)\mathcal{I}_a$ and $\mathcal{I}^{-{\bgpi}}=\big\{\sum_{a\in \mathcal{A}} \pi(a)\mathcal{I}_a\big\}^{-1}$.

\end{lemma}

\begin{remark}
    Based on \eqref{equ:Phi_q_Gateaux}, and the Riesz representation theorem over the Hilbert space of symmetric matrix, for any positive definite $\bgSigma$, there exists unique symmetric matrix $\nabla \Phi_q(\bgSigma)=\{ \frac{\partial}{\partial \Sigma_{ij}} \Phi_q(\bgSigma) \}_{1\leq i,j \leq n}$, such that for any symmetric matrix $\blH$ of comparable size, 
\begin{equation*}
    \nabla_{\blH} \Phi_q(\bgSigma)=\innerpoduct{\nabla  \Phi_q(\bgSigma)}{\blH}. 
\end{equation*}   
\end{remark}

\begin{proof}[Proof of Lemma \ref{lem:phi_q}]
    Let $q=0$. By the definition of the Gateaux derivative, for any symmetric $\blH$ and positive definite matrix $\bgSigma$, 
\begin{equation*}
\begin{split}
&\nabla_{\blH}\Phi_0(\bgSigma)=\lim_{\varepsilon\to 0} \frac{ \Phi_0(\bgSigma+\varepsilon \blH)-\Phi_0(\bgSigma ) }{\varepsilon}\\
=&\lim_{\varepsilon\to 0}\frac{1}{\varepsilon}\log |I+\varepsilon \blH \bgSigma^{-1} |=\lim_{\varepsilon\to 0}\frac{ \log (1+\varepsilon \lambda_{1})+\log(1+\varepsilon \lambda_{2})+\cdots +\log (1+\varepsilon \lambda_{n}) }{\varepsilon} \\
=&(\lambda_1+\lambda_2+\cdots+\lambda_n)=\operatorname{tr}(\blH\bgSigma^{-1} ).
\end{split}
\end{equation*}
where $\lambda_{1},\lambda_{2},\cdots, \lambda_{n}$ denote all eigenvalues of $\blH\bgSigma^{-1}$ counting multiplicity.

When $q$ is a positive integer, by expanding $(\bgSigma+\varepsilon \blH)^q$, we have
\begin{equation*}%
    \operatorname{tr} ((\bgSigma+\varepsilon \blH)^q)-\operatorname{tr} (\bgSigma^q) = \varepsilon \cdot q \cdot \operatorname{tr}(\bgSigma^{q-1} \blH ) + o(\varepsilon).
\end{equation*}
Note that when $q$ is a positive integer
\begin{equation}\label{equ:phi_q}
\begin{split}
&\frac{ \Phi_q(\bgSigma+\varepsilon \blH)-\Phi_q(\bgSigma ) }{\varepsilon}\\
=&\frac{1}{ \varepsilon } (\operatorname{tr} \ \bgSigma^q)^{1/q} \left[ \Big(1+ \big( \operatorname{tr}(\bgSigma+\varepsilon \blH)^q-\operatorname{tr} (\bgSigma^q) \big)/{ \operatorname{tr} (\bgSigma^q)}  \Big)^{1/q} -1\right]\\
=&  \frac{1}{\varepsilon \cdot q} (\operatorname{tr} (\bgSigma^q) )^{1/q-1}\cdot \Big[ \operatorname{tr} ((\bgSigma+\varepsilon \blH)^q)-\operatorname{tr} (\bgSigma^q)\Big] +o(1)\\
=&   (\operatorname{tr} (\bgSigma^q) )^{1/q-1}\cdot \operatorname{tr}(\bgSigma^{q-1}\blH)+o(1).
\end{split}
\end{equation}
Thus, when $q$ is a positive integer, \eqref{equ:Phi_q_Gateaux} holds.

Now, consider the case when $q>0$ and $q$ is not an integer. Because we can not expand $(\bgSigma+\varepsilon \blH)^q$ and due to the lack of commutative between $\bgSigma$ and $\blH$, we need some more complicated techniques.
Assume $\bgSigma$ is of size $n\times n$. Let $\lambda_1(\varepsilon)\geq\lambda_2(\varepsilon) \geq \cdots \geq  \lambda_n (\varepsilon)$ be all the eigenvalues of $\bgSigma+\varepsilon \blH$. Denote the corresponding eigenvectors by $\blu_1(\varepsilon),\cdots,\blu_n(\varepsilon)$.

Let $\lambda_i=\lambda_i(0)$, and $\blu_i=\blu_i(0)$ for $1\leq i\leq n$. Set $\lambda_0=-\infty, \lambda_{n+1}=\infty$.

Let $\lambda$ be an eigenvalue of $\bgSigma+\varepsilon \blH$, there exists $0\leq r<s\leq n+1$ such that
\begin{equation*}
    \lambda_{r-1}>\lambda=\lambda_r=\cdots =\lambda_s > \lambda_{s+1}. 
\end{equation*}
Let $d=s-r+1$, $\blU_\lambda=[\blu_r,\blr_{r+1},\cdots,\blu_s]$ and $\blU_\lambda(\varepsilon)=[\blu_r(\varepsilon),\blu_{r+1}(\varepsilon),\cdots, \blu_s(\varepsilon)]$.

By Wely's inequality, $|\lambda_i(\varepsilon)-\lambda|\leq |\varepsilon| \norm{\blH}_{op}$.

Notice that when $|\varepsilon|$ is small enough, we have
\begin{equation}\label{equ:57}
\begin{split}
    &\operatorname{tr}( \blU^T_\lambda(\varepsilon) (\bgSigma+\varepsilon \blH)^q \blU_\lambda(\varepsilon))-\operatorname{tr}( \blU^T_\lambda  \bgSigma^q \blU_\lambda )\\
    =&(\lambda^q_r(\varepsilon)-\lambda^q )+\cdots+(\lambda^q_s(\varepsilon)-\lambda^q )\\
    =&\lambda^q\left[  \left(\Big(1+\frac{\lambda_r(\varepsilon)-\lambda }{\lambda} \Big)^q-1\right) + \cdots +\left(\Big(1+\frac{\lambda_s(\varepsilon)-\lambda }{\lambda} \Big)^q-1\right)   \right]\\
    =&q\lambda^{q-1} ( (\lambda_r(\varepsilon) -\lambda) +\cdots+(\lambda_s(\varepsilon) -\lambda) )+o(\varepsilon),
\end{split}
\end{equation}
and
\begin{equation}\label{equ:58}
\begin{split}
    &( (\lambda_r(\varepsilon) -\lambda) +\cdots+(\lambda_s(\varepsilon) -\lambda) )\\
    =&\operatorname{tr}( \blU^T_\lambda(\varepsilon) (\bgSigma+\varepsilon \blH) \blU_\lambda(\varepsilon))-\operatorname{tr}( \blU^T_\lambda  \bgSigma  \blU_\lambda )\\
    =&\varepsilon \cdot \operatorname{tr}( \blU^T_\lambda(\varepsilon)   \blH \blU_\lambda(\varepsilon)) + \operatorname{tr}( \blU^T_\lambda(\varepsilon)  \bgSigma  \blU_\lambda(\varepsilon) ) - \operatorname{tr}( \blU^T_\lambda   \bgSigma  \blU_\lambda  ).
\end{split}
\end{equation}
To proceed, we first prove the following equation,
\begin{equation}\label{claim:ind}
    \operatorname{tr}( (\bgSigma-\lambda I) (\blU_\lambda(\varepsilon)-\blU_\lambda)(\blU_\lambda(\varepsilon)-\blU_\lambda)^T )=\operatorname{tr}( \blU^T_\lambda(\varepsilon)  \bgSigma  \blU_\lambda(\varepsilon) ) - \operatorname{tr}( \blU^T_\lambda   \bgSigma  \blU_\lambda  ).
\end{equation}
\eqref{claim:ind} is justified by the following matrix calculation,
\begin{equation*}
\begin{split}
   &\operatorname{tr}\big( (\bgSigma-\lambda I) (\blU_\lambda(\varepsilon)-\blU_\lambda)(\blU_\lambda(\varepsilon)-\blU_\lambda)^T \big)-\big[\operatorname{tr}( \blU^T_\lambda(\varepsilon)  \bgSigma  \blU_\lambda(\varepsilon) ) - \operatorname{tr}( \blU^T_\lambda   \bgSigma  \blU_\lambda  )\big]\\
   =&\operatorname{tr}\big( (\blU_\lambda(\varepsilon)-\blU_\lambda)^T (\bgSigma-\lambda I) (\blU_\lambda(\varepsilon)-\blU_\lambda) \big)- \big[\operatorname{tr}( \blU^T_\lambda(\varepsilon)  (\bgSigma-\lambda I)  \blU_\lambda(\varepsilon) ) - \operatorname{tr}( \blU^T_\lambda   (\bgSigma-\lambda I) \blU_\lambda  )\big]\\
   =&-2\operatorname{tr}\big( \blU^T_\lambda(\varepsilon)  (\bgSigma-\lambda I) \blU_\lambda  \big)-2\operatorname{tr}\big( \blU^T_\lambda   (\bgSigma-\lambda I) \blU_\lambda  \big)=0, \text{ because $(\bgSigma-\lambda I) \blU_\lambda =0$.}
\end{split}
\end{equation*}
Note that $\bgSigma-\lambda I = (I-\blU_\lambda \blU^T_\lambda)(\bgSigma-\lambda I)(I-\blU_\lambda \blU^T_\lambda) $, $-\norm{\bgSigma-\lambda I}_{op} I\preceq  \bgSigma-\lambda I \preceq \norm{\bgSigma-\lambda I}_{op} I$, and  $(\blU_\lambda(\varepsilon)-\blU_\lambda)(\blU_\lambda(\varepsilon)-\blU_\lambda)^T$ is positive semidefinite, we obtain
\begin{equation}\label{equ:60}
\begin{split}
    &|\operatorname{tr}\big( (\bgSigma-\lambda I) (\blU_\lambda(\varepsilon)-\blU_\lambda)(\blU_\lambda(\varepsilon)-\blU_\lambda)^T \big)|\\
    \leq& \norm{\bgSigma-\lambda I}_{op} \cdot \operatorname{tr}\big( (I-\blU_\lambda \blU^T_\lambda) (\blU_\lambda(\varepsilon)-\blU_\lambda)(\blU_\lambda(\varepsilon)-\blU_\lambda)^T(I-\blU_\lambda \blU^T_\lambda) \big)\\
    =&\norm{\bgSigma-\lambda I}_{op}\cdot  \operatorname{tr}\big((I-\blU_\lambda \blU^T_\lambda) \blU_\lambda(\varepsilon) \blU^T_\lambda(\varepsilon) \big)\\
    =&\norm{\bgSigma-\lambda I}_{op}\cdot   (d- \norm{\blU^T_\lambda  \blU_\lambda(\varepsilon)}^2_F)\\
    =&\norm{\bgSigma-\lambda I}_{op}\cdot   (d- \norm{\cos \bgTheta ( \blU_\lambda,  \blU_\lambda(\varepsilon) )}_F^2 )\\
    =&\norm{\bgSigma-\lambda I}_{op}\cdot \norm{\sin \bgTheta ( \blU_\lambda,  \blU_\lambda(\varepsilon) )}_F^2,
\end{split}
\end{equation}
where $\norm{\cos \bgTheta ( \blU_\lambda,  \blU_\lambda(\varepsilon) )}_F= \norm{\blU^T_\lambda  \blU_\lambda(\varepsilon)}_F$ is due to the definition of principal angles between column spaces of $\blU_{\lambda}$ and $\blU_{\lambda}( \varepsilon)$ (see \cite{yu2015useful}).

By Davis-Kahan theorem (see \cite{yu2015useful}), we obtain
\begin{equation}\label{equ:61}
\|\sin \bgTheta(  \blU_\lambda, \blU_\lambda(\varepsilon))\|_{ {F}} \leq \frac{2    \| (\bgSigma+\varepsilon \blH)-\bgSigma \|_{  {F}} }{\min \left(\lambda_{r-1}-\lambda_r, \lambda_s-\lambda_{s+1}\right)}=O(\varepsilon).
\end{equation}
Combining \eqref{equ:57},\eqref{equ:58},\eqref{claim:ind},\eqref{equ:60} and \eqref{equ:61}, we obtain
\begin{equation*}
\begin{split}
    &( (\lambda_r(\varepsilon) -\lambda) +\cdots+(\lambda_s(\varepsilon) -\lambda) )=\varepsilon \cdot \operatorname{tr}( \blU^T_\lambda(\varepsilon)   \blH \blU_\lambda(\varepsilon))  +O(\varepsilon^2), \text{ and}\\
    &\operatorname{tr}( \blU^T_\lambda(\varepsilon) (\bgSigma+\varepsilon \blH)^q \blU_\lambda(\varepsilon))-\operatorname{tr}( \blU^T_\lambda  \bgSigma^q \blU_\lambda )
    =\varepsilon \cdot q \lambda^{q-1} \operatorname{tr}( \blU^T_\lambda(\varepsilon)   \blH \blU_\lambda(\varepsilon)) +o(\varepsilon).
\end{split}
\end{equation*}
Let $\blP_{\varepsilon}=\blU_\lambda(\varepsilon)\blU^T_\lambda(\varepsilon)$ and $\blP =\blU_\lambda \blU^T_\lambda $, we know that
\begin{equation}
\begin{split}
    &\norm{\blP_{\varepsilon}-\blP }_F^2=\operatorname{tr}(\blP_{\varepsilon}-2\blP_{\varepsilon}\blP+\blP)=2(d-\operatorname{tr}(\blP_{\varepsilon} \blP) )=2(d- \norm{ \blU^T_\lambda \blU_\lambda(\varepsilon)}_F^2 )\\
    =&2(d-\norm{\cos \bgTheta ( \blU_\lambda,  \blU_\lambda(\varepsilon) )}^2_F )=2\norm{\sin \bgTheta ( \blU_\lambda,  \blU_\lambda(\varepsilon) )}^2_F,
\end{split}
\end{equation}
and $\norm{\blP_{\varepsilon}-\blP }_F=\sqrt{2}\norm{\sin \bgTheta ( \blU_\lambda,  \blU_\lambda(\varepsilon) ) }_F=O(\varepsilon)$. Due to $\bgSigma \blU_{\lambda}=\lambda \blU_{\lambda}$, we know that $\bgSigma^{q-1} \blU_{\lambda}=\lambda^{q-1} \blU_{\lambda}$, which means that
\begin{equation*}
\begin{split}
    &( (\lambda_r(\varepsilon) -\lambda) +\cdots+(\lambda_s(\varepsilon) -\lambda) )=\varepsilon \cdot q \lambda^{q-1} \operatorname{tr}(\blH \blU_\lambda(\varepsilon) \blU^T_\lambda(\varepsilon) ) +o(\varepsilon)\\
    =&\varepsilon \cdot q \lambda^{q-1} \operatorname{tr}(\blH \blU_\lambda \blU^T_\lambda) +o(\varepsilon)=\varepsilon \cdot q \lambda^{q-1} \operatorname{tr}(\blU^T_\lambda\blH \blU_\lambda ) +o(\varepsilon)=\varepsilon \cdot  q \cdot  \operatorname{tr}(\bgSigma^{q-1}\blH \blU_\lambda \blU^T_\lambda)+o(\varepsilon).
\end{split}
\end{equation*}
This leads to
\begin{equation}\label{equ:trace_dir}
\begin{split}
    &\operatorname{tr} ((\bgSigma+\varepsilon \blH)^q)-\operatorname{tr} (\bgSigma^q)=\sum_{\lambda} \operatorname{tr}( \blU^T_\lambda(\varepsilon) (\bgSigma+\varepsilon \blH)^q \blU_\lambda(\varepsilon))-\operatorname{tr}( \blU^T_\lambda  \bgSigma^q \blU_\lambda )\\
    =&\sum_{\lambda}\varepsilon \cdot q \cdot  \operatorname{tr}(\bgSigma^{q-1}\blH \blU_\lambda \blU^T_\lambda)+o(\varepsilon)=\varepsilon \cdot q \cdot  \operatorname{tr}(\bgSigma^{q-1}\blH  )+o(\varepsilon),
\end{split}
\end{equation}
where the last equation holds due to $\sum_{\lambda} \blU_\lambda \blU_\lambda^T = I$.

Thus, when $0<q<1$, we know that
\[
\nabla_{\blH} \Phi_q(\bgSigma) = q \cdot \tr \big( \bgSigma^{q-1} \blH \big).
\]

Combining the first two equations in \eqref{equ:phi_q} with the equation \eqref{equ:trace_dir}, if $q\geq 1$, we know that 
\[
\nabla_{\blH} \Phi_q(\Sigma)= (\operatorname{tr} (\bgSigma^q) )^{1/q-1}\cdot \operatorname{tr}(\bgSigma^{q-1}\blH).
\]

Applying \eqref{equ:phi_q} again, we complete the proof of  \eqref{equ:Phi_q_Gateaux} in Lemma \ref{lem:phi_q}.

By applying the chain rule and combining  \eqref{equ:derivative_pi(a)} in Lemma~\ref{lem:Differentiation inverse Matrix} with \eqref{equ:Phi_q_Gateaux}, we obtain that
\[
\frac{\partial\Phi_q( \mathcal{I}^{-\pi})   }{\partial \pi(a)} =\innerpoduct{ \nabla \Phi_q( \mathcal{I}^{-\pi} ) }{ \frac{\partial \mathcal{I}^{-\pi} }{\partial \pi(a)} }=\innerpoduct{ \nabla \Phi_q( \mathcal{I}^{-\pi} ) }{ {-\mathcal{I}^{-\pi} \mathcal{I}_a \mathcal{I}^{-\pi}} }=\nabla_{-\mathcal{I}^{-\pi} \mathcal{I}_a \mathcal{I}^{-\pi}} \Phi_q( \mathcal{I}^{-\pi} ),
\]
which completes the proof of  \eqref{equ:Phi_q_Gateaux_pi} in Lemma \ref{lem:phi_q}.
\end{proof}

\begin{lemma}\label{lem:convex F}
    Assume $\mathcal{I}_a,a\in \ShatA$ are positive semidefinite matrices. Consider a convex matrix function, $g$, such that for any pair of positive semidefinite matrices, $\blA$ and $\blB$, if $\blA-\blB$ is a positive semidefinite matrix (denoted as $\blA\succeq \blB$), then $g(\blA)\geq g(\blB)$. For any ${\bgpi}\in \ShatA$, define
\begin{equation}\label{equ:S17_F}
    F({\bgpi})=g\left( \Big\{\sum_{a \in \mathcal{A}} \pi(a) \mathcal{I}_a\Big\}^{-1}\right).
\end{equation}
Then, $F({\bgpi})$ is a convex function over the set $\{{\bgpi}\in \ShatA; \sum_{a\in \mathcal{A} } \pi(a) \mathcal{I}_a\text{ is nonsingular}\}$.

In particular, under Assumption~\ref{ass:5}, the function $\mathbb{F}_{\bgtheta}(\bgpi)$ is a convex function over the set 
$
\{{\bgpi}\in \ShatA; \sum_{a\in \mathcal{A} } \pi(a) \mathcal{I}_a(\bgtheta)\text{ is nonsingular}\}.
$

\end{lemma}

\begin{proof}[Proof of Lemma \ref{lem:convex F}]
Let $\mathcal{C}_0=\{{\bgpi}\in \ShatA; \sum_{a\in \mathcal{A} } \pi(a) \mathcal{I}_a\text{ is nonsingular}\}$. Assume that ${\bgpi},{\bgpi}'\in \mathcal{C}_0$. Then, $\blA=\sum_{a \in \mathcal{A}}\pi(a) \mathcal{I}_a$ and $\blB=\sum_{a\in \mathcal{A}}\pi'(a) \mathcal{I}_a$ are positive definite. 

For any $0<t<1$, $t\blA+(1-t)\blB$ is positive definite. Note that for any vector $v$, applying Schur complement condition (see Theorem 1.12 (b) in \cite{zhang2006schur}), %
we obtain that
\begin{equation*}
    \left[\begin{array}{cc}
v^T \blA^{-1} v & v^T \\
v & \blA
\end{array}\right] \text{ and } \left[\begin{array}{cc}
v^T \blB^{-1} v & v^T \\
v & \blB
\end{array}\right]
\end{equation*}
are positive semi-definite matrices.
Notice that the following matrix is positive semi-definite
\begin{equation*}
t\left[\begin{array}{cc}
v^T \blA^{-1} v & v^T \\
v & \blA
\end{array}\right]+(1-t)\left[\begin{array}{cc}
v^T \blB^{-1} v & v^T \\
v & \blB
\end{array}\right]=\left[\begin{array}{cc}
t v^T \blA^{-1} v+(1-t) v^T \blB^{-1} v & v^T \\
v & t \blA+(1-t) \blB
\end{array}\right], 
\end{equation*}
by Theorem 1.12 (b) in \cite{zhang2006schur}, we obtain that
$$
tv^T\blA^{-1}v+ (1-t)v^T\blB^{-1}v\geq v^T(t\blA+(1-t)\blB)^{-1}v.
$$
Since $v$ is arbitrary, we have
\begin{equation}\label{ineq:convex inverse}
t\blA^{-1}+(1-t)\blB^{-1}\succeq (t\blA+(1-t)\blB)^{-1}.    
\end{equation}
Now, we have
\begin{align*}
    &tF({\bgpi})+(1-t)F({\bgpi}')=t g \left(   \blA^{-1}\right)+(1-t)g \left(  \blB^{-1}\right)\\
    \geq&g (t\blA^{-1}+(1-t)\blB^{-1})\geq g \big( \big\{t\blA+(1-t)\blB\big\}^{-1}\big)=F(t{\bgpi}+(1-t){\bgpi}').
\end{align*}
This shows that $F$ is convex.

We proceed to the proof of the `In particular' part of the lemma. 
{Note that under Assumption~\ref{ass:5}, there are two cases:
Case 1: $\mathbb{G}_{\bgtheta}(\bgSigma )=\Phi_q(\bgSigma)$ for $q\geq 0$, and Case 2: $\mathbb{G}_{\bgtheta}(\cdot)$ is a convex function satisfying $\mathbb{G}_{\bgtheta}(\blA)\geq \mathbb{G}_{\bgtheta}(\blB)$ whenever $\blA\succeq\blB$. For Case 2, we can apply our previous analysis directly for $g(\cdot)=\mathbb{G}_{\bgtheta}(\cdot)$, and obtain that $\mathbb{F}_{\bgtheta}(\bgpi)$ is convex. Thus, we focus our analysis on Case 1 in the rest of the proof.
}
By Courant-Fischer-Wely minimax principle (see Corollary III.1.2 in \cite{bhatia2013matrix}), the $i-$th largest eigenvalue satisfies $\lambda_{i}(\blA)\geq \lambda_{i}(\blB)$ for any $1\leq i\leq n$. Thus, $\Phi_q(\blA)\geq \Phi_q(\blB)$ for any $q\geq 0$.

If $\mathbb{G}_{\bgtheta}(\bgSigma )=\Phi_q(\bgSigma)= \big(\operatorname{tr} ( \bgSigma^q) \big)^{1/q}$ with $q\geq 1$, then $\Phi_q(\bgSigma)$ is the Schatten $q-$norm (see equation (IV.31) in \cite{bhatia2013matrix}), which implies that $\mathbb{G}_{\bgtheta}(\bgSigma )$ is convex. More generally, if $\mathbb{G}_{\bgtheta}(\bgSigma )$ is convex in $\bgSigma$, then by \eqref{equ:S17_F}, we obtain that 
\[
 \mathbb{F}_{\bgtheta}(\bgpi)= \mathbb{G}_{\bgtheta} (\{ \mathcal{I}^{\bgpi}(\bgtheta) \}^{-1})
\]
is convex in $\bgpi$ over
\[
\{{\bgpi}\in \ShatA; \sum_{a\in \mathcal{A} } \pi(a) \mathcal{I}_a(\bgtheta)\text{ is nonsingular}\}.
\]

If $\mathbb{G}_{\bgtheta}(\bgSigma )=\Phi_0(\bgSigma)=\log \det \bgSigma$, we know that 
\[
\Phi_0( \{\mathcal{I}^{\bgpi}(\bgtheta) \}^{-1} )=-\log \det (\sum_{a\in \mathcal{A}}\pi(a) \mathcal{I}_a(\bgtheta) ).
\]
We aim to show that $-\log \det (\blA)$ is convex over positive definite matrices.

Notice that for any $p\times p$ positive definite matrix $\blA$, 
\[
\int_{\mathbb{R}^p} e^{-1/2\innerpoduct{ \blA   \blx}{\blx}} d\blx= \frac{1 }{   (2 \pi )^{p/2} \det (\blA)^{1/2} }.
\]
By Hölder's inequality, for any positive definite $p\times p$ matrices $\blA$ and $\blB$, we have
\begin{equation}
    \int_{\mathbb{R}^p} e^{-1/2\innerpoduct{(t\blA+(1-t)\blB) \blx}{\blx}} d\blx \leq \Big( \int_{\mathbb{R}^p} e^{-1/2\innerpoduct{ \blA   \blx}{\blx}} d\blx \Big)^{t} \Big( \int_{\mathbb{R}^p} e^{-1/2\innerpoduct{ \blB   \blx}{\blx}} d\blx \Big)^{1-t},
\end{equation}
which implies that
\[
-\log \det (t\blA +(1-t)\blB) \leq -t\log \det (\blA) - (1-t)\log \det (\blB).
\]
This shows that $-\log \det (\blA)$ is convex over positive definite matrices. \\
Thus, $\mathbb{F}_{\bgtheta}(\bgpi)=-\log \det (\sum_{a\in \mathcal{A}}\pi(a) \mathcal{I}_a(\bgtheta) )$ is convex in $\bgpi$ over 
\[
\{{\bgpi}\in \ShatA; \sum_{a\in \mathcal{A} } \pi(a) \mathcal{I}_a(\bgtheta)\text{ is nonsingular}\}.
\]
If $\mathbb{G}_{\bgtheta}(\bgSigma )=\Phi_q(\bgSigma)= \operatorname{tr} \big( \bgSigma^q\big)$ with $0<q<1$, we know that 
\[
\Phi_q( \{\mathcal{I}^{\bgpi}(\bgtheta) \}^{-1} )=  \operatorname{tr} \Big(\sum_{a\in \mathcal{A}}\pi(a) \mathcal{I}_a(\bgtheta) \Big)^{-q}.
\]
By L\"owner-Heinz Theorem (see Theorem 2.6 in \cite{carlen2010trace}), we know that $\operatorname{tr} \big( \blA^{-q} \big)$ is operator convex, which means that for all positive definite matrices $\blA$ and $\blB$,
\[
\Big( t\blA+(1-t)\blB \Big)^{-q} \preceq  t  \blA^{-q} +(1-t)  \blB^{-q},
\]
which implies that $\operatorname{tr}\big(\blA^{-q}\big)$ is a convex function over positive definite matrices. In conclusion, we obtain that \(\mathbb{F}_{\bgtheta}(\bgpi)=\Phi_q( \{\mathcal{I}^{\bgpi}(\bgtheta) \}^{-1} )\) is convex in $\bgpi$ over 
\[
\{{\bgpi}\in \ShatA; \sum_{a\in \mathcal{A} } \pi(a) \mathcal{I}_a(\bgtheta)\text{ is nonsingular}\}.
\]

\end{proof}

\subsection{Decoupling Active Sequential Sampling}
The next lemma provides a decoupling result which makes it easier to analyze the likelihood for problems with adaptive experiment selection.
\begin{lemma}\label{lem:same dist}
Consider deterministic sequential selection functions $h_m(\cdot)$,
\begin{equation*}
h_m(a_1,Y_1,a_2,Y_2,\cdots,a_{m-1},Y_{m-1}) \in \mathcal{A}, \text{ for } m \geq 2,
\end{equation*}
and $h_1 \in \mathcal{A}$.
Consider two
random vectors generated from the following procedures.
\begin{enumerate}
    \item (Decoupled Sampling) Independently generate $\{X^a_m\}_{m\geq 1, a \in \mathcal{A}}$, where $X^a_m \sim f_{\boldsymbol{\theta}^*, a}(\cdot)$. Let 
    \[ 
    a_1 = h_1,  a_2 = h_2(a_1, X_1^{a_1}),  \cdots,  a_{m+1} = h_{m+1}(a_1, X_1^{a_1}, a_2, X_2^{a_2}, \cdots, a_{m}, X_{m}^{a_{m}}),\cdots.
    \]
    \item (Iterative Sampling) Let $a'_1 = h_1$. Generate $X_1 \sim f_{\boldsymbol{\theta}^*, a'_1}(\cdot)$. For $m \geq 1$, 
    obtain $a'_{m+1}$ and $X_{m+1}$ iteratively as
    \[ 
    a'_{m+1} = h_{m+1}(a'_1, X_1, \cdots, a'_m, X_m),
    \]
    then generate $X_{m+1} | \mathcal{F}_m \sim f_{\boldsymbol{\theta}, a'_{m+1}}$, where the $\sigma$-algebra $\mathcal{F}_m = \sigma(a'_1, X_1, a'_2, X_2, \cdots, a'_{m}, X_m)$.
\end{enumerate}
Then, the random vectors $(X_1^{a_1},  \cdots, X_m^{a_m}, a_1, \cdots, a_m)$ and $(X_1,  \cdots, X_m, a'_1, \cdots, a'_m)$ have the same distribution for all $m\geq 1$.
\end{lemma}
\begin{proof}[Proof of Lemma \ref{lem:same dist}]
We prove the lemma by induction.
When $m=1$, we know that $a_1=h_1=a_1'$, and $X_1|a'_1$ and $X^{a_1}_1|a_1$ have the same distribution. 
Thus, $(X_1^{a_1},a_1)$ and $(X_1,a_1')$ have the same distribution. 

By induction, assume that when $m=n$, random vectors $(X_1^{a_1}, X_2^{a_2}, \cdots, X_n^{a_n}, a_1, \cdots, a_n)$ and $(X_1, X_2, \cdots, X_n, a'_1, \cdots, a'_n)$ have the same distribution.

Let $m=n+1$. Define $\blX_{n}^{\mathcal{A}}=\{X_i^a\}_{1\leq i\leq n,a\in \mathcal{A}}$, and $\va^n=(a^1,\cdots,a^n)\in \mathcal{A}^n$. The density of $\blX_{n}^{\mathcal{A}}$ is given by
\begin{equation*}
    f_{\bgtheta}(\blX_{n}^{\mathcal{A}})=\prod_{i=1}^n\prod_{a\in \mathcal{A}}f_{\bgtheta,a}(X_i^a).
\end{equation*}
{Let $\boldsymbol{a}_m = (a_1, \cdots, a_m)$, where $a_1,\cdots,a_m$ are obtained from the decoupled sampling.}
Given $\blX_{n}^{\mathcal{A}}$, the conditional probability mass function of {$\bla_n$ and $\bla_{n+1}$ are}
\begin{equation*}
    f_{\bgtheta}(\va^{n}|\blX_{n}^{\mathcal{A}})=I(a^1=a_1,\cdots,a^{n}=a_{n} )\text{ and }
    f_{\bgtheta}(\va^{n+1}|\blX_{n}^{\mathcal{A}})=I(a^1=a_1,\cdots,a^{n+1}=a_{n+1} ),
\end{equation*}
{where we used the fact that the $\{a_m\}_{1\leq m\leq n+1}$ is measurable with respect to $\sigma(\blX_{n}^{\mathcal{A}})$.} As a result, the joint density functions for { $(\blX_{n}^{\mathcal{A}}, \bla_n)$ and $(\blX_{n+1}^{\mathcal{A}}, \bla_{n+1})$} 
are
\begin{equation*}
     f_{\bgtheta}(\blX_{n}^{\mathcal{A}},\va^{n})=\prod_{i=1}^n\prod_{a\in \mathcal{A}}f_{\bgtheta,a}(X_i^a) I(a^1=a_1,\cdots,a^{n}=a_{n} )
\end{equation*}
and
\begin{equation*}
     f_{\bgtheta}(\blX_{n+1}^{\mathcal{A}},\va^{n+1})=f_{\bgtheta}(\blX_{n}^{\mathcal{A}},\va^{n})\prod_{a\in \mathcal{A}}f_{\bgtheta,a}(X_{n+1}^a) I( a^{n+1}=a_{n+1} ).
\end{equation*}
Thus, given $\blX^{\mathcal{A}}_n$ and $\va_{n}$, the condition density for $\{X^a_{n+1}\}_{a\in \mathcal{A}}, a_{n+1}$ is 
\begin{equation*}
    f_{\bgtheta}(\{X_{n+1}^{a}\}_{a\in \mathcal{A}},a^{n+1}|\blX_{n}^{\mathcal{A}},\va^{n})=\prod_{a\in \mathcal{A}}f_{\bgtheta,a}(X_{n+1}^a) I( a^{n+1}=a_{n+1} ).
\end{equation*}
Note that $a_{n+1} = h_{n+1}(a_1, X_1^{a_1}, a_2, X_2^{a_2}, \cdots, a_{n}, X_{n}^{a_{n}})$. So $f_{\bgtheta}(\{X_{n+1}^{a}\}_{a\in \mathcal{A}},a^{n+1}|\blX_{n}^{\mathcal{A}},\va^{n})$ depends on $\blX_{n}^{\mathcal{A}},\va^{n}$ only through  $a_1,X_1^{a_1},\cdots,a_n,X_n^{a_n}$. 

Define $\sigma$-algebra $\mathcal{F}'_{n}=\sigma(a_1,X_1^{a_1},\cdots,a_n,X_n^{a_n})$. Because $f_{\bgtheta}(\{X_{n+1}^{a}\}_{a\in \mathcal{A}},a^{n+1}|\blX_{n}^{\mathcal{A}},\va^{n})$ is measurable in $\mathcal{F}'_{n}$, we have
\begin{equation}
    f_{\bgtheta}(\{X_{n+1}^{a}\}_{a\in \mathcal{A}},a^{n+1}|\mathcal{F}'_{n})=f_{\bgtheta}(\{X_{n+1}^{a}\}_{a\in \mathcal{A}},a^{n+1}|\blX_{n}^{\mathcal{A}},\va_{n})=\prod_{a\in \mathcal{A}}f_{\bgtheta,a}(X_{n+1}^a) I( a^{n+1}=a_{n+1} )
\end{equation}
We have 
$X_{n+1}^{a_{n+1}}|\{a_1,X^{a_1}_{1},\cdots,a_n,X^{a_n}_{n}\}\sim f_{\bgtheta ,a_{n+1}}(\cdot)$, and $a_{n+1}=h_n(a_1,X_1^{a_1},\cdots,a_n,X_n^{a_n})$. Recall that  $X_{n+1} |\{a'_1,X_{1},\cdots,a'_n,X_{n}\}\sim f_{\bgtheta ,a'_{n+1}}(\cdot)$, $a'_{n+1}=h_n(a'_1,X_1 ,\cdots,a'_n,X_n )$, as well as the induction assumption that random vectors $(X_1^{a_1}, X_2^{a_2}, \cdots, X_n^{a_n}, a_1, \cdots, a_n)$ and $(X_1, X_2, \cdots, X_n, a'_1, \cdots, a'_n)$ have the same distribution. Consequently, random vectors $(X_1^{a_1}, X_2^{a_2}, \cdots, X_{n+1}^{a_{n+1}}, a_1, \cdots, a_{n+1})$ and $(X_1, X_2, \cdots, X_{n+1}, a'_1, \cdots, a'_{n+1})$ also have the same distribution. We complete the proof of Lemma \ref{lem:same dist} by induction.

\end{proof}

\subsection{Results on Linear Spaces Indexed by a Parameter}
Linear spaces spanned by the Fisher information play a crucial role in the proof of the theorems. Note that the Fisher information matrices are depending on the parameter $\bgtheta$. In this section, we present useful linear algebra results where the linear spaces are indexed by a parameter.

{
Recall $V_{Q}(\bgtheta)=\sum_{a\in Q}\mathcal{R}(\mathcal{I}_a(\bgtheta))$, and $\mathcal{I}_a(\bgtheta)$ is the Fisher information matrix at the parameter $\bgtheta$ with the experiment $a$. 
Throughout the section, we only used the property that $\mathcal{I}_a(\bgtheta)$ is a positive semidefinite matrix and is continuous in $\bgtheta$, for all $a\in\mathcal{A}$, which is guaranteed under the regularity assumptions in Section~\ref{sec:assumptions}. The results in this section still hold even when $\mathcal{I}_a(\bgtheta)$ is not the Fisher information matrix, as long as it is still positive semidefinite and continuous in $\bgtheta$, for all $a\in\mathcal{A}$. We do not require any further assumptions.
}
\begin{lemma}\label{lem:dim-rank}
    For all $Q\subset \mathcal{A}$, and $x_a>0,a\in Q$, we have 
\begin{equation*}
    \dim(V_{Q}(\bgtheta))=\operatorname{rank}\Big(\sum_{a\in Q}x_a\mathcal{I}_a(\bgtheta)\Big).
\end{equation*}    
\end{lemma}
\begin{proof}[Proof of Lemma \ref{lem:dim-rank}]
    It suffices to show that $V_Q(\bgtheta)^{\perp}=\operatorname{ker}(\sum_{a\in Q }x_a\mathcal{I}_a(\bgtheta))$.
This equation holds because $\blu\in V_Q(\bgtheta)^{\perp}$ if and only if $\innerpoduct{\mathcal{I}_a(\bgtheta)\bly_a}{\blu}=0$ for all $a\in Q$ and $\bly_a\in \mathbb{R}^p$, if and only if $\mathcal{I}_a(\bgtheta)\blu=0$ for all $a\in Q$, if and only if $\blu^T\mathcal{I}_a(\bgtheta)\blu=0$ for all $a\in Q$, if and only if $\blu^T(\sum_{a\in Q}x_a\mathcal{I}_a(\bgtheta))\blu=0$, if and only if $\blu\in \operatorname{ker}(\sum_{a\in Q }x_a\mathcal{I}_a(\bgtheta))$.
\end{proof}

\begin{lemma}\label{lem:dim-rank-c}
    Assume $\bgTheta$ is a path connected and compact set. The following statements are equivalent:
    \begin{enumerate}
        \item for all $Q\subset \mathcal{A}$, $\dim(V_{Q}({\bgtheta}))$ does not depend on ${\bgtheta}$,
        \item for all $Q\subset \mathcal{A}$, $\operatorname{rank}(\sum_{a\in Q}\mathcal{I}_a({\bgtheta}))$ does not depend on ${\bgtheta}$,
        \item  there exists $0<\underline{c}<\overline{c}<\infty$, which does not depend on $Q$, such that
\begin{equation*}
 \underline{c} \cdot \blP_{V_Q({\bgtheta})}\preceq  \sum_{a\in Q} \mathcal{I}_a({\bgtheta}) \preceq \overline{c} \cdot \blP_{V_Q({\bgtheta})}, \forall Q\subset \mathcal{A},
\end{equation*}
where $\blP_V$ denotes the orthogonal projection matrix onto vector space $V$. 
    \end{enumerate}
   
\end{lemma}
\begin{proof}[Proof of Lemma \ref{lem:dim-rank-c}]
~~
\paragraph*{1 $\iff$ 2}
{ This equivalency holds because $\dim(V_Q({\bgtheta}))=\text{rank}(\sum_{a\in Q}\mathcal{I}_a({\bgtheta}))$, according to Lemma~\ref{lem:dim-rank}.}

\paragraph*{3 $\implies$ 2} For $Q\subset \mathcal{A}$, let $r({\bgtheta})=\operatorname{rank}( \sum_{a\in Q}\mathcal{I}_a({\bgtheta}) )$.
Also, let 
\begin{equation}\label{def:maxrank}
    r=\sup_{{\bgtheta}\in \bgTheta} \operatorname{rank}( \sum_{a\in Q}\mathcal{I}_a({\bgtheta}) ).
\end{equation}
By the definition of supremum, there exists ${\bgtheta}_0\in \bgTheta$ such that
\[
r-1/2\leq \operatorname{rank}( \sum_{a\in Q}\mathcal{I}_a({\bgtheta}_0) )\leq r.
\]
Because the rank of a  matrix can only take integer values, we know that $\operatorname{rank}( \sum_{a\in Q}\mathcal{I}_a({\bgtheta}_0) ) = r$.
 Let $\mu_1(\blA)\geq \mu_2(\blA) \geq \cdots \geq \mu_p(\blA)$ be the eigenvalues of a positive semidefinite matrix $\blA$. Applying Courant–Fischer–Weyl min-max principle (see Chapter I of \cite{hilbert1953methods} or Corollary III.1.2 in \cite{bhatia2013matrix}) to $\underline{c} \cdot \blP_{V_Q({\bgtheta})}\preceq  \sum_{a\in Q} \mathcal{I}_a({\bgtheta})$, and $r({\bgtheta})=\dim(V_Q({\bgtheta}))$ (see Lemma~\ref{lem:dim-rank}), we obtain 
\begin{equation}\label{equ:rank5.3}
    \mu_{r({\bgtheta})}( \sum_{a\in Q} \mathcal{I}_a({\bgtheta}) )\geq \mu_{r({\bgtheta})}(  \underline{c} \cdot \blP_{V_Q({\bgtheta})} )=\underline{c}>0, \forall {\bgtheta}\in \bgTheta.
\end{equation}
Applying Courant–Fischer–Weyl min-max principle to $  \sum_{a\in Q} \mathcal{I}_a({\bgtheta}) \preceq \overline{c} \cdot \blP_{V_Q({\bgtheta})}$, and $r({\bgtheta})=\dim(V_Q({\bgtheta}))$, we obtain
\begin{equation*}
    \mu_{r({\bgtheta})+1}( \sum_{a\in Q} \mathcal{I}_a({\bgtheta}) )\leq \mu_{r({\bgtheta})+1}(  \overline{c} \cdot \blP_{V_Q({\bgtheta})} )=0, \forall {\bgtheta}\in \bgTheta.
\end{equation*}
We will prove $r(\bgtheta)=r$ for all $\bgtheta\in\bgTheta$ by contradiction. Assume, in contrast,  that there exists $r_1=r({\bgtheta}_1)<r({\bgtheta}_0)=r$. Then, there exists a continuous path $h:[0,1]\to \bgTheta$ such that $h(0)={\bgtheta}_0$, and $h(1)={\bgtheta}_1$. Set 
$$
u(t)=\mu_{r }( \sum_{a\in Q} \mathcal{I}_a( h(t) ) ).
$$
Note that $u(0)=\mu_{r }( \sum_{a\in Q} \mathcal{I}_a({\bgtheta}_0) )\geq \underline{c}$ and $u( 1 )=\mu_{r }( \sum_{a\in Q} \mathcal{I}_a({\bgtheta}_1) )=0$. Because $u(t)$ is a continuous function in $t\in [0,1]$, by the intermediate value theorem, there exists $t'\in(0,1)$ such that $u(t')=\underline{c}/2$. 

Let ${\bgtheta}_2=h(t')$. Because
$\mu_{r }( \sum_{a\in Q} \mathcal{I}_a({\bgtheta}_2) ) = \underline{c}/2>0$, we know that $rank( \sum_{a\in Q} \mathcal{I}_a({\bgtheta}_2))\geq r$. By definition \eqref{def:maxrank}, we know that $rank( \sum_{a\in Q} \mathcal{I}_a({\bgtheta}_2))\leq r$. Thus, $r({\bgtheta}_2)=r$. 
However, $\mu_{r({\bgtheta}_2) }( \sum_{a\in Q} \mathcal{I}_a({\bgtheta}_2) )=\underline{c}/2$ contradicts inequality \eqref{equ:rank5.3}. This completes the proof that $r({\bgtheta})=r$ for all ${\bgtheta}\in \bgTheta$.

\paragraph*{2 $\implies$ 3}

For $Q\subset \mathcal{A}$, define
\begin{equation*}
    c_{min}(Q)=\min_{{\bgtheta}\in \bgTheta }\Lambda_{min}\Big(\sum_{a\in Q}\mathcal{I}_a({\bgtheta})\Big), \text{ and }c_{max}(Q)=\max_{{\bgtheta}\in \bgTheta }\Lambda_{max}\Big(\sum_{a\in Q}\mathcal{I}_a({\bgtheta})\Big),
\end{equation*}
where $\Lambda_{min}$ and $\Lambda_{max}$ represent the smallest and largest non-zero eigenvalue of a positive semidefinite matrix, respectively.

Because $\operatorname{rank}(\sum_{a\in Q}\mathcal{I}_a({\bgtheta}))$ does not depend on ${\bgtheta}$, let ${r}=\operatorname{rank}(\sum_{a\in Q}\mathcal{I}_a({\bgtheta}))$. Let $\lambda_{(s)}(\sum_{a\in Q}\mathcal{I}_a({\bgtheta}) )$ denote the $s-$th largest eigenvalue of $\sum_{a\in Q}\mathcal{I}_a({\bgtheta})$, $s=1,2,\cdots,p$.
Note that $\Lambda_{min}(\sum_{a\in Q}\mathcal{I}_a({\bgtheta}))=\lambda_{(r)}(\sum_{a\in Q}\mathcal{I}_a({\bgtheta}) )$ and $\Lambda_{max}(\sum_{a\in Q}\mathcal{I}_a({\bgtheta}))=\lambda_{(1)}(\sum_{a\in Q}\mathcal{I}_a({\bgtheta}) )$. Now, we know that $\Lambda_{min}$ and $\Lambda_{max}$ are continuous functions provided $\operatorname{rank}(\sum_{a\in Q}\mathcal{I}_a({\bgtheta}))$ does not depend on ${\bgtheta}$, and $\mathcal{I}_a({\bgtheta})$ is continuous over compact set $\bgTheta$. Thus,
$0<c_{min}(Q)\leq c_{max}(Q)<\infty$.

Recall that $V_Q({\bgtheta})^{\perp}=\operatorname{ker}(\sum_{a\in Q}\mathcal{I}_a({\bgtheta}))$ from Lemma~\ref{lem:dim-rank}. Because for the positive semidefinite matrix $\sum_{a\in Q}\mathcal{I}_a({\bgtheta})$, $\operatorname{ker}(\sum_{a\in Q}\mathcal{I}_a({\bgtheta}))^{\perp}=\mathcal{R}(\sum_{a\in Q}\mathcal{I}_a({\bgtheta}) )$, we obtain $V_Q({\bgtheta})=\mathcal{R}(\sum_{a\in Q}\mathcal{I}_a({\bgtheta}))$.

Applying eigendecomposition of $\sum_{a\in Q}\mathcal{I}_a({\bgtheta}) $, we have
\begin{equation*}
    c_{min}(Q)\cdot \blP_{V_Q({\bgtheta})} \preceq \sum_{a\in Q}\mathcal{I}_a({\bgtheta}) \preceq c_{max}(Q)\cdot \blP_{V_Q({\bgtheta})}.
\end{equation*}
Set $\underline{c}=\min_{Q\subset \mathcal{A}}c_{min}(Q)$ and $\overline{c}=\max_{Q\subset \mathcal{A}}c_{max}(Q)$. Since $\mathcal{A}$ is a finite set, we know that $\underline{c}>0$ and $\overline{c}<\infty$.
\end{proof}
\subsection{Other Supporting Lemmas}
\begin{lemma}\label{lemma:trace}
    Assume that positive semidefinite matrices $\blA,\blB$ and $\blC$ have the same size. If $\blA\succeq \blC$, then
    \begin{equation}
        \operatorname{tr}(\blA\blB)\geq \operatorname{tr}(\blC\blB).
    \end{equation}
\end{lemma}
\begin{proof}[Proof of Lemma \ref{lemma:trace}]
    Because $\blB^{1/2}(\blA-\blC) \blB^{1/2}$ is positive semi-definite, 
    \begin{equation}
        \operatorname{tr}(\blA\blB)-\operatorname{tr}(\blC\blB)=\operatorname{tr}(\blB^{1/2}(\blA-\blC) \blB^{1/2})\geq 0.
    \end{equation}
    This completes the proof.
\end{proof}

\begin{lemma}[Multivariate Cauchy-Schwartz Inequality]
\label{lemma:MCS}
    For any random variable $z$ and random vector $\bly$, if $\operatorname{cov}(\bly)=\bgSigma_{\bly}$ is positive definite matrix, then
\begin{equation}\label{ineq:MCR}
    \operatorname{var}(z)\geq \operatorname{cov}(z,\bly) \operatorname{cov}(\bly )^{-1}\operatorname{cov}(\bly,z).
\end{equation}
\end{lemma}
\begin{proof}[Proof of Lemma~\ref{lemma:MCS}]

Because 
\begin{equation*}
\begin{split}
    0\leq&\operatorname{var} \big(z- \operatorname{cov}(z,\bly) \bgSigma_{\bly}^{-1}\bly \big)\\
    =&\operatorname{var} (z)-2 \operatorname{cov}\Big(\operatorname{cov}(z,\bly) \bgSigma_{\bly}^{-1}\bly ,z \Big)+\operatorname{var}\Big(\operatorname{cov}(z,\bly) \bgSigma_{\bly}^{-1}  {\bly} \Big)\\
    =&\operatorname{var} (z)-2 \operatorname{cov}(z,\bly) \bgSigma_{\bly}^{-1}\operatorname{cov}(\bly ,z)+  \operatorname{cov}(z,\bly) \bgSigma_{\bly}^{-1}\bgSigma_{\bly}    \bgSigma_{\bly}^{-1}\operatorname{cov}(\bly,z) \\
    =&\operatorname{var} (z)-\operatorname{cov}(z,\bly)  \{\operatorname{cov}(\bly)\}^{-1} \operatorname{cov}(\bly ,z),
\end{split}
\end{equation*}
we complete the proof of Lemma~\ref{lemma:MCS}.

\end{proof}

\begin{lemma}\label{lemma:assumption-AB}
Assumptions~\ref{ass:6A} and \ref{ass:7A} imply Assumptions~\ref{ass:6B} and \ref{ass:7B}.
\end{lemma}
\begin{proof}[Proof of Lemma~\ref{lemma:assumption-AB}]
Under Assumption~\ref{ass:6A}, we have
\begin{equation*}
    \mathcal{I}_a(\bgtheta)=\bm{Z}_a^T \mathcal{I}_{\bgxi_a,a}(\bm{Z}_a \bgtheta ) \bm{Z}_a.
\end{equation*}
{Let $\blZ_a^\dagger$ be the Moore-Penrose inverse of $\blZ_a$.}
Because $\blZ_a$ has full row rank, we know that $\blZ_a \blZ_a^\dagger=I_{p_a}$, and
\[
\mathcal{I}_{\bgxi_a,a}( \blZ_a \bgtheta )=\{\blZ_a^\dagger\}^T \mathcal{I}_a(\bgtheta) \blZ_a^\dagger,
\]
which implies that $\mathcal{I}_{\bgxi_a,a}( \blZ_a \bgtheta )$ is continuous in $\bgtheta$. Thus, there exists $0<c_1<c_2<\infty$ such that
\begin{equation*}
    c_1 I_{p_a} \preceq \mathcal{I}_{\bgxi_a,a}(\bm{Z}_a \bgtheta ) \preceq c_2 I_{p_a},
\end{equation*}
and for any $Q\subset \mathcal{A}$, 
\begin{equation*}
    V_{Q}(\bgtheta)=\sum_{a\in Q}\mathcal{R}(\mathcal{I}_a(\bgtheta))=\sum_{a\in Q}\mathcal{R}(\bm{Z}^T_a).
\end{equation*}
By Lemma~\ref{lem:dim-rank-c}, we know that
\begin{equation*}
    \dim(V_{Q}(\bgtheta))=\dim\Big(\sum_{a\in Q}\mathcal{R}(\bm{Z}^T_a)\Big)
\end{equation*}
does not depend on $\bgtheta$.

By Lemma \ref{lem:dim-rank-c}, we obtain inequality \eqref{ineq:double c bound}. This proves that Assumption~\ref{ass:6A} implies Assumption~\ref{ass:6B}. 

Next, we consider Assumption \ref{ass:7B}. Under Assumptions~\ref{ass:6A} and \ref{ass:7A}, we obtain
\begin{equation*}
\begin{split}
    &D_{\mathrm{KL}}( f_{{\bgtheta}^*,a} \| f_{{\bgtheta},a} )\\
    = &D_{\mathrm{KL}}(h_{\bgxi^*_{a},a}\| h_{\bgxi_{a},a} )\\
    \geq  &C \norm{ \bgxi^*_{a} - \bgxi_{a} }^2\\
    = &C(\bgtheta-\bgtheta^*)^T\bm{Z}^T_a\bm{Z}_a (\bgtheta-\bgtheta^*)\\
    \geq &\frac{C}{c_2}(\bgtheta-\bgtheta^*)^T\bm{Z}^T_a \mathcal{I}_{\bgxi_a,a}(\bm{Z}_a \bgtheta^* ) \bm{Z}_a (\bgtheta-\bgtheta^*)\\
    = &\frac{C}{c_2}(\bgtheta-\bgtheta^*)^T  \mathcal{I}_{a}(\bgtheta^*)  (\bgtheta-\bgtheta^*).
\end{split}
\end{equation*}
Thus, for any ${\bgpi} \in \ShatA$ and $\bgtheta\in \bgTheta$, we have
\begin{equation*}
    \sum_{a \in \mathcal{A}}\pi(a)D_{\mathrm{KL}}(f_{\bgtheta^*,a}\|f_{\bgtheta,a} )\geq  \frac{C}{c_2}   \sum_{a\in \mathcal{A}}\pi(a)(\bgtheta-\bgtheta^*)^T  \mathcal{I}_a(\bgtheta^*)(\bgtheta-\bgtheta^*). 
\end{equation*}
Replacing the constant $\frac{C}{c_2}$ by $C$, we obtain inequality \eqref{ineq:piD_KL} in Assumption~\ref{ass:7B}. Thus, we obtain Assumptions~\ref{ass:6B}-\ref{ass:7B}. 

\end{proof}

\section{Proof of Theoretical Results}
\subsection{Proof of Lemma \ref{lem:complexity}}\label{app:computational_complexity}
\begin{proof}[Proof of Lemma \ref{lem:complexity}]
The standard computational complexity for both matrix multiplication and matrix inversion of a matrix of size $p\times p$ is $O(p^3)$. Consequently, the computational complexity of evaluating 
$  \mathbb{G}_{\widehat{\bgtheta}_{n}^{\text{ML}} } \left[ \big\{ \mathcal{I}(\widehat{\bgtheta}_{n}^{\text{ML}};\va_n,a) \big\}^{-1} \right]$ for each $ a\in \mathcal{A}$ is $O(p^3)$. Therefore, the computational complexity for the \textrm{GI0} selection is $O(kp^3)$.

The \textrm{GI1} selection rule \eqref{eq:GI1} can be reformulated as
\begin{equation*}
\begin{aligned}
    &a_{n+1}\\
    =& \arg\max_{a\in \mathcal{A}}\tr \left[  \nabla\mathbb{G}_{\widehat{\bgtheta}_{n}^{\text{ML}}}\big(\{   \mathcal{I}(\widehat{\bgtheta}_{n}^{\text{ML}}; \va_n ) \}^{-1}   \big) \{\mathcal{I} (\widehat{\bgtheta}_{n}^{\text{ML}}; \va_n)\}^{-1} \mathcal{I}_a(\widehat{\bgtheta}_{n}^{\text{ML}})\{\mathcal{I} (\widehat{\bgtheta}_{n}^{\text{ML}}; \va_n)\}^{-1} \right] \\
    =& \arg\max_{a\in \mathcal{A}}\tr\left[ L^T_a(\widehat{\bgtheta}_{n}^{\text{ML}})
    \{\mathcal{I} (\widehat{\bgtheta}_{n}^{\text{ML}}; \va_n)\}^{-1} \nabla\mathbb{G}_{\widehat{\bgtheta}_{n}^{\text{ML}}}\big(\{   \mathcal{I}(\widehat{\bgtheta}_{n}^{\text{ML}}; \va_n ) \}^{-1}   \big) 
    \{\mathcal{I} (\widehat{\bgtheta}_{n}^{\text{ML}}; \va_n)\}^{-1} L_a(\widehat{\bgtheta}_{n}^{\text{ML}})\right] .
\end{aligned}
\end{equation*}
Thus, Algorithm~\ref{alg:gi1-Accelerated} produce the same outcome as Algorithm~\ref{alg:gi1}.

Note that the matrix 
\[ 
\blM = \{\mathcal{I} (\widehat{\bgtheta}_{n}^{\text{ML}}; \va_n)\}^{-1} \nabla\mathbb{G}_{\widehat{\bgtheta}_{n}^{\text{ML}}}\big(\{   \mathcal{I}(\widehat{\bgtheta}_{n}^{\text{ML}}; \va_n ) \}^{-1}   \big)  \{\mathcal{I} (\widehat{\bgtheta}_{n}^{\text{ML}}; \va_n)\}^{-1}  \]
only needs to be computed once. Matrix multiplication involving matrices of sizes $p\times p$ and $p\times s_a$ is of order $O(p^2s)$, given $s_a \leq s$. Under the assumption that $L_a(\bgtheta)$ has size $p\times s_a$, the computational complexity of the \textrm{GI1} is bounded by $O(p^3+ksp^2)$.

Furthermore, if the matrices $\{L_a(\bgtheta)\}$ are primarily supported on an $s\times s$ submatrix, then the computational cost of the multiplication $\bgSigma  L_a(\widehat{\bgtheta}_{n}^{\text{ML}})$ is at most equivalent to multiplying matrices of sizes $p\times s$ and $s\times s$ {for any matrix $\bgSigma$}, which has a complexity of $O(s^2p)$. Therefore, the overall computational complexity of the \textrm{GI1} selection is $O(p^3+ks^2p)$.
\end{proof}

\subsection{Proof of Proposition~\ref{prop:exploration}}
We prove the following 
Theorem~\ref{thm:selection bound} instead, which is a generalized version of Proposition~\ref{prop:exploration} allowing for an arbitrary sequence of $\bgtheta_n$ that is not necessarily the MLE.
\begin{theorem}\label{thm:selection bound}Under the regularity conditions described in Section~\ref{sec:assumptions}, and also
assume that the initial experiments $a_1,\cdots, a_{n_0}\in \mathcal{A}$ are such that $\mathcal{I}( \bgtheta;\va_{n_0})$ is nonsingular.
For any sequence of (random or non-random) vectors $\bgtheta_1,\bgtheta_2,\cdots$ in $\bgTheta$, if we consider the following generalized \textrm{GI0} or \textrm{GI1} selection rules: for any $n\geq n_0$
\begin{equation}\label{eq:GI0-exploration}
 \textrm{GI0}: \quad   a_{n+1} = \arg\min_{a\in\mathcal{A}}\mathbb{G}_{ {\bgtheta}_{n} }\Big[\big\{   \mathcal{I}( {\bgtheta}_{n} ;\va_n,a)\big\}^{-1} \Big], \text{ and}
\end{equation}
\begin{equation}\label{eq:GI1-exploration}
\begin{split}
      \textrm{GI1}:  a_{n+1} = \arg\max_{a\in\mathcal{A}}  \operatorname{tr} \Big[  \nabla\mathbb{G}_{ \bgtheta_{n} }\big( {\bgSigma}_n  \big) 
  {\bgSigma}_n \mathcal{I}_a( {\bgtheta}_{n} ) {\bgSigma}_n  \Big],\text{ where we define }{\bgSigma}_n=\{   \mathcal{I}( {\bgtheta}_{n} ; \va_n ) \}^{-1},
  \end{split}
\end{equation}
then there exists \( C > 0 \) such that
\[
\inf_{n\geq n_0}\frac{n_I}{n} \geq C,
\]
and the lower bound is independent of the choice of \( \bgtheta_n \). {Moreover, under the same settings, 
\begin{equation}\label{ineq:Ipi_bound}
    \mathcal{I}^{\overline{{\bgpi}}_n}(\bgtheta)\succeq \underline{c} \cdot C  \cdot I_p.
\end{equation}
for all $\bgtheta\in\bgTheta$. %
}
\end{theorem}
\begin{proof}[Proof of Proposition~\ref{prop:exploration}]
Applying Theorem~\ref{thm:selection bound} with { $\bgtheta_n=\widehat{\bgtheta}_n^{\text{ML}}$}, we complete the proof of Proposition~\ref{prop:exploration}.
\end{proof}
 
The proof of Theorem~\ref{thm:selection bound} is involved. We break it down to the following series of lemmas and steps.

\paragraph*{Step 1: Define the order statistics of experiments counts, permutations, and find their connections with $n_{\max}$ and $n_I$}
Let $m^n_a=|\{ i; a_i=a, 1\leq i\leq n \}|$ be the number of times that the experiment $a$ has been selected up to time $n$. Without loss of generality, let $\mathcal{A}=[k]=\{1,2,\cdots,k\}$. Let $\mathscr{P}_k$ be the set of all permutations over $[k]$. 

For each $m^{n}=(m^n_a)_{a\in\mathcal{A}}$, define $\mathscr{P}^{m^n}_k\subset \mathscr{P}_k$, which is described by the following statements: permutation $\tau\in\mathscr{P}^{m^n}_k$ if and only if $\tau\in \mathscr{P}_k$ and
\begin{equation*}
    m^n_{\tau(1)}\geq m^n_{\tau(2)}\geq \cdots \geq m^n_{\tau(k)}.
\end{equation*}
The set $\mathscr{P}^{m^n}_k$ is not empty, because order statistic exists.

For any permutation $\tau\in\mathscr{P}^{m^n}_k$, define the set $Q_s(\tau)=\{\tau(1),\tau(2),\cdots,\tau(s)\}$ for $s\in [k]$. {$Q_s(\tau)$ collects the indices of the top-$s$ most frequently selected experiments. Here, $\tau$ is introduced to handle the case where there may be ties among $m^n_a$ for $a\in\mathcal{A}$.}

{Define
\begin{equation}\label{def:t_n}
    t_n=t_n(m^n)=\min_{\tau\in \mathscr{P}^{m^n}_k} \{ s\in [k]; \dim(V_{Q_s(\tau)}(\bgtheta))=p \}.
\end{equation}
{
Note that $\dim(V_{Q_k(\tau)}(\bgtheta))= \text{rank}(\sum_{a\in\mathcal{A}}\mathcal{I}_a(\bgtheta))=p$ for all $\tau\in\mathscr{P}^{m^n}_k$, according to Lemma~\ref{lem:dim-rank} and Assumption~\ref{ass:3}. Also note that according to Lemma~\ref{lem:dim-rank-c} and Assumption~\ref{ass:6B}, $\dim(V_{Q_s(\tau)}(\bgtheta))$ does not depend on $\bgtheta$. Thus, the above $t_n$ is well defined and does not depend on $\bgtheta$. For the same reason, we will drop `$\bgtheta$' and write $\dim(V_A)$ for $\dim(V_{A}(\bgtheta))$ for $A\subset\mathcal{A}$ in the rest of the proof when the context is clear.

The next lemma specifies the permutations that we would like to focus on when there may be ties among $m^n_a$.
\begin{lemma}\label{lemma:tau_order}
There exists $\tau_n\in \mathscr{P}_k$ such that 
\begin{equation}\label{ineq:tau_order}
    m^n_{\tau_n(1)}\geq m^n_{\tau_n(2)}\geq \cdots \geq m^n_{\tau_n(k)},
\end{equation}
$\dim (V_{Q_{t_n} (\tau_n) } )=p$, and for all $\tau' \in \mathscr{P}^{m^n}_k$ and all $s\leq t_n-1$, $\dim (V_{Q_{s} (\tau') } ) <p$.
\end{lemma}
\begin{proof}[Proof of Lemma~\ref{lemma:tau_order}]
First, according to the definition of $t_n$ in \eqref{def:t_n} and $\mathscr{P}^{m^n}_k\neq\emptyset$, we know that
\[
S'=\arg\min_{\tau\in \mathscr{P}^{m^n}_k} \{ s\in [k]; \dim(V_{Q_s(\tau)}(\bgtheta))=p \}
\]
is not empty. Let \(\tau_n\in S'\). We know that $\tau_n$ satisfies \eqref{ineq:tau_order}, and $\dim (V_{Q_{t_n} (\tau_n) } )=p$.

Assume there exist $\tau' \in \mathscr{P}^{m^n}_k$ and $s\leq t_n-1$, such that $\dim (V_{Q_{s} (\tau') } ) = p$. This leads to the following contradiction
\[
t_n=\min_{\tau\in \mathscr{P}^{m^n}_k} \{ s\in [k]; \dim(V_{Q_s(\tau)}(\bgtheta))=p \} \leq s \leq t_n-1.
\]

This completes the proof of Lemma~\ref{lemma:tau_order}.    
\end{proof}
}

{Recall that $n_{\max} = \max_{a\in\mathcal{A}}n_a=\max_{a\in\mathcal{A}} m_a^{n}$ is defined in Section~\ref{sec:proof-sketch}.} We obtain that $n_{\max}=m^n_{\tau_n(1)}$. The following Lemma shows that $n_I=m^n_{\tau_n(t_n)}$.
\begin{lemma}\label{lem:n_I equation}
{Let $\tau_n$ be a permutation satisfying the properties described in Lemma~\ref{lemma:tau_order}.} Then,
    $n_I=m^n_{\tau_n(t_n)}$, where $n_I$ is defined in \eqref{eq:exploration-condition}.
\end{lemma}
\begin{proof}[Proof of Lemma~\ref{lem:n_I equation}]
    
Because  $\dim (V_{Q_{t_n} (\tau_n) } (\bgtheta))=p$, we know that $Q=Q_{t_n}(\tau_n)$ %
is relevant, which means that $\sum_{a\in Q} \mathcal{I}_a(\bgtheta)$ is non-singular for all $\bgtheta\in \bgTheta$. By the definition of $n_I$ in \eqref{eq:exploration-condition},
\[
n_I\geq \min_{a\in Q} m^n_a =m^n_{\tau_n(t_n)}.
\]

It suffices to prove $n_I\leq m^n_{\tau_n(t_n)}$. Assume, on the contrary, that $n_I>m^n_{\tau_n(t_n)}$. In the rest of the proof, we aim to find a contradiction.

For any $S\subset \mathcal{A}$ such that $S$ is relevant, define $Q(S)=\{a\in \mathcal{A}; m_a^n\geq \min_{a\in S}m_a^n\}$. Since $S\subset Q(S)$, $Q(S)$ is also relevant, and 
\[
\min_{a\in S}m^n_a=\min_{a\in Q(S)}m^n_a.
\]
By the definition of $n_I$ in \eqref{eq:exploration-condition}, there exists a relevant $S'\subset \mathcal{A}$ such that
\(\min_{ a\in S' }m^n_{a}=n_I.\)
Thus, $n_I=\min_{ a\in S' }m^n_{a}=\min_{ a\in Q(S') }m^n_{a}$, $Q(S')$ is relevant and $ \min_{ a\in Q(S') }m^n_{a} > m^n_{\tau_n(t_n)}$.This implies that
\[
\min_{ a\in Q(S') }m^n_{a} \geq m^n_{\tau_n(t_n-1)}.%
\]
Thus, \(Q(S')\subset Q_{t_n-1}(\tau_n),\) which implies that $\dim(V_{Q(S')}) \leq \dim(Q_{t_n-1}(\tau_n))<p$. By Assumption~\ref{ass:6B} and Lemma~\ref{lem:dim-rank-c}, we know that ${Q(S')}$ is not relevant, which contradicts the previous assumption that $Q(S')$ is relevant.
\end{proof}

{The next lemma compares 
the ratio between $n_{\max}=m^n_{\tau_n(1)}$ and $n_I=m^n_{\tau_n(t_n)}$ with the ratio between the maximum and minimum counts of experiments for a  set of relevant experiments.
}
\begin{lemma}\label{lem:kappa bound}
    For any $Q\subset \mathcal{A}$ such that $\dim(V_Q)=p$, 
\begin{equation*}
    \frac{m^n_{\tau_n(1)}}{m^n_{\tau_n(t_n)}}\leq \frac{\max_{a\in \mathcal{A}} m_a^n}{ \min_{a\in Q } m_a^n }.
\end{equation*}
\end{lemma}
\begin{proof}[Proof of Lemma \ref{lem:kappa bound}]
By the definition of $\tau_n$, we know that
\begin{equation*}
    m^n_{\tau_n(1)}\geq\cdots \geq   m^n_{\tau_n(k)}. %
\end{equation*}
Define $m^n_{\tau_n(0)}=\infty$ and $m^n_{\tau_n(k+1)}=-\infty$. %

{Because $(m^n_{\tau_n(1)},\cdots, m^n_{\tau_n(k)})$ forms the order statistic of $(m^n_a)_{a\in[k]}$ (with possibly ties), } there exists $s\in [k]$ such that $Q\subset\{\tau_n(1),\ \cdots, \tau_n(s)\}=Q_s(\tau_n)$, and
\begin{equation*}
    m^n_{\tau_n(s)} = \min_{a\in Q} m_a^n \text{ and }  m^n_{\tau_n(s+1)}<\min_{a\in Q} m_a^n.
\end{equation*}
Because $V_Q(\bgtheta)\subset V_{Q_s(\tau_n)}(\bgtheta)$, we have $p=\dim(V_Q)\leq \dim(V_{Q_s(\tau_n)})$. By the definition of $t_n$ in \eqref{def:t_n}, we obtain that $t_n\leq s$. {This implies $m^n_{\tau_n(t_n)} \geq m^n_{\tau_n(s)} =\min_{a\in Q }m_a^n$. We complete the proof by noting that $ m^n_{\tau_n(1)} = \max_{a\in \mathcal{A}} m^n_a$.} 

\end{proof}

{ 
\paragraph*{Step 2: Unify the proof for generalized \textrm{GI0} and \textrm{GI1}}
To simplify the analysis, we use the next lemma to extract a key property shared by generalized \textrm{GI0} and \textrm{GI1}.  
\begin{lemma}\label{lemma:unify}
    Assume that $\sum_{i=1}^{n_0}\mathcal{I}_{a_i}({\bgtheta})$ is non-singular.

For a fixed (or random) sequence $\bgtheta_n\in \bgTheta$ and for any $n\geq n_0$, we consider the following generalized \textrm{GI0} selection rule
\begin{equation*}
a_{n+1}=\arg\min_{a\in\mathcal{A}} {  \mathbb{F}_{  {\bgtheta}_{n }  }(\frac{n}{n+1}\overline{{\bgpi}}_{n } + \frac{1}{n+1}\delta_a )}  =\arg\min_{a\in \mathcal{A}} \mathbb{G}_{\bgtheta_n}\left( \left\{ \frac{1}{n+1}\blA+ \frac{1}{n+1}\mathcal{I}_a(\bgtheta_n)\right\}  ^{-1}\right),
\end{equation*}
and \textrm{GI1} selection rule
\begin{equation}\label{equ:GI1_genera_select_S5.5}
a_{n+1}=\arg\min_{a\in\mathcal{A}}  \frac{\partial {  \mathbb{F}_{{\bgtheta}_{n }}(\overline{{\bgpi}}_{n })} }{\partial \pi(a)}  =\arg\max_{a\in \mathcal{A}}\innerpoduct{\nabla \mathbb{G}_{\bgtheta_n}( \{\blA/n\}^{-1})}{  {\blA}^{-1}\mathcal{I}_a(\bgtheta_n ){\blA}^{-1}},
\end{equation}
where $\blA=\sum_{a\in \mathcal{A}} m^n_{a}\mathcal{I}_a(\bgtheta_n)$,$\delta_a = (\delta_a(a'))_{a'\in\mathcal{A}}$, and $\delta_a(a')=I(a=a')$.

\sloppy Let $A_2(t_1,t_2)=\sum_{a\in \mathcal{A}} \mathcal{I}_a(\bgtheta_n)+ t_1\mathcal{I}_{a'}(\bgtheta_n) +t_2 \mathcal{I}_{a''}(\bgtheta_n)$, and ${\blS}_n = {\blS}_n(t_1,t_2)= \nabla\mathbb{G}_{\bgtheta_n}(\{A_2(t_1,t_2)/(n+1)\}^{-1})$.
Then,  both generalized \textrm{GI0} and \textrm{GI1} satisfy the following property for all $n\geq n_0$:
\begin{itemize}
    \item []If $a',a''\in\mathcal{A}$ are such that
\begin{equation}\label{eq:unify-selection-property}
     \innerpoduct{ {\blS}_n  }{ \{A_2(t_1,t_2) \}^{-1}  \mathcal{I}_{a'}(\bgtheta_n)   \{A_2(t_1,t_2) \}^{-1}  } 
     >\innerpoduct{{\blS}_n  }{ \{A_2(t_1,t_2) \}^{-1}  \mathcal{I}_{a''}(\bgtheta_n)   \{A_2(t_1,t_2) \}^{-1}  }, 
\end{equation}
for all $t_1,t_2\in[0,1]$ 
, then
$
a_{n+1}\neq a''.
$
\end{itemize}

\end{lemma}
\begin{remark}
The generalized \textrm{GI0} and \textrm{GI1}  defined in \eqref{eq:GI0-exploration} and \eqref{eq:GI1-exploration} are the same as \textrm{GI0} and \textrm{GI1} selections described in Lemma~\ref{lemma:unify}, respectively.
\end{remark}

\begin{proof}[Proof of Lemma~\ref{lemma:unify}]
Let $a',a''$ satisfy \eqref{eq:unify-selection-property}. Assume, in the contrast, that $a_{n+1}=a''$. We will find contradictions for both \textrm{GI0} and \textrm{GI1} in the rest of the proof.

We start with \textrm{GI0}, which selects 
$a_{n+1}=\arg\min_{a\in\mathcal{A}} {  \mathbb{F}_{  {\bgtheta}_{n }  }(\frac{n}{n+1}\overline{{\bgpi}}_{n } + \frac{1}{n+1}\delta_a )}$. Thus, $a''=a_{n+1}$ satisfies
\begin{equation}\label{eq:gi-unify-first}
     \mathbb{F}_{  {\bgtheta}_{n }  }(\frac{n}{n+1}\overline{{\bgpi}}_{n } + \frac{1}{n+1}\delta_{a''} )\leq \mathbb{F}_{  {\bgtheta}_{n }  }(\frac{n}{n+1}\overline{{\bgpi}}_{n } + \frac{1}{n+1}\delta_{a'} ).
\end{equation}
Define $h(t)=\mathbb{F}_{  {\bgtheta}_{n }  }(\frac{n}{n+1}\overline{{\bgpi}}_{n } + \frac{1}{n+1}\{(1-t)\delta_{a'} +t \delta_{a''} \} )$. Then, \eqref{eq:gi-unify-first} is equivalent to that $h(1)-h(0)\leq 0$.

Let $A_1(t)=\sum_{a\in \mathcal{A}} \mathcal{I}_a(\bgtheta_n)  +(1-t)\mathcal{I}_{a'}(\bgtheta_n) +t \mathcal{I}_{a''}(\bgtheta_n)$. By Lemma \ref{lem:Differentiation inverse Matrix}, we know that
\begin{equation*}
    h'(t)=\innerpoduct{ \nabla \mathbb{G}_{\bgtheta_n}(\{A_1(t)/(n+1)\}^{-1})  }{ -\{A_1(t)/(n+1)\}^{-1} \{\mathcal{I}_{a''}(\bgtheta_n)- \mathcal{I}_{a'}(\bgtheta_n) \}  \{A_1(t)/(n+1)\}^{-1}  }.
\end{equation*}
Note that $A_1(t)=A_2(1-t,t)$. Thus, \eqref{eq:unify-selection-property} holds for all $t_1,t_2\in[0,1]$ implies that it holds for $(t_1,t_2)=(1-t,t)$, which further implies $h'(t)>0$ for any $t\in(0,1)$. This contradicts with $h(1)-h(0)\leq 0$. Thus, $a_{n+1}\neq a''$.

We proceed to the analysis of \textrm{GI1}. By the definition of the generalized \textrm{GI1}, $a''=a_{n+1}$ satisfies
$$
\innerpoduct{\nabla \mathbb{G}_{\bgtheta_n}( \{\blA/n\}^{-1})}{  {\blA}^{-1}\mathcal{I}_{a'}(\bgtheta_n ){\blA}^{-1}}\leq \innerpoduct{\nabla \mathbb{G}_{\bgtheta_n}( \{\blA/n\}^{-1})}{  {\blA}^{-1}\mathcal{I}_{a''}(\bgtheta_n ){\blA}^{-1}},
$$
Note that $\blA = A_2(0,0)$. Thus, the above inequality contradicts with \eqref{eq:unify-selection-property} with $(t_1,t_2)=(0,0)$.
\end{proof}

}

{ 
\paragraph*{Step 3: Regularization effect of \textrm{GI0} and \textrm{GI1}}
In this step, we show that both \textrm{GI0} and \textrm{GI1} regularize the experiment selection process through the property established in Lemma~\ref{lemma:unify}. This is proved through the following Lemma~\ref{lem:Li-Bound}, Lemma~\ref{lem:Li-Bound2}, and Lemma~\ref{lem:m^n_a norm bound}. 
}
\begin{lemma}\label{lem:Li-Bound}
    Assume that $\sum_{i=1}^{n_0}\mathcal{I}_{a_i}({\bgtheta})$ is non-singular for some $n_0$. 
    
    Assume the condition number $\max_{{\bgtheta}\in \bgTheta,\bgSigma\succeq 0}\kappa \Big(\nabla \mathbb{G}_{\bgtheta}  \Big(  \bgSigma \Big)\Big)\leq K$ for some $0<K<\infty$. Let $(a^{(1)},a^{(2)},\cdots, a^{(k)})$ be a permutation of $\mathcal{A}$ such that $m^n_{a^{(1)} } \geq \cdots \geq m^n_{a^{(k)} }$ and $\dim(Q_{t_n-1})<\dim(Q_{t_n})=p$, where $Q_s=\{a^{(1)},a^{(2)},\cdots, a^{(s)} \}$, and $t_n=t_n(m^n)$.

If for some {$n\geq n_0$ and } $1\leq s\leq t_n-1$
\begin{equation}\label{ineq:Negative_feedback}
    \left( \frac{m^n_{a^{(s)}}}{m^n_{a^{(s+1)}}} \right)^2 >\frac{8\overline{c}^3  p K}{\underline{c}^3  } \Big(1+16p\frac{\overline{c}^2}{\underline{c}^2} \Big( \frac{   m^n_{a^{(s+1)} }     }{   m^n_{a^{(t_n)}}    }\Big)^2      \Big),
\end{equation}
then 
{ 
\begin{equation*}
     a_{n+1}\in \mathcal{G}(Q_s)=\{a\in \mathcal{A}: \dim(V_{Q_s\cup \{ a\}})>\dim(V_{Q_s })  \}
\end{equation*}
}
for \textrm{GI0} and \textrm{GI1}, where
\begin{equation*}
    V_{Q}=V_{Q}(\bgtheta_n) = \sum_{a\in Q }\mathcal{R}(\mathcal{I}_a(\bgtheta_n)).
\end{equation*}
\end{lemma}
\begin{proof}[Proof of Lemma \ref{lem:Li-Bound}]
Let $t=t_n$.{  \eqref{ineq:Negative_feedback} implies that $t\geq 2$}. By Lemma~\ref{lemma:tau_order}, for any $s\leq t-1$, we have $\dim(V_{Q_t})=p$. Because $\dim(V_{Q_s})<p$, we know that $|\mathcal{G}(Q_s)| >0$.

{
For any $Q\subset \mathcal{A}$, let $\blP_{V_Q }$ be the orthogonal projection matrix on $V_Q({\bgtheta_n})$. We will simplify the notation and write it as $\blP_Q$ for the ease of exposition when the context is clear. Then $\blP_{Q_s}$ denote the orthogonal projection matrix on $V_{Q_s}$.%

According to Lemma~\ref{lemma:unify}, it is sufficient to show that, if \eqref{ineq:Negative_feedback} holds, then for all $a''\not\in \mathcal{G}(Q_{ s})$, $a'\in \mathcal{G}(Q_{ s})$ and all $t_1,t_2\in[0,1]$.
\begin{equation}
     \innerpoduct{ {\blS}_n  }{ \{A_2(t_1,t_2) \}^{-1}  \mathcal{I}_{a'}(\bgtheta_n)   \{A_2(t_1,t_2) \}^{-1}  } 
     >\innerpoduct{{\blS}_n  }{ \{A_2(t_1,t_2) \}^{-1}  \mathcal{I}_{a''}(\bgtheta_n)   \{A_2(t_1,t_2) \}^{-1}  },
\end{equation}
where we recall that ${\blS}_n = \nabla\mathbb{G}_{\bgtheta_n}(\{A_2(t_1,t_2)/(n+1)\}^{-1})$. 
In the rest of the proof, we abuse the notation a little and write ${\blA}=A_2(t_1,t_2)$ for the ease of exposition. Then, it is sufficient to show that for all $a''\not\in \mathcal{G}(Q_{ s})$, $a'\in \mathcal{G}(Q_{ s})$ and all $t_1,t_2\in[0,1]$
\begin{equation}\label{eq:suff-A}
         \innerpoduct{ {\blS}_n  }{ \blA^{-1}  \mathcal{I}_{a'}(\bgtheta_n)   \blA^{-1}  } >\innerpoduct{{\blS}_n  }{ \blA^{-1}  \mathcal{I}_{a''}(\bgtheta_n)   \blA^{-1}  }.
\end{equation}

The rest proof of the consists of the following three steps:
\paragraph*{Step A: Connect \eqref{eq:suff-A} with $\operatorname{tr}(\blP_{Q_s\cup \{ a' \}} {\blA}^{-2} )$ and $\operatorname{tr}(  \blP_{Q_{s}} {\blA}^{-2})$} 

Let
\begin{equation}
\overline{m}^n_{a}=
    \begin{cases}
    {m}^n_{a'}+t_1 &\text{ if } a= a'\\
    {m}^n_{a'}+t_2 &\text{ if } a = a''\\
        {m}^n_{a} &\text{ otherwise}
    \end{cases}.
\end{equation}
Then, for $0\leq t_1,t_2\leq 1 $, $\overline{m}^{n}_{a}\leq m_a^n+1$ for all $a\in \mathcal{A}$, and ${\blA}=\sum_{a\in \mathcal{A}} \overline{m}^n_{a}\mathcal{I}_a(\bgtheta_n)$.

}

{ 
Note that $\lambda_{\max} ({\blS}_n)I_p\succeq\blS_n\succeq \lambda_{\min} ({\blS}_n)I_p$.  
}
we have
\begin{equation}\label{eq:trace-projection-connection-a'}
\begin{split}
 &\innerpoduct{{\blS}_n}{  {\blA}^{-1} \mathcal{I}_{a'}(\bgtheta_n){\blA}^{-1}}\\
   = &\operatorname{tr}({\blS}_n{\blA}^{-1} 
    \mathcal{I}_{a'}(\bgtheta_n){\blA}^{-1} )\\
    =&\operatorname{tr}({\blS}_n{\blA}^{-1} \sum_{a\in Q_s\cup \{ a' \} }
    \mathcal{I}_{a}(\bgtheta_n){\blA}^{-1})-\operatorname{tr}({\blS}_n{\blA}^{-1} 
    \sum_{a  \in  Q_{s} }
    \mathcal{I}_{a}(\bgtheta_n){\blA}^{-1})\\
    \geq & \underline{c} \cdot \lambda_{min} ({\blS}_n)\operatorname{tr}(\blP_{Q_s\cup \{ a' \}}{\blA}^{-2} )-\overline{c}\cdot \lambda_{max}({\blS}_n) \operatorname{tr}( {\blA}^{-1} \blP_{Q_{s}} {\blA}^{-1}),
\end{split}
\end{equation}
{where the last inequality is due to Assumption~\ref{ass:6B} and
Lemma~\ref{lemma:trace}}.

Notice that $a''\not\in \mathcal{G}(Q_s)$  implies
$$
\operatorname{dim}(V_{Q_s\cup \{a''\} }) = \operatorname{dim}(V_{Q_s}) .
$$
Combined with $V_{Q_s}\subset V_{Q_s\cup \{a''\} }$, we know that $V_{Q_s}=V_{Q_s\cup \{a''\} }$. This implies 
\begin{equation}\label{subset:space_32}
    \mathcal{R}(\mathcal{I}_{a''}(\bgtheta_n)) \subset V_{Q_s}.
\end{equation}
By Assumption~\ref{ass:6B}, we obtain 
\[
\mathcal{I}_{a''}(\bgtheta_n) \preceq \overline{c}\cdot \blP_{\mathcal{R}(\mathcal{I}_{a''}(\bgtheta_n)) }  \preceq \overline{c}\cdot \blP_{Q_s}.
\]
Hence, 
\begin{equation}\label{eq:trace-projection-connection-a''}
   \innerpoduct{{\blS}_n}{  {\blA}^{-1} \mathcal{I}_{a''}(\bgtheta_n){\blA}^{-1}} = \operatorname{tr}({\blS}_n{\blA}^{-1} \mathcal{I}_{a''}(\bgtheta_n){\blA}^{-1})
    \leq \overline{c}\cdot \lambda_{\max}({\blS}_n )  \operatorname{tr}( {\blA}^{-1} 
     \blP_{Q_{s}} {\blA}^{-1}).
\end{equation}
Thus, to show \eqref{eq:suff-A}, it is sufficient to show that \eqref{ineq:Negative_feedback} implies 
\begin{equation}\label{ineq:4.20}
    \underline{c} \cdot \lambda_{min} ({\blS}_n)\operatorname{tr}({\blA}^{-1} \blP_{Q_s\cup \{ a' \}} {\blA}^{-1} )> 2\overline{c}\cdot \lambda_{max}({\blS}_n) \operatorname{tr}( {\blA}^{-1} \blP_{Q_{s}} {\blA}^{-1}),
\end{equation}
which is equivalent to\begin{equation}\label{ineq:tr_ratio}
    \frac{\operatorname{tr}(\blP_{Q_s\cup \{ a' \}} {\blA}^{-2} )}{\operatorname{tr}(  \blP_{Q_{s}} {\blA}^{-2})}>\frac{2\overline{c}\cdot \lambda_{max}({\blS}_n) }{\underline{c} \cdot \lambda_{min} ({\blS}_n)}=\frac{2\overline{c} }{\underline{c}  }\kappa({\blS}_n).
\end{equation}
We focus on proving the above inequality in the rest of the proof.

\paragraph*{Step B: Establish a lower bound for $\operatorname{tr}(\blP_{Q_s\cup \{ a' \}} {\blA}^{-2} )$}
Because $V_{Q_s}\subset V_{Q_s\cup \{ a' \}}$ and $\dim(V_{Q_s})< \dim(V_{Q_s\cup \{ a' \}})$, we know that $(I-\blP_{Q_s})\blP_{Q_s\cup \{ a' \}}=\blP_{Q_s\cup \{ a' \}}-\blP_{Q_s}\neq{\bf 0}$. Thus, there exists a unit vector $\blu\in \mathbb{R}^p $ such that $\norm{(I-\blP_{Q})\blP_{Q_s\cup \{ a' \}} \blu}=1$.

Applying the Rayleigh–Ritz quotient for the largest eigenvalue, we know that
\begin{equation}\label{eq:v-bound-2}
\begin{split}
   &\operatorname{tr}( \blP_{Q_s\cup \{ a' \}} {\blA}^{-2})=\operatorname{tr}( \blP_{Q_s\cup \{ a' \}} {\blA}^{-2}\blP_{Q_s\cup \{ a' \}} )\\
   \geq &\lambda_{max}(\blP_{Q_s\cup \{ a' \}} {\blA}^{-2}\blP_{Q_s\cup \{ a' \}} )\geq \blu^T\blP_{Q_s\cup \{ a' \}} {\blA}^{-2}\blP_{Q_s\cup \{ a' \}}  \blu =\norm{{\blA}^{-1}\blP_{Q_s\cup \{ a' \}}  \blu}^2.
\end{split}
\end{equation}
Set $\blv={\blA}^{-1}\blP_{Q_s\cup \{ a' \}}  \blu$. Then,
\begin{equation}\label{eq:v-bound-1}
     \norm{(I-\blP_{Q_{s}})\blA}_{op} \norm{\blv}\geq \norm{(I-\blP_{Q_{s}})\blA \blv}=\norm{(I-\blP_{Q_{s}})\blP_{Q_s\cup \{ a' \}} \blu}=1.
\end{equation}
Notice that $m^n_{a^{(s+1)}}\geq m^n_{a^{(t_n)}} \geq 1$ for any $s\leq t-1$. 
Thus, $m^n_{a^{(s+1)}}+1\leq 2 m^n_{a^{(s+1)}}$.
{ 
Also note that by the definition of $\mathcal{G}(Q_{s})$, $a_{s'}\notin \mathcal{G}(Q_{s})$ for all $s'\in[s]$. Thus, for all $a\in \mathcal{G}(Q_{s})$, $m_a\leq m_a^{(s+1)}$.}
The above analysis, together with Assumption~\ref{ass:6B}, implies
\begin{equation}\label{eq:g-qs-A-bound}
    \sum_{a\in \mathcal{G}(Q_{s})} \overline{m}_a^n  \mathcal{I}_a(\bgtheta_n) \preceq \overline{c}\cdot  ({m}_a^{(s+1)} +1)\cdot \blP_{\mathcal{G}(Q_{s}) } \mathcal{I}_a(\bgtheta_n) \preceq 2\overline{c}\cdot  {m}_a^{(s+1)}  \cdot \blP_{\mathcal{G}(Q_{s}) }.
\end{equation}
{
Note that 
$\blA= \sum_{a  \in \mathcal{G}(Q_{s})} \overline{m}^n_{a}\mathcal{I}_a(\bgtheta_n) + \sum_{a  \notin \mathcal{G}(Q_{s})} \overline{m}^n_{a}\mathcal{I}_a(\bgtheta_n)$. Also note that if  $a  \notin \mathcal{G}(Q_{s})$, then $\mathcal{I}_a(\bgtheta_n)\in V_{Q_{s}}$, which implies $(I-\blP_{Q_{s}})\mathcal{I}_a(\bgtheta_n) =\mathbf{0}$. Thus, \eqref{eq:g-qs-A-bound}} further implies
\begin{equation*}
\begin{split}
    &\norm{(I-\blP_{Q_{s}})\blA}_{op}\\
    =&\norm{\sum_{a\in \mathcal{G}(Q_{s})} \overline{m}_a^n (I-\blP_{Q_{s}})\mathcal{I}_a({\bgtheta}_n)}_{op}\\
    \leq &\norm{(I-\blP_{Q_{s}})}_{op}\norm{\sum_{a\in \mathcal{G}(Q_{s})} \overline{m}_a^n \mathcal{I}_a({\bgtheta}_n)}_{op}\\
    \leq& 2\overline{c}\cdot {m}_a^{(s+1)}  . 
\end{split}
\end{equation*}
The above inequality and \eqref{eq:v-bound-1} implies
$\norm{\blv}\geq \frac{1 }{2\overline{c}\cdot  m^n_{a^{(s+1)}}   }$. This, along with \eqref{eq:v-bound-2}, implies
\begin{equation}\label{eq:qs-lower}
    \operatorname{tr}( \blP_{Q_s\cup \{ a' \}} {\blA}^{-2})\geq \norm{\blv}^2\geq  \frac{1}{4\overline{c}^2\cdot (m^n_{a^{(s+1)}})^2 }.
\end{equation}
 
\paragraph*{Step C: Establish an upper bound for $\operatorname{tr}(  \blP_{Q_{s}} {\blA}^{-2})$}
\begin{equation}\label{eq:qs-upper-1}
    \operatorname{tr}( {\blA}^{-1} \blP_{Q_{s}} {\blA}^{-1}) \leq p \cdot \lambda_{max}( {\blA}^{-1} \blP_{Q_{s}} {\blA}^{-1}).
\end{equation}
Set $\blA_{s} = \sum_{a\not\in \mathcal{G}(Q_s)}\overline{m_a^n} \mathcal{I}_a({\bgtheta}_n)$. Let $r=\dim(V_{Q_s})$.
We first show that $\text{rank}(\blA_s)=r$ and $\mathcal{R}(\blA_s)=V_{Q_s}$.
By Lemma~\ref{lem:dim-rank},  
\begin{equation}\label{equ:rank_dim_equ}
    \operatorname{rank}(\blA_s)=\operatorname{dim}(V_{\mathcal{A}\backslash \mathcal{G}(Q_s) 
\cap \{a\in \mathcal{A}; m^n_a\geq 1 \} } ).
\end{equation}
Because $Q_s\subset\mathcal{A}\backslash \mathcal{G}(Q_s) 
\cap \{a\in \mathcal{A}; m^n_a\geq 1 \}$, we know that
\begin{equation}\label{equ:dim_ineq40}
     \operatorname{dim}(V_{Q_s})\leq \operatorname{dim}(V_{\mathcal{A}\backslash \mathcal{G}(Q_s) 
\cap \{a\in \mathcal{A}; m^n_a\geq 1 \} } ) \leq \operatorname{dim}(V_{\mathcal{A}\backslash \mathcal{G}(Q_s)   } ) .
\end{equation}
By \eqref{subset:space_32} and $V_{Q_s}\subset V_{\mathcal{A}\backslash \mathcal{G}(Q_s)} $, we know that 
\begin{equation}\label{equ:def_sin_Theta}
    V_{\mathcal{A}\backslash \mathcal{G}(Q_s)   } =\sum_{a\not\in  \mathcal{G}(Q_s) }\mathcal{R}(\mathcal{I}_a(\bgtheta_n)) \subset V_{Q_s}\subset V_{\mathcal{A}\backslash \mathcal{G}(Q_s).   } 
\end{equation}
Hence, $V_{\mathcal{A}\backslash \mathcal{G}(Q_s)   }=V_{Q_s}$, and $\operatorname{dim}(V_{\mathcal{A}\backslash \mathcal{G}(Q_s)   })=\operatorname{dim}(V_{Q_s})$. Combined with \eqref{equ:rank_dim_equ} and \eqref{equ:dim_ineq40}, we know that
\[
\operatorname{rank}(\blA_s)=\operatorname{dim}(V_{Q_s})= \operatorname{dim}(V_{\mathcal{A}\backslash \mathcal{G}(Q_s) 
\cap \{a\in \mathcal{A}; m^n_a\geq 1 \} } ) = \operatorname{dim}(V_{\mathcal{A}\backslash \mathcal{G}(Q_s)   } ) =r.
\]
The above analysis and 
\[
\mathcal{R}(\blA_s )\subset \sum_{a\not\in \mathcal{G}(Q_s)} \mathcal{R}(\mathcal{I}_a({\bgtheta}_n)) = V_{\mathcal{A} \backslash \mathcal{G}(Q_s)}=V_{Q_s}
\]
together imply that $\mathcal{R}(\blA_s )=V_{Q_s}$.

Assume the eigendecomposition $\blA_s\widehat{\blu}_i=\widehat\lambda_i \widehat{\blu}_i$, and $\blA {\blu}_i=\lambda_i {\blu}_i$, $1\leq i\leq p$, where $\widehat\lambda_1 \geq \cdots \geq \widehat\lambda_r> \widehat\lambda_{r+1}=\cdots=\widehat\lambda_{p}=0$, $ \lambda_1\geq \cdots \geq \lambda_p$ with $\widehat{\blU}_s=[\widehat{\blu}_1,\widehat{\blu}_2,\cdots,\widehat{\blu}_r]$, $\widehat{\blU}_{-s}=[\widehat{\blu}_{r+1},\widehat{\blu}_{r+2} ,\cdots,\widehat{\blu}_p]$, $\widehat{\blU}=[\widehat{\blU}_s, \widehat{\blU}_{-s}]$, ${\blU}_s=[\blu_1,\blu_2,\cdots,\blu_r]$, ${\blU}_{-s}=[\blu_{r+1},\blu_{r+2} ,\cdots, \blu_p]$, and ${\blU}=[ {\blU}_s,  {\blU}_{-s}]$. Based on the previous notation, we know that $\blP_{Q_s}$ is the orthogonal projection on $\mathcal{R}(\blA_s)$, and thus, it equals $\widehat{\blU}_s \widehat{\blU}_s^T$. 

Let $\bgTheta(\widehat{\blU}_s, {\blU}_s)$ denote the $r \times r$ diagonal matrix whose $j-$th diagonal entry is the $j-$th principal angle $\cos^{-1}(\sigma_j)$, where $\sigma_1\geq \sigma_2\geq \cdots \geq \sigma_r$ are the singular values of $\widehat{\blU}_s^T {\blU}_s$.

Applying a variant of the Davis–Kahan theorem (Theorem 2 in \cite{yu2015useful}), we have 
\begin{equation}\label{eq:DK-theorem}
    \norm{\sin \bgTheta(\widehat{\blU}_s,  {\blU}_s)}_F\leq \frac{2
    \norm{\blA-\blA_s}_F }{\widehat\lambda_r-\widehat\lambda_{r+1}}= \frac{2
    \norm{\blA-\blA_s}_F }{\widehat\lambda_r }.
\end{equation}
Note that 
\begin{equation}\label{equ:A-A_s}
    \blA-\blA_s=\sum_{a\in \mathcal{G}(Q_s)}\overline{m}_a^n\mathcal{I}_a({\bgtheta}_n)\preceq 2\overline{c} m^n_{a^{(s+1)}}  \blP_{\mathcal{G}(Q_s)},
\end{equation}
and
\begin{equation}\label{equ:4p_upper_bound}
    \norm{\blA-\blA_s}^2_F\leq 4p\cdot \overline{c}^2(m^n_{a^{(s+1)}}) ^2.
\end{equation}

Note that $Q_s\subset \mathcal{A}\backslash \mathcal{G}(Q_s)$.
Because $\blA\succeq \blA_s\succeq \sum_{a\in  Q_s}\overline{m}_a^n \mathcal{I}_a(\bgtheta_n)\succeq \underline{c} {m}^n_{a^{(s)}} \cdot \blP_{Q_s}$ %
and $\blA \succeq \sum_{a\in  Q_t}m_a^n\mathcal{I}_a(\bgtheta_n)\succeq \underline{c} m^n_{a^{(t)}}  I_p$, by Courant–Fischer–Weyl min-max principle (see Chapter I of \cite{hilbert1953methods} or Corollary III.1.2 in \cite{bhatia2013matrix}), we have
\begin{equation}\label{eq:lambda-bound}
    \lambda_r\geq \widehat\lambda_r  \geq  \underline{c} m^n_{a^{(s )}} \text{ and } \ \lambda_p\geq  \underline{c} m^n_{a^{(t)}}.
\end{equation}
{
Combining \eqref{eq:DK-theorem}, \eqref{equ:4p_upper_bound}, and \eqref{eq:lambda-bound}, we obtain
\begin{equation}\label{eq:sin-perturb-f}
       \norm{\sin \bgTheta(\widehat{\blU}_s,  {\blU}_s)}_F^2\leq  \frac{16p\cdot \overline{c}^2(m^n_{a^{(s+1)}}) ^2 }{\underline{c}^2 (m^n_{a^{(s )}})^2} .
\end{equation}
}
By definition, we obtain
\begin{equation*}
    \norm{\sin \bgTheta(\widehat{\blU}_s,  {\blU}_s) }_F^2=r-(\cos^2(\sigma_1)+\cos^2(\sigma_2)+\cdots +\cos^2(\sigma_r))=r-\norm{\widehat{\blU}_s^T {\blU}_s}_F^2,
\end{equation*}
and 
\begin{equation*}
    r=\norm{\widehat{\blU}_s^T [ {\blU}_s,  {\blU}_{-s}]  }_F^2=\norm{\widehat{\blU}_s^T {\blU}_s}_F^2+\norm{\widehat{\blU}_s^T {\blU}_{-s}}_F^2.
\end{equation*}
Thus,
\begin{equation}\label{eq:sine-matrix}
    \norm{\sin \bgTheta(\widehat{\blU}_s,  {\blU}_s) }_F^2=\norm{\widehat{\blU}_s^T {\blU}_{-s}}_F^2.
\end{equation}
{Combining \eqref{eq:sin-perturb-f}  and \eqref{eq:sine-matrix}, we have}
\begin{equation}\label{eq:long-lam-max-bound}
\begin{split}
    &\lambda_{max}( {\blA}^{-1} \blP_{Q_{s}} {\blA}^{-1})\\
    =&\lambda_{max}(\widehat{\blU}^T_s {\blA}^{-2} \widehat{\blU}_s)\\
    =&\lambda_{max}(\widehat{\blU}^T_s  [ {\blU}_s,  {\blU}_{-s}]diag(\lambda^{-2}_1,\lambda^{-2}_2,\cdots, \lambda^{-2}_p)[ {\blU}_s,  {\blU}_{-s}]^T \widehat{\blU}_s)\\
    =&\lambda_{max}(\widehat{\blU}^T_s   {\blU}_s diag(\lambda^{-2}_1, \cdots, \lambda^{-2}_r)  {\blU}^T_s  \widehat{\blU}_s + \widehat{\blU}^T_s   {\blU}_{-s} diag(\lambda^{-2}_{r+1}, \cdots, \lambda^{-2}_p)  {\blU}^T_{-s}  \widehat{\blU}_s)\\
    \leq &\lambda_{max}(\widehat{\blU}^T_s   {\blU}_s diag(\lambda^{-2}_1, \cdots, \lambda^{-2}_r)  {\blU}^T_s  \widehat{\blU}_s )+\lambda_{max}( \widehat{\blU}^T_s   {\blU}_{-s} diag(\lambda^{-2}_{r+1}, \cdots, \lambda^{-2}_p)  {\blU}^T_{-s}  \widehat{\blU}_s)\\
    \leq & \lambda_r^{-2}\norm{\widehat{\blU}^T_s   {\blU}_s}_{op}^2 +\lambda_p^{-2} \norm{{\blU}^T_{-s}  \widehat{\blU}_s}_{op}^2 \\
    \leq & \lambda_r^{-2} +\lambda_p^{-2} \norm{\sin \bgTheta(\widehat{\blU}_s,  {\blU}_s) }_F^2\\
    \leq & \frac{1}{(\underline{c} m^n_{a^{(s)}})^2}\left(1+\frac{16p (\overline{c}m^n_{a^{(s+1)}})^2}{(\underline{c} m^n_{a^{(t)}})^2} \right).
\end{split}
\end{equation}

{ 
The above display and \eqref{eq:qs-upper-1} implies
\begin{equation}\label{eq:qs-upper-2} \operatorname{tr}( {\blA}^{-1} \blP_{Q_{s}} {\blA}^{-1}) \leq p \frac{1}{(\underline{c} m^n_{a^{(s)}})^2}\left(1+\frac{16p (\overline{c}m^n_{a^{(s+1)}})^2}{(\underline{c} m^n_{a^{(t)}})^2} \right).
\end{equation}
Combining \eqref{eq:qs-lower}, \eqref{eq:qs-upper-1} and \eqref{eq:qs-upper-2}, we have
\begin{equation} 
\frac{\operatorname{tr}( \blP_{Q_s\cup \{ a' \}} {\blA}^{-2})}{\operatorname{tr}( {\blA}^{-1} \blP_{Q_{s}} {\blA}^{-1}) }
\geq \Big\{\frac{4 \overline{c}^2  p }{\underline{c}^2}\left(1+\frac{16p (\overline{c}m^n_{a^{(s+1)}})^2}{(\underline{c} m^n_{a^{(t)}})^2} \right)\Big\}^{-1}  \left( \frac{m^n_{a^{(s)}}}{m^n_{a^{(s+1)}}} \right)^2
\end{equation}
Thus, \eqref{ineq:Negative_feedback}
implies \eqref{ineq:tr_ratio}. 
}

\end{proof}

\begin{lemma}\label{lem:Li-Bound2}
    Assume that $\sum_{i=1}^{n_0}\mathcal{I}_{a_i}({\bgtheta})$ is nonsingular for some $n_0$. 
    
    Consider the pre-specified criteria function $\mathbb{G}_{\bgtheta}(\bgSigma)=\Phi_q(\bgSigma)$. Let $(a^{(1)},a^{(2)},\cdots, a^{(k)})$ be a permutation of $\mathcal{A}$ such that $m^n_{a^{(1)} } \geq \cdots \geq m^n_{a^{(k)} }$ and $\dim(Q_{t_n-1})<\dim(Q_{t_n})=p$, where $Q_s=\{a^{(1)},a^{(2)},\cdots, a^{(s)} \}$, and $t_n=t_n(m^n)$.

For a fixed (or random) sequence $\bgtheta_n\in \bgTheta$ and for any $n\geq n_0$, we consider the generalized \textrm{GI0} selection rule
\begin{equation*}
a_{n+1}= \arg\min_{a\in \mathcal{A}} \Phi_q\left( \left\{ \frac{1}{n+1}\blA+ \frac{1}{n+1}\mathcal{I}_a(\bgtheta_n)\right\}  ^{-1}\right),
\end{equation*}
and \textrm{GI1} selection rule 
\begin{equation}\label{equ:GI1_Phi_q}
a_{n+1}=\arg\max_{a\in\mathcal{A}}\operatorname{tr} \Big( {\blA}^{-(q+1)} \mathcal{I}_a(\bgtheta_n)  \Big),
\end{equation}
where $\blA=\sum_{a\in \mathcal{A}} m^n_{a}\mathcal{I}_a(\bgtheta_n)$.%
The generalized \textrm{GI1} selection based on \eqref{equ:GI1_Phi_q} coincides with the selection \eqref{equ:GI1_genera_select_S5.5}.

If for some $1\leq s\leq t_n-1$
\begin{equation}\label{ineq:Negative_feedback_q}
    \left( \frac{m^n_{a^{(s)}}}{m^n_{a^{(s+1)}}} \right)^{q+1} \Big(1- \frac{16p \overline{c}^2 }{\underline{c}^2   } \big(\frac{m^n_{a^{(s+1)}} }{ m^n_{a^{(s)}} }\big)^2  \Big) >\Big(\frac{2\overline{c} }{\underline{c}}\Big)^{q+2}  p \left(1+  \frac{16p\overline{c}^{2} }{\underline{c}^{2} } \Big( \frac{   m^n_{a^{(s+1)} }     }{   m^n_{a^{(t_n)}}    }\Big)^{q+1}  \Big(\frac{ m^n_{a^{(s+1)}} }{ m^n_{a^{(s)}} } \Big)^{1-q} \right), 
\end{equation}
then 
\begin{equation*}
     a_{n+1}\in \mathcal{G}(Q_s)=\{a\in \mathcal{A}: \dim(V_{Q_s\cup \{ a\}})>\dim(V_{Q_s })  \}
\end{equation*}
for generalized \text{GI0} and \textrm{GI1}, where
\begin{equation*}
    V_{Q} = \sum_{a\in Q }\mathcal{R}(\mathcal{I}_a(\bgtheta_n)).
\end{equation*}
\end{lemma}

\begin{proof}[Proof of Lemma \ref{lem:Li-Bound2} ]
{
To prove the lemma, we follow similar steps as those in the proof of Lemma~\ref{lem:Li-Bound}. We will omit the repetitive details and only state the main differences. 

By assumption $\mathbb{G}_{\bgtheta}(\bgSigma)=\Phi_q(\bgSigma)$, Lemma~\ref{lem:Differentiation inverse Matrix} and \ref{lem:phi_q}, we know that for any $a,a'\in \mathcal{A}$ and positive definte matrix $\blA\in \mathbb{R}^{p\times p}$,
\begin{equation}\label{equ:equiv_GI1_select}
\begin{split}
    &\innerpoduct{\nabla \Phi_q( \{\blA/n\}^{-1} )}{\blA^{-1} \mathcal{I}_{a}(\bgtheta_n) \blA^{-1}} >\innerpoduct{\nabla \Phi_q( \{\blA/n\}^{-1} )}{\blA^{-1} \mathcal{I}_{a'}(\bgtheta_n) \blA^{-1}}\\
    &\text{if and only if } \operatorname{tr}( {\blA}^{-(q+1)} \mathcal{I}_a(\bgtheta_n) )> \operatorname{tr}( {\blA}^{-(q+1)} \mathcal{I}_{a'}(\bgtheta_n) ).  
\end{split}
\end{equation}
Thus, the generalized \textrm{GI1} selection based on \eqref{equ:GI1_Phi_q} coincides with the selection \eqref{equ:GI1_genera_select_S5.5}.

Similar to the arguments for \eqref{eq:suff-A}, to prove the lemma, it is sufficient to show that \eqref{ineq:Negative_feedback_q} implies that for all $a''\not\in \mathcal{G}(Q_{ s})$, $a'\in \mathcal{G}(Q_{ s})$ and all $t_1,t_2\in[0,1]$,
\begin{equation}\label{equ:S55}
         \innerpoduct{ {\blS}_n  }{ \blA^{-1}  \mathcal{I}_{a'}(\bgtheta_n)   \blA^{-1}  } >\innerpoduct{{\blS}_n  }{ \blA^{-1}  \mathcal{I}_{a''}(\bgtheta_n)   \blA^{-1}  },
\end{equation}
where ${\blA}$ is redefined as $A_2(t_1,t_2)$ and ${\blS}_n = \nabla\mathbb{G}_{\bgtheta_n}(\{A_2(t_1,t_2)/(n+1)\}^{-1}) = \nabla\Phi_q(\{\blA/(n+1)\}^{-1})$. %
Applying \eqref{equ:equiv_GI1_select}, we know that \eqref{equ:S55} is equivalent to
\begin{equation}\label{eq:suff-A2}
 \operatorname{tr}\left({\blA}^{-(q+1)}\mathcal{I}_{a'}(\bgtheta_n) \right) > \operatorname{tr}\left({\blA}^{-(q+1)}\mathcal{I}_{a''}(\bgtheta_n) \right).
\end{equation}
It is sufficient to show that \eqref{ineq:Negative_feedback_q} implies \eqref{eq:suff-A2} for all $a''\not\in \mathcal{G}(Q_{ s})$, $a'\in \mathcal{G}(Q_{ s})$ and all $t_1,t_2\in[0,1]$. Similar to the proof of Lemma~\ref{lem:Li-Bound}, this is proved using the following 3 Steps.
}

\paragraph*{Step A: Connect \eqref{eq:suff-A2} with   $\operatorname{tr}(  {\blA}^{-(q+1) } \blP_{Q_s\cup \{a'\}} )$ and ${\operatorname{tr}( {\blA}^{-(q+1)} \blP_{Q_s} )}$ }

{Similar to the derivation leading to \eqref{eq:trace-projection-connection-a'}, 
we have
\begin{equation*}
\begin{split}
    \operatorname{tr}( {\blA}^{-(q+1)} 
    \mathcal{I}_{a'}(\bgtheta_n) )
    \geq  \underline{c} \cdot  \operatorname{tr}( {\blA}^{-(q+1)}\blP_{Q_s\cup \{a'\}})-\overline{c}\cdot   \operatorname{tr}( {\blA}^{-(q+1)} \blP_{Q_{s}}  ).
\end{split}
\end{equation*}
Similar to \eqref{eq:trace-projection-connection-a''}, we have
\begin{equation}
   \text{tr}(  {\blA}^{-(q+1)} \mathcal{I}_{a''}(\bgtheta_n) )  
    \leq \overline{c} \cdot \operatorname{tr}( {\blA}^{-(q+1)} 
     \blP_{Q_{s}}).
\end{equation}
Thus, to prove \eqref{eq:suff-A2}, it is sufficient to show 
\eqref{ineq:Negative_feedback_q} implies that
\begin{equation}\label{ineq:4.21}
    \underline{c} \cdot \operatorname{tr}( {\blA}^{-(q+1)}\blP_{Q_s\cup \{a'\}} )> 2\overline{c}\cdot   \operatorname{tr}( {\blA}^{-(q+1) } \blP_{Q_{s}}  ),
\end{equation}
which is equivalent to
\begin{equation}\label{ineq:tr_ratio2}
    \frac{\operatorname{tr}(  {\blA}^{-(q+1) } \blP_{Q_s\cup \{a'\}} )}{\operatorname{tr}( {\blA}^{-(q+1)} \blP_{Q_s} )}>\frac{2\overline{c}  }{\underline{c}  }.
\end{equation}
}

\paragraph*{Step B: Establish a lower bound for $\operatorname{tr}(  {\blA}^{-(q+1) } \blP_{Q_s\cup \{a'\}} )$}
Similar to the proof of Lemma~\ref{lem:Li-Bound}, we define $\blA_s=\sum_{a\in \mathcal{G}(Q_s) } \overline{m_a^n} \mathcal{I}_a(\bgtheta_n)$, then $r=\dim(V_{Q_s})=\operatorname{rank}(\blA_s)$ and $\mathcal{R}(\blA_s)=V_{Q_s}$. Assume the eigendecomposition $\blA_s\widehat{\blu}_i=\widehat\lambda_i \widehat{\blu}_i$, and $\blA \blu_i=\lambda_i \blu_i$, $1\leq i\leq p$, where $\widehat\lambda_1 \geq \cdots \geq \widehat\lambda_r> \widehat\lambda_{r+1}=\cdots=\widehat\lambda_{p}=0$, $ \lambda_1\geq \cdots \geq \lambda_p$ with $\widehat{\blU}_s=[\widehat{\blu}_1,\widehat{\blu}_2,\cdots,\widehat{\blu}_r]$, $\widehat{\blU}_{-s}=[\widehat{\blu}_{r+1},\widehat{\blu}_{r+2} ,\cdots,\widehat{\blu}_p]$, $\widehat{\blU}=[\widehat{\blU}_s, \widehat{\blU}_{-s}]$, ${\blU}_s=[\blu_1,\blu_2,\cdots,\blu_r]$, ${\blU}_{-s}=[\blu_{r+1},\blu_{r+2} ,\cdots,\blu_p]$, and ${\blU}=[ {\blU}_s,  {\blU}_{-s}]$.

{Because $a'\notin V_{Q_s}$, there exists a unit vector $\blu\in \mathbb{R}^p$ such that $\blP_{Q_s}\blu=\bm{0}$, and $\blP_{Q_s\cup \{a'\} }\blu=\blu$. Then,
}
\begin{equation*}
    \lambda_{max}( {\blA}^{-(q+1)} \blP_{Q_{s} \cup \{a'\} })=\lambda_{max}( \blP_{Q_{s} \cup \{a'\} } {\blA}^{-(q+1)} \blP_{Q_{s} \cup \{a'\} })\geq  \blu^T {\blA}^{-(q+1)} \blu.
\end{equation*}
Assume $\blu=\sum_{i=1}^p b_i \blu_i=\sum_{i=1}^p \widehat{b}_i \widehat{\blu}_i$. Because $\blP_{Q_s}=\sum_{i=1}^r \widehat{\blu}^T_i \widehat{\blu}_i $ and
\[
\bm{0}=\blP_{Q_s}\blu=\sum_{i=1}^p \widehat{b}_i \blP_{Q_s}\widehat{\blu}_i=\sum_{i=1}^{r} \widehat{b}_i  \widehat{\blu}_i,
\]
we obtain that $\widehat{b}_1=\widehat{b}_2=\cdots=\widehat{b}_r=0$. Thus, we can rewrite $\blu$ as $\widehat\blU_{-s}\widehat{\bgbeta}_1$ with $\norm{\widehat{\bgbeta}_1}=1$.

Note that
\begin{equation*}
    \blu^T {\blA}^{-(q+1)} \blu=\sum_{i=1}^p \lambda^{-(q+1)}_{i} b_i^2\geq \lambda_{r+1}^{-(q+1) } \sum_{i=r+1}^pb_i^2 = \lambda_{r+1}^{-(q+1) } \norm{\blU^T_{-s} \blu}^2.
\end{equation*}
Because $\widehat\bgbeta_1\in \mathbb{R}^{p-r}$ and $\norm{\widehat\bgbeta_1}=1$, we know that there exist unit vectors $\widehat\bgbeta_2,\widehat\bgbeta_3,\cdots,\widehat\bgbeta_{p-r}$ such that $\widehat{\bgbeta}=[\widehat\bgbeta_1,\widehat\bgbeta_2,\cdots,\widehat\bgbeta_{p-r}]$ is an orthogonal matrix. Thus, we know that 
\begin{equation*}
\begin{split}
    &\norm{\blU^T_{-s} \blu}^2=\norm{\blU^T_{-s} \widehat{\blU}_{-s}\widehat\bgbeta_1}^2=\norm{\blU^T_{-s} \widehat{\blU}_{-s}\widehat\bgbeta}^2-\sum_{i=2}^{p-r} \norm{\blU^T_{-s} \widehat{\blU}_{-s}\widehat\bgbeta_i}^2\\
    \geq&  \norm{    \blU^T_{-s} \widehat{\blU}_{-s} }_F^2-(p-r-1)=1-\norm{\operatorname{sin} \bgTheta (\widehat{\blU}_{-s},\blU_{-s}) }_F^2,
\end{split}
\end{equation*}
where the last equation holds because of the definition of $\operatorname{sin} \bgTheta (\widehat{\blU}_{-s},\blU_{-s})$.

{
Combining the above inequalities, we obtain that
\begin{equation}
    \lambda_{max}( {\blA}^{-(q+1)} \blP_{Q_{s} \cup \{a'\} })
    \geq \lambda_{r+1}^{-(q+1) }\cdot \Big(1- \norm{\operatorname{sin} \bgTheta (\widehat{\blU}_{-s},\blU_{-s}) }_F^2\Big).
\end{equation}
By Weyl's inequality and \eqref{equ:A-A_s}, we know that
\begin{equation*}
    \lambda_{r+1}=|\lambda_{r+1}-\widehat{\lambda}_{r+1}|\leq \norm{\blA-\blA_s}_{op} \leq  2\overline{c}m^n_{a^{(s+1)}}.
\end{equation*}
Similar to how we show \eqref{eq:sin-perturb-f}, we also have that
\begin{equation}\label{eq:sin-perturb-f-s}
       \norm{\sin \bgTheta(\widehat{\blU}_{-s},  {\blU}_{-s})}_F^2\leq  \frac{16p\cdot \overline{c}^2(m^n_{a^{(s+1)}}) ^2 }{\underline{c}^2 (m^n_{a^{(s )}})^2} .
\end{equation}
Thus, we obtain
\begin{equation}\label{eq:a-q-lower}
     \operatorname{tr}(  {\blA}^{-(q+1) } \blP_{Q_s\cup \{a'\}} )
     \geq \lambda_{max}( {\blA}^{-(q+1)} \blP_{Q_{s} \cup \{a'\} }) \geq (2\overline{c} m^n_{a^{(s +1)}})^{-q-1}\Big\{1-\frac{16p\cdot \overline{c}^2(m^n_{a^{(s+1)}}) ^2 }{\underline{c}^2 (m^n_{a^{(s )}})^2} 
\Big\}.
\end{equation}
}

\paragraph*{Step C: Establish an upper bound for $\operatorname{tr}( {\blA}^{-(q+1)} \blP_{Q_s} )$}
{Similar to \eqref{eq:long-lam-max-bound} and according to \eqref{eq:sin-perturb-f}, we have}
\begin{equation*}
\begin{split}
    &\lambda_{max}( {\blA}^{-(q+1)} \blP_{Q_{s}}  )\\
    =&\lambda_{max}(\widehat{\blU}^T_s {\blA}^{-(q+1) } \widehat{\blU}_s)\\
    =&\lambda_{max}(\widehat{\blU}^T_s  [ {\blU}_s,  {\blU}_{-s}]diag(\lambda^{-(q+1)}_1,\lambda^{-(q+1) }_2,\cdots, \lambda^{-(q+1) }_p)[ {\blU}_s,  {\blU}_{-s}]^T \widehat{\blU}_s)\\
    =&\lambda_{max}(\widehat{\blU}^T_s   {\blU}_s diag(\lambda^{-(q+1)}_1, \cdots, \lambda^{-(q+1)}_r)  {\blU}^T_s  \widehat{\blU}_s + \widehat{\blU}^T_s   {\blU}_{-s} diag(\lambda^{-(q+1)}_{r+1}, \cdots, \lambda^{-(q+1)}_p)  {\blU}^T_{-s}  \widehat{\blU}_s)\\
    \leq &\lambda_{max}(\widehat{\blU}^T_s   {\blU}_s diag(\lambda^{-(q+1)}_1, \cdots, \lambda^{-(q+1)}_r)  {\blU}^T_s  \widehat{\blU}_s )+\lambda_{max}( \widehat{\blU}^T_s   {\blU}_{-s} diag(\lambda^{-(q+1)}_{r+1}, \cdots, \lambda^{-(q+1)}_p)  {\blU}^T_{-s}  \widehat{\blU}_s)\\
    \leq & \lambda_r^{-(q+1)}\norm{\widehat{\blU}^T_s   {\blU}_s}_{op}^2 +\lambda_p^{-(q+1)} \norm{{\blU}^T_{-s}  \widehat{\blU}_s}_{op}^2\\ 
    \leq & \lambda_r^{-(q+1)} +\lambda_p^{-(q+1)} \norm{\sin \bgTheta(\widehat{\blU}_s,  {\blU}_s) }_F^2\\
    \leq & \frac{1}{(\underline{c} m^n_{a^{(s)}})^{q+1}  }\left(1+\frac{16p (\overline{c})^2 (m^n_{a^{(s+1)}})^{2}}{(\underline{c})^2 ( m^n_{a^{(t_n)}})^{q+1} ( m^n_{a^{(s)}})^{1-q}  } \right),
\end{split}
\end{equation*}
{where we used $\lambda_p\geq  \underline{c} m^n_{a^{(t)}}$ in the last inequality.} Thus,
\begin{equation}\label{eq:a-q-upper}
    \operatorname{tr}( {\blA}^{-(q+1)} \blP_{Q_s} )\leq p \lambda_{max}( {\blA}^{-(q+1)} \blP_{Q_{s}}  ) \leq   \frac{p}{(\underline{c} m^n_{a^{(s)}})^{q+1}  }\left(1+\frac{16p (\overline{c})^2 (m^n_{a^{(s+1)}})^{2}}{(\underline{c})^2 ( m^n_{a^{(t_n)}})^{q+1} ( m^n_{a^{(s)}})^{1-q}  } \right).
\end{equation}

{Combining \eqref{eq:a-q-lower} and \eqref{eq:a-q-upper}, we obtain
\begin{equation}
    \frac{ \operatorname{tr}(  {\blA}^{-(q+1) } \blP_{Q_s\cup \{a'\}} )}{\operatorname{tr}( {\blA}^{-(q+1)} \blP_{Q_s} )}
     \geq \Big(\frac{\underline{c}}{2\overline{c} }\Big)^{q+1} \frac{ (m^n_{a^{(s)}} )^{q+1} }{p(m^n_{a^{(s+1)}} )^{q+1} } 
     \frac{\Big(1-\frac{16p\cdot \overline{c}^2(m^n_{a^{(s+1)}}) ^2 }{\underline{c}^2 (m^n_{a^{(s )}})^2} 
\Big)}{\left(1+\frac{16p (\overline{c})^2 (m^n_{a^{(s+1)}})^{2}}{(\underline{c})^2 ( m^n_{a^{(t_n)}})^{q+1} ( m^n_{a^{(s)}})^{1-q}  } \right)}.
\end{equation}
}
Hence, if  
\begin{equation*}
    \left( \frac{m^n_{a^{(s)}}}{m^n_{a^{(s+1)}}} \right)^{q+1} \Big(1- \frac{16p \overline{c}^2 }{\underline{c}^2   } \big(\frac{m^n_{a^{(s+1)}} }{ m^n_{a^{(s)}} }\big)^2  \Big) >\Big(\frac{2\overline{c} }{\underline{c}}\Big)^{q+2}  p \left(1+  \frac{16p\overline{c}^{2} }{\underline{c}^{2} } \Big( \frac{   m^n_{a^{(s+1)} }     }{   m^n_{a^{(t_n)}}    }\Big)^{q+1}  \Big(\frac{ m^n_{a^{(s+1)}} }{ m^n_{a^{(s)}} } \Big)^{1-q} \right), 
\end{equation*}
then \eqref{ineq:tr_ratio2} holds.
\end{proof}
{
The next lemma is useful for controlling a sequence based on some iterative inequality.
\begin{lemma}\label{lemma:iterative-sequence}
Given $M_0>0, A'>0$, $B',C'\geq 0$, and $q\geq 0$, there exists $M_1,M_2,\cdots$ such that 
\begin{equation*}
    x^{q+1}\leq A' x^{q-1}+B'(1+C' M_j^{q+1} x^{q-1}),
\end{equation*}
then $x\leq M_{j+1}/M_j$. Here, each $M_j$ depends only on $M_0,\cdots,M_{j-1}$, $q$, and $A',B',C'$.
\end{lemma}
\begin{proof}[Proof of Lemma~\ref{lemma:iterative-sequence}]
Let 
\begin{equation}\label{equ:M_j_seq}
    M_{j+1}:=M_{j} \cdot \sup  \{x;x^{q+1}\leq A' x^{q-1}+B'(1+C' M_j^{q+1} x^{q-1}) \}.
\end{equation}
By induction, assume $M_j<\infty$. Due to
\[
\lim_{x\to \infty} \frac{1}{x^{q+1}}\Big( A' x^{q-1}+B'(1+C' M_j^{q+1} x^{q-1}) \Big) =0,
\]
we know that the set $\{x;x^{q+1}\leq A' x^{q-1}+B'(1+C' M_j^{q+1} x^{q-1}) \}$ is bounded from above.

Thus, $M_{j+1}$ defined by \eqref{equ:M_j_seq} is bounded from above. By induction, we complete the proof.
\end{proof}
}

{
\begin{lemma}\label{lem:m^n_a norm bound}
Assume that $\sum_{i=1}^{n_0}\mathcal{I}_{a_i}(\bgtheta)$ is nonsingular. 

Let $(a_n^{(1)},a_n^{(2)},\cdots, a_n^{(k)})$ be a permutation of $\mathcal{A}$ such that $m^n_{a_n^{(1)} } \geq \cdots \geq m^n_{a_n^{(k)} }$ and $\dim(Q_{t_n-1})<\dim(Q_{t_n})=p$, where $Q_{s,n}=\{a_n^{(1)},a_n^{(2)},\cdots, a_n^{(s)} \}$, and $t_n=t_n(m^n)$. To simplify the notation, let $(a^{(1)},a^{(2)},\cdots, a^{(k)})=(a_n^{(1)},a_n^{(2)},\cdots, a_n^{(k)})$ and $Q_s=Q_{s,n}$.

Also assume that  the experiment selection rule satisfy the following property:
\begin{itemize}
    \item []There exists constants $A'\geq 0,B'>0,C'>0$ such that for all $n\geq n_0$, if for some $1\leq s \leq t_n-1$, 
    \begin{equation*}
        \left( \frac{m^n_{a^{(s)}}}{m^n_{a^{(s+1)}}} \right)^{q+1} \Big( 1-A'  \big(\frac{m^n_{a^{(s+1)}}}{m^n_{a^{(s)} } }\big)^2   \Big) > B' \left(1+C' \Big( \frac{   m^n_{a^{(s+1)} }     }{   m^n_{a^{(t_n)}}    }\Big)^{q+1}  \Big(\frac{ m^n_{a^{(s+1)}} }{ m^n_{a^{(s)}} } \Big)^{1-q} \right),
    \end{equation*}
then  $a_{n+1}\in \mathcal{G}(Q_s)$. 
\end{itemize}
Then, this experiment selection rule also satisfies that
$\sup_{n\geq n_0}   \frac{m^n_{a^{(1)}}}{m^n_{a^{(t_n)}}}  <\infty$ and $\inf_{n\geq n_0} \frac{n_I}{n_{max}} \geq C>0$, where $C$ depending only on $A',B',C',k$ and $\frac{m^{n_0}_{a^{(1)}} }{ m^{n_0}_{a^{(t_{n_0})}} }$.
\end{lemma}
}

\begin{proof}[Proof of Lemma \ref{lem:m^n_a norm bound}]
Recall the definition of $t_n=t_n(m^n)$, there exists a permutation $\tau_n\in \mathscr{P}$ over $\mathcal{A}$ such that
\begin{equation*}
    m^n_{\tau_n(1)}\geq m^n_{\tau_n(2)}\geq \cdots\geq m^n_{\tau_n(k)},
\end{equation*}
\begin{equation*}
    a^{(1)}=\tau_n(1),\cdots, a^{(k)}=\tau_n(k),
\end{equation*}
and $\dim(V_{Q_{t_n-1}})<\dim(V_{Q_{t_n}})=p$.
Define
\begin{equation*}
    Ind(n)= \frac{ m^{n}_{a^{(1)}} }{ m^{n}_{a^{(t_n)}} }.
\end{equation*}
To show $\sup_{n\geq n_0}   \frac{m^n_{a^{(1)}}}{m^n_{a^{(t_n)}}}  <\infty$, it is sufficient to show
 that if $n\geq n_0$,
\begin{equation*}
    \sup_{n\geq n_0 }  Ind(n)<\infty.
\end{equation*}
{Let $M_0=\frac{m^{n_0}_{a^{(1)}} }{m^{n_0}_{a^{(t_{n_0})}}}$. According to the definition of $a^{(1)}$ and $a^{(t_{n_0})}$, we know that $M_0\geq 1$.
Next, we use induction to prove that for all $n\geq n_0$
\begin{equation}
    Ind(n)\leq 2\max_{0\leq i\leq k-1} M_{i},
\end{equation}
where the sequence $\{M_{i}\}_{i=1}^{k-1}$ is defined in Lemma~\ref{lemma:iterative-sequence}.}

For the base case, when $n=n_0$, we know that $\sum_{i=1}^{n_0}\mathcal{I}_{a_i}(\bgtheta)$ is nonsingular, and thus $m^{n_0}_{a^{(t_{n_0})}}\geq 1$. This implies,
\begin{equation*}
    Ind(n_0)\leq M_0\leq  m^{n_0}_{a^{(1)}}<\infty.
\end{equation*}

For the induction step $n > n_0$, assume that $Ind(n)\leq 2\max_{0\leq i\leq k-1} M_{i}$. We discuss two cases
\begin{itemize}
    \item [Case 1:] if $a_{n+1}\in \arg\max_{a} m_a^n$, then we will show that $  Ind(n) \leq  \max_{0\leq i\leq k-1} M_{i}$ and $  Ind(n+1) \leq 2 \max_{0\leq i\leq k-1} M_{i}$,
    \item[Case 2:] if $a_{n+1}\not \in \arg\max_{a} m_a^n$, then we will show that $Ind(n+1)\leq Ind(n)$.
\end{itemize}
Below are the detailed analysis for these two cases.
\paragraph*{Case 1: $a_{n+1}\in \arg\max_{a} m_a^n$}
Without loss of generality, we assume that $\tau_n(1)=a^{(1)}=a_{n+1}$. If $t_n=1$, then $Ind(n)=1\leq \max_{0\leq i\leq k-1} M_{i}$. Now, we focus on the case where $t_n\geq 2$. 

Note that $\tau_n(1)$ has been selected as $a_{n+1}$,
and $\tau_n(1)\notin \mathcal{G}(Q_1)$. According to the lemma's assumption, we know that for all $1\leq s \leq t_n-1$,    \begin{equation}\label{ineq:5.43}
    \left( \frac{m^n_{a^{(s)}}}{m^n_{a^{(s+1)}}} \right)^{q+1} \Big( 1-A'  \big(\frac{m^n_{a^{(s+1)}}}{m^n_{a^{(s)} } }\big)^2   \Big) \leq B' \left(1+C' \Big( \frac{   m^n_{a^{(s+1)} }     }{   m^n_{a^{(t_n)}}    }\Big)^{q+1}  \Big(\frac{ m^n_{a^{(s+1)}} }{ m^n_{a^{(s)}} } \Big)^{1-q} \right).
\end{equation}
Next, we use this inequality iteratively for $s=t_n-1,t_n-2,\cdots, 1$ to show that $Ind(n)\leq  \max_{0\leq i\leq k-1} M_{i}$.
We start with setting $s=t_n-1$ in \eqref{ineq:5.43}, we obtain that
\begin{equation*}
    x^{q+1}\leq A' x^{q-1}+B'(1+C' M_0^{q+1} x^{q-1}),
\end{equation*}
where $x=\frac{ m^n_{a^{(t_n-1)}} }{ m^n_{a^{(t_n)}} }$. 
According to Lemma~\ref{lemma:iterative-sequence}, this implies $x= \frac{m^n_{a^{(t_n-1)}}}{m^n_{a^{(t_n)}}}\leq M_1.$ 

Set $s=t_n-2$ in \eqref{ineq:5.43}, and combine it with $\frac{m^n_{a^{(t_n-1)}}}{m^n_{a^{(t_n)}}}\leq M_1$, we have
\begin{equation*}
    x^{q+1}\leq A' x^{q-1}+B'(1+C' M_1^{q+1} x^{q-1}),
\end{equation*}
where $x=\frac{ m^n_{a^{(t_n-2)}} }{ m^n_{a^{(t_n-1)}} }$. Apply Lemma~\ref{lemma:iterative-sequence} again, we obtain that $\frac{m^n_{a^{(t_n-2)}}}{m^n_{a^{(t_n-1)}}}   \leq  M_2/M_1,$ which further implies $\frac{m^n_{a^{(t_n-2)}}}{m^n_{a^{(t_n)}}} \leq M_2$. 
By similar arguments, set $s=t_n-3,t_n-4,\cdots,1$, we obtain that
\begin{equation*}
      \frac{m^n_{a^{(1)}}}{m^n_{a^{(t_n)}}}   \leq M_{t_n-1}\leq \max_{0\leq i\leq k-1} M_{i}.
\end{equation*}
That is, $Ind(n)\leq  \max_{0\leq i\leq k-1} M_{i} $.

Note that in this case $m^{n+1}_{\tau_n(1)}=m^{n}_{\tau_n(1) }+1$, and $m^{n+1}_{\tau_n(s)}=m^{n}_{\tau_n(s)}$, for any $s\geq 2$. Set $Q=\{ \tau_n(1), \tau_n(2),\cdots, \tau_n(t_n) \}$. We know that $\dim(V_Q)=p$. Hence, 
\begin{equation*}
    \frac{ \max_{a}m_a^{n+1}  }{ \min_{a\in Q} m^{n+1}_a }=\frac{ m^{n+1}_{\tau_n(1) }}{m^{n+1}_{\tau_n(t_{n}) } } = \frac{m^{n}_{\tau_n(1) }+1 }{m^{n}_{\tau_n( t_n )} }\leq 2\frac{ m^{n}_{\tau_n(1) } }{ m^{n}_{\tau_n( t_n )}  }\leq 2  \max_{0\leq i\leq k-1} M_{i}.
\end{equation*}
By Lemma \ref{lem:kappa bound}, we know that 
\begin{equation*}
    Ind(n+1)=  \frac{  m^{n+1}_{\tau_{n+1}(1)}   }{ m^{n+1}_{\tau_{n+1}(t_{n+1})} } \leq  \frac{ \max_{a}m_a^{n+1}  }{ \min_{a\in Q} m^{n+1}_a } \leq 2  \max_{0\leq i\leq k-1} M_{i}.
\end{equation*}
\paragraph*{Case 2: $a_{n+1}\not \in \arg\max_{a} m_a^n$} In this case, $\max_{a}m_a^{n+1}=m^n_{\tau_n(1)}$ and $m^{n+1}_{\tau_n(t_n)}\geq m^{n}_{\tau_n(t_n)}=\min_{a\in Q} m^{n}_a$, where we let $Q=\{ \tau_n(1), \tau_n(2),\cdots, \tau_n(t_n) \}$.

Applying Lemma \ref{lem:kappa bound}, we have
\begin{equation*}
    Ind(n+1)\leq \frac{ \max_{a}m_a^{n+1}  }{ \min_{a\in Q} m^{n+1}_a }\leq \frac{ m^n_{\tau_n(1)} }{ m^{n}_{\tau_n(t_n)} }=Ind(n)\leq 2\max_{0\leq i\leq k-1} M_{i},
\end{equation*}
where the last inequality in the above display is due to the induction assumption.

Combine the results from both cases. By induction, we have
\begin{equation*}
     \frac{m^n_{a^{(1)}}}{m^n_{a^{(t_n )}}}    \leq  2\max_{0\leq i\leq k-1} M_{i},
\end{equation*}
for all $n\geq n_0$.
Combined with Lemma~\ref{lem:n_I equation}, we know that
\[
\frac{n_I}{n}=\frac{m^n_{a^{(t_n )}}}{n}\geq \frac{m^n_{a^{(t_n )}}}{k\cdot m^n_{a^{(1)}}} \geq \frac{1}{2 k \cdot\max_{0\leq i\leq k-1} M_{i}}>0,
\]
where $k \cdot\max_{0\leq i\leq k-1} M_{k-1}$ only depend on $A',B',C',k$ and $\frac{m^{n_0}_{a^{(1)}} }{ m^{n_0}_{a^{(t_{n_0})} } }$.
\end{proof}

\begin{proof}[Proof of Theorem~\ref{thm:selection bound}]
Combining Lemmas~\ref{lem:Li-Bound}, \ref{lem:Li-Bound2} and \ref{lem:m^n_a norm bound}, we compete the proof of 
\[
\inf_{n\geq n_0} \frac{n_I}{n}\geq C>0.
\]
{By Assumption~\ref{ass:6B}, we know that
\begin{equation}
    \mathcal{I}^{\overline{{\bgpi}}_n}(\bgtheta)\succeq\sum_{a;m^n_a\geq n_I } \frac{m^n_a}{ n } \mathcal{I}_a(\bgtheta) \succeq \frac{n_I}{ n }\sum_{a;m^n_a\geq n_I }  \mathcal{I}_a(\bgtheta) \succeq \underline{c} \cdot C  \cdot I_p.
\end{equation}
for all $\bgtheta\in\bgTheta$. %
}

\end{proof}

\subsection{Proof of Theorem \ref{thm:consistency_final}}
To show Theorem~\ref{thm:consistency_final}, we  prove the following more general Theorem~\ref{thm:GI0-GI1consistent} instead, which applies to general experiment selection rules that are not necessarily  \textrm{GI0} and \textrm{GI1}.\begin{theorem}\label{thm:GI0-GI1consistent}
    Let $U\in (0,1)$. Assume the experiment selection rule satisfies that $ \overline{\bgpi}_n (\va_n)\in K_U$ for large enough $n$, where
\begin{equation}\label{def:K_U}
    K_U= \left\{ {\bgpi}\in \ShatA: \max_{S\subset\mathcal{A}: S \text{ is relevant} }\min_{a\in S} \pi(a) \geq U \right\}.
\end{equation}
Here, we say that a set of experiments $S$ is relevant if $\sum_{a\in \mathcal{A}} \mathcal{I}_a(\bgtheta)$ is nonsingular for any $\bgtheta\in \bgTheta$. Given Assumptions~\ref{ass:1}-\ref{ass:4} along with either Assumptions~\ref{ass:6A}-\ref{ass:7A} or \ref{ass:6B}-\ref{ass:7B}, the $\widehat{\bgtheta}_{n}^{\text{ML}}$ converges to $\bgtheta^*$ almost surely.
\end{theorem}

\begin{proof}[Proof of Theorem~\ref{thm:consistency_final}]
{ According to Proposition~\ref{prop:exploration}, there exists $U>0$ such that \eqref{def:K_U} holds for $n$ large enough, following \textrm{GI0} or \textrm{GI1}. Theorem~\ref{thm:consistency_final} then follows by applying Theorem~\ref{thm:GI0-GI1consistent}.

}

\end{proof}

\begin{proof}[Proof of Theorem \ref{thm:GI0-GI1consistent}]

Let $\widehat{\bgtheta}_n=\widehat{\bgtheta}_{n}^{\text{ML}}$ for the ease of exposition. According to \eqref{def:MLE}, we know that $l_n(\widehat{\bgtheta}_n; \va_n) \geq l_n(\bgtheta^*; \va_n)$. According to \eqref{eq:uslln} in Assumption~\ref{ass:4}, with probability 1, for any $\eta>0$, there exists $N$ such that for $n>N$, $\overline{{\bgpi}}_n\in K_U$ 
\begin{equation*}
|l_n(\bgtheta^*;\va_n) - M(\bgtheta^*;\overline{{\bgpi}}_n)| \leq  \frac{\eta}{3} \text{ and } |l_n(\widehat{\bgtheta}_n; \va_n) -  M(\widehat{\bgtheta}_n;\overline{{\bgpi}}_n)|\leq \frac{\eta}{3}.
\end{equation*}
It follows that 
\begin{equation*}
l_n(\widehat{\bgtheta}_n; \va_n) \geq M(\bgtheta^*;\overline{{\bgpi}}_n) -\frac{\eta}{3}.
\end{equation*}
Also, we have 
\begin{equation*}
M(\bgtheta^*;\overline{{\bgpi}}_n)- M(\widehat{\bgtheta}_n;\overline{{\bgpi}}_n) \leq l_n(\widehat{\bgtheta}_n; \va_n) -  M(\widehat{\bgtheta}_n;\overline{{\bgpi}}_n)+ \frac{\eta}{3}\leq \frac{2\eta}{3}.
\end{equation*}
That is, for $\eta>0$, 
\begin{equation*}
\mathbb{P} \left\{ \bigcup_{m=1}^{\infty} \bigcap_{n=m}^{\infty} \{ M(\widehat{\bgtheta}_n;\overline{{\bgpi}}_n) -M(\bgtheta^*;\overline{{\bgpi}}_n) \geq - \frac{2}{3}\eta \} \right\} = 1. 
\end{equation*}
It follows that 
\begin{equation*}
\mathbb{P} \left\{ \bigcap_{m=1}^{\infty} \bigcup_{n=m}^{\infty} \{ M(\widehat{\bgtheta}_n;\overline{{\bgpi}}_n) -M(\bgtheta^*;\overline{{\bgpi}}_n) \leq - \eta \} \right\} =0. 
\end{equation*}
Notice that
\[
M(\bgtheta^*;\bgpi)-M(\bgtheta;\bgpi)=\sum_{a\in \mathcal{A}} \pi(a) D_{KL}( f_{\bgtheta^*,a}\|f_{\bgtheta,a})
\]
By Assumption~\ref{ass:7B}, we can show that for any ${\bgpi}\in K_U$, and any $\varepsilon>0$, there exists a finite positive number $\eta=\eta(U,\varepsilon)$, such that
\begin{equation*} %
        \sup_{\bgtheta:   \norm{\bgtheta - \bgtheta^* } \geq \varepsilon} M(\bgtheta;{\bgpi})\leq M\left(\bgtheta^*,{\bgpi}\right)-\eta.
\end{equation*} 
This means that for large enough $n$,
\begin{equation*}
    \left\{ ||\widehat{\bgtheta}_n-\bgtheta^* || \geq \varepsilon \right\} \subset  
    \left\{ M(\widehat{\bgtheta}_n;\overline{{\bgpi}}_n) \leq M(\bgtheta^*;\overline{{\bgpi}}_n)- \eta \right\}.  
\end{equation*}
It follows that for any $\varepsilon>0$, 
\begin{equation*}
    \mathbb{P} \left( \bigcap_{m=1}^{\infty} \bigcup_{n=m}^{\infty} \left\{\norm{\widehat{\bgtheta}_n-\bgtheta^* } \geq \varepsilon \right\} \right) = 0. 
\end{equation*}
Thus,
\begin{equation*}
\begin{split}
    &\mathbb{P}\left( \lim_{n\to \infty} \widehat{\bgtheta}_n=\bgtheta^*\right)=\mathbb{P}\left( \bigcap_{l=1}^\infty \bigcup_{m=1}^\infty \bigcap_{n=m}^\infty \left\{\norm{\widehat{\bgtheta}_n-\bgtheta^*}< \frac{1}{l}\right\}  \right)\\
    =&\lim_{l\to \infty } \mathbb{P}\left(   \bigcup_{m=1}^\infty \bigcap_{n=m}^\infty \left\{\norm{\widehat{\bgtheta}_n-\bgtheta^*}< \frac{1}{l}\right\}  \right)=1.
\end{split}    
\end{equation*}
\end{proof}
\subsection{Proof of Theorem \ref{thm:GI0-GI1AN}}
{The proof of Theorem~\ref{thm:GI0-GI1AN} follows the similar strategy as the proof of the classic asymptotic normality result for MLE with i.i.d. observations, which involves the asymptotic analysis of the Taylor expansion of the score equation. However, the proof for Theorem~\ref{thm:GI0-GI1AN} requires the analysis of dependent stochastic processes and is more delicate.

In the following series of lemmas, we first justify the use of the score equation in Lemma~\ref{lemma:MLE score}. Then, we provide (almost surely) asymptotic bounds for the Hessian of the log-likelihood and the score statistic in Lemma~\ref{lem:a.s.bound}. Lemma~\ref{lem:taylor ln} provides a Taylor expansion for the score function around the true parameter and the MLE, and gives an upper bound for the remaining terms. Finally, these lemmas are combined together to obtain the proof of Theorem~\ref{thm:GI0-GI1AN}.}

\begin{lemma}\label{lemma:MLE score}
Under the setting of Theorem \ref{thm:GI0-GI1consistent}, if $\bgtheta^*\in \operatorname{int}(\bgTheta)$, we have
\begin{equation*}
    \mathbb{P}\left( \bigcup_{m=1}^\infty \bigcap_{n=m}^\infty\{ \nabla_{\bgtheta} l_n(\widehat\bgtheta_n;\va_n)=\bm{0} \} \right)=1.
\end{equation*}
\end{lemma}
\begin{proof}[Proof of Lemma~\ref{lemma:MLE score}]
Let $B(\bgtheta^*,\delta)$ denote the open ball with the center $\bgtheta^*$ and radius $\delta>0$ such that $B(\bgtheta^*,\delta)\subset \operatorname{int}(\bgTheta)$.

    Because $l_n(\bgtheta;\va_n)$ is differentiable in $\bgtheta$, we know that 
    \[
    \Big\{  \norm{\widehat\bgtheta_n- \bgtheta^*}<\delta \Big\}\subset\Big\{ \nabla_{\bgtheta} l_n(\widehat\bgtheta_n;\va_n)=\bm{0}\Big\}.
    \]
Thus,
\begin{equation*}
    1\geq\mathbb{P}\left( \bigcup_{m=1}^\infty \bigcap_{n=m}^\infty\{ \nabla_{\bgtheta} l_n(\widehat\bgtheta_n;\va_n)=\bm{0} \} \right)\geq \mathbb{P}\left( \bigcup_{m=1}^\infty \bigcap_{n=m}^\infty\Big\{  \norm{\widehat\bgtheta_n- \bgtheta^*}<\delta \Big\} \right)=1,
\end{equation*}
{where the last equation is due to the almost sure convergence of $\widehat\bgtheta_n$ obtained from Theorem~\ref{thm:GI0-GI1consistent}.}
\end{proof}
\begin{lemma}\label{lem:a.s.bound}
    Under Assumptions~\ref{ass:1}-\ref{ass:4}, if $\overline{{\bgpi}}_n\in K_U$ for large enough $n$, i.e.,
\begin{equation*}
    \mathbb{P} \left( \bigcup_{m=1}^\infty\bigcap_{n=m}^\infty \{\overline{{\bgpi}}_n\in K_U\} \right) = 1,
\end{equation*}
Also assume that the estimator  $\widehat{\bgtheta}_n\to \bgtheta^*$ a.s. $\mathbb{P}_*$.
Then, 
 with probability 1,
\begin{equation*}
    \limsup_{n\to \infty} \frac{1}{n-1}\sum_{j=1}^{n-1}\Psi_2^{a_j}(X_j)\leq \sum_{a=1}^k \mathbb{E}_{X\sim f_{\bgtheta^*,a} } \Psi_2^a(X)=:\mu_Y<\infty,
\end{equation*}
\begin{equation*}%
\begin{split}
    &\lim_{n\to \infty}\norm{-\nabla_{\bgtheta}^2l_{n-1}(\bgtheta^*;\va_{n-1})-\mathcal{I}^{\overline{{\bgpi}}_{n-1}}(\bgtheta^*)  }_{op}= 0,\\
    &\norm{-\nabla_{\bgtheta}^2 l_{n-1}(\bgtheta;\va_{n-1})+\nabla_{\bgtheta}^2 l_{n-1}(\bgtheta^*;\va_{n-1})}_{op}\leq \frac{1}{n-1}\sum_{i=1}^{n-1}\Psi_2^{a_i}(X_i)\norm{\bgtheta-\bgtheta^*},\\
    &\limsup_{n\to \infty}\norm{(-\nabla_{\bgtheta}^2 l_{n-1}(\widehat{\bgtheta}_{n-2} ;\va_{n-1}))^{-1}}_{op} \leq \frac{1}{\min_{{\bgpi}\in K_U} \lambda_{min}(\mathcal{I}^{\bgpi}(\bgtheta^*))}<\infty,\\
    &\limsup_{n\to \infty}\norm{(-\nabla_{\bgtheta}^2 l_{n-2}(\widehat{\bgtheta}_{n-2};\va_{n-2} ))^{-1}}_{op} \leq \frac{1}{\min_{{\bgpi}\in K_U} \lambda_{min}(\mathcal{I}^{\bgpi}(\bgtheta^*))}<\infty, \\  
    &\limsup_{n\to \infty}\norm{(-\nabla_{\bgtheta}^2 l_{n}( \bgtheta^*;\va_{n}))^{-1}}_{op} \leq \frac{1}{\min_{{\bgpi}\in K_U} \lambda_{min}(\mathcal{I}^{\bgpi}(\bgtheta^*))}<\infty.   
\end{split}
\end{equation*}
\end{lemma}
\begin{proof}[Proof of Lemma \ref{lem:a.s.bound}]
Let the information filtration be
\[
\mathcal{F}_n =\sigma(\{a_1,X_1,\cdots, a_n, X_n\}).
\]
{In the rest of the proof, we restrict the analysis to the event $\bigcup_{m=1}^\infty\bigcap_{n=m}^\infty \{\overline{{\bgpi}}_n\in K_U\} $, which has probability 1 by the assumption.
}
Applying Lemma \ref{lem:SLLN} on each entry of $-\nabla_{\bgtheta}^2 \log f_{\bgtheta^*,a_i}(X_i)$, { 
and note that
$
\mathbb{E}(\nabla_{\bgtheta}^2 \log f_{\bgtheta^*,a_i}(X_i)|\mathcal{F}_{i-1}) =  -\mathcal{I}_{a_i}(\bgtheta^*)
$, 
} we obtain
\begin{equation}\label{eq:information-true-LLN}
    -\nabla_{\bgtheta}^2l_{n-1}(\bgtheta^*;\va_{n-1})-\mathcal{I}^{\overline{{\bgpi}}_{n-1}}(\bgtheta^*)=\frac{1}{n-1}\sum_{i=1}^{n-1}\left( -\nabla_{\bgtheta}^2 \log f_{\bgtheta^*,a_i}(X_i)-I_{a_i}(\bgtheta^*) \right)\stackrel{\text { a.s. }}{\longrightarrow} 0.
\end{equation} 
By Assumption~\ref{ass:2} and the relaxed condition \eqref{cond:relaxed}, we know that 
\begin{equation*}
\norm{-\nabla_{\bgtheta}^2 l_{n-1}(\bgtheta; \va_{n-1})+\nabla_{\bgtheta}^2 l_{n-1}(\bgtheta^*; \va_{n-1})}_{op}\leq \frac{1}{n-1}\sum_{i=1}^{n-1}\Psi_2^{a_i}(X_i)\psi(\norm{\bgtheta-\bgtheta^*}).
\end{equation*}
Let $\{X_j^a\}_{1\leq j\leq n,a\in \mathcal{A}}$ be a sequence of independent random variables such that $X_j^a\sim f_{\bgtheta^*,a}$ for all $j\geq 1$ and $a\in \mathcal{A}$. By Lemma~\ref{lem:same dist}, we can replace $(X_1, X_2, \cdots, X_n)$ with $(X_1^{a_1}, X_2^{a_2}, \cdots, X_n^{a_n})$ without changing the joint distribution for all $n$.

Let $Y_i= \sum_{a\in \mathcal{A}}\Psi_2^{a}(X^a_i)$. %
We know that $\{Y_i\}_{i=1}^\infty$ are i.i.d. and by Assumption~\ref{ass:2}, $\mu_Y:=\mathbb{E}_{\bgtheta^*}Y_1<\infty$.

The strong law of large numbers (see Theorem 2.1 in \cite{ross2014introduction}) %
implies that with probability 1,
\begin{equation*} 
\frac{1}{n}\sum_{j=1}^nY_j\to \mu_Y. 
\end{equation*}
Thus,
{
with probability $1$
\begin{equation*}
    \limsup_{n\to \infty} \frac{1}{n-1}\sum_{j=1}^{n-1}\Psi_2^{a_j}(X^{a_j}_j)
    \leq  \limsup_{n\to \infty} \frac{1}{n-1}\sum_{j=1}^{n-1} \sum_{a\in\mathcal{A}}\Psi_2^{a}(X^{a}_j) 
    = \sum_{a=1}^k \mathbb{E}_{X\sim f_{\bgtheta^*,a}} \Psi_2^a(X)=\mu_Y<\infty.
\end{equation*}
}
Set $A_{n}=\mathcal{I}^{\overline{\bgpi}_{n-1}}(\bgtheta^*)$, and $\Delta A_{n}=-\nabla^2 l_{n-1}(\widehat{\bgtheta}_{n-2};\va_{n-1} ) -A_{n}$ for all $n\geq 1$. Notice that
\begin{equation*}
    A_n^{-1}-(A_n+\Delta A_n)^{-1}\Delta A_nA_n^{-1}= (A_n+\Delta A_n)^{-1}(A_n+\Delta A_n-\Delta A_n)A_n^{-1}=(A_n+\Delta A_n)^{-1},
\end{equation*}
\begin{equation*}
    \norm{(A_n+\Delta A_n)^{-1}-A_n^{-1}}_{op}= \norm{-(A_n+\Delta A_n)^{-1}\Delta A_nA_n^{-1}}_{op} \leq \norm{(A_n+\Delta A_n)^{-1} }_{op} \norm{\Delta A_n}_{op}\norm{A_n^{-1}}_{op},
\end{equation*}
as well as 
\begin{equation*}
    \norm{A_n^{-1}}_{op}\leq \frac{1}{ \min_{{\bgpi}\in K_U} \lambda_{min}(\mathcal{I}^{\bgpi}(\bgtheta^*)) }<\infty,
\end{equation*}
{for any $n$.}
Furthermore,
\begin{equation*}
    \norm{(A_n+\Delta A_n)^{-1}}_{op}\leq \norm{A_n^{-1}}_{op}+\norm{(A_n+\Delta A_n)^{-1}}_{op}\norm{A_n^{-1}}_{op} \norm{\Delta A_n}_{op},
\end{equation*}
{which implies
$$
 \norm{(A_n+\Delta A_n)^{-1}}_{op}\leq \frac{\norm{A_n^{-1}}_{op}}{1-\norm{A_n^{-1}}_{op} \norm{\Delta A_n}_{op}}
$$
given that $\norm{A_n^{-1}}_{op} \norm{\Delta A_n}_{op}<1$.
}
Note that
\begin{equation*}
    \norm{\Delta A_{n}}_{op}\leq \frac{1}{n-1}\sum_{j=1}^{n-1}\Psi_2^{a_j}(X_j)\norm{\widehat{\bgtheta}_{n-2}-\bgtheta^*} +\norm{-\nabla^2l_{n-1}(\bgtheta^*)-\mathcal{I}^{\overline{{\bgpi}}_{n-1}}(\bgtheta^*)  }_{op}.
\end{equation*}
{The first term on the right-hand side of the above inequality converges to $0$ a.s., because of the almost sure convergence assumption on $\hat{\bgtheta}_n$, and the second term converges to $0$ a.s. because of \eqref{eq:information-true-LLN}. Consequently, $ \norm{\Delta A_{n}}_{op}\stackrel{\text { a.s. }}{\longrightarrow} 0$.
This further implies that, for $n$ large enough, $\norm{\Delta A_n}_{op}\leq \frac{1}{2}\min_{{\bgpi}\in K_U} \lambda_{min}(\mathcal{I}^{\bgpi}(\bgtheta^*))$. For such $n$, we have
}
\begin{equation*}
    \norm{(A_n+\Delta A_n)^{-1}-A_n^{-1}}_{op}\leq  \frac{\norm{\Delta A_n}_{op} \norm{A_n^{-1}}^2_{op}}{1-\norm{\Delta A_n}_{op} \norm{A_n^{-1}}_{op}}\stackrel{\text { a.s. }}{\longrightarrow}0.
\end{equation*}
{Note that $A_n+\Delta A_n = -\nabla^2 l_{n-1}(\widehat{\bgtheta}_{n-2};\va_{n-1} )$.
}
Thus, with probability 1,
\begin{equation*}
    \limsup_{n\to \infty}\norm{(-\nabla_{\bgtheta}^2 l_{n-1}(\widehat{\bgtheta}_{n-2}; \va_{n-1} ))^{-1}}_{op} \leq \limsup_{n\to \infty} \norm{A_n^{-1}}_{op} \leq \frac{1}{\min_{{\bgpi}\in K_U} \lambda_{min}(\mathcal{I}^{\bgpi}(\bgtheta^*))}<\infty.
\end{equation*}
Similarly, we can also show that with probability 1,
\begin{equation*}
    \limsup_{n\to \infty}\norm{(-\nabla_{\bgtheta}^2 l_{n-2}(\widehat{\bgtheta}_{n-2} ;\va_{n-2}))^{-1}}_{op} \leq \frac{1}{\min_{{\bgpi}\in K_U} \lambda_{min}(\mathcal{I}^{\bgpi}(\bgtheta^*))}<\infty,
\end{equation*}
as well as 
\begin{equation*}
    \limsup_{n\to \infty}\norm{(-\nabla_{\bgtheta}^2 l_{n}( \bgtheta^*;\va_{n}))^{-1}}_{op} \leq \frac{1}{\min_{{\bgpi}\in K_U} \lambda_{min}(\mathcal{I}^{\bgpi}(\bgtheta^*))}<\infty.
\end{equation*}
\end{proof}
\begin{lemma}\label{lem:taylor ln}
    Under Assumptions~\ref{ass:1}-\ref{ass:4}, the Taylor expansion for the score function $\nabla_{\bgtheta}l_n(\bgtheta;\va_n)$ is given by
\begin{equation}
\begin{split}
    &\nabla_{\bgtheta} l_n(\bgtheta^*+\bm{W}_n/\sqrt{n};\va_n)=\nabla_{\bgtheta} l_n(\bgtheta^* ;\va_n )+\nabla^2_{\bgtheta} l_n(\bgtheta^*;\va_n)\bm{W}_n/{\sqrt{n}}+R(\bgtheta^*,\bm{W}_n),\\
    &\nabla_{\bgtheta} l_{n}(\widehat{\bgtheta}_{n'};\va_{n})=\nabla_{\bgtheta} l_{n}(\widehat{\bgtheta}_{n};\va_{n} )+\nabla^2_{\bgtheta} l_{n}(\widehat{\bgtheta}_{n} ;\va_{n})(\widehat{\bgtheta}_{n'}-\widehat{\bgtheta}_{n}) +R'(\widehat{\bgtheta}_{n'},\widehat{\bgtheta}_{n}),
\end{split}
\end{equation}
where {$\widehat{\bgtheta}_n$ 
and $\widehat{\bgtheta}_{n'}$
     are  the MLE based on $l_n(\bgtheta;\va_n)$ and $l_{n'}(\bgtheta;\va_{n'}), respectively$,}  $\bm{W}_n=\sqrt{n}(\widehat{\bgtheta}_n-\bgtheta^*)$, 
\begin{equation*}
    \norm{R(\bgtheta^*,\bm{W}_n)}\leq \frac{1}{n}\sum_{j=1}^n\Psi^{a_j}_2(X_j) \norm{\widehat{\bgtheta}_n-\bgtheta^* } \psi\Big(\norm{\widehat{\bgtheta}_n-\bgtheta^* }\Big),\text{ and}
\end{equation*}
\begin{equation*}
\norm{R'(\widehat{\bgtheta}_{n'},\widehat{\bgtheta}_{n})}\leq \frac{1}{n}\sum_{j=1}^{n}\Psi^{a_j}_2(X_j) \norm{\widehat{\bgtheta}_{n'}-\widehat{\bgtheta}_{n} }\psi\Big(\norm{\widehat{\bgtheta}_{n'}-\widehat{\bgtheta}_{n} } \Big).
\end{equation*}
\end{lemma}
\begin{proof}[Proof of Lemma \ref{lem:taylor ln}]
    Set $g(t)=\innerpoduct{\blb}{\nabla_{\bgtheta} l_n(\bgtheta^*+t(\widehat{\bgtheta}_n-\bgtheta^*);\va_n)}$. By the Lagrange mean value theorem, there exists $0<t^*<1$ such that
\begin{equation*}
    g(1)-g(0)=g'(t^*),
\end{equation*}
i.e.,
\begin{equation*}
    \innerpoduct{\blb}{\nabla_{\bgtheta} l_n(\widehat{\bgtheta}_n ;\va_n)-\nabla_{\bgtheta} l_n(\bgtheta^* ;\va_n)  }=\innerpoduct{\blb}{\nabla^2_{\bgtheta} l_n(\bgtheta^*+t^*(\widehat{\bgtheta}_n-\bgtheta^*);\va_n ) (\widehat{\bgtheta}_n-\bgtheta^*))}.
\end{equation*}
Then,
\begin{equation}
    \innerpoduct{\blb}{R(\bgtheta^*,\bm{W}_n) }=\innerpoduct{\blb}{\left\{\nabla^2_{\bgtheta} l_n(\bgtheta^*+t^*(\widehat{\bgtheta}_n-\bgtheta^*);\va_n)-\nabla^2_{\bgtheta} l_n(\bgtheta^* ;\va_n) \right\} (\widehat{\bgtheta}_n-\bgtheta^*)}.
\end{equation}
Under Assumption~\ref{ass:2}, we have  
\begin{equation*}
\begin{split}
\norm{R(\bgtheta^*,\bm{W}_n)}&=\sup_{\norm{\blb}\leq 1}\innerpoduct{\blb}{R(\bgtheta^*,\bm{W}_n)}\\
&\leq \max_{0\leq t^*\leq 1}\norm{  \nabla^2_{\bgtheta} l_n(\bgtheta^*+t^*(\widehat{\bgtheta}_n-\bgtheta^*);\va_n)-\nabla^2_{\bgtheta} l_n(\bgtheta^* ;\va_n)  }_{op}\norm{\widehat{\bgtheta}_n-\bgtheta^*}\\
&\leq \frac{1}{n}\sum_{j=1}^n\Psi^{a_j}_2(X_j)\norm{\widehat{\bgtheta}_n-\bgtheta^*}\psi\Big(\norm{\widehat{\bgtheta}_n-\bgtheta^* }\Big).
\end{split} 
\end{equation*}
Similarly, we can show that
\begin{equation*}
    \innerpoduct{\blb}{R'(\widehat{\bgtheta}_{n'},\widehat{\bgtheta}_{n}) }=\innerpoduct{\blb}{\left\{\nabla^2_{\bgtheta} l_{n}(\widehat{\bgtheta}_{n}+t^*(\widehat{\bgtheta}_{n'}-\widehat{\bgtheta}_{n});\va_{n})-\nabla^2_{\bgtheta} l_{n}(\widehat{\bgtheta}_{n} ;\va_{n}) \right\} (\widehat{\bgtheta}_{n'}-\widehat{\bgtheta}_{n})}.
\end{equation*}
Under Assumption~\ref{ass:2}, we have  
\begin{equation*}
\begin{split}
&\norm{R'(\widehat{\bgtheta}_{n'},\widehat{\bgtheta}_{n}) }\\
=&\sup_{\norm{\blb}\leq 1}\innerpoduct{\blb}{R'(\widehat{\bgtheta}_{n'},\widehat{\bgtheta}_{n}) }\\
\leq &\max_{0\leq t^* \leq 1}\norm{  \nabla^2_{\bgtheta} l_{n}(\widehat{\bgtheta}_{n}+t^*(\widehat{\bgtheta}_{n'}-\widehat{\bgtheta}_{n});\va_{n})-\nabla^2_{\bgtheta} l_{n}(\widehat{\bgtheta}_{n} ;\va_{n} )  }_{op}\norm{\widehat{\bgtheta}_{n'}-\widehat{\bgtheta}_{n}}\\
\leq &\frac{1}{n}\sum_{j=1}^{n}\Psi^{a_j}_2(X_j)\norm{\widehat{\bgtheta}_{n'}-\widehat{\bgtheta}_{n}} \psi\Big(\norm{\widehat{\bgtheta}_{n'}-\widehat{\bgtheta}_{n}}\Big).
\end{split} 
\end{equation*}
\end{proof}

\begin{proof}[Proof of Theorem \ref{thm:GI0-GI1AN}]
Write $\widehat\bgtheta_n=\widehat{\bgtheta}_{n}^{\text{ML}}$ for the ease of exposition. By Theorem \ref{thm:consistency_final}, we know that $\widehat\bgtheta_n$ converges to $\bgtheta^*$ almost surely.
By Lemma \ref{lemma:MLE score}, with probability 1, there exists random integer $N<\infty$ such that for any $n\geq N$, $\nabla_{\bgtheta} l_n(\widehat{\bgtheta}_n; \va_n)={\mathbf 0}$. Let $\blW_n = \sqrt{n}(\widehat{\bgtheta}_n - {\bgtheta}^*).$ 

Recall the remainder function defined in Lemma~\ref{lem:taylor ln},
\begin{equation*}
R({\bgtheta}^*, \blW_n) := \nabla_{\bgtheta} l_n({\bgtheta}^* + \blW_n/\sqrt{n};\va_n) - \nabla_{\bgtheta} l_n({\bgtheta}^*;\va_n) - \nabla_{\bgtheta}^2 l_n({\bgtheta}^*;\va_n) \blW_n/\sqrt{n}. 
\end{equation*}
With $\nabla_{\bgtheta} l_n({\bgtheta}^* + \blW_n/\sqrt{n};\va_n)={\mathbf 0}$ provided $n\geq N$ in mind, we can write $\blW_n$
\begin{equation} \label{eq:Wn-decomp}
    \blW_n = - \left\{\nabla_{\bgtheta}^2 l_n({\bgtheta}^*;\va_n) \right\}^{-1} \left\{  \sqrt{n}\nabla_{\bgtheta} l_n({\bgtheta}^*;\va_n) + \sqrt{n}R({\bgtheta}^*, \blW_n) \right\}. 
\end{equation}
The rest of the proof consists of three parts: in Part I, we show that $\sqrt{n}\nabla_{\bgtheta} l_n({\bgtheta}^*;\va_n) \inD N(0, \sum_{a\in \mathcal{A}}  \pi(a) \mathcal{I}_{a}({\bgtheta}^*))$; in Part II, we show that $\nabla_{\bgtheta}^2 l_n({\bgtheta}^*;\va_n) \inP -\sum_{a \in \mathcal{A} } \pi(a) \mathcal{I}_{a}({\bgtheta}^*)$; and in Part III, we show that $\sqrt{n}R({\bgtheta}^*, \blW_n) = o_p(1)$.

\paragraph*{Part I: Show that $\sqrt{n}\nabla_{\bgtheta} l_n({\bgtheta}^*;\va_n) \inD N(0, \sum_{a\in \mathcal{A}}  \pi(a) \mathcal{I}_{a}({\bgtheta}^*))$ as $n\rightarrow\infty$}
Let $\blb$ be any constant vector in $\mathbb{R}^p$ with $\norm{\blb} =1$. For $i = 1,\ldots, n$, let
\begin{equation*}
    \xi_{n,i} := \frac{1}{\sqrt{n}} \blb^T \nabla_{\bgtheta} \log f_{\bgtheta^*,a_i} (X_i).
\end{equation*}
Set $\mathcal{F}_{i}=\sigma\{a_1,X_1,a_2,X_2,\cdots,a_i,X_i\}$ for any $i\geq 1$ and $\mathcal{F}_0$ denote the trivial $\sigma-$algebra. Applying the Dominated Convergence Theorem, coupled with the classical proof of differentiation under the integral sign, we arrive at the conclusion that $\mathbb{E}(\xi_{n,i}|\mathcal{F}_{i-1}) = 0$,
which implies that $\mathbb{E}(\xi_{n,i}) = 0$. Denote  $\sigma_{n,i}^2 := \mathbb{E}(\xi_{n,i}^2 |\mathcal{F}_{i-1}) = \frac{1}{n} \blb^T \mathcal{I}_{a_i}({\bgtheta}^*) \blb$. 

Let $S_{n}:= \sum_{i=1}^{n} \xi_{n,i}$. Note that $\mathbb{E}(S_{n}) = 0$ and $\mathbb{E}(S_n^2) < \infty$, since $\mathbb{E}(\xi_{n,i}^2) < \infty$ for all $i$. Then, $\{S_n, \mathcal{F}_n\}_{n\geq1}$ is a martingale array with mean 0 and finite variance. We will apply the martingale central limit theorem to $S_n$. We check the conditions first.

We first check the conditional variance condition. We write  
\begin{equation*}
    \sum_{i=1}^n \sigma_{n,i}^2 = \blb^T \left\{\frac{1}{n} \sum_{i=1}^n  \mathcal{I}_{a_i}({\bgtheta}^*) \right\} \blb = \blb^T \left\{ \sum_{a\in \mathcal{A}}  \overline{\pi}_n(a) \mathcal{I}_{a}({\bgtheta}^*) \right\} \blb.  
\end{equation*}
Due to the convergence assumption of $\overline{{\bgpi}}_n$, we have 
\begin{equation*}
    \blb^T \left\{ \sum_{a\in \mathcal{A}}  \overline{\pi}_n(a) \mathcal{I}_{a}({\bgtheta}^*) \right\} \blb \inP \blb^T \left\{ \sum_{a\in \mathcal{A}}  \pi(a) \mathcal{I}_{a}({\bgtheta}^*) \right\} \blb. 
\end{equation*}
Then the conditional variance condition holds. 

We then check the conditional Lindeberg's condition. Assume random variables $\{X^a\}_{a\in \mathcal{A}} $ have densities $\{f_{\bgtheta^*,a} \}_{a\in \mathcal{A}}$, respectively. For any $\varepsilon >0$, with probability 1,
\begin{equation*} 
    \sum_{i=1}^n \mathbb{E}\left\{ \xi_{n,i}^2 I (|\xi_{n,i}|>\varepsilon) | \mathcal{F}_{i-1} \right\} 
    \leq  \sum_{a=1}^k \mathbb{E} \left\{\norm{\nabla_{\bgtheta} \log f_{\bgtheta^*,a}(X^a)}^{2 }I (\norm{\nabla_{\bgtheta} \log f_{\bgtheta^*,a}(X^a)}>\sqrt{n}\varepsilon)  \right\} .
\end{equation*}
By Assumption~\ref{ass:2},
\begin{equation*}
    \lim_{n\to \infty}\mathbb{E}\left\{\norm{\nabla_{\bgtheta} \log f_{\bgtheta^*,a}(X^a)}^{2 }I (\norm{\nabla_{\bgtheta} \log f_{\bgtheta^*,a}(X^a)}>\sqrt{n}\varepsilon)  \right\} = 0.
\end{equation*}
Thus, the conditional Lindeberg condition holds. 

By the Martingale Central Limit Theorem (Corollary 3.1 in \cite{hall2014martingale}), we have
\begin{equation*}
    \sum_{i=1}^n \xi_{n,i} \inD N \left(0, \blb^T \left\{ \sum_{a \in \mathcal{A} } \pi(a) \mathcal{I}_{a}({\bgtheta}^*) \right\} \blb\right). 
\end{equation*}
It follows by Cramér–Wold theorem (see \cite{billingsley1999convergence} p383) that
\begin{equation}\label{equ:lim_grad_ln}
    \sqrt{n}\nabla_{\bgtheta} l_n({\bgtheta}^*;\va_n) \inD N \left(0, \sum_{a \in \mathcal{A} } \pi(a) \mathcal{I}_{a}({\bgtheta}^*)\right). 
\end{equation}

\paragraph*{ Part II: Show that $\nabla_{\bgtheta}^2 l_n({\bgtheta}^*;\va_n) \inP -\sum_{a \in \mathcal{A} } \pi(a) \mathcal{I}_{a}({\bgtheta}^*)$}

For each $i=1,\ldots, n$, by Assumption A3, we have
\begin{equation*}
    \mathbb{E} \left\{ \nabla_{\bgtheta}^2 \log f_{\bgtheta^*,a_i}(X_i) | \mathcal{F}_{i-1} \right\} = -\mathcal{I}_{a_i}({\bgtheta}^*).
\end{equation*}
Also, the conditional expectation has
\begin{equation*} 
    \frac{1}{n} \sum_{i=1}^n  \mathcal{I}_{a_i}({\bgtheta}^*) = \mathcal{I}^{\overline{\bgpi}_n}(\bgtheta^*) \inP \sum_{a \in \mathcal{A} } \pi(a) \mathcal{I}_{a}({\bgtheta}^*),
\end{equation*}
due to the convergence assumption of $\overline{{\bgpi}}_n$. 

For each $i,l = 1,\ldots, p$, define 
\begin{equation*}
    G_{i,l} = \sum_{a\in \mathcal{A}, X^a\sim f_{\bgtheta^*,a}(\cdot)} \left| \left( \nabla_{\bgtheta}^2 \log f_{\bgtheta^*,a}(X^a) \right)_{i,l} \right|. 
\end{equation*}
Then, for all $x\geq0$ and $i,l\geq1$, 
\begin{equation*}
    \mathbb{P} \left\{ \left| \left( \nabla_{\bgtheta}^2 \log f_{\bgtheta^*,a}(X^a) \right)_{i,l} \right|
    >x \right\} \leq  \mathbb{P} \{ G_{i,l} >x\}
\end{equation*}
{This implies that $\sum_{a\in\mathcal{A}}\mathbb{E}(\left| \left( \nabla_{\bgtheta}^2 \log f_{\bgtheta^*,a}(X^a) \right)_{i,l} \right|)\leq \sum_{i,l}\mathbb{E}(G_{i,l})<\infty$ under Assumption~\ref{ass:2}.}

By Lemma~\ref{lem:SLLN} {and the Slutsky's theorem}, we arrive at
\begin{equation} \label{eq:lln}
    \frac{1}{n} \sum_{i=1}^{n}\nabla_{\bgtheta}^2 \log f_{\bgtheta^*,a_i}(X_i) = \nabla_{\bgtheta}^2 l_n({\bgtheta}^*;\va_n) \inP -\sum_{a \in \mathcal{A} } \pi(a) \mathcal{I}_{a}({\bgtheta}^*).
\end{equation}

\paragraph*{Part III: Show that $\sqrt{n}R({\bgtheta}^*, \blW_n) = o_p(1)$ }

According to Lemma~\ref{lem:a.s.bound} and Lemma~\ref{lem:taylor ln}, we know that
\begin{equation} \label{eq:R-ineq}
    \norm{ \sqrt{n} R({\bgtheta}^*, \blW_n)} \leq \frac{1}{ n }\sum_{j=1}^n\Psi_2^{a_j}(X_j)\norm{\blW_n}  \psi\Big(\norm{\widehat{\bgtheta}_n-{\bgtheta}^*}\Big). 
\end{equation}
By \eqref{eq:Wn-decomp}, we have 
\begin{align} \label{eq:Wn-ineq}
     \norm{\blW_n} &\leq \norm{\left\{\nabla_{\bgtheta}^2 l_n({\bgtheta}^*;\va_n) \right\}^{-1}}_{op} \left\{  \norm{\sqrt{n}\nabla_{\bgtheta} l_n({\bgtheta}^*;\va_n)} + \norm{\sqrt{n}R({\bgtheta}^*, \blW_n)} \right\}.
\end{align}    
{ 
The above two inequalities together implies
\begin{equation} \label{eq:Wn-ineq2}
\begin{split}
     \norm{ \sqrt{n} R({\bgtheta}^*, \blW_n)}  \leq &\frac{\frac{1}{ n }\sum_{j=1}^n\Psi_2^{a_j}(X_j)\psi\Big(\norm{\widehat{\bgtheta}_n-{\bgtheta}^*}\Big)\norm{\left\{\nabla_{\bgtheta}^2 l_n({\bgtheta}^*) \right\}^{-1}}_{op}   \norm{\sqrt{n}\nabla_{\bgtheta} l_n({\bgtheta}^*)}}{1-\frac{1}{n}\sum_{j=1}^n\Psi_2^{a_j}(X_j)\norm{\left\{\nabla_{\bgtheta}^2 l_n({\bgtheta}^*) \right\}^{-1}}_{op}  \psi\left(\norm{\widehat{\bgtheta}_n-{\bgtheta}^*}\right) }\\
\end{split}
\end{equation}
given that $\frac{1}{n}\sum_{j=1}^n\Psi_2^{a_j}(X_j)\norm{\left\{\nabla_{\bgtheta}^2 l_n({\bgtheta}^*) \right\}^{-1}}_{op}  \psi\left(\norm{\widehat{\bgtheta}_n-{\bgtheta}^*}\right)<1$.
}
We have shown in \eqref{eq:lln} in Part II that 
\begin{equation*}
    \nabla_{\bgtheta}^2 l_n({\bgtheta}^*;\va_n) = -\sum_{a \in \mathcal{A} } \pi(a) \mathcal{I}_{a}({\bgtheta}^*) + o_p(1) , 
\end{equation*}
and thus
\begin{equation} \label{eq:hessian-op1}
     \norm{\left\{\nabla_{\bgtheta}^2 l_n({\bgtheta}^*;\va_n) \right\}^{-1}}_{op} = \lambda_{\min}^{-1} \left( \sum_{a \in \mathcal{A} } \pi(a) \mathcal{I}_{a}({\bgtheta}^*) \right) +o_p(1).
\end{equation}
Furthermore, under Assumption~\ref{ass:5}, by the consistency result in Theorem~\ref{thm:consistency_final} and Lemma \ref{lem:a.s.bound}, we have
\begin{equation*}
    \norm{\widehat{\bgtheta}_n-{\bgtheta}^*} = o_p(1),
\end{equation*}
which implies that 
\begin{equation} \label{eq:denom-op1}
    \frac{1}{n}\sum_{j=1}^n\Psi_2^{a_j}(X_j)\norm{\left\{\nabla_{\bgtheta}^2 l_n({\bgtheta}^*) \right\}^{-1}}_{op}   \psi\left(\norm{\widehat{\bgtheta}_n-{\bgtheta}^*}\right) = o_p(1).
\end{equation}
Let the event $D_n := \big\{ 1-\frac{1}{n}\sum_{j=1}^n\Psi_2^{a_j}(X_j)\norm{\left\{\nabla_{\bgtheta}^2 l_n({\bgtheta}^*) \right\}^{-1}}_{op}  \psi\left(\norm{\widehat{\bgtheta}_n-{\bgtheta}^*}\right) > \frac{1}{2}\big\}$. We have 
\begin{equation*}
    \mathbb{P} (D_n) \to 1, \text{ as } n \to \infty. 
\end{equation*}
On the event $D_n$, according to \eqref{equ:lim_grad_ln} %
, we have
\begin{equation*}%
    \norm{\sqrt{n}\nabla_{\bgtheta} l_n({\bgtheta}^*)} = O_p(1),
\end{equation*}
which together with \eqref{eq:Wn-ineq2}, \eqref{eq:denom-op1} and Lemma~\ref{lem:a.s.bound} yields $\norm{\blW_n} = O_p(1)$.
It follows from \eqref{eq:R-ineq} that
\begin{equation*}
    \norm{ \sqrt{n} R({\bgtheta}^*, \blW_n)} =o_p(1). 
\end{equation*}
Therefore, applying Slutsky's Theorem and the continuous mapping Theorem to \eqref{eq:Wn-decomp}, we have 
\begin{equation*}
    \blW_n \inD N \left(0, \left(\sum_{a \in \mathcal{A} } \pi(a) \mathcal{I}_{a}({\bgtheta}^*)\right)^{-1} \right), 
\end{equation*}
which concludes the proof. 
\end{proof}
\subsection{Proof of Theorem \ref{thm:empirical pi as converge}}
{We first present an extension of the classic convergence theorem by \cite{robbins1971convergence}, which is frequently employed to prove convergence of stochastic processes within the fields of stochastic approximation and reinforcement learning. It provides conditions on a stochastic process $\{Z_n\}$ for it to converge almost surely. The following modified version of the Robbins-Siegmund Theorem allows us to obtain a better estimate of the convergence rate of $\{Z_n\}$. Later in this section, we will apply this result to  $Z_n=\mathbb{F}_{ \bgtheta^*}( \overline{{\bgpi}}_{n } )-\mathbb{F}_{ \bgtheta^*}({\bgpi}^*)$ for proving Theorem~\ref{thm:empirical pi as converge}.
}
\begin{lemma}[Modified Robbins-Siegmund Theorem]\label{lem:RS lem}
Let $a_n, c_n$ be integrable random variables and $Z_n$ be a non-negative integrable random variable adaptive to filtration $\mathcal{F}_n$ for all $n\geq 1$, and $\mathcal{F}_1\subset \mathcal{F}_2\cdots$. Assume that
\begin{equation}\label{ineq:RStm}
    \mathbb{E}\left[Z_{n+1} \mid \mathcal{F}_n\right] \leq\left(1-a_n \right) Z_n+c_n, \text{ for all }n\geq 1.
\end{equation} 
Set $a_n^-=\max\{ 0, -a_n \}$ and $a_n^+=\max\{ 0, a_n \}$. Assume
    \begin{equation*}
        \sum_{n=1}^\infty a^-_n<\infty.
    \end{equation*}
Then, the following statements hold.    
\begin{enumerate}
    \item If
    \begin{equation}\label{equ:sum_c_finite}
        \sum_{n=1}^\infty c_n \text{ exists with probability $1$,}
    \end{equation}
    then there exists non-negative random variable $Z_\infty$ such that $\lim_{n\to \infty} Z_n=Z_\infty$ with probability 1.
    \item If we assume \eqref{equ:sum_c_finite} holds and further require
\begin{equation}\label{equ:aZ_intergrable}
    \{a_n^+Z_n\}_{n=1}^\infty \text{ are all intergrable, and }\sum_{n=1}^\infty a_n=+\infty \text{  with probability $1$},
\end{equation}
then $\lim_{n\to \infty} Z_n=0$ with probability 1.
    \item Assume \eqref{equ:aZ_intergrable} holds. If there exists $0<\beta<c$ such that $a_n\geq\frac{c}{n}$ and the limit $\sum_{n=1}^\infty n^\beta c_n$ exists with probability 1, then $\lim_{n\to \infty}n^\beta Z_n=0$ with probability 1.
\end{enumerate}
\end{lemma}

\begin{proof}[Proof of Lemma \ref{lem:RS lem}]
~~
\paragraph*{Part 1} 

First of all, \eqref{ineq:RStm} implies
\begin{equation*}
    \mathbb{E}\left[Z_{n+1} \mid \mathcal{F}_n\right] \leq\left(1+a^-_n \right) Z_n+c_n.
\end{equation*}
Set $Z'_n=\frac{Z_n}{\prod_{i=1}^{n-1}(1+a_i^-) }$, and $c_n'=\frac{c_n}{\prod_{i=1}^{n}(1+a_i^-)}$. Notice that $\sum_{n=1}^\infty a_n^-<\infty$ implies $\prod_{n=1}^\infty (1+a_n^-)<\infty$. By Abel's test for series (see Exercise 9.15 in \cite{ghorpade2006course}), we know that 
\begin{equation}\label{equ:c_n_sum_prob1}
    \mathbb{P}\Big(\sum_{n=1}^\infty c_n' \text{ exsits }\Big)=1,
\end{equation}
Because $|c_n'|\leq |c_n|$, $0\leq Z_n' \leq  Z_n$ as well as $c_n$ and $Z_n$ are integrable, we know that $c_n'$ and $Z_n'$ are also integrable. Note that
\begin{equation}\label{ineq:sup_martingale}
    \mathbb{E}\left[Z_{n+1}' \mid \mathcal{F}_n\right] \leq Z'_n+c'_n.
\end{equation}
Let $Y_1=Z_1'$ and $Y_n=Z_n'-(c_1'+\cdots+c_{n-1}')$, which are integrable for all $n\geq 2$. We know that $Y_n$ is integrable for all $n\geq 1$. By \eqref{ineq:sup_martingale}, we obtain
\begin{equation}\label{ineq:sup_martingale_Y}
    \mathbb{E}\left[Y_{n+1}  \mid \mathcal{F}_n\right] \leq Y_n.
\end{equation}
Let $\tau_T=\inf\{n; \sum_{k=1}^nc_k'>T\}$, for any $T\geq 0$. Note that $\{ \tau_T>n \}=\{ \sum_{i=1}^n c'_n\leq T \}\in \mathcal{F}_n$. Because
\begin{equation*}
    Y_{n\wedge \tau_T}=\sum_{l=1}^nY_l\cdot  I(\tau_T=l)+Y_n\cdot I(\tau_T>n), \text{ and }|Y_{n\wedge \tau_T}|\leq \sum_{i=1}^n|Y_i|,
\end{equation*}
we obtain that $Y_{n\wedge \tau_T}$ is integrable for all $n\geq 1$. 

By the definition of $\tau_T$, we know that $i\leq \tau_T-1\implies\sum_{k=1}^{i} c_k^{\prime} \leq T$, which implies that for any $n\geq 1$
\begin{equation*}
Y_{n \wedge \tau_T} \geq-\sum_{k=1}^{n \wedge \tau_T-1} c_k^{\prime} \geq-T.
\end{equation*}
By \eqref{ineq:sup_martingale_Y}, we know that for any $n\geq 1$,
\begin{equation*}
\begin{split}
    \mathbb{E}[Y_{(n+1) \wedge \tau_T}|\mathcal{F}_n]&=Y_{\tau_T} I(\tau_T\leq n)+\mathbb{E}[Y_{n+1}|\mathcal{F}_n] I(\tau_T > n)\\
    &\leq \sum_{l=1}^nY_{\tau_T} I(\tau_T = l)+ Y_{n }  I(\tau_T > n)= Y_{n\wedge \tau_T},\text{ and}
\end{split}
\end{equation*}
\begin{equation*}
    0\leq \mathbb{E}[Y_{n \wedge \tau_T} +T ]\leq \cdots\leq  \mathbb{E}[Y_{1\wedge \tau_T} +T ]=\mathbb{E}Z'_1+T<\infty.
\end{equation*}
This concludes that $Y_{n \wedge \tau_T}+T$ is a non-negative supermartingale (see Section 1.1 in \cite{hall2014martingale}). Applying Doob's convergence theorem (see Theorem 2.5 in \cite{hall2014martingale}) to $L^1$ uniformly bounded submartingale $-(Y_{n \wedge \tau_T}+T)$
, we know that $\lim_{n\to \infty }Y_{n \wedge \tau_T}$ exists and is finite for any $T\geq 0$. 

In conclusion, $\lim_{n\to \infty }Y_{n }$ exists and is finite almost surely on event 
\[
\{ \tau_T=\infty \} = \big\{ \sum_{i=1}^n c'_i\leq T \text{ for any }n\geq 1 \big\} \text{ for any }T\geq 0.
\]
Combining this with \eqref{equ:c_n_sum_prob1}, we know that $\lim_{n\to \infty} Y_n$ exists and is finite almost surely. Hence, with probability 1, we have
\[
Z_\infty=\lim_{n\to \infty} Z_n=\prod_{n=1}^\infty (1+a_n^-)\Big(\lim_{n\to \infty} Y_n-\sum_{k=1}^{\infty}c'_{k}\Big).
\]
\paragraph*{Part 2} 

Because $a_n^+=\max\{0,a_n\}$, we have $a_n'=\frac{a_n^+}{1+a_n^-}\geq 0$. Similar to the arguments in Part 1, we have
\begin{equation*}
    \mathbb{E}\left[Z'_{n+1} \mid \mathcal{F}_n\right] \leq\left(1-a'_n \right) Z'_n+c'_n.
\end{equation*}
Because we assume that $\sum_{n=1}^\infty a_n=+\infty$ with probability 1, we have
$$
\sum_{n=1}^N a_n'\geq \frac{1}{1+\sup_{n}a_n^-}\sum_{n=1}^N(a_n-a_n^- )\to +\infty,
$$
as $N\to \infty$ with probability 1. Let $Y_1'=Z_1'$ and for any $n\geq 2$
$$
Y'_n=Z_n'+\sum_{k=1}^{n-1}a_k'Z_k'-\sum_{k=1}^{n-1}c_k'.
$$
Since $|a_n'|\leq |a_n|$, $|a'_nZ_n'|\leq a_n^+Z_n$, $|a_n|$ and $a_n^+Z_n$ are intergrable, we know that $Y_n'$ is intergrable for any $n\geq 1$. Similar to the arguments in Part 1, we obtain that $Y'_{n\wedge \tau_T}$ is intergrable for any $T\geq 0$,
\begin{equation*}
    \mathbb{E}[Y'_{ n+1}|\mathcal{F}_n] \leq Y'_n \text{ and}
\end{equation*}
\begin{equation*}
    Y'_{ n \wedge \tau_T} \geq-\sum_{k=0}^{n \wedge \tau_T-1} c_k^{\prime} \geq-T, \text{ and}
\end{equation*}
\begin{equation*}
\begin{split}
    \mathbb{E}[Y'_{(n+1) \wedge \tau_T}|\mathcal{F}_n]&=Y'_{\tau_T} I(\tau_T\leq n)+\mathbb{E}[Y'_{n+1}|\mathcal{F}_n] I(\tau_T > n)\\
    &\leq \sum_{l=1}^nY'_{\tau_T} I(\tau_T = l)+ Y'_{n }  I(\tau_T > n)= Y'_{n\wedge \tau_T}.
\end{split}
\end{equation*}
In conclusion, we obtain that $Y'_{ n \wedge \tau_T}$ is a super-martingale, such that 
\[
0\leq   \mathbb{E}[Y'_{ n \wedge \tau_T}+T]\leq \cdots \leq  \mathbb{E}[Y'_{ 1 \wedge \tau_T}+T]=\mathbb{E}Z'_1+T<\infty.
\]
This concludes that $\{Y'_{n \wedge \tau_T}+T\}_{n=1}^\infty$ is a $L^1$ uniformly bounded supermartingale (see Section 1.1 in \cite{hall2014martingale}). Applying Doob's convergence theorem (see Theorem 2.5 in \cite{hall2014martingale}) to $L^1$ uniformly bounded submartingale $-(Y'_{n \wedge \tau_T}+T)$, we know that $\lim_{n\to \infty }Y'_{n \wedge \tau_T}$ exists and is finite for any $T\geq0$. Similar to the arguments in Part 1, we know that with probability 1, $\lim_{n\to \infty}Y'_n$ exists and is finite.

Notice that with probability 1,
\begin{equation*}
    0\leq \sum_{k=1}^{n-1}a_k'Z_k'= Y'_n-Z_n' +\sum_{k=1}^{n-1}c_k'\leq Y'_n+\sum_{k=1}^{n-1}c_k'<\infty.
\end{equation*}
Combined with \eqref{equ:c_n_sum_prob1}, we obtain that $\sum_{k=1}^{\infty}a_k'Z_k'$ exists with probability 1.

Because with probability 1,
\begin{equation*}
\sum_{n=1}^\infty a'_n=+\infty,\quad\sum_{k=1}^{\infty}a_k'Z_k'<\infty,\text{ and }\lim_{n\to \infty}Z_n' \text{ exists},
\end{equation*}
we obtain that with probability 1
$$
\lim_{n\to \infty}Z'_n=0.
$$
\paragraph*{Part 3} We define $g(t)=(1-ct)(1+t)^\beta, t\geq 0$. Notice that
$\lim_{t\to 0+}\frac{g(t)-g(0)}{t}=g'(0)= -(c-\beta)<0$. Thus, there exists $N>0$ such that
$g(\frac{1}{n})\leq 1-\frac{c-\beta}{2n}$, for all $n\geq N$. Define $C_n=(n+1)^\beta c_n$ and $A_n=1-A_n'$, where
\begin{equation*}
    A'_n=\begin{cases}
        &g(\frac{1}{n}) ,n< N\\
        &1-\frac{c-\beta}{2n}, n\geq N.
        \end{cases}
\end{equation*}
Note that 
\[
\frac{(n+1)^\beta}{n^\beta}(1-a_n)\leq (1+\frac{1}{n})^\beta (1-\frac{c}{n}) =g\big(\frac{1}{n}\big) \leq A'_n=1-A_n, n\geq 1.
\]
This implies that 
\[
\mathbb{E}\left[(n+1)^\beta Z_{n+1} \mid \mathcal{F}_n\right]  \leq\left( 1-A_n \right)n^\beta Z_n+C_n.
\]

Because the limit $\sum_{n=1}^\infty n^\beta c_n$ exists and $\{(\frac{n+1}{n})^\beta\}_{n=1}^\infty$ is a monotone and bounded sequence with probability 1,  by Abel's test for series (see Exercise 9.15 in \cite{ghorpade2006course}), the limit $\sum_{n=1}^\infty   C_n$ exists with probability 1.

It is straightforward to check that
$$
\sum_{n=1}^\infty A_n=\infty, \text{ and }\sum_{n=1}^\infty A_n^{-}<\infty.
$$
Applying the second conclusion in Lemma~\ref{lem:RS lem}, we obtain that with probability 1
$$
\lim_{n\to \infty}n^{\beta}Z_n=0.
$$
\end{proof}

{
In the rest of the section, let
$Z_{n}=\mathbb{F}_{\bgtheta^*}(\overline{{\bgpi}}_n)-\mathbb{F}_{\bgtheta^*}({\bgpi}^*)$.
} Applying Assumption~\ref{ass:5} and Lemma~\ref{lem:convex F}, we know that $\mathbb{F}_{\bgtheta}( {\bgpi} )$ is convex in $\bgpi$. Notice
$$
\mathbb{E}\left[Z_{n} \mid \mathcal{F}_{n-1}\right]=  \mathbb{F}_{\bgtheta^*}\left(\frac{n-1}{n}\overline{{\bgpi}}_{n-1}+\frac{1}{n}\delta_{a_n}\right)-\mathbb{F}_{\bgtheta^*}\left(  {\bgpi}^* \right),
$$
and
\begin{equation*}%
Z_{n-1}= \mathbb{F}_{\bgtheta^*} \left( \overline{{\bgpi}}_{n-1} \right)-\mathbb{F}_{\bgtheta^*}\left(  {\bgpi}^* \right). 
\end{equation*}

\begin{lemma}\label{lem:K_U}
$K_U$ defined in \eqref{def:K_U} satisfies that
\begin{itemize}
    \item $\bgpi,\bgpi'\in K_U, t\in(0,1)\implies t\bgpi+(1-t)\bgpi'\in K_{U/2 }$, and
    \item $\bgpi\in K_U\implies \lambda_{max}\big(\{\mathcal{I}^{\bgpi}(\bgtheta)\}^{-1} \big)\leq \frac{1}{\underline{c}\cdot U}$.
\end{itemize}
Moreover, there exists $U_0>0$ such that \begin{equation}\label{subset:argmin K_U}
\bigcup_{\widehat{\bgtheta} \in \bgTheta} \arg\min_{\bgpi\in \ShatA}\mathbb{F}_{\widehat{\bgtheta} }(\bgpi) \subset K_{U_0}.
\end{equation}
and for both {generalized \textrm{GI0} and \textrm{GI1} defined in \eqref{eq:GI0-exploration} and \eqref{eq:GI1-exploration}}, and for all $n\geq n_0$,
we have
\[
\overline{\bgpi}_n \in K_{U_0}, \forall n\geq n_0,
\]
{where $n_0$ satisfies that $\sum_{i=1}^{n_0} \mathcal{I}_{a_i} (\widehat\bgtheta_0)$ is non-singular for some $\widehat\bgtheta_0\in \bgTheta$.}

\end{lemma}
\begin{proof}[Proof of Lemma \ref{lem:K_U}]
For any $\bgpi,\bgpi'\in K_{U}$, and $t\in (0,1/2]$, 
\begin{equation*}
     \max_{S\subset\mathcal{A}: S \text{ is relevant} }\min_{a\in S} t\pi(a) +(1-t)\pi'(a)\geq \frac{1}{2}\max_{S\subset\mathcal{A}: S \text{ is relevant} }\min_{a\in S}\pi'(a) \geq \frac{U}{2}.
\end{equation*}
When $t\in[1/2,1)$, we can obtain the same lower bound, which means that $t\bgpi+(1-t)\bgpi'\in K_{U/2}$ for any $t\in(0,1)$. For any $\bgpi\in \ShatA$, define 
\begin{equation}\label{def:pi_I}
    \bgpi_I:=\max_{S\subset\mathcal{A}: S \text{ is relevant} }\min_{a\in S} \pi(a).
\end{equation}
By Assumption~\ref{ass:6B}, 
\begin{equation}\label{ineq:inv 2 side bound}
    \underline{c}\bgpi_I\cdot I_p\preceq \sum_{a\in \mathcal{A}} \pi(a)\mathcal{I}_a(\bgtheta)   \preceq\sum_{a: \pi(a)>0} \mathcal{I}_a(\bgtheta) \preceq \overline{c}\cdot \blP_{V_{\{a;\pi(a)>0\}} (\bgtheta)}.
\end{equation}
Notice that for any $\bgtheta$, 
\begin{equation*}
    \bgpi_{I}>0 \implies \mathcal{I}^{\bgpi}(\bgtheta)\succ 0\implies  \dim \Big( V_{\{a;\pi(a)>0\}} (\bgtheta) \Big)=p \implies \bgpi_I\geq \min_{a;\pi(a)>0} \pi(a)>0.
\end{equation*}
Thus, we know that $\bgpi_I>0$, if and only if $\mathcal{I}^{\bgpi}(\bgtheta^*)$ is nonsingular, if and only if $\mathcal{I}^{\bgpi}(\bgtheta)$ is nonsingular for any $\bgtheta\in \bgTheta$.

Let $Q=\{a\in \mathcal{A}; \pi(a)>\bgpi_I \}$.
We will show by contradiction that if \( Q \) is not empty, then \(\dim(V_Q) < p\). 
If $\dim(V_Q)=p$, then by Assumption~\ref{ass:6B}, we know that $\blP_{V_{Q}(\bgtheta) }=I_p$. By \eqref{ineq:double c bound}, we know that $\mathcal{I}^{\bgpi}(\bgtheta^*)$ is nonsingluar, which means that $Q\subset \mathcal{A}$ is relevant. However, $\min_{a\in Q} \pi(a)>\bgpi_I$, which contradicts the definition of $\bgpi_I$ in \eqref{def:pi_I}.

Thus, \(\dim(V_Q) < p\) and $\blP_{V_{Q}(\bgtheta) }\neq I_p$. By Assumption~\ref{ass:6B}, we obtain
\begin{equation}\label{ineq:bgpi_I}
     \underline{c}\bgpi_I\cdot I_p\preceq \sum_{a\in \mathcal{A}} \pi(a)\mathcal{I}_a(\bgtheta)   \preceq\sum_{a\in Q} \pi(a)\mathcal{I}_a(\bgtheta) +\sum_{a\not\in Q} \bgpi_I\mathcal{I}_a(\bgtheta) \preceq \overline{c}\cdot  \blP_{V_{Q}(\bgtheta) }+\overline{c} \bgpi_I \cdot I_p.
\end{equation}
Applying Courant–Fischer–Weyl min-max principle (see Chapter I of \cite{hilbert1953methods} or Corollary III.1.2 in \cite{bhatia2013matrix}) for Rayleigh quotient on \eqref{ineq:bgpi_I}, we obtain that
\begin{equation}\label{equ:lambda_min_range}
    \lambda_{\min}(\mathcal{I}^{\bgpi}(\bgtheta))\in [\underline{c} \bgpi_I, \overline{c} \bgpi_I].
\end{equation}
{Applying Theorem~\ref{thm:selection bound}, $\bgpi\in K_U$ implies 
$
\lambda_{\min}(\mathcal{I}^{\bgpi}(\bgtheta))\geq \underline{c} U
$, which further implies $\lambda_{max}\big(\{\mathcal{I}^{\bgpi}(\bgtheta)\}^{-1} \big)\leq \frac{1}{\underline{c}\cdot U}$.}

Also, we have
\begin{equation}\label{equ:equv_bgpi_I}
\lambda_{\max}\big(\{ \mathcal{I}^{\bgpi} ( \bgtheta^*) \}^{-1} \big)\to \infty
    \iff \lambda_{\max}\big(\{ \mathcal{I}^{\bgpi}(\bgtheta) \}^{-1} \big)\to \infty, \forall \bgtheta\in \bgTheta
    \iff \bgpi_I \to 0.
\end{equation}
We will show \eqref{subset:argmin K_U} by contradiction. Set $U_n=\frac{1}{n}$. Assume, in contrast to  \eqref{subset:argmin K_U}, that there exists $\widehat{\bgtheta}^n\in \bgTheta$ and $\bgpi^n\in\ShatA$,  such that
\[
\mathbb{F}_{\widehat{\bgtheta}^n}(\bgpi^n)=\min_{\bgpi\in \ShatA}\mathbb{F}_{\widehat{\bgtheta}^n}(\bgpi), \text{ and }\bgpi^n_I\leq U_n=\frac{1}{n}.
\]
Then, 
\begin{equation}\label{lim:limsup_upper bound}
    \limsup_{n\to \infty}\mathbb{F}_{\widehat{\bgtheta}^n}(\bgpi^n)=\limsup_{n\to \infty}\min_{\bgpi\in \ShatA}\mathbb{F}_{\widehat{\bgtheta}^n}(\bgpi)\leq \max_{\bgtheta\in \bgTheta} \min_{\bgpi\in\ShatA}\mathbb{F}_{\bgtheta}(\bgpi)<\infty
\end{equation}
Set $\blA_n= \mathcal{I}^{\bgpi^n}(\bgtheta)$. We know that $\norm{\blA_n}_{op}\leq \overline{c}$ and by \eqref{equ:lambda_min_range} $\lambda_{min}(\blA_n)\leq \overline{c} \bgpi^n_I\leq \overline{c}/n$. Let $\mathbb{G}_{\bgtheta}(\blA_n^{-1})=\Phi_q(\blA_n^{-1})$. When $q=0$, we know that
\begin{equation}\label{equ:S106}
    \lim_{n\to \infty}\min_{\bgtheta\in \bgTheta }\Phi_0(\blA_n^{-1})\geq \lim_{n\to \infty}\log  ( {n}/{\overline{c}}  ) + (p-1)\log (1/\overline{c})=\infty.
\end{equation}
When $q>0$, we know that
\begin{equation}\label{equ:S107}
    \lim_{n\to \infty}\min_{\bgtheta\in \bgTheta }\Phi_q(\blA_n^{-1}) \geq \lim_{n\to \infty}\min_{\bgtheta\in \bgTheta } \lambda_{max}(\blA_n^{-1}) \geq \lim_{n\to \infty}n/\overline{c}=\infty.
\end{equation}
Combining \eqref{equ:S106} and \eqref{equ:S107} with Assumption~\ref{ass:5} {and \eqref{equ:equv_bgpi_I}}, we know that
\begin{equation}\label{lim:Fbgpi_I}
    \lim_{\bgpi^n_I\to 0} \min_{\bgtheta\in\bgTheta} \mathbb{F}_{\bgtheta}(\bgpi^n)=\lim_{\bgpi^n_I\to 0} \min_{\bgtheta\in\bgTheta} \mathbb{G}_{\bgtheta}( \{\mathcal{I}^{\bgpi^n}(\bgtheta)\}^{-1} )=\infty.
\end{equation}
By equivalence result \eqref{equ:equv_bgpi_I} and limit result \eqref{lim:Fbgpi_I}, since $\bgpi^n_I\to 0$ as $n\to\infty$, we obtain
\[
\liminf_{n\to \infty} \mathbb{F}_{\widehat{\bgtheta}^n}(\bgpi^n)\geq \liminf_{\bgpi^n_I\to 0} \min_{\bgtheta\in \bgTheta}\mathbb{F}_{ {\bgtheta}}(\bgpi^n)\to \infty,
\]
which contradicts \eqref{lim:limsup_upper bound}. Thus, \eqref{subset:argmin K_U} holds.

By Theorem~\ref{thm:selection bound}, we know that there exists $U_0>0$ such that
\[
 \overline{\bgpi}_n\in K_{U_0}, n\geq n_0.
\]
Combined with \eqref{equ:lambda_min_range}, we completes the proof.
\end{proof}
\begin{lemma}\label{lem:grad_F}
Under Assumptions~\ref{ass:1}-\ref{ass:5}, there exists $L_U<\infty$ such that
\[
\norm{\nabla \nabla_{\bgtheta} \mathbb{F}_{\bgtheta}(\bgpi) }_{op}\leq L_U, \text{ and }\norm{\nabla^2  \mathbb{F}_{\bgtheta}(\bgpi) }_{op}\leq L_U, 
\]
for any $\bgtheta\in \bgTheta$ and $\bgpi\in K_U$.
\end{lemma}
\begin{proof}[Proof of Lemma \ref{lem:grad_F}]
Define $u_{\bgtheta}(\blA)=\mathbb{G}_{\bgtheta}(\blA^{-1})$.

For any positive definite matrix \(\blA\), each element of \( \blA^{-1} \) is a well-defined composition of elementary functions of \( \blA \). Therefore, each element of \( \blA^{-1} \) is infinitely differentiable.

By Assumptions~\ref{ass:5} and Lemma~\ref{lem:phi_q}, $ \nabla_{\bgtheta}\mathbb{G}_{\bgtheta}(\bgSigma)$ and $\nabla^2\mathbb{G}_{\bgtheta}(\bgSigma)$ are continuous in $(\bgtheta,\bgSigma)$ for any $\bgtheta\in \bgTheta$ and positive definite matrix $\bgSigma$. Set $\blA=\mathcal{I}^{\bgpi}({\bgtheta})$. We have $\mathbb{F}_{\bgtheta}(\bgpi)=u_{\bgtheta}(\mathcal{I}^{\bgpi}({\bgtheta}))$.

By the chain rule, we know that
\begin{equation}\label{equ:S95_lem}
    \frac{\partial }{\partial{\theta}_i}\mathbb{F}_{\bgtheta}(\bgpi)=\innerpoduct{\left.\frac{\partial}{\partial \blA}u_{\bgtheta}(\blA)\right|_{\blA=\mathcal{I}^{\bgpi}({\bgtheta})} }{\frac{\partial \mathcal{I}^{\bgpi}({\bgtheta})}{\partial{\theta}_i}}+\left.\frac{\partial}{\partial{\theta}_i}\mathbb{G}_{\bgtheta}(\blA^{-1} ) \right|_{\blA=\mathcal{I}^{\bgpi}({\bgtheta})}.
\end{equation}
Notice that each element of $\left.\frac{\partial}{\partial \blA}u_{\bgtheta}(\blA)\right|_{\blA=\mathcal{I}^{\bgpi}({\bgtheta})}$ is continuously differentiable in $\blA$. Thus
\begin{equation}\label{equ:S96_lem}
        \frac{\partial}{\partial \pi(a)}\frac{\partial}{\partial \blA}u_{\bgtheta}(\mathcal{I}^{\bgpi}({\bgtheta}))
\end{equation}
exists and is continuous in $(\bgpi,\bgtheta)\in  K_U\times \bgTheta$. %

Furthermore, we know that
\begin{equation}\label{equ:S97_lem}
    \frac{\partial}{\partial \pi(a) }\frac{\partial \mathcal{I}^{\bgpi}({\bgtheta})}{\partial{\theta}_i}=\frac{\partial \mathcal{I}_a({\bgtheta})}{\partial{\theta}_i},\text{ and}
\end{equation}
\begin{equation}\label{equ:S98_lem}
    \frac{\partial}{\partial \pi(a)} \Big(\left.\frac{\partial}{\partial \theta_i} \mathbb{G}_{\bgtheta}(\blA^{-1}) \right|_{\blA=\mathcal{I}^{\bgpi}({\bgtheta})} \Big) = \innerpoduct{   \left.\frac{\partial}{\partial \theta_i} \frac{\partial}{\partial \blA}\mathbb{G}_{\bgtheta}(\blA^{-1}) \right|_{\blA=\mathcal{I}^{\bgpi}({\bgtheta})}   }{ \mathcal{I}_a(\bgtheta)}.
\end{equation}
Combining \eqref{equ:S95_lem}, \eqref{equ:S96_lem}, \eqref{equ:S97_lem} and \eqref{equ:S98_lem}, we obtain that $\norm{\nabla \nabla_{\bgtheta} \mathbb{F}_{\bgtheta}(\bgpi) }$ is continuous over $\bgTheta\times K_U$ for any $U>0$. By Lemma~\ref{lem:K_U} and the definition of $K_U$ in \eqref{def:K_U}, we know that $K_U$ is a close subset of $\ShatA$. Thus, $\bgTheta\times K_U$ is compact.

By chain rule, we know that
\begin{equation}\label{equ:S99_lem}
    \frac{\partial }{\partial\pi(a)}\mathbb{F}_{\bgtheta}(\bgpi)=\innerpoduct{\left.\frac{\partial}{\partial \blA}u_{\bgtheta}(\blA)\right|_{\blA=\mathcal{I}^{\bgpi}({\bgtheta})} }{\frac{\partial \mathcal{I}^{\bgpi}({\bgtheta})}{\partial\pi(a)}}=\innerpoduct{ \frac{\partial}{\partial \blA}u_{\bgtheta}( \mathcal{I}^{\bgpi}({\bgtheta}))  }{ \mathcal{I}_a({\bgtheta}) }.
\end{equation}
Because $u_{\bgtheta}( \mathcal{I}^{\bgpi}({\bgtheta}))$ is twice continously differentiable in $\blA$, we know that $\nabla^2 \mathbb{F}_{\bgtheta}(\bgpi)$ is continuous over compact set $\bgTheta\times K_U$ for any $U$. 

In conclusion, there exists $L_U<\infty$ such that
\[
\norm{\nabla \nabla_{\bgtheta} \mathbb{F}_{\bgtheta}(\bgpi) }_{op}\leq L_U, \text{ and }\norm{\nabla^2  \mathbb{F}_{\bgtheta}(\bgpi) }_{op}\leq L_U, 
\]
for any $\bgtheta\in \bgTheta$ and $\bgpi\in K_U$.
\end{proof}

\begin{lemma}\label{lem:replacement principle before}
Under Assumptions~\ref{ass:1}-\ref{ass:5} as well as \ref{ass:6A}-\ref{ass:7A} (or \ref{ass:6B}-\ref{ass:7B}), the {generalized } \(\textrm{GI0}\) and \eqref{eq:GI0-exploration} and \(\textrm{GI1}\), defined in  \eqref{eq:GI1-exploration}, satisfy that there exists a constant \( L > 0 \) such that,
\begin{equation}\label{equ:replacement principle 1}
   \mathbb{F}_{{\bgtheta}_{n-1}} (\overline{{\bgpi}}_n)-\mathbb{F}_{{\bgtheta}_{n-1}} ({\bgpi}^*_n)\leq (1-\frac{1}{n})(\mathbb{F}_{{\bgtheta}_{n-1}}(\overline{{\bgpi}}_{n-1})-\mathbb{F}_{{\bgtheta}_{n-1}} ({\bgpi}_n^*))+\frac{L}{n^2}, n\geq n_0.
\end{equation}
\end{lemma}

\begin{proof}[Proof of Lemma \ref{lem:replacement principle before}]
By Theorem~\ref{thm:selection bound} and Lemma~\ref{lem:K_U}, there exists $0<U<\infty$ such that $\overline{{\bgpi}}_n,{\bgpi}_n^*\in K_U$, where we define 
\[
\bgpi_n^*\in \arg\min_{\bgpi\in \ShatA}\mathbb{F}_{{\bgtheta}_{n-1}}(\bgpi).
\]
By Lemma~\ref{lem:convex F}, $\mathbb{F}_{{\bgtheta}_{n-1}}({\bgpi})$ is convex in $\bgpi$. According to Jensen's inequality, 
\begin{equation}\label{ineq:convex-line1}
    \mathbb{F}_{{\bgtheta}_{n-1}}(\frac{n-1}{n}\overline{{\bgpi}}_{n-1}+\frac{1}{n}{\bgpi}_n^* ) \leq \frac{n-1}{n} \mathbb{F}_{{\bgtheta}_{n-1}}(\overline{{\bgpi}}_{n-1}) +\frac{1}{n}\mathbb{F}_{{\bgtheta}_{n-1}}({\bgpi}^*_n).
\end{equation}
Thus,
\begin{equation}\label{ineq:convex-line2}
    \mathbb{F}_{{\bgtheta}_{n-1}}(\frac{n-1}{n}\overline{{\bgpi}}_{n-1}+\frac{1}{n}{\bgpi}_n^*)-\mathbb{F}_{{\bgtheta}_{n-1}}({\bgpi}^*_n)\leq (1-\frac{1}{n})(\mathbb{F}_{{\bgtheta}_{n-1}}(\overline{{\bgpi}}_{n-1})-\mathbb{F}_{{\bgtheta}_{n-1}} ({\bgpi}^*_n)).
\end{equation} 
Notice that
\begin{equation}\label{equ:taylor 7.11}
   \mathbb{F}_{{\bgtheta}_{n-1}}(\frac{n-1}{n}\overline{{\bgpi}}_{n-1}+\frac{1}{n}{\bgpi}_n^*)-\mathbb{F}_{{\bgtheta}_{n-1}}(\overline{{\bgpi}}_{n-1})=\innerpoduct{\nabla\mathbb{F}_{{\bgtheta}_{n-1}}(\overline{{\bgpi}}_{n-1}) }{\frac{1}{n}{\bgpi}_n^*-\frac{1}{n}\overline{{\bgpi}}_{n-1}}+\overline{R}(\overline{{\bgpi}}_{n-1}, {\bgpi}_n^*),
\end{equation}
where 
$$
\overline{R}(\overline{{\bgpi}}_{n-1}, {\bgpi}_n^*)=\innerpoduct{\nabla\mathbb{F}_{{\bgtheta}_{n-1}}( {\bgpi}' )-\nabla\mathbb{F}_{{\bgtheta}_{n-1}}( \overline{{\bgpi}}_{n-1} )}{ \frac{1}{n}{\bgpi}_n^*-\frac{1}{n}\overline{{\bgpi}}_{n-1} },
$$
for some ${\bgpi}'$ between $\overline{{\bgpi}}_{n-1}$ and $\frac{n-1}{n}\overline{{\bgpi}}_{n-1}+\frac{1}{n}{\bgpi}_n^*$. By Lemma~\ref{lem:K_U}, we know that ${\bgpi}'\in K_{U/2}$.

By Assumptions~\ref{ass:1}-\ref{ass:5} and Lemma~\ref{lem:grad_F}, there exists a constant $C'<\infty$ such that
{\begin{equation}\label{eq:taylor-L}
\begin{split}
    & |\overline{R}(\overline{{\bgpi}}_{n-1}, {\bgpi}_n^*)|\\
\leq &
\|\frac{1}{n}{\bgpi}_n^*-\frac{1}{n}\overline{{\bgpi}}_{n-1}\|^2 \sup_{{\bgpi}\in K_{U/2},\bgtheta \in \bgTheta }\norm{\nabla^2 \mathbb{F}_{\bgtheta } ({\bgpi})}_{op}\\
\leq & \frac{C'}{n^2} \sup_{{\bgpi}\in K_{U/2},\bgtheta \in \bgTheta }\norm{\nabla^2 \mathbb{F}_{\bgtheta } ({\bgpi})}_{op}\\
= &\frac{L}{n^2},
\end{split}
\end{equation}
}
where
\begin{equation*}
    L=C'\sup_{{\bgpi}\in K_{U/2},\bgtheta \in \bgTheta }\norm{\nabla^2 \mathbb{F}_{\bgtheta } ({\bgpi})}_{op}<\infty.
\end{equation*}
By Lemma~\ref{lem:Differentiation inverse Matrix}, we know that
\[
\innerpoduct{ \nabla \mathbb{F}_{\bgtheta}(\bgpi) }{ \delta_a } = \frac{\partial }{\partial \pi(a) }\mathbb{F}_{\bgtheta}(\bgpi)=-\innerpoduct{ \nabla \mathbb{G}_{\bgtheta}(\{\mathcal{I}^{\bgpi}(\bgtheta)\}^{-1} ) }{\{\mathcal{I}^{\bgpi}(\bgtheta)\}^{-1} \mathcal{I}_a(\bgtheta) \{\mathcal{I}^{\bgpi}(\bgtheta)\}^{-1} }.
\]
{Let $a_{n}^{(1)}$ be the experiment selected following the generalized \textrm{GI1}.}
Then, according to the definition of \textrm{GI1}, it minimizes the following function over $\ShatA$ with respect to $a$:
\begin{equation*}
\innerpoduct{\nabla\mathbb{F}_{{\bgtheta}_{n-1}}(\overline{{\bgpi}}_{n-1}) }{\frac{1}{n}{\bgpi}-\frac{1}{n}\overline{{\bgpi}}_{n-1}}=\sum_{a=1}^k{\bgpi}(a)\innerpoduct{\nabla\mathbb{F}_{{\bgtheta}_{n-1}}(\overline{{\bgpi}}_{n-1}) }{\frac{1}{n}\delta_a-\frac{1}{n}\overline{{\bgpi}}_{n-1}}.
\end{equation*}
By similar Taylor expansion arguments {as those for \eqref{eq:taylor-L}}, we have for all $n\geq n_0$,
\begin{equation*}
    \left|\mathbb{F}_{{\bgtheta}_{n-1}}(\frac{n-1}{n}\overline{{\bgpi}}_{n-1}+\frac{1}{n} \delta_{a^{(1)}_n} )-\mathbb{F}_{{\bgtheta}_{n-1}}(\overline{{\bgpi}}_{n-1})-\innerpoduct{\nabla\mathbb{F}_{{\bgtheta}_{n-1}}(\overline{{\bgpi}}_{n-1}) }{\frac{1}{n}\delta_{a^{(1)}_n}-\frac{1}{n}\overline{{\bgpi}}_{n-1}}\right| \leq \frac{L}{n^2}.
\end{equation*}

The above inequality implies that for all $n\geq n_0$, \textrm{GI1} satisfies
\begin{equation}\label{ineq:det rule}
    \mathbb{F}_{\widehat{\bgtheta}_{n-1}}( \overline{{\bgpi}}_{n})=\mathbb{F}_{\widehat{\bgtheta}_{n-1}}(\frac{n-1}{n}\overline{{\bgpi}}_{n-1}+\frac{1}{n} \delta_{a^{(1)}_n} )\leq\mathbb{F}_{\widehat{\bgtheta}_{n-1}}(\frac{n-1}{n}\overline{{\bgpi}}_{n-1}+\frac{1}{n} {\bgpi}_n^* )+\frac{L}{n^2}.
\end{equation}
For \textrm{GI0}, {let $a_{n}^{(0)}$ be the experiment selected at time $n$. Then, according to its definition} we have
\begin{equation}\label{ineq:GI0<GI1}\mathbb{F}_{\widehat{\bgtheta}_{n-1}}(\frac{n-1}{n}\overline{{\bgpi}}_{n-1}+\frac{1}{n} \delta_{a^{(0)}_n} )\leq 
\mathbb{F}_{\widehat{\bgtheta}_{n-1}}(\frac{n-1}{n}\overline{{\bgpi}}_{n-1}+\frac{1}{n} \delta_{a^{(1)}_n} ).
\end{equation}
Therefore, the proof of Lemma \ref{lem:replacement principle before} is concluded by combining  inequalities {\eqref{ineq:convex-line2} -- \eqref{ineq:GI0<GI1}}.

\end{proof}

\begin{lemma}\label{lem:replacement principle after}
Under Assumptions~\ref{ass:1}-\ref{ass:5} as well as \ref{ass:6A}-\ref{ass:7A} (or \ref{ass:6B}-\ref{ass:7B}),  the generalized \(\textrm{GI0}\) selection \eqref{eq:GI0-exploration} and \(\textrm{GI1}\) selection \eqref{eq:GI1-exploration} satisfy that there exists $0<C<\infty$ such that 
\begin{equation}\label{equ:replacement principle 2}
    \mathbb{F}_{ {\bgtheta}^* }(\overline{{\bgpi}}_n)-\mathbb{F}_{ {\bgtheta}^* }({\bgpi}^* )\leq (1-\frac{1}{n})(\mathbb{F}_{ {\bgtheta}^* }(\overline{{\bgpi}}_{n-1})-\mathbb{F}_{ {\bgtheta}^* }({\bgpi}^* ))+c_{n-1}, n\geq n_0,
\end{equation}
where $c_{n-1}  = \frac{C}{n^2} +\frac{C}{n}\norm{{\bgtheta}_{n-1}-\bgtheta^*}$.
\end{lemma}

\begin{proof}[Proof of Lemma \ref{lem:replacement principle after}]
We can rewrite \eqref{equ:replacement principle 1} as
\begin{equation}\label{equ:140}
    \mathbb{F}_{{\bgtheta}_{n-1}}(\overline{{\bgpi}}_n)-\mathbb{F}_{{\bgtheta}_{n-1}}(\overline{{\bgpi}}_{n-1}) +  \frac{1}{n} (\mathbb{F}_{{\bgtheta}_{n-1}}(\overline{{\bgpi}}_{n-1})-\mathbb{F}_{{\bgtheta}_{n-1}}({\bgpi}^*_n))\leq \frac{L}{n^2}.
\end{equation}
{We first show that for all ${\bgpi}_0,{\bgpi}_1\in K_U$, and $\bgtheta_1,\bgtheta_2\in\bgTheta$,
\begin{equation}\label{eq:f-bound-pdif}
   | \mathbb{F}_{\bgtheta_1}({\bgpi}_1)-\mathbb{F}_{\bgtheta_1}({\bgpi}_0)- \{\mathbb{F}_{\bgtheta_2}({\bgpi}_1) -\mathbb{F}_{\bgtheta_2}({\bgpi}_0) \}|\leq  C_1 \norm{{\bgpi}_0-{\bgpi}_1}\norm{ {\bgtheta}_{1}-\bgtheta_{2} },
\end{equation}
where $C_1=\sup_{{\bgpi}\in K_{U/2},\bgtheta\in \bgTheta}\norm{\nabla_{\bgtheta} \nabla  \mathbb{F}_{\bgtheta}({\bgpi}) }_{op}$ is a positive constant.
To show this,
}
set $g(t)=\mathbb{F}_{\bgtheta_1}({\bgpi}(t))-\mathbb{F}_{\bgtheta_2}({\bgpi}(t))$, where ${\bgpi}(t)=t{\bgpi}_1+(1-t){\bgpi}_0, t\in [0,1]$ and ${\bgpi}_0,{\bgpi}_1\in K_U$, {where $K_U$ is chosen according to the proof of Lemma \ref{lem:replacement principle before}}.
By Lagrange mean value theorem, there exists $0<t<1$ such that
\begin{equation*}
    g(1)-g(0)=g'(t)=\innerpoduct{\nabla \mathbb{F}_{\bgtheta_1}({\bgpi}(t))-\nabla \mathbb{F}_{\bgtheta_2}({\bgpi}(t))}{ {\bgpi}_1-{\bgpi}_0 }.
\end{equation*}
By Assumptions~\ref{ass:1}-\ref{ass:5}, Lemma~\ref{lem:K_U} and Lemma~\ref{lem:grad_F}, we know that
\begin{equation*}
    \frac{\norm{\nabla \mathbb{F}_{\bgtheta_1}({\bgpi}(t))-\nabla \mathbb{F}_{\bgtheta_2}({\bgpi}(t))}}{\norm{\bgtheta_1-\bgtheta_2}}\leq  \sup_{{\bgpi}\in K_{U/2},\bgtheta\in \bgTheta}\norm{\nabla_{\bgtheta} \nabla  \mathbb{F}_{\bgtheta}({\bgpi}) }_{op}
     <\infty.
\end{equation*}
Set
\begin{equation*}
C_1=\sup_{{\bgpi}\in K_{U/2},\bgtheta\in \bgTheta}\norm{\nabla_{\bgtheta} \nabla  \mathbb{F}_{\bgtheta}({\bgpi}) }_{op}.
\end{equation*}
Then, the above inequality implies \eqref{eq:f-bound-pdif}. Note that
\begin{equation*}
    \norm{\overline{{\bgpi}}_n-\overline{{\bgpi}}_{n-1}}\leq \frac{2}{n}.
\end{equation*}
The above inequality together with \eqref{eq:f-bound-pdif}  implies
\begin{equation}\label{equ:142}
\begin{split}
     &|(\mathbb{F}_{ \bgtheta^*}(\overline{{\bgpi}}_n)-\mathbb{F}_{ \bgtheta^*}(\overline{{\bgpi}}_{n-1}))- (\mathbb{F}_{{\bgtheta}_{n-1}}(\overline{{\bgpi}}_n)-\mathbb{F}_{{\bgtheta}_{n-1}}(\overline{{\bgpi}}_{n-1}))| \leq  \frac{2C_1}{n}\norm{{\bgtheta}_{n-1}-\bgtheta^*},\text{ and}\\ 
     &|(\mathbb{F}_{ \bgtheta^*}(\overline{{\bgpi}}_{n-1})-\mathbb{F}_{ \bgtheta^*}({{\bgpi}}^*))- (\mathbb{F}_{{\bgtheta}_{n-1}}(\overline{{\bgpi}}_{n-1})-\mathbb{F}_{{\bgtheta}_{n-1}}( {{\bgpi}}^{*}))|\leq  {2C_1} \norm{{\bgtheta}_{n-1}-\bgtheta^*}.
\end{split}
\end{equation}
Because $\mathbb{F}_{{\bgtheta}_{n-1}}({\bgpi}^*)\geq \mathbb{F}_{{\bgtheta}_{n-1}}({\bgpi}^*_n) $, we have
\begin{equation}\label{equ:143}
\begin{split}
    &(\mathbb{F}_{ \bgtheta^*}(\overline{{\bgpi}}_{n-1})-\mathbb{F}_{ \bgtheta^*}({\bgpi}^*))-(\mathbb{F}_{{\bgtheta}_{n-1}}(\overline{{\bgpi}}_{n-1})-\mathbb{F}_{{\bgtheta}_{n-1}}({\bgpi}^*_n))\\
    \leq&(\mathbb{F}_{ \bgtheta^*}(\overline{{\bgpi}}_{n-1})-\mathbb{F}_{ \bgtheta^*}({\bgpi}^*))-(\mathbb{F}_{{\bgtheta}_{n-1}}(\overline{{\bgpi}}_{n-1})-\mathbb{F}_{{\bgtheta}_{n-1}}({\bgpi}^*))\\
    \leq & 2C_1  \norm{{\bgtheta}_{n-1}-\bgtheta^*}.
\end{split}
\end{equation}
By triangular inequality, inequalities \eqref{equ:140}, \eqref{equ:142} and \eqref{equ:143}, we obtain
\begin{equation}
    \begin{split}
          &\mathbb{F}_{ \bgtheta^*}(\overline{{\bgpi}}_n)-\mathbb{F}_{ \bgtheta^*}(\overline{{\bgpi}}_{n-1}) +  \frac{1}{n} (\mathbb{F}_{ \bgtheta^*}(\overline{{\bgpi}}_{n-1})-\mathbb{F}_{ \bgtheta^*}({\bgpi}^*))\\
          \leq & |(\mathbb{F}_{ \bgtheta^*}(\overline{{\bgpi}}_n)-\mathbb{F}_{ \bgtheta^*}(\overline{{\bgpi}}_{n-1}))- (\mathbb{F}_{{\bgtheta}_{n-1}}(\overline{{\bgpi}}_n)-\mathbb{F}_{{\bgtheta}_{n-1}}(\overline{{\bgpi}}_{n-1}))| \\ 
         & +  (\mathbb{F}_{{\bgtheta}_{n-1}}(\overline{{\bgpi}}_n)-\mathbb{F}_{{\bgtheta}_{n-1}}(\overline{{\bgpi}}_{n-1}))+\frac{1}{n}\Big(\mathbb{F}_{{\bgtheta}_{n-1}}(\overline{{\bgpi}}_{n-1})-\mathbb{F}_{{\bgtheta}_{n-1}}({\bgpi}^*_n)\Big)+
         \frac{2C_1  \norm{{\bgtheta}_{n-1}-\bgtheta^*}}{n}\\
         \leq &  \frac{4C_1}{n}\norm{\bgtheta_{n-1}-\bgtheta^*}+\frac{L}{n^2}.
    \end{split}
\end{equation}
In conclusion, we know that
\begin{equation*}
    \mathbb{F}_{ \bgtheta^*}(\overline{{\bgpi}}_n)-\mathbb{F}_{ \bgtheta^*}(\overline{{\bgpi}}_{n-1}) +  \frac{1}{n} (\mathbb{F}_{ \bgtheta^*}(\overline{{\bgpi}}_{n-1})-\mathbb{F}_{ \bgtheta^*}({\bgpi}^*))\leq c_{n-1},
\end{equation*}
where $c_{n-1}  = \frac{C}{n^2} +\frac{C}{n}\norm{{\bgtheta}_{n-1}-\bgtheta^*},$ and $C=  4C_1+L $.
\end{proof}

As a corollary of Theorem \ref{thm:GI0-GI1AN}, we establish the following:
\begin{corollary}\label{cor:bound}
Under Assumptions~\ref{ass:1}-\ref{ass:5} as well as Assumptions~\ref{ass:6A}-\ref{ass:7A} (or Assumptions~\ref{ass:6B}-\ref{ass:7B}), if there exists $U>0$ such that $n\geq n_0\implies \overline{{\bgpi}}_n\in K_U$, 
then with probability 1,
\begin{equation*}%
\sum_{n=1}^\infty n^{-s}\norm{\sqrt{n}(\widehat{\bgtheta}^{\text{ML}}_{n} - \bgtheta^*)}^t<\infty. 
\end{equation*}
provided $s>1,0<t\leq 2$.

\end{corollary}

\begin{proof}[Proof of Corollary \ref{cor:bound}]
Let $\widehat{\bgtheta}_n=\widehat{\bgtheta}^{\text{ML}}_{n}$. We first assume $t=2$. Applying Lemma \ref{lem:a.s.bound}, 
$$
 \limsup_{n\to \infty}\norm{\{-\nabla_{\bgtheta}^2 l_n({\bgtheta^*} )\}^{-1}}_{op}\leq \frac{1}{ \min_{\pi\in K_U}\lambda_{min}(I^{\pi}({\bgtheta^*}))}.
$$
Set $D_n := \big\{ 1-\frac{1}{n}\sum_{j=1}^n\Psi_2^{a_j}(X_j)\norm{\left\{\nabla_{\bgtheta}^2 l_n({\bgtheta^*}) \right\}^{-1}}_{op}  \psi\left(\norm{{{\widehat\bgtheta}_n}-{\bgtheta^*}} \right)  > \frac{1}{2}, \norm{\{-\nabla_{\bgtheta}^2 l_n({\bgtheta^*})\}^{-1}}_{op}\leq \frac{2}{ \min_{\pi\in K}\lambda_{min}(I^{\pi}({\bgtheta^*})) } \big\}$. By Lemma~\ref{thm:consistency_final}, we know that 
\begin{equation}\label{equ:D_m_lim_prob}
    \mathbb{P} \left( \bigcup_{n=1}^\infty\bigcap_{m=n}^\infty D_m \right) = 1.
\end{equation}
Note that \eqref{eq:R-ineq} and \eqref{eq:Wn-ineq} yield
\begin{equation*}  
    \norm{\blW_n}I_{D_n} \leq I_{D_n}\frac{\norm{\left\{\nabla_{\bgtheta}^2 l_n({\bgtheta^*}) \right\}^{-1}}_{op}   \norm{\sqrt{n}\nabla_{\bgtheta} l_n({\bgtheta^*})}}{1-\frac{1}{n}\sum_{j=1}^n\Psi_2^{a_j}(X_j)\norm{\left\{\nabla_{\bgtheta}^2 l_n({\bgtheta^*}) \right\}^{-1}}_{op}  \psi\left(\norm{{{\widehat\bgtheta}_n}-{\bgtheta^*}} \right) }\leq C'\norm{\sqrt{n}\nabla_{\bgtheta} l_n({\bgtheta^*})},
\end{equation*}
where $C'<\infty$ {and $\blW_n = \sqrt{n}(\widehat{\bgtheta}_n^{\text{ML}}-\bgtheta^*)$}. Set $S_n=\sum_{i=1}^n\nabla_{\bgtheta}\log f_{{\bgtheta^*},a_i}(X_i)$. By Assumption \ref{ass:2}, we know that 
\begin{equation*}
    \sigma^2:=\max_{a\in \mathcal{A} }\mathbb{E}_{X\sim f_{{\bgtheta^*},a}}\left\{ \norm{\nabla_{\bgtheta}\log  f_{{\bgtheta^*},a}(X)}^2 \right\}<\infty.
\end{equation*}
By induction, we obtain that
\begin{equation*}
\begin{split}
    &\mathbb{E}\norm{\sqrt{n}\nabla_{\bgtheta} l_n({\bgtheta^*})}^2\\
    = &\frac{1}{n} \mathbb{E}\norm{S_n}^2 \\
    =& \frac{1}{n}\mathbb{E}\left[\mathbb{E}\left\{\left.\norm{\log f_{{\bgtheta^*},a_n}(X_n)+S_{n-1}}^2\right| \mathcal{F}_{n-1}\right\}\right]\\
    =& \frac{1}{n}\mathbb{E}\left[\norm{S_{n-1}}^2+2\innerpoduct{S_{n-1}}{\mathbb{E}\left\{\left. {\log f_{{\bgtheta^*},a_n}(X_n)}\right| \mathcal{F}_{n-1}\right\}}+\mathbb{E}\left\{\left.\norm{\log f_{{\bgtheta^*},a_n}(X_n)}^2\right| \mathcal{F}_{n-1}\right\}\right]\\   
    =&\frac{1}{n}\mathbb{E}\left[\norm{S_{n-1}}^2 +\mathbb{E}\left\{\left.\norm{\log f_{{\bgtheta^*},a_n}(X_n)}^2\right| \mathcal{F}_{n-1}\right\}\right]\\ 
    \leq &\frac{1}{n}\left(\mathbb{E}\norm{S_{n-1}}^2 +\sigma^2 \right) \leq   \cdots\leq  \sigma^2.
\end{split}
\end{equation*}
{Apply Lemma \ref{lem:K two-series theorem} with $X_n=\frac{1}{n^{s}}\norm{\sqrt{n}(\widehat{{\bgtheta}}_n - {\bgtheta^*})}^{2} $, $E_n=D_n $, $\gamma=1$, and $\varepsilon_{n-1} =\frac{1}{n^{s}}(C')^2\mathbb{E}[\norm{\sqrt{n}\nabla_{\bgtheta}  l_n({\bgtheta^*})   }^2|\mathcal{F}_{n-1}] $,}
because 
\[
\sum_{n=0}^\infty \mathbb{E}\varepsilon_{n} \leq \sum_{n=1}^\infty\frac{\sigma^2(C')^2}{n^{s}}<\infty,
\]
we obtain that with probability 1,
\begin{equation*}
    \sum_{n=1}^\infty n^{-s} \cdot\mathbb{E}\left\{ \norm{\sqrt{n}(\widehat{{\bgtheta}}_n - {\bgtheta^*})}^{2}I_{D_n}\right\} <\infty.
\end{equation*}
Combined with \eqref{equ:D_m_lim_prob}, we obtain that
\begin{equation*}
\begin{split}
    &\mathbb{P}\Big( \sum_{n=1}^\infty n^{-s}\norm{\sqrt{n}(\widehat{{\bgtheta}}_n - {\bgtheta^*})}^2= \infty \Big)\\
    \leq & \mathbb{P}\Big( \sum_{n=1}^\infty n^{-s}\norm{\sqrt{n}(\widehat{{\bgtheta}}_n - {\bgtheta^*})}^2I_{D_n} = \infty \Big)+\mathbb{P}\Big( \sum_{n=1}^\infty n^{-s}\norm{\sqrt{n}(\widehat{{\bgtheta}}_n - {\bgtheta^*})}^2I_{D^c_n} = \infty \Big)\\
    \leq & 0+\mathbb{P}\Big( \sum_{n=1}^\infty   I_{D_n} = \infty \Big)=\mathbb{P} \left( \bigcap_{n=1}^\infty\bigcup_{m=n}^\infty D^c_m \right)=0,
\end{split}
\end{equation*}
that is, $\sum_{n=1}^\infty n^{-s}\norm{\sqrt{n}(\widehat{{\bgtheta}}_n - {\bgtheta^*})}^2<\infty$ with probability 1.

If $0<t<2$, set $s_1=1-\frac{t}{2}$, $s_2=\frac{t}{2}$, $s_0=\frac{s-1}{2}$, $p=\frac{1}{s_1}>1$, and $q=\frac{1}{s_2}>1$. We have $1/p+1/q=1$, and $s_1+s_2+2s_0=s$. Notice that 
$$
\left(\sum_{n=1}^\infty n^{-(s_1+s_0)p}\right)^{1/p} =\left(\sum_{n=1}^\infty n^{-(1+s_0p)}\right)^{1/p} <\infty,
$$
and with probability 1,
$$
\sum_{n=1}^\infty n^{-(s_2+s_0)q} \norm{\sqrt{n}(\widehat{{\bgtheta}}_n - {\bgtheta^*})}^{tq}=\sum_{n=1}^\infty n^{-1-s_0 q} \norm{\sqrt{n}(\widehat{{\bgtheta}}_n - {\bgtheta^*})}^{2}<\infty.
$$
By Hölder's inequality, with probability 1
$$
\sum_{n=1}^\infty n^{-s}\norm{\sqrt{n}(\widehat{{\bgtheta}}_n - {\bgtheta^*})}^t\leq \left(\sum_{n=1}^\infty n^{-(s_1+s_0)p}\right)^{1/p}  \left(\sum_{n=1}^\infty n^{-(s_2+s_0)q} \norm{\sqrt{n}(\widehat{{\bgtheta}}_n - {\bgtheta^*})}^{tq} \right)^{1/q}<\infty. 
$$

\end{proof}

\begin{lemma}\label{lem:converge}
    Under Assumptions~\ref{ass:1}-\ref{ass:5} as well as  Assumptions~\ref{ass:6A}-\ref{ass:7A} (or Assumptions~\ref{ass:6B}-\ref{ass:7B}), if the sequence of estimators $\widehat{\bgtheta}_n$ satisfies that for $0\leq \beta< \frac{1}{2}$,{
\begin{equation}\label{eq:converge-condition}
\sum_{n\geq n_0} n^{\beta-1}\norm{ \widehat{\bgtheta}_{n-1}-\bgtheta^* }<\infty \text{ a.s.}
\end{equation}
}
then
    the {generalized} \textrm{GI0} and \textrm{GI1} {(with $\bgtheta_n$ replaced by $\widehat{\bgtheta}_{n}$)} satisfy
\begin{equation*}
    n^\beta Z_n\stackrel{\text { a.s. }}{\longrightarrow} 0,
\end{equation*}
where $Z_n=\mathbb{F}_{ \bgtheta^*}( \overline{{\bgpi}}_{n } )-\mathbb{F}_{ \bgtheta^*}({\bgpi}^*)$. 
\end{lemma}
\begin{proof}[Proof of Lemma \ref{lem:converge}]
Based on Lemma \ref{lem:replacement principle after}, the \textrm{GI0} and \textrm{GI1} selection rules satisfy that there exists $C<\infty$ such that for any $n\geq n_0$,
\begin{equation*}
    \mathbb{E}\left[Z_{n} \mid \mathcal{F}_{n-1}\right] \leq\left(1-\frac{1}{n} \right) Z_{n-1}+c_{n-1},
\end{equation*}   
where $c_{n-1}  = C\Big(\frac{1}{n^2}+\frac{ \norm{\widehat{\bgtheta}_{n-1}-\bgtheta^*}}{n}\Big)$. Notice that with probability $1$, we have
\begin{equation*}
\begin{split}
    &\sum_{n=n_0+1}^\infty ( n-1 )^\beta c_{n-1}\leq \sum_{n=n_0}^\infty C\cdot n^\beta \Big(\frac{1}{n^2}+\frac{ \norm{\widehat{\bgtheta}_{n-1}-\bgtheta^*}}{n}\Big)\\
    =& \sum_{n=n_0+1}^\infty C\cdot    \frac{1}{n^{2-\beta}} + \sum_{n=n_0+1}^\infty \frac{C}{n^{1-\beta}}   { \norm{  \widehat{ \bgtheta}_{n-1}-\bgtheta^*    }} <\infty.
\end{split}
\end{equation*}
Applying the third part of Lemma~\ref{lem:RS lem} to $Z_n$ {with $a_n=1/(n+1)$, $c=1$, }we obtain
\begin{equation*}
    n^\beta Z_n\stackrel{\text { a.s. }}{\longrightarrow} 0.
\end{equation*}
\end{proof}
\begin{proof}[Proof of Theorem \ref{thm:empirical pi as converge}]
By Corollary \ref{cor:bound}, we know that with probability 1, if $0\leq \beta<\frac{1}{2}$, then with probability $1$,
\begin{equation*}
\sum_{n=1}^\infty n^{\beta-1}\norm{ \widehat{\bgtheta}^{\text{ML}}_{n} - \bgtheta^*} = \sum_{n=1}^\infty n^{\beta-3/2}\norm{\sqrt{n}(\widehat{\bgtheta}^{\text{ML}}_{n} - \bgtheta^*)}<\infty.
\end{equation*}
By Lemma~\ref{lem:converge}, we obtain that $n^{\beta}Z_n\to 0$ a.s. $\mathbb{P}_*$. {That is, 
$\lim_{n\to\infty}n^{\beta}\{\mathbb{F}_{ \bgtheta^*}( \overline{{\bgpi}}_{n } ) - \mathbb{F}_{ \bgtheta^*}({\bgpi}^*)\}$, a.s.

Next, we prove by contradiction that, when $\mathbb{F}_{\bgtheta^*}(\cdot)$ has a unique minimizer, we also have $\lim_{n\to\infty}\overline{{\bgpi}}_{n }={\bgpi}^*$ a.s. Assume, on the contrary, that  there exists a sub-sequence such that $\overline{{\bgpi}}_{n_l }   \to {\bgpi}_1\neq {\bgpi}^*$, as $l\to \infty$. Then, by the continuity of $\mathbb{F}_{\bgtheta^*}(\cdot)$, we have $ \mathbb{F}_{ \bgtheta^*}( \overline{{\bgpi}}_{n_l } )  \to \mathbb{F}_{ \bgtheta^*}({\bgpi}_1)$.}

Set $\beta=0$. We obtain that
\begin{equation*}
    \mathbb{F}_{ \bgtheta^*}( \overline{{\bgpi}}_{n } ) \to \mathbb{F}_{ \bgtheta^*}({\bgpi}^*)\text{ a.s. }\mathbb{P}_* .
\end{equation*}
Given that $\mathbb{F}_{ \bgtheta^*}({\bgpi})$ has a unique global minimizer, it must be the case that $\mathbb{F}_{ \bgtheta^*}({\bgpi}_1)\neq \mathbb{F}_{ \bgtheta^*}({\bgpi}^*)$. This contradicts with the above display.
\end{proof}

\subsection{Proof of Theorem~\ref{thm:Asy_Normal_final}}
\begin{proof}[Proof of Theorem~\ref{thm:Asy_Normal_final}]
By applying Theorem~\ref{thm:GI0-GI1AN} and Theorem~\ref{thm:empirical pi as converge}, we conclude the proof of Theorem~\ref{thm:Asy_Normal_final}.
\end{proof}

\subsection{Proof of Theorem~\ref{thm:Asy_cov_mle}}
\begin{proof}[Proof of Theorem~\ref{thm:Asy_cov_mle}]
Under the assumptions of Theorem~\ref{thm:Asy_Normal_final}, the conclusions from Theorem~\ref{thm:consistency_final} and Theorem~\ref{thm:empirical pi as converge} still apply. Hence, we have
    \begin{equation*}
        \lim_{n\to \infty}\mathcal{I}^{ \overline{{\bgpi}}_{ n} }(\widehat{\bgtheta}_{ n}^{\text{ML}})
        = \lim_{n\to \infty}\sum_{a\in \mathcal{A}}  \overline{{\bgpi}}_n (a) \mathcal{I}_a(\widehat{\bgtheta}_{ n}^{\text{ML}})
        = \mathcal{I}^{{\bgpi}^*}(\bgtheta^*)  \text{ a.s. },
    \end{equation*}
and 
    \begin{equation*}
        \lim_{n\to \infty} \norm{ \{ \mathcal{I}^{\overline{{\bgpi}}_{n} }(\widehat{\bgtheta}_{ n}^{\text{ML}}) \}^{-1/2}  \nabla g(\widehat{\bgtheta}_{ n}^{\text{ML}}) }
        = \norm{ \{ \mathcal{I}^{{\bgpi}^*}(\bgtheta^*) \}^{-1/2}  \nabla g(\bgtheta^*) }\text{ a.s. }
    \end{equation*}
    By Slutsky’s theorem and Theorem~\ref{thm:Asy_Normal_final}, we derive the limit result as in \eqref{lim:AN}.

    Moreover, through the Delta method, we find
    \begin{equation*}
        \frac{\sqrt{n} (g(\widehat{\bgtheta}_{ n}^{\text{ML}})-g(\bgtheta^*)) }{\norm{ \{ \mathcal{I}^{{\bgpi}^*}(\bgtheta^*) \}^{-1/2}  \nabla g(\bgtheta^*) }}
        \overset{d}{\longrightarrow} N(0,1).
    \end{equation*}
    Once again, by Slutsky’s theorem, we establish the limit result in \eqref{lim:lim_g}.
\end{proof}

\subsection{Proof of Theorem~\ref{thm:ultimate}}
We first provide an extension of the Cram\'er-Rao lower bound for unbiased estimators based on sequential observations following an active experiment selection rule.
\begin{lemma}[Cra\'mer-Rao lower bound for sequential data]\label{thm:C-R ineq}
Assume that for some initial values $a_1^0, \cdots, a_{n_0}^0\in \mathcal{A}$, we consider initial selections \(a_i = a_i^0\) for \(i = 1, \ldots, n_0,\) such that the sum
\(\sum_{i=1}^{n_0} \mathcal{I}_{a_i^0}(\bgtheta) \)
is nonsingular for all \( \bgtheta \in \bgTheta \). Given any deterministic selection function \( h_n \), we consider the selections 
\(  a_n = h_n(a_1, X_1, \cdots, a_{n-1}, X_{n-1}) \in \mathcal{A}, \forall n>n_0.\)
Let $\blT_n=T(X_1,X_2,\cdots,X_n,\va_n)$ be an unbiased estimator of vector $\blh(\bgtheta)$ with a finite second moment, for all $\bgtheta \in \bgTheta$, that is $\blh(\bgtheta)=\mathbb{E}_{\bgtheta}[\blT_n]$ and $\sup_{\bgtheta\in \bgTheta}\mathbb{E}_{\bgtheta}\norm{\blT_n}^2<\infty$. Then, under Assumptions~\ref{ass:1}-\ref{ass:4}, 
we have
\begin{equation*}
    \operatorname{cov}_{\bgtheta}(\blT_n ) \succeq \frac{1}{n}\Big\{ \nabla_{\bgtheta} \blh(\bgtheta)\Big\}^T \{\mathcal{I}^{\mathbb{E}_{\bgtheta}\overline{{\bgpi}}_n} (\bgtheta)\}^{-1} \nabla_{\bgtheta} \blh(\bgtheta).
\end{equation*}
Specifically, if $\blh(\bgtheta)=\bgtheta$, then
\begin{equation*}
    \mathbb{G}_{\bgtheta}(n\operatorname{cov}_{\bgtheta}(\blT_n ) )\geq \inf_{\pi\in \ShatA} \mathbb{G}_{\bgtheta}(    \{\mathcal{I}^{ \pi } (\bgtheta)\}^{-1}).
\end{equation*}
\end{lemma}

\begin{proof}[Proof of Lemma~\ref{thm:C-R ineq}]

Assume $\blh(\bgtheta)\in \mathbb{R}^l$. For any $\blb\in \mathbb{R}^l$, define $h_{\blb}(\bgtheta) = \blb^T\mathbb{E}_{\bgtheta}\Big[\blT_n \Big]$. 

Let $\{X_i^a\}_{a\in \mathcal{A},i\geq 1}$ be a sequence of independent random elements, such that $X_i^a \sim f_{\bgtheta,a}(\cdot)$. According to Lemma \ref{lem:same dist},  we can assume that the observations and experiments are $a_1, X_1^{a_1},\cdots, a_n, X_n^{a_n}$ in the rest of the proof, where $a_{n+1}=h_{n+1}(a_1,X_1^{a_1},\cdots,a_n,X_n^{a_n})$, for all $n\geq n_0$.

The joint density for $\blX^{\mathcal{A}}_n=\{X^a_i\}_{1\leq i\leq n, a\in \mathcal{A}}$ and $\bla_n = (a_1,\cdots,a_n)$
is given by
\begin{equation*}
    f_{\bgtheta} (\blX^{\mathcal{A}}_n,\va^n)=\prod_{i=1}^n\prod_{a\in \mathcal{A}}f_{\bgtheta,  {a}}(X_i^{a}) I(a_1=a^1, \cdots,a_{n_0}=a^{n_0},a_{n_0+1} =a^{n_0+1},\cdots,a_n=a^n).
\end{equation*}
Notice that
\begin{equation*}
    \nabla_{\bgtheta} f_{\bgtheta} (\blX^{\mathcal{A}}_n,\va^n)= f_{\bgtheta} (\blX^{\mathcal{A}}_n,\va^n)\sum_{i=1}^n\sum_{a\in \mathcal{A}}  \nabla_{\bgtheta}\log f_{\bgtheta,  {a}}(X_i^{a}).  
\end{equation*}
Assume that probability density $f_{\bgtheta,a}(\cdot)$ is with respect to baseline measure $\mu_a(\cdot)$. By Assumption \ref{ass:2}, denote the support of probability density $f_{\bgtheta,a}(\cdot)$ by $\Omega_a=\operatorname{supp}(f_{\bgtheta,a})$, which does not depend on $\bgtheta$. Let product measure $d\bgmu^n(\blX^{\mathcal{A}}_{n})= \prod_{1 \leq i \leq n, a\in \mathcal{A}}d \mu_a(X^a_i)$, and product space $\bgOmega^{1}=\times_{a\in \mathcal{A}} \Omega_a$, $\bgOmega^n= \times_{a\in \mathcal{A}} \Omega_a \times \bgOmega^{n-1}$.

Set $\blT_n=T_n(a_1,X_1^{a_1},\cdots,a_n,X_n^{a_n})$. Because
\begin{equation*}
    \mathbb{E}_{\bgtheta}[\blb^T\blT_n] = \sum_{\va^n\in \mathcal{A}^n}\int_{\bgOmega^n}\blb^TT_n(a_1,X_1^{a_1},\cdots,a_n,X_n^{a_n}) f_{\bgtheta} (\blX^{\mathcal{A}}_n,\va^n) d\bgmu^n(\blX_n^{\mathcal{A}}),
\end{equation*}
we know that
\begin{equation*}
   \nabla_{\bgtheta}h_{\blb}(\bgtheta)=\nabla_{\bgtheta} \mathbb{E}_{\bgtheta}[\blb^T\blT_n] = \sum_{\va^n\in \mathcal{A}^n}\nabla_{\bgtheta} \int_{\bgOmega^n}\blb^TT_n(a_1,X_1^{a_1},\cdots,a_n,X_n^{a_n}) f_{\bgtheta} (\blX^{\mathcal{A}}_n,\va^n) d\bgmu^n(\blX_n^{\mathcal{A}}).
\end{equation*}
By Assumption \ref{ass:2}, we know that for any $a\in \mathcal{A}$
\begin{equation*}
    \norm{ \nabla_{\bgtheta} 
    \log f_{\bgtheta,a}(X^a) }\leq \norm{ \nabla_{\bgtheta}\log f_{\bgtheta^*,a}(X^a) }+\Psi_1^a(X^a) \max_{\bgtheta,\bgtheta'\in \bgTheta} \norm{\bgtheta-\bgtheta'}=:F_a(X^a), \forall  X^a\in \Omega_a,
\end{equation*}
where the dominate function $F_a$ satisfies that
\[
\sup_{\bgtheta\in \bgTheta}\mathbb{E}_{X^a\sim f_{\bgtheta,a}} \{F_a(X^a)\}^2<\infty.
\]
Notice that
\begin{equation*}
\begin{split}
&\int_{\bgOmega^n}\blb^TT_n(a_1,X_1^{a_1},\cdots,a_n,X_n^{a_n}) \nabla_{\bgtheta} f_{\bgtheta} (\blX^{\mathcal{A}}_n,\va^n) d\bgmu^n(\blX_n^{\mathcal{A}})\\
=&\int_{\bgOmega^n}\blb^TT_n(a_1,X_1^{a_1},\cdots,a_n,X_n^{a_n}) \sum_{i=1}^n\sum_{a\in \mathcal{A}}  \nabla_{\bgtheta}\log f_{\bgtheta,  {a}}(X_i^{a}) f_{\bgtheta} (\blX^{\mathcal{A}}_n,\va^n) d\bgmu^n(\blX_n^{\mathcal{A}}),
\end{split}
\end{equation*}
\begin{equation*}
\norm{\blb^T\blT_n \sum_{i=1}^n\sum_{a\in \mathcal{A}}  \nabla_{\bgtheta}\log f_{\bgtheta,  {a}}(X_i^{a})  } 
    \leq |\blb^T\blT_n| \cdot  \sum_{i=1}^n\sum_{a\in \mathcal{A}} F_a(X^a_i),
\end{equation*}
and by Hölder's inequality,
\[
\mathbb{E}_{\bgtheta}[ {|\blb^T\blT_n| \cdot  \sum_{i=1}^n\sum_{a\in \mathcal{A}} F_a(X^a_i)}\leq\sum_{i=1}^n\sum_{a\in \mathcal{A}}\Big(\mathbb{E}_{\bgtheta}[ (\blb^T\blT_n)^2]\cdot \mathbb{E}_{\bgtheta}[\{F_a(X^a_i) \}^2 ]\Big)^{1/2}<\infty
\]

Taking into account that \(\bgOmega^n\) is independent of \(\bgtheta\), and by applying the Dominated Convergence Theorem together with the classical proof of differentiation under the integral sign, we arrive at
\begin{equation*}
\begin{split}
&\nabla_{\bgtheta}\int_{\bgOmega^n}\blb^TT_n(a_1,X_1^{a_1},\cdots,a_n,X_n^{a_n}) f_{\bgtheta} (\blX^{\mathcal{A}}_n,\va^n) d\bgmu^n(\blX_n^{\mathcal{A}}) \\
=& \int_{\bgOmega^n}\blb^TT_n(a_1,X_1^{a_1},\cdots,a_n,X_n^{a_n}) \nabla_{\bgtheta}f_{\bgtheta} (\blX^{\mathcal{A}}_n,\va^n) d\bgmu^n(\blX_n^{\mathcal{A}}).
\end{split}
\end{equation*}
In conclusion, we know that
\[
\nabla_{\bgtheta}h_{\blb}(\bgtheta)=\mathbb{E}_{\bgtheta}\left[\blb^T\blT_n\sum_{i=1}^n\sum_{a\in \mathcal{A}} \nabla_{\bgtheta}\log f_{\bgtheta,  {a}}(X_i^{a}   )\right]. 
\]
Next, we show that
\begin{equation}\label{equ:212}
\mathbb{E}_{\bgtheta}\left[\blb^T\blT_n\sum_{i=1}^n\sum_{a\in \mathcal{A},a\neq a_i}\nabla_{\bgtheta}\log f_{\bgtheta,  {a}}(X_i^{a}   ) \right]=0.
\end{equation}
First of all, let $\mathcal{F}_{i+1}=\sigma\{a_1, X^{a_1}_1, a_2, X^{a_2}_2, \cdots, a_{i},X^{a_i}_i \}$. Note that $a_{i+1}$ is measurable with respect to  $\mathcal{F}_{i}$ for all $i$. Notice that $\{X^{a}_n\}_{a\in \mathcal{A}}$ are independent of $\mathcal{F}_{n-1}$, as well as $\{X^{a}_n\}_{a\in \mathcal{A},a\neq a_n}$ and $X^{a_n}_n$ are independent, given $\mathcal{F}_{n-1}$. Also recall that $a_n$ is measurable in $\mathcal{F}_{n-1}$. Thus,
\begin{equation*}
\begin{split}
    &\mathbb{E}_{\bgtheta}\Big[\blb^T\blT_n \sum_{a\in \mathcal{A},a\neq a_n}\nabla_{\bgtheta}\log f_{\bgtheta,  {a}}(X_n^{a}    ) \Big|\mathcal{F}_{n-1}, X^{a_n}_n\Big]\\
    =&\blb^T\blT_n\sum_{a\in \mathcal{A},a\neq a_n}\mathbb{E}_{\bgtheta}\Big[\nabla_{\bgtheta} \log f_{\bgtheta,  {a}}(X_n^{a}    ) \Big|\mathcal{F}_{n-1}, X^{a_n}_n\Big]\\
    =&\blb^T\blT_n\sum_{a\in \mathcal{A},a\neq a_n}\mathbb{E}_{\bgtheta}\Big[\nabla_{\bgtheta} \log f_{\bgtheta,  {a}}(X_n^{a}    )  \Big]=0.    
\end{split}
\end{equation*}
Note that for fixed $1\leq i<n$, $\{X^{a}_j\}_{j\geq i,a\in \mathcal{A}}$ and $\mathcal{F}_{i-1}$ are independent. Define another $\sigma$-algebra, $\mathcal{G}_{i-1}=\sigma(\mathcal{F}_{i-1},\{X^a_j\}_{i+1\leq j\leq n ,a\in \mathcal{A} })$. Note that $a_i,a_{i+1},\cdots, a_n$ are measurable in $\sigma(\mathcal{G}_{i-1},X_i^{a_i})$. Furthermore, $\{X^{a}_i\}_{a\in \mathcal{A},a\neq a_i}$ and $X^{a_i}_i$ are independent, given $\mathcal{G}_{i-1}$. Thus, for any $1\leq i<n$
\begin{equation*}
\begin{split}
    &\mathbb{E}_{\bgtheta}\Big[\blb^T\blT_n \sum_{a\in \mathcal{A},a\neq a_i}\nabla_{\bgtheta}\log f_{\bgtheta,  {a}}(X_i^{a}    ) \Big|\mathcal{G}_{i-1}, X^{a_i}_i\Big]\\
    =&\blb^T\blT_n\sum_{a\in \mathcal{A},a\neq a_i}\mathbb{E}_{\bgtheta}\Big[\nabla_{\bgtheta} \log f_{\bgtheta,  {a}}(X_i^{a}    ) \Big|\mathcal{G}_{i-1}, X^{a_i}_i\Big]\\
    =&\blb^T\blT_n\sum_{a\in \mathcal{A},a\neq a_i}\mathbb{E}_{\bgtheta}\Big[\nabla_{\bgtheta}\log f_{\bgtheta,  {a}}(X_i^{a}    )  \Big]=0.    
\end{split}
\end{equation*}
By the law of iterated expectation, we have proved \eqref{equ:212}. Hence, we know that
\begin{equation*}
    \nabla_{\bgtheta}h_{\blb}(\bgtheta)=\mathbb{E}_{\bgtheta}\left[\blb^T\blT_n\nabla_{\bgtheta}\sum_{i=1}^n \log f_{\bgtheta,  {a_i}}(X_i^{a_i}   )\right]=0.
\end{equation*}
Set $\blY_n=\nabla_{\bgtheta} \sum_{i=1}^n\log f_{\bgtheta , a_i}(X^{a_i}_i)$, and we have 
\begin{equation*}
    \nabla_{\bgtheta} h_{\blb}(\bgtheta )= \mathbb{E}_{\bgtheta}  [  \blY_n \blT_n^T \blb   ] = \operatorname{cov}_{\bgtheta} (  \blY_n, \blb^T\blT_n  ).
\end{equation*}
By multivariate Cauchy-Schwartz inequality \eqref{ineq:MCR}, for any $\blb\in \mathbb{R}^l$,
\begin{equation*}
\begin{split}
&\blb^T\operatorname{cov}_{\bgtheta}(\blT_n )\blb\\
=&\operatorname{var}_{\bgtheta}(\blb^T\blT_n )\\
\geq&\operatorname{cov}_{\bgtheta}\Big(\blb^T\blT_n, \blY_n  \Big) \big\{\operatorname{cov}_{\bgtheta } ( \blY_n  )\big\}^{-1}\operatorname{cov}_{\bgtheta}\Big(\blY_n , \blb^T\blT_n \Big)\\
=&\blb^T \big\{\nabla_{\bgtheta} h(\bgtheta) \big\}^T  \big\{\operatorname{cov}_{\bgtheta } ( \blY_n  )\big\}^{-1} \nabla_{\bgtheta} h(\bgtheta)\blb.
\end{split}
\end{equation*}
Note that
\[
\mathbb{E}_{\bgtheta}[\blY_{i-1}  \{\nabla_{\bgtheta} \log f_{\bgtheta,a_i}(X^{a_i}_i)\}^T ]=\mathbb{E}_{\bgtheta}\Big\{\blY_{i-1}\cdot \mathbb{E}_{\bgtheta}\big[   \{\nabla_{\bgtheta} \log f_{\bgtheta,a_i}(X^{a_i}_i)\}^T|\mathcal{F}_{i-1}\big]\Big\}=0.
\]
Thus
\begin{equation*}
\begin{split}
    &\operatorname{cov}_{\bgtheta } ( \blY_n  ) = \mathbb{E}_{\bgtheta}[\blY_n \blY_n^T]=\mathbb{E}_{\bgtheta}[\blY_{n-1} \blY_{n-1}^T]+ \mathbb{E}_{\bgtheta}[ \mathcal{I}_{a_i}(\bgtheta) ]=n\cdot \mathcal{I}^{\mathbb{E}_{\bgtheta}\overline{{\bgpi}}_n}(\bgtheta). 
\end{split}
\end{equation*}
In conclusion, for any $\blb\in \mathbb{R}^l$, we obtain that
\begin{equation*}
    \blb^T \operatorname{cov}_{\bgtheta }(\blT_n ) \blb=\operatorname{var}_{\bgtheta }(\blb^T \blT_n )\geq \blb^T\left[\frac{1}{n}\Big\{\nabla_{\bgtheta} h(\bgtheta )\Big\}^T \Big\{\mathcal{I}^{\mathbb{E}_{\bgtheta}\overline{{\bgpi}}_n} (\bgtheta )\Big\}^{-1} \nabla_{\bgtheta} h(\bgtheta )\right]\blb.
\end{equation*}
This implies that
\begin{equation*}
    \operatorname{cov}_{\bgtheta }(\blT_n ) \succeq \frac{1}{n}\{\nabla_{\bgtheta} h(\bgtheta )\}^T \{\mathcal{I}^{\mathbb{E}_{\bgtheta}\overline{{\bgpi}}_n} (\bgtheta )\}^{-1} \nabla_{\bgtheta} h(\bgtheta ).
\end{equation*}   
If $h(\bgtheta)=\bgtheta$, we know that
\begin{equation*}
    n\operatorname{cov}_{\bgtheta}(\blT_n )  \succeq \{\mathcal{I}^{\mathbb{E}_{\bgtheta}\overline{{\bgpi}}_n} (\bgtheta)\}^{-1}.
\end{equation*}
By assumption~\ref{ass:5}, we obtain
\begin{equation*}
    \mathbb{G}_{\bgtheta}(n\operatorname{cov}_{\bgtheta}(\blT_n )  ) \geq \mathbb{G}_{\bgtheta}(\{ \mathcal{I}^{\mathbb{E}_{\bgtheta} [\overline{{\bgpi}}_n] }(\bgtheta ) \}^{- 1} ) \geq \inf_{{\bgpi}\in \ShatA }\mathbb{G}_{\bgtheta }(\{\mathcal{I}^{ {\bgpi} } (\bgtheta )\}^{-1}  ). 
\end{equation*}
\end{proof}

\begin{proof}[Proof of Theorem~\ref{thm:ultimate}]
~~

\paragraph*{Part 1}

Notice that $L(\bgtheta^*,\widehat{\bgtheta})$ is a loss function, which means that $L(\bgtheta^*,\widehat{\bgtheta}) \geq L(\bgtheta^*,\bgtheta^*)=0$. Due to $L(\bgtheta^*,\widehat{\bgtheta})$ is differentiable in $\widehat{\bgtheta}$, we know that $ \nabla_{\widehat{\bgtheta}} L(\bgtheta^*,\bgtheta^*)=\bm{0}$.

Applying first order Taylor expansion to $L(\bgtheta^*,\widehat{\bgtheta})$ with respect to $\widehat{\bgtheta}$, we obtain that
\begin{equation}\label{ineq:L33}
    L(\bgtheta^*,\blT_n)= \frac{1}{2}\innerpoduct{ \nabla^2_{\widehat \bgtheta} L(\bgtheta^*,\widehat\bgtheta)\Big\vert_{\widehat\bgtheta=\widetilde{\bgtheta}_n }  (\bgtheta^*-\blT_n)}{\bgtheta^*-\blT_n}\geq \eta \norm{\bgtheta^* -\blT_n   }^2,
\end{equation}
where $\widetilde{\bgtheta}_n=t_n\bgtheta^*+(1-t_n)\blT_n$ for some $t_n\in (0,1)$. Thus,
\begin{equation*}
    \mathbb{E}_{\bgtheta^*}n\cdot L(\bgtheta^*,\blT_n)\geq \eta \mathbb{E}_{\bgtheta^*}\norm{\sqrt{n} (\blT_n-\bgtheta^*)}^2.
\end{equation*}
To show \eqref{lim:efficiency}, without loss of generality, we assume that 
\[
\limsup_{n\to \infty}\mathbb{E}_{\bgtheta^*}\norm{\sqrt{n} (\blT_n-\bgtheta^*)}^2<\infty.
\]
This implies $\blT_n\inP \bgtheta^*$. {It also implies that $\blT_n$ has finite second moment, and,  thus, conditions of Lemma~\ref{thm:C-R ineq} are satisfied.}
By Lemma~\ref{thm:C-R ineq}, we have
\(\operatorname{cov}_{\bgtheta}(\blT_n ) \succeq \frac{1}{n}  \{\mathcal{I}^{\mathbb{E}\overline{{\bgpi}}_n} (\bgtheta)\}^{-1}.\) 

If $L(\bgtheta^*,\widehat{\bgtheta})  \equiv \innerpoduct{H_{\bgtheta^*} (\bgtheta^*-\widehat{\bgtheta})}{\bgtheta^*-\widehat{\bgtheta}}$, we obtain
\[
    \mathbb{E}_{\bgtheta^*}[nL(\bgtheta^*,\blT_n)] = n\langle H_{\bgtheta^*}  , \operatorname{cov}_{\bgtheta^*}(\blT_n) \rangle \geq \tr(H_{\bgtheta^*}\{\mathcal{I}^{\mathbb{E}\overline{{\bgpi}}_n} (\bgtheta^*)\}^{-1}).
\] 
If $L(\bgtheta^*,\widehat{\bgtheta})  \not\equiv \innerpoduct{H_{\bgtheta^*} (\bgtheta^*-\widehat{\bgtheta})}{\bgtheta^*-\widehat{\bgtheta}}$, under the theorem's assumption
\[   \limsup_{n\to \infty}\mathbb{E}_{\bgtheta^*}n\norm{\blT_n-\bgtheta^*}^2 I(\norm{\blT_n-\bgtheta^*} >\varepsilon)= 0.\]
Define $\blV_n=\sqrt{n} \big(\bgtheta^*-\blT_n \big)$, and its truncation $\blV^M_n=\blV_n I(\norm{\blV_n} \leq M )$. Define 
\[
H(\bgtheta^*,\blT_n)=\frac{1}{2}\nabla^2_{\widehat \bgtheta} L(\bgtheta^*,\widehat\bgtheta)\Big\vert_{\widehat\bgtheta=\widetilde{\bgtheta}_n } , 
\]
where $\widetilde{\bgtheta}_n =t_n \bgtheta^*+(1-t_n)\blT_n$. {According to \eqref{ineq:L33},}$L(\bgtheta^*,\blT_n)=\innerpoduct{H(\bgtheta^*,\blT_n) (\bgtheta^*-\blT_n)}{\bgtheta^*-\blT_n}$. Furthermore, for any $\varepsilon >0$,
\begin{equation*}
\begin{split}
    &\mathbb{E}_{\bgtheta^*} \Big| n\cdot L(\bgtheta^*,\blT_n) -\innerpoduct{H_{\bgtheta^*}  \blV_n}{\blV_n}\Big|I(\norm{\blT_n -\bgtheta^* } \leq \varepsilon    )\\
    \leq&\max_{ \norm{\widehat{\bgtheta}-\bgtheta^*} \leq \varepsilon } \norm{H(\bgtheta^*,\widehat{\bgtheta})-H_{\bgtheta^*} }_{op} \mathbb{E}_{\bgtheta^*} \norm{\blV_n}^2 \\
    = & o(1). 
\end{split}
\end{equation*}
Now, we obtain that
\begin{equation*}
\begin{split}
    \mathbb{E}_{\bgtheta^*}\Big[ n\cdot L(\bgtheta^*,\blT_n) \Big]&\geq \mathbb{E}_{\bgtheta^*}\Big[ n\cdot L(\bgtheta^*,\blT_n)\cdot  I(\norm{\blV_n }\leq \varepsilon \sqrt{n}) \Big]\\
    &\geq \mathbb{E}_{\bgtheta^*}    \innerpoduct{H_{\bgtheta^*}  \blV^{\varepsilon \sqrt{n}}_n}{\blV^{\varepsilon \sqrt{n}}_n} -\max_{ \norm{\widehat{\bgtheta}-\bgtheta^*} \leq \varepsilon } \norm{H(\bgtheta^*,\widehat{\bgtheta})-H_{\bgtheta^*} }_{op} \mathbb{E}_{\bgtheta^*} \norm{\blV_n}^2\\
    &= \mathbb{E}_{\bgtheta^*}    \innerpoduct{H_{\bgtheta^*}  \blV_n}{\blV_n}-\mathbb{E}_{\bgtheta^*}n\norm{\blT_n-\bgtheta^*}^2 I(\norm{\blT_n-\bgtheta^*} >\varepsilon) -o(1)\\
    &\geq  \min_{{\bgpi} \in \ShatA} \operatorname{tr} ( H_{\bgtheta^*}\{\mathcal{I}^{ {\bgpi} } (\bgtheta^*)\}^{-1} )-\mathbb{E}_{\bgtheta^*}n\norm{\blT_n-\bgtheta^*}^2 I(\norm{\blT_n-\bgtheta^*} >\varepsilon) -o(1),
\end{split}
\end{equation*}
{where the last inequality is due to Lemma~\ref{thm:C-R ineq}.
}
Taking the inferior limit as $n\to \infty$ and then taking the inferior limit as $\varepsilon\to 0^+$, we obtain
\[
\liminf_{n\to \infty}\mathbb{E}_{\bgtheta^*}\Big[ n\cdot L(\bgtheta^*,\blT_n) \Big]\geq    \min_{{\bgpi} \in \ShatA}\operatorname{tr} ( H_{\bgtheta^*}\{\mathcal{I}^{ {\bgpi} } (\bgtheta^*)\}^{-1} ).
\]
{The `in particular' part is proved by noting that $H_{\bgtheta^*} = I_p$ in this case.}
\paragraph*{Part 2} 
Recall the log-likelihood defined in \ref{equ:log-like}. Because \(f_{\bgtheta,a}(\cdot) = h_{\bgxi_a,a}(\cdot)\), we obtain that
\begin{equation}\label{ineq:31}
    -\nabla^2_{\bgtheta}l_n(\bgtheta;\va_n)=-\frac{1}{n}\sum_{i=1}^n \blZ_{a_i}^T\nabla^2_{\bgxi_{a_i}}\log h_{\bgxi_{a_i},a_i}(X_i)  \blZ_{a_i} \succeq \alpha  \sum_{a\in \mathcal{A}}\overline{{\bgpi}}_n(a)  \blZ^T_{a}\blZ_{a },
\end{equation}
and \(\mathcal{I}_{\bgxi_a, a}(\bgxi_a)=-\mathbb{E}_{X^a\sim h_{\bgxi_a,a}} \nabla^2_{\bgxi_a} \log h_{\bgxi_a,a}(X^a)\succeq \alpha I, \text{ where } \bgxi_a=\blZ_a \bgtheta.\)

Under Assumption~\ref{ass:6A}, we obtain \(\mathcal{I}_a(\bgtheta) = \blZ^T_a \mathcal{I}_{\bgxi_a,a}(\bgxi_a)\blZ_a\) and
\[
\nabla^2_{\bgtheta} \log f_{\bgtheta, a}(X^a) = \blZ^T_a\nabla^2_{\bgxi_a}\log h_{\bgxi_a,a}(X^a)  \blZ_a,
\]
and
\[
\mathcal{I}^{{\bgpi}}(\bgtheta^*)=\sum_{a\in \mathcal{A}}{\bgpi}(a) \blZ^T_a \mathcal{I}_{\bgxi_a,a}(\bgxi^*_a)\blZ_a \succeq \alpha \sum_{a\in \mathcal{A}}{\bgpi}(a) \blZ^T_a  \blZ_a
\]
is a positive definite matrix.

Applying the first order Taylor expansion of $L(\bgtheta^*,\widehat{\bgtheta})$ over $\widehat{\bgtheta}$, we obtain that
\begin{equation}\label{ineq:L33_2}
    L(\bgtheta^*,\widehat{\bgtheta}_n)=\frac{1}{2} \innerpoduct{ \nabla^2_{\widehat \bgtheta} L(\bgtheta^*,\widehat\bgtheta)\Big\vert_{\widehat\bgtheta=\widetilde{\bgtheta}_n }  (\bgtheta^*-\widehat\bgtheta_n)}{\bgtheta^*-\widehat\bgtheta_n}\leq \eta' \norm{\bgtheta^* -\widehat\bgtheta_n   }^2,
\end{equation}
where $\widetilde{\bgtheta}_n=t_n\bgtheta^*+(1-t_n)\widehat \bgtheta_n$ for some $t_n\in (0,1)$. Recall that $\mathbb{G}_{\bgtheta}(\bgSigma)=\tr (H_{\bgtheta} \bgSigma)$. Note that  $\nabla \mathbb{G}_{\bgtheta}(\bgSigma)=H_{\bgtheta}$ and $\kappa(H_{\bgtheta})\leq \frac{\eta'}{\eta}<\infty$, which implies that $\mathbb{G}_{\bgtheta}(\bgSigma)$ satisfies Assumption \ref{ass:5}.

By Theorem~\ref{thm:empirical pi as converge}, $\overline{{\bgpi}}_n \inP {\bgpi^*}= \arg\min_{\bgpi\in \ShatA}\mathbb{F}_{\bgtheta^*}(\bgpi) $ and $\mathcal{I}^{{\bgpi^*}}(\bgtheta)$ is nonsingular for any $\bgtheta\in\bgTheta$, applying Theorem~\ref{thm:GI0-GI1AN}, we obtain
\begin{equation}\label{lim:31}
    \sqrt{n} (\widehat{\bgtheta}_n^{\text{ML}} - \bgtheta^*) \inD N_p\Big(0,\big\{\mathcal{I}^{{\bgpi^*}}(\bgtheta^*)\big\}^{-1} \Big) \text{ as } n\to\infty.
\end{equation}
Notice that for any $n\geq n_0$, by Lemma~\ref{thm:selection bound} and Assumption~\ref{ass:6B}, there exists $\underline{C}>0$ such that
\[
-\nabla^2_{\bgtheta}l_n(\bgtheta)\succeq \alpha \sum_{a\in \mathcal{A}} \overline{{\bgpi}}_n(a) \blZ_a^T \blZ_a \succeq \alpha \inf_{n\geq n_0}   \lambda_{min}\Big( \overline{{\bgpi}}_n(a) \blZ_a^T \blZ_a \Big)I_p\succeq 2\underline{C}I_p.
\]
By Taylor expansion, we obtain
\begin{equation*}
     0\leq l_n(\widehat{\bgtheta}_n^{\text{ML}};\va_n) -l_n(\bgtheta^*;\va_n)\leq  \innerpoduct{\nabla_{\bgtheta} l_n(\bgtheta^*;\va_n)}{\widehat{\bgtheta}_n^{\text{ML}} - \bgtheta^*} - \underline{C} \norm{\widehat{\bgtheta}_n^{\text{ML}} - \bgtheta^*}^2.
\end{equation*}
Thus,
\[
\norm{\widehat{\bgtheta}_n^{\text{ML}} - \bgtheta^*} \leq \frac{1}{\underline{C}}    \norm{\nabla_{\bgtheta} l_n(\bgtheta^*;\va_n)}.
\]
By Theorem 6.2 in \cite{dasgupta2008asymptotic}, to show that
\[
\mathbb{E}_{\bgtheta^*}\Big[n \norm{\widehat{\bgtheta}_n^{\text{ML}}- \bgtheta^* }^2\Big]\to \tr(\big\{\mathcal{I}^{{\bgpi}}(\bgtheta^*)\big\}^{-1}),
\]
combined with \eqref{lim:31}, it suffices to show that 
{
\[
\limsup_{n\to\infty}\mathbb{E}_{\bgtheta^*}\Big(\sqrt{n} \norm{\widehat{\bgtheta}_n^{\text{ML}}- \bgtheta^* }\Big)^{2+\delta}<\infty.\]
}
Note that
{ 
\begin{equation*}
\mathbb{E}_{\bgtheta^*}\Big(\sqrt{n} \norm{\widehat{\bgtheta}_n^{\text{ML}}- \bgtheta^* }\Big)^{2+\delta}\leq\frac{ 1 }{( \underline{C} )^{1+\delta} } \mathbb{E}_{\bgtheta^*} \norm{ \sqrt{n}  \nabla_{\bgtheta} l_n(\bgtheta^*;\va_n) }^{2+\delta}.
\end{equation*}
}
By classical $c_r$-inequality (see Chapter 9 of \cite{lin2010probability}), we have
\[
\mathbb{E}_{\bgtheta^*} \norm{ \sqrt{n} \nabla_{\bgtheta} l_n(\bgtheta^*;\va_n) }^{2+\delta}\leq   {p^{\delta/2}} 
\sum_{j=1}^p \mathbb{E}_{\bgtheta^*}\Big|  \sum_{i=1}^n  \frac{1}{\sqrt{n}} \ble_j^T \nabla_{\bgtheta} \log f_{\bgtheta^*,a_i}(X_i) \Big|^{2+\delta }.
\]
Since $\sum_{i=1}^n \ble_j^T \nabla_{\bgtheta} \log f_{\bgtheta^*,a_i}(X_i)$ is a martingale, 
applying inequality $(45)$ in \cite{lin2010probability}, 
we obtain
\begin{equation*}
\begin{split}
    \mathbb{E}_{\bgtheta^*}\Big|  \sum_{i=1}^n  \frac{1}{\sqrt{n}} \ble_j^T \nabla_{\bgtheta} \log f_{\bgtheta^*,a_i}(X_i) \Big|^{2+\delta }&\leq C_{2+\delta}\cdot n^{\delta/2} \sum_{i=1}^n \mathbb{E}_{\bgtheta^*}\Big| \frac{1}{\sqrt{n}}  \ble_j^T \nabla_{\bgtheta} \log f_{\bgtheta^*,a_i}(X_i)\Big|^{2+\delta}\\
    &\leq C_{2+\delta}   \sum_{a\in \mathcal{A}}\mathbb{E}_{X^a\sim f_{\bgtheta^*,a} } \norm{\nabla_{\bgtheta} \log f_{\bgtheta^*,a}(X^a)}^{2+\delta}\\
    &\leq C_{2+\delta}   \sum_{a\in \mathcal{A}}\mathbb{E}_{X^a\sim f_{\bgtheta^*,a} } \norm{\nabla_{\bgtheta} \log h_{\bgxi_a^*,a}(X^a)}^{2+\delta} \norm{\blZ_a}^{2+\delta}_{op}.
\end{split}
\end{equation*}
In conclusion, we obtain
\begin{equation}\label{ineq:L2_33}
    \sup_{n\geq n_0}\mathbb{E}_{\bgtheta^*}\Big(\sqrt{n} \norm{\widehat{\bgtheta}_n^{\text{ML}}- \bgtheta^* }\Big)^{2+\delta}<\infty.
\end{equation}
Notice that as $\widehat\bgtheta_n\inP \bgtheta^*$, we know that
\[
\frac{1}{2}\nabla^2_{\widehat \bgtheta} L(\bgtheta^*,\widehat\bgtheta)\Big\vert_{\widehat\bgtheta=\widetilde{\bgtheta}_n }\inP H_{\bgtheta^*}.
\]
Thus, we obtain that
\begin{equation}\label{equ:lim35}
nL(\bgtheta^*,\widehat{\bgtheta}_n^{\text{ML}})=n\innerpoduct{H_{\bgtheta^*} (\bgtheta^*-\widehat{\bgtheta}_n^{\text{ML}})}{\bgtheta^*-\widehat{\bgtheta}_n^{\text{ML}}}+o_p(1).
\end{equation}
By \eqref{ineq:L2_33} and \eqref{ineq:L33_2}, we obtain that
\[
\sup_{n\geq n_0}\mathbb{E}_{\bgtheta^*}\Big[nL(\bgtheta^*,\widehat{\bgtheta}_n^{\text{ML}}) \Big]^{1+\delta/2} \leq (\eta')^{1+\delta/2}\sup_{n\geq n_0}\mathbb{E}_{\bgtheta^*}\Big(\sqrt{n} \norm{\widehat{\bgtheta}_n^{\text{ML}}- \bgtheta^* }\Big)^{2+\delta}<\infty.
\]
Applying Theorem 6.2 in \cite{dasgupta2008asymptotic}, we obtain that as $n\to \infty$,
\[
\mathbb{E}_{\bgtheta^*} \Big[nL(\bgtheta^*,\widehat{\bgtheta}_n^{\text{ML}}) \Big]\to \mathbb{E}\innerpoduct{H_{\bgtheta^*} \blV }{\blV}=\tr (H_{\bgtheta^*} \{\mathcal{I}^{{\bgpi}}(\bgtheta^*) \}^{-1}), \blV \sim N_p(\bm{0}_p, \{\mathcal{I}^{{\bgpi}}(\bgtheta^*) \}^{-1}).
\]
Applying Theorem~\ref{thm:opt_selection}, the proof of the second part of Theorem~\ref{thm:ultimate} is completed.
\end{proof}

\subsection{Proof of Theorem~\ref{thm:convolution_theorem}}
\begin{proof}[Proof of Theorem~\ref{thm:convolution_theorem}]
The proof of Theorem~\ref{thm:convolution_theorem} is similar to that of Theorem~8.8 and Theorem~8.11 in~\cite{van2000asymptotic}. Thus, we will only state the main differences and omit the repetitive details.

{For proving the first part of the theorem, }
we follow the proof of Theorem~8.8 in~\cite{van2000asymptotic}. We need to verify Theorem~8.3, Theorem~7.10, as well as Proposition~8.4,  as presented in~\cite{van2000asymptotic}, {under our sequential setting}. %

{For proving the second part of the theorem, we follow the proof of }Theorem~8.11 in~\cite{van2000asymptotic}. It is sufficient to modify and prove Theorem~7.2 %
and Proposition 8.6, as presented in~\cite{van2000asymptotic}, {under our sequential setting}.

{Below we verify the above mentioned results in our context.}

\paragraph*{Differentiable in quadratic mean}
We need to show that densities $\{f_{\bgtheta,a}(\cdot)\}_{a\in \mathcal{A}}$ are differentiable in quadratic mean at $\bgtheta$, which means that 
\begin{equation}\label{equ:differentiable_in_quadratic_mean}
    \int \Big[  \sqrt{f_{\bgtheta+\blh,a}(x) }-\sqrt{f_{\bgtheta,a}(x)}-\frac{1}{2} \blh^T \nabla_{\bgtheta}\log f_{\bgtheta,a}(x) \sqrt{f_{\bgtheta,a}(x) }  \Big]^2d\mu(x)=o(\norm{\blh}^2).
\end{equation}
By applying the regularity conditions and using Lemma~7.6 from \cite{van2000asymptotic}, we have completed the proof of \eqref{equ:differentiable_in_quadratic_mean} for any $a\in \mathcal{A}$ and $\bgtheta$, where $\bgtheta$ is an interior point of $\bgTheta$.

\paragraph*{Modified Theorem~7.2 in \cite{van2000asymptotic}}
{We modified Theorem~7.2 in \cite{van2000asymptotic} in our context as follows.
Let $P_{n,\bgtheta}$ denote the joint distribution of
$(a_1,X_1,\cdots, a_n,X_n)$ following some experiment selection rule with the empirical selection proportion $\overline{\bgpi}_n$.
Then, {given that $\blh_n=\blh+ o(1)$,}
\begin{equation}
  \begin{split}     &\log\frac{P_{n,\bgtheta+\blh_n/\sqrt{n}}}{P_{n,\bgtheta}}(a_1, X_1,\cdots,a_n,X_n)\\
  \sim  &   \log \prod_{j=1}^n \frac{f_{\bgtheta+ {\blh_n}/{\sqrt{n}},a_j}(X_j^{a_j}) }{f_{\bgtheta ,a_j}(X_j^{a_j})}\\
  = &\frac{1}{\sqrt{n}} \sum_{j=1}^n \blh^T \nabla_{\bgtheta}  \log f_{\bgtheta,a_j}(X_j^{a_j})-\frac{1}{2} \blh^T \mathcal{I}^{\bgpi}(\bgtheta) \blh +o_{p}(1) 
  \end{split}
\end{equation}
where $\{ X^a_{j} \}_{a\in \mathcal{A},j\geq 1}$, where $X^a_j\sim f_{\bgtheta,a}(\cdot)$ are independent random variables, `$\sim$' means that random variables on both sides share the same distribution, the second line is due to Lemma~\ref{lem:same dist}, and the last line is obtained following a similar proof as that of Theorem 7.2 in \cite{van2000asymptotic}, which is detailed below.

By Assumptions~\ref{ass:1}-\ref{ass:4}, the Dominated Convergence Theorem, and the  proof of the classical differentiation under the integral sign, we arrive at
\[
\mathbb{E}[\nabla_{\bgtheta}  \log f_{\bgtheta,a_j}(X_j^{a_j}) |\mathcal{F}_{j-1}] =\bm{0} \text{ and }\mathbb{E}\Big[ \sum_{j=1}^n  \nabla_{\bgtheta}  \log f_{\bgtheta,a_j}(X_j^{a_j})\Big]=\bm{0}.
\]
The proof of the first Equation (7.3) in \cite{van2000asymptotic} needs to be modified as follows. Let $W_{nj}=2\Big(\sqrt{ \frac{f_{\bgtheta +\blh_n/\sqrt{n},a_j } (X_j^{a_j})}{f_{\bgtheta ,a_j } (X_j^{a_j})}  } - 1 \Big)$ and 
\(V_n=\sum_{j=1}^n W_{nj} -\frac{1}{\sqrt{n} }\blh^T\sum_{j=1}^n \nabla_{\bgtheta}\log f_{\bgtheta,a_j}(X_j^{a_j}) \). 
We know that
\begin{equation*}
\begin{split}
    &\operatorname{var}\Big(V_n \Big)\\
    =&\operatorname{var}\Big(V_{n-1}\Big)+ \operatorname{var}\Big( W_{nn} -\frac{1}{\sqrt{n} } \blh^T\nabla_{\bgtheta}\log f_{\bgtheta,a_n}(X_n^{a_n}) \Big) +2\operatorname{cov}\Big(V_{n-1},W_{nn} -\frac{1}{\sqrt{n} }\blh^T \nabla_{\bgtheta}\log f_{\bgtheta,a_n}(X_n^{a_n})\Big)
\end{split}
\end{equation*}
and
\begin{equation*}
\begin{split}
    &\operatorname{cov}\Big(V_{n-1},W_{nn} -\frac{1}{\sqrt{n} } \blh^T\nabla_{\bgtheta}\log f_{\bgtheta,a_n}(X_n^{a_n})\Big)\\
    =&\mathbb{E}\Big[\big(V_{n-1}-\mathbb{E}V_{n-1}\big)\big( W_{nn} -\frac{1}{\sqrt{n} } \blh^T\nabla_{\bgtheta}\log f_{\bgtheta,a_n}(X_n^{a_n})-\mathbb{E}\big(W_{nn} -\frac{1}{\sqrt{n} }\blh^T\nabla_{\bgtheta} \log f_{\bgtheta,a_n}(X_n^{a_n})\big) \big) \Big]\\
    =&\mathbb{E}\Big\{\big(V_{n-1}-\mathbb{E}V_{n-1}\big)\mathbb{E}\Big[\big( W_{nn} -\frac{1}{\sqrt{n} } \blh^T\nabla_{\bgtheta}\log f_{\bgtheta,a_n}(X_n^{a_n})-\mathbb{E}\big(W_{nn} -\frac{1}{\sqrt{n} }\blh^T \nabla_{\bgtheta}\log f_{\bgtheta,a_n}(X_n^{a_n})\big) \big)\Big|\mathcal{F}_{n-1} \Big]\Big\}\\
    =&0.
\end{split}
\end{equation*}
By induction and \eqref{equ:differentiable_in_quadratic_mean}, we obtain that as $n\to \infty$,
\begin{equation*}
\begin{split}
    &\operatorname{var}\Big(V_n \Big)= \sum_{j=1}^n \operatorname{var}\Big( W_{nj} -\frac{1}{\sqrt{n} } \blh^T\nabla_{\bgtheta}\log f_{\bgtheta,a_j}(X_j^{a_j}) \Big) \\
    \leq&\sum_{j=1}^n \mathbb{E}\Big( W_{nj} -\frac{1}{\sqrt{n} }\blh^T\nabla_{\bgtheta} \log f_{\bgtheta,a_j}(X_j^{a_j}) \Big)^2\\
     \leq & \sum_{j=1}^n \sum_{a\in\mathcal{A}} \mathbb{E}\Big( 2\Big(\sqrt{ \frac{f_{\bgtheta +\blh_n/\sqrt{n},a  } (X_j^{a })}{f_{\bgtheta ,a  } (X_j^{a })}  } - 1 \Big) -\frac{1}{\sqrt{n} }\blh^T\nabla_{\bgtheta} \log f_{\bgtheta,a}(X_j^{a}) \Big)^2\\
    \leq&8 n\sum_{a\in \mathcal{A}}\int \Big[  \sqrt{f_{\bgtheta+\blh_n/\sqrt{n},a}(x) }-\sqrt{f_{\bgtheta,a}(x)}-\frac{1}{2\sqrt{n}} \blh_n^T \nabla_{\bgtheta}\log f_{\bgtheta,a}(x) \sqrt{f_{\bgtheta,a}(x) }  \Big]^2d\mu(x)\\
    &+2 (\blh-\blh_n)^T \sum_{a\in \mathcal{A}} \mathcal{I}_a(\bgtheta)(\blh-\blh_n)\\
    =&8 o(\norm{\blh}^2 ) + 2 (\blh-\blh_n)^T \sum_{a\in \mathcal{A}} \mathcal{I}_a(\bgtheta)(\blh-\blh_n)\to 0.
\end{split}
\end{equation*}

Because of \eqref{equ:differentiable_in_quadratic_mean}, we obtain that
\begin{equation*}
\begin{split}
    &\left|\norm{\sqrt{  f_{\bgtheta +\blh_n/\sqrt{n},a } (x)} - \sqrt{    {f_{\bgtheta ,a } (x)}} }_{L^2(\mu)} - \norm{\frac{1}{2\sqrt{n}} \blh_n^T \nabla_{\bgtheta}\log f_{\bgtheta,a}(x) \sqrt{f_{\bgtheta,a}(x) }}_{L^2(\mu)} \right|\\
    \leq& \Big( \int \Big[  \sqrt{f_{\bgtheta+\blh_n/\sqrt{n},a}(x) }-\sqrt{f_{\bgtheta,a}(x)}-\frac{1}{2\sqrt{n}} \blh_n^T \nabla_{\bgtheta}\log f_{\bgtheta,a}(x) \sqrt{f_{\bgtheta,a}(x) }  \Big]^2d\mu(x)\Big)^{1/2}=o\Big(\frac{\norm{\blh}}{\sqrt{n}}\Big). 
\end{split}
\end{equation*}
Note that
\[
\norm{\frac{1}{2\sqrt{n}} \blh_n^T \nabla_{\bgtheta}\log f_{\bgtheta,a}(x) \sqrt{f_{\bgtheta,a}(x) }}_{L^2(\mu)}^2=\frac{1}{4n} \blh_n^T \mathcal{I}_a(\bgtheta)\blh_n=O\big(\frac{\norm{\blh}^2 }{n}\big).
\]
By inequality $|x^2-y^2|\leq|x-y|\big||x|+|y|\big|\leq |x-y|\big|2|x|+|x-y|\big|$, we obtain
\begin{equation*}
\begin{split}
    &\left|\norm{\sqrt{  f_{\bgtheta +\blh_n/\sqrt{n},a } (x)} - \sqrt{    {f_{\bgtheta ,a } (x)}} }_{L^2(\mu)}^2-\frac{1}{4n} \blh_n^T \mathcal{I}_a(\bgtheta) \blh_n\right|\\
    \leq& o\Big(\frac{\norm{\blh}}{\sqrt{n}}\Big) \left|  2\cdot O(\frac{\norm{\blh}}{\sqrt{n}})+ o\Big(\frac{\norm{\blh}}{\sqrt{n}}\Big) \right|=o\Big(\frac{\norm{\blh}^2}{ {n}}\Big).
\end{split}
\end{equation*}
Thus, the second Equation (7.3) in \cite{van2000asymptotic} is modified by
\begin{equation}
\begin{split}
&\mathbb{E} [W_{nj}|\mathcal{F}_{j-1}]=2 \Big(\int\sqrt{  {f_{\bgtheta +\blh_n/\sqrt{n},a_j } (x)}     {f_{\bgtheta ,a_j } (x)}} d\mu(x)-1 \Big)\\
=&-  \int\Big[\sqrt{  f_{\bgtheta +\blh_n/\sqrt{n},a_j } (x)} - \sqrt{    {f_{\bgtheta ,a_j } (x)}} \Big]^2d\mu(x)  =-\frac{1}{4n}\blh_n^T\mathcal{I}_{a_j}(\bgtheta)\blh_n+o(\frac{1}{n})\\
=&-\frac{1}{4n}\blh^T\mathcal{I}_{a_j}(\bgtheta)\blh+\frac{1}{n}o(1),
\end{split}
\end{equation}
where the $o(1)$ converges uniformly to $0$ as $n\to \infty$. Now, we obtain that
\[
\mathbb{E}\overline{\bgpi}_n=\bgpi+o(1),\text{ and }\mathbb{E} \sum_{j=1}^n W_{nj} = -\frac{1}{4}\blh^T\mathcal{I}^{\mathbb{E}\overline{\bgpi}_n}(\bgtheta)\blh+o(1)=-\frac{1}{4}\blh^T\mathcal{I}^{{\bgpi} }(\bgtheta)\blh+o(1).
\]
We define 
\[
A_{ni}=nW_{ni}^2-\Big(\blh^T \nabla_{\bgtheta}\log f_{\bgtheta,a_i}(X_i^{a_i})\Big)^2
\]
and
\[
A'_{ni}=\sum_{a\in \mathcal{A}}\Big|4n\Big(\sqrt{ \frac{f_{\bgtheta +\blh_n/\sqrt{n},a  } (X_i^{a })}{f_{\bgtheta ,a  } (X_i^{a })}  } - 1 \Big)^2-\Big(\blh^T \nabla_{\bgtheta}\log f_{\bgtheta,a}(X_i^{a})\Big)^2\Big|.
\]
Notice that
\begin{equation*}
\begin{split}
    A'_{ni}=& \sum_{a\in \mathcal{A}}\Big|4n\Big(\sqrt{ \frac{f_{\bgtheta +\blh_n/\sqrt{n},a  } (X_i^{a })}{f_{\bgtheta ,a  } (X_i^{a })}  } - 1 \Big)^2-\Big(\blh^T \nabla_{\bgtheta}\log f_{\bgtheta,a}(X_i^{a})\Big)^2\Big|\\
    \leq  &\sum_{a\in \mathcal{A}}\Big|2\sqrt{n}\Big(\sqrt{ \frac{f_{\bgtheta +\blh_n/\sqrt{n},a  } (X_i^{a })}{f_{\bgtheta ,a  } (X_i^{a })}  } - 1 \Big)-\Big(\blh^T \nabla_{\bgtheta}\log f_{\bgtheta,a}(X_i^{a})\Big) \Big|\cdot \\
    &\Big|2\sqrt{n}\Big(\sqrt{ \frac{f_{\bgtheta +\blh_n/\sqrt{n},a  } (X_i^{a })}{f_{\bgtheta ,a  } (X_i^{a })}  } - 1 \Big)+ \Big(\blh^T \nabla_{\bgtheta}\log f_{\bgtheta,a}(X_i^{a})\Big) \Big|
\end{split}
\end{equation*}
By Hölder's inequality and the definition of differentiable in quadratic mean at $\bgtheta$, we obtain that
\begin{equation*}
\begin{split}
    \mathbb{E}|A'_{ni}|
    \leq  &\sum_{a\in \mathcal{A}}\Big(\mathbb{E}\Big|2\sqrt{n}\Big(\sqrt{ \frac{f_{\bgtheta +\blh_n/\sqrt{n},a  } (X_i^{a })}{f_{\bgtheta ,a  } (X_i^{a })}  } - 1 \Big)-\Big(\blh^T \nabla_{\bgtheta}\log f_{\bgtheta,a}(X_i^{a})\Big) \Big|^2\Big)^{1/2}\cdot \\
    &\Big(\mathbb{E}\Big|2\sqrt{n}\Big(\sqrt{ \frac{f_{\bgtheta +\blh_n/\sqrt{n},a  } (X_i^{a })}{f_{\bgtheta ,a  } (X_i^{a })}  } - 1 \Big)+ \Big(\blh^T \nabla_{\bgtheta}\log f_{\bgtheta,a_i}(X_i^{a_i})\Big) \Big|^2\Big)^{1/2}\\
    \leq&\sum_{a\in \mathcal{A}}\Big(\mathbb{E}\Big|2\sqrt{n}\Big(\sqrt{ \frac{f_{\bgtheta +\blh_n/\sqrt{n},a  } (X_i^{a })}{f_{\bgtheta ,a  } (X_i^{a })}  } - 1 \Big)-\Big(\blh^T \nabla_{\bgtheta}\log f_{\bgtheta,a}(X_i^{a})\Big) \Big|^2\Big)^{1/2}\cdot \\
    &\Big[\Big(\mathbb{E}\Big|2\sqrt{n}\Big(\sqrt{ \frac{f_{\bgtheta +\blh_n/\sqrt{n},a  } (X_i^{a })}{f_{\bgtheta ,a  } (X_i^{a })}  } - 1 \Big) -\blh^T \nabla_{\bgtheta}\log f_{\bgtheta,a}(X_i^{a})  \Big|^2\Big)^{1/2}\\
    &+2\Big(\mathbb{E}\Big|  \Big(\blh^T \nabla_{\bgtheta}\log f_{\bgtheta,a}(X_i^{a})\Big) \Big|^2\Big)^{1/2}\Big].
\end{split}
\end{equation*}
Due to
\begin{equation*}
\begin{split}
&\Big(\mathbb{E}\Big|2\sqrt{n}\Big(\sqrt{ \frac{f_{\bgtheta +\blh_n/\sqrt{n},a  } (X_i^{a })}{f_{\bgtheta ,a  } (X_i^{a })}  } - 1 \Big)-\Big(\blh^T \nabla_{\bgtheta}\log f_{\bgtheta,a}(X_i^{a})\Big) \Big|^2\Big)^{1/2}\\
\leq& \Big(\mathbb{E}\Big|2\sqrt{n}\Big(\sqrt{ \frac{f_{\bgtheta +\blh_n/\sqrt{n},a  } (X_i^{a })}{f_{\bgtheta ,a  } (X_i^{a })}  } - 1 \Big)-\Big(\blh_n^T \nabla_{\bgtheta}\log f_{\bgtheta,a}(X_i^{a})\Big) \Big|^2\Big)^{1/2} \\
&+\Big(\mathbb{E}\Big| \Big(\blh_n-\blh\Big)^T \nabla_{\bgtheta}\log f_{\bgtheta,a}(X_i^{a})  \Big|^2\Big)^{1/2}\\
=&o(\norm{\blh_n} )+ \Big( (\blh-\blh_n)^T \mathcal{I}_a(\bgtheta) (\blh-\blh_n) \Big)^{1/2}=o(1),
\end{split}
\end{equation*}
we obtain that
\begin{equation*}
    \mathbb{E}|A'_{ni}|=\sum_{a\in \mathcal{A}}o(1)\Big( o(1)+2\big( \blh^T \mathcal{I}_a(\bgtheta) \blh  \big)^{1/2} \Big)=o(1).
\end{equation*}
Because $|A_{ni}|\leq A'_{ni}$, we know that $\mathbb{E}|A_{ni}| \to 0$ and $\mathbb{E} \frac{1}{n}\sum_{i=1}^n|A_{ni}|\to 0$ as $n\to \infty$. By Lemma~\ref{lem:SLLN} and  \eqref{equ:lim_T_n}, we know that
\[
\sum_{i=1}^nW_{ni}^2=\frac{1}{n}\sum_{i=1}^n \Big(\blh^T \nabla_{\bgtheta}\log f_{\bgtheta,a_i}(X_i^{a_i})\Big)^2 + \frac{1}{n}\sum_{i=1}^n A_{ni}  \stackrel{\mathbb{P}_{\bgtheta}}{\rightarrow}  \blh^T \mathcal{I}^{\bgpi}(\bgtheta) \blh.
\]
By triangle inequality and Markov's inequality, as $n\to \infty$,
\begin{equation*}
\begin{split}
    &\mathbb{P}(\max_{1\leq i\leq n} |W_{ni}|>\varepsilon \sqrt{2} )\leq n\mathbb{P}(|W_{ni}|>\varepsilon \sqrt{2}) \\
    \leq& n \mathbb{P}\Big( \big(\blh^T \nabla_{\bgtheta}\log f_{\bgtheta,a_i}(X_i^{a_i})\big)^2 >n\varepsilon^2\Big)+n \mathbb{P}\Big( |A_{ni}| >n\varepsilon^2  \Big)\\
    \leq& n \mathbb{P}\Big(\sum_{a\in \mathcal{A}} \big(\blh^T \nabla_{\bgtheta}\log f_{\bgtheta,a}(X_i^{a})\big)^2 >n\varepsilon^2\Big)+n \mathbb{P}\Big( |A'_{ni}| >n\varepsilon^2  \Big)\\
    \leq&\frac{1}{\varepsilon^2} \mathbb{E}\sum_{a\in \mathcal{A}} \big(\blh^T \nabla_{\bgtheta}\log f_{\bgtheta,a}(X_i^{a})\big)^2 I\Big(\sum_{a\in \mathcal{A}} \big(\blh^T \nabla_{\bgtheta}\log f_{\bgtheta,a}(X_i^{a})\big)^2>n\varepsilon^2\Big)+\frac{\mathbb{E}A'_{ni}}{\varepsilon^2}\\
    \to& 0.
\end{split}
\end{equation*}
Based on the rest of the proof of Theorem 7.2 in \cite{van2000asymptotic}, we complete the proof of modified Theorem 7.2.

}

\paragraph*{Modified Theorem~7.10 in~\cite{van2000asymptotic}}. The modified theorem is as follows: if statistics $\blT_n=\blT_n(a_1,X_1^{a_1},\cdots,a_n,X_n^{a_n})$ satisfies the limit results in \eqref{equ:lim_T_n} under every $\blh$, then there exists a randomized statistic $\blT$ in the experiment $\{N_p(\blh,\{\mathcal{I}^{\bgpi}(\bgtheta)\}^{-1}  ):\blh \in \mathbb{R}^p \}$ such that $\blT_n \stackrel{\blh}{\rightsquigarrow} \blT$ for every $\blh$.
 
The proof mostly follows that of Theorem~7.10 in~\cite{van2000asymptotic} with the following modifications. Without loss of generality, let
\[
P_{n,\blh}={P_{n,\bgtheta+\blh/\sqrt{n}}}(a_1, X^{a_1}_1,\cdots,a_n,X^{a_n}_n), \blJ=\mathcal{I}^{\bgpi}(\bgtheta), \Delta_n= \frac{1}{\sqrt{n}} \sum_{j=1}^n \nabla_{\bgtheta} \log f_{\bgtheta,a_j}(X_j^{a_j}).
\]
There exists random vector $(\blS,\Delta)$ such that 
\begin{equation*}
\left(\blT_n, \Delta_n\right) \stackrel{ {\mathbf 0}}{\rightsquigarrow}(\blS, \Delta).
\end{equation*}

Applying the modified Theorem 7.2 and follow similar arguments as those in the proof of Theorem 7.10 in \cite{van2000asymptotic}, we obtain 
\begin{equation*}
\left( \blT_n, \log \frac{d P_{n, \blh}}{d P_{n, {\mathbf 0}}}\right) \stackrel{ {\mathbf 0}}{\rightsquigarrow}\left( \blS, \blh^T \Delta-\frac{1}{2} \blh^T \blJ \blh\right).
\end{equation*}
The rest of the proof remains unchanged. %

\paragraph*{Modified Theorem 8.3 in~\cite{van2000asymptotic}}
With a similar proof,  the conclusion in Theorem 8.3 in~\cite{van2000asymptotic} is modified as follows. If the limit results in \eqref{equ:lim_T_n} hold, then {there exists a randomized statistic $\blT$ in 
$\{N_p(\blh, \{\mathcal{I}^{\bgpi}(\bgtheta)\}^{-1}):\blh\in \mathbb{R}^p \}$ such that $\blT-\blh$ has the distribution $L^{\bgpi}_{\bgtheta}$ for every $\blh$.
}

\paragraph*{Proposition 8.4 in~\cite{van2000asymptotic}} 
{This proposition directly apply to our setting and } does not required to be changed.

With the above modifications, we follow a similar proof as that for Theorem~8.8 in \cite{van2000asymptotic}, we obtain \eqref{equ:Convolution_Theorem} as well as the first part of the theorem.

\paragraph*{Proposition~8.6 in \cite{van2000asymptotic}}
{This proposition directly applies to our problem and does not need to be modified.}

Following the proof of Theorem 8.11~\cite{van2000asymptotic} with the above modifications, we complete the proof of \eqref{conj:local_minimax} and the second part of the theorem.
\end{proof}
\subsection{Proof of Theorem~\ref{thm:a.s. stopping time}}
\begin{proof}[Proof of Theorem~\ref{thm:a.s. stopping time}]
    By Theorem~\ref{thm:consistency_final}, we know that 
\[
\lim_{n\to \infty}\widehat{\bgtheta}_{n}^{\text{ML}} = \bgtheta^*, a.s.
\]
Because $\lim_{n\to \infty} \tau_n = \infty$ a.s., we obtain that
\[
\lim_{n\to \infty}\widehat{\bgtheta}_{\tau_n}^{\text{ML}} = \bgtheta^*, a.s.
\]   
\end{proof}

\subsection{Proof of Theorem~\ref{thm:Asy_Normal_final_stopping_time_sp}}
{
We prove the theorem for a class more general stopping rules instead. 
}
We first define a deterministic stopping rule
\begin{equation*}
    \tau(\Gamma_{\bgtheta^*} ,c, \bgtheta, {\bgpi})=\min \Big\{m\geq n_0; \frac{1}{m} \Gamma_{\bgtheta^*} ( \{\mathcal{I}^{{\bgpi}}(\bgtheta)\}^{-1}  )\leq c \Big\},
\end{equation*}
where $\Gamma_{\bgtheta}$ is a continuous function that maps a positive definite matrix to a positive number, $c$ is a positive number, $\bgtheta\in \bgTheta$, and ${\bgpi} \in \ShatA$. Note that $\tau(\Gamma,c, \bgtheta, {\bgpi})=\max\{\lceil \frac{\Gamma( \{\mathcal{I}^{{\bgpi}}(\bgtheta)\}^{-1}  )}{c}  \rceil , n_0\}$, where $\lceil \cdot  \rceil$ is the ceiling function.
 
Consider a class of functions $\{ \Gamma_{\bgtheta}\}_{\bgtheta\in \bgTheta}$, such that for any $0<u_1<u_2<\infty$,
\begin{equation}\label{lim:Gamma_theta}
    \lim_{\bgtheta \to \bgtheta^*} \max_{{u_1 I\preceq \bgSigma \preceq u_2 I  }}  \left|\Gamma_{\bgtheta}(\bgSigma)-\Gamma_{\bgtheta^*}(\bgSigma)\right| =0, \text{ and } \min_{\bgtheta \in \bgTheta}\min_{{u_1 I\preceq \bgSigma \preceq u_2 I  }} \Gamma_{\bgtheta}(\bgSigma)>0.
\end{equation}
Define a random stopping time 
\begin{equation}\label{equ:tau_c}
\tau_c=\min \Big\{m\geq n_0; \frac{1}{m} \Gamma_{\widehat{\bgtheta}_m} ( \{\mathcal{I}^{\overline{{\bgpi}}_m }(\widehat{\bgtheta}_m )\}^{-1}  )\leq c \Big\},
\end{equation}
where $\widehat{\bgtheta}_m $ is an estimator of $\bgtheta$ based on $m$ observations. Later, we will show that the stopping rules considered in Theorem~\ref{thm:Asy_Normal_final_stopping_time_sp} are special cases of the general stopping rule defined above.

The following theorem generalizes Theorem~\ref{thm:Asy_Normal_final_stopping_time_sp}.
\begin{theorem}[General result for Asymptotic normality with stopping time] \label{thm:Asy_Normal_final_stopping_time}
Let $\widehat{\bgtheta}_{n}^{\text{ML}}$ be the MLE following the experiment selection rule \textrm{GI0} or \textrm{GI1}, as described in Algorithm~\ref{alg:gi0} and Algorithm~\ref{alg:gi1}. Assume that $\mathbb{F}_{\bgtheta^*} ({\bgpi})$ has a unique minimizer ${\bgpi}^*$. Let $\{c_n\}_{n\geq 0}$ be a positive decreasing sequence such that $c_n\to 0$ as $n\to \infty$. Consider the stopping time $\tau_{c_n}$ given by \eqref{equ:tau_c}, where $\Gamma_{\bgtheta}$ satisfies \eqref{lim:Gamma_theta}. Then,
\begin{equation}\label{lim:stopping_time_lim}
    \sqrt{\tau_n} \Big\{  \mathcal{I}^{ \overline{{\bgpi}}_{\tau_{c_n }} }(\widehat{\bgtheta}_{\tau_{c_n }}^{\text{ML}})\Big\}^{ 1/2}(\widehat{\bgtheta}_{\tau_{c_n } }^{\text{ML}} -\bgtheta^*) \inD N_p\Big(\bm{0}_p,I_p\Big).
\end{equation}
Furthermore, for any continuously differentiable function $g: \bgTheta\to \mathbb{R}$ such that $\nabla g(\bgtheta^*)\neq 0$,
\begin{equation}\label{lim:stopping_time_lim_g}
     \frac{\sqrt{\tau_{c_n }}(g(\widehat{\bgtheta}_{\tau_{c_n }}^{\text{ML}}) -g(\bgtheta^*))}{\norm{\Big\{  \mathcal{I}^{ \overline{{\bgpi}}_{\tau_{c_n }} }(\widehat{\bgtheta}_{\tau_{c_n }}^{\text{ML}})\Big\}^{ -1/2}\nabla g( \widehat{\bgtheta}_{\tau_{c_n }}^{\text{ML}} )}}\inD N\Big(0,1\Big).
\end{equation}
\end{theorem}
{Given the above generalized theorem, the proof of \ref{thm:Asy_Normal_final_stopping_time_sp} is provided below. The proof of Theorem~\ref{thm:Asy_Normal_final_stopping_time}} is provided later in this section.}
\begin{proof}[Proof of Theorem~\ref{thm:Asy_Normal_final_stopping_time_sp}]
{
Note that $\tau^{(1)}_c$ and $\tau^{(2)}_c$ can be rewritten as
\begin{equation*}
\begin{split}
    &\tau^{(1)}_c=\min\{m\geq n_0;\ \frac{1}{m}\Gamma^{(1)}_{\bgtheta}\big( \{ \mathcal{I}( \widehat{\bgtheta}_{\tau_n}^{\text{ML}};\va_m ) \}^{-1} \big) \leq c^2 \}\\
    &\tau^{(2)}_c=\min\{m\geq n_0;\ \frac{1}{m}\Gamma^{(1)}_{\bgtheta}\big( \{ \mathcal{I}( \widehat{\bgtheta}_{\tau_n}^{\text{ML}};\va_m ) \}^{-1} \big) \leq c \},  
\end{split}
\end{equation*}
where \begin{equation*}
    \Gamma^{(1)}_{\bgtheta}(\bgSigma)=\tr \Big( \{\nabla h(\bgtheta)\}^T  \bgSigma  h(\bgtheta) \Big),\Gamma^{(2)}_{\bgtheta}(\bgSigma)=\tr \big( \bgSigma   \big),
\end{equation*}
Both $\Gamma^{(l)}_{\bgtheta}(\bgSigma)$ ($l=1,2$) are continuously differentiable in $\bgtheta$ and $\bgSigma$ so the first part of \eqref{lim:Gamma_theta} is satisfied. The second part of \eqref{lim:Gamma_theta} is satisfied for $\Gamma^{(2)}$ is straightforward. For $\Gamma^{(1)}_{\bgtheta}(\bgSigma)$, the second part of \eqref{lim:Gamma_theta} is satisfied due to the assumption that $\nabla h(\bgtheta)\neq \mathbf{0}$ for all $\bgtheta$. Thus, conditions of Theorem~\ref{thm:Asy_Normal_final_stopping_time} are satisfied, and the proof is completed by applying this theorem.
}    
\end{proof}
{In the rest of the section, we present the proof of Theorem~\ref{thm:Asy_Normal_final_stopping_time}. Roughly, Theorem~\ref{thm:Asy_Normal_final_stopping_time} is proved by combining the following lemma, compares the random and deterministic stopping times,  with the multivariate Anscombe’s theorem (Lemma~\ref{thm:mult_random_clt}).}

\begin{lemma}\label{lem:tau_n_prop}
Assume a family of function $\Gamma_{\bgtheta}$ satisfies \eqref{lim:Gamma_theta}. 
Assume there exists  constants $0<u_1<u_2<\infty$ such that for any $n\geq n_0$ and $\bgtheta \in \bgTheta$,
\begin{equation}\label{ineq:u1u2}
      u_1 I\preceq \mathcal{I}^{\overline{{\bgpi}}_n }( \bgtheta) \preceq u_2 I.
\end{equation}
If as $n\to \infty$, 
\begin{equation}\label{lim:3lim}
\begin{split}
    &c_n\to 0, c_n\geq c_{n+1}>0, \\
    &\widehat{\bgtheta}_n\to \bgtheta^*\ a.s. \ \mathbb{P}_*,\ \text{and}\\
    &\overline{{\bgpi}}_n \to {\bgpi}\ a.s. \ \mathbb{P}_*,
\end{split}
\end{equation}
where $\mathcal{I}^{{\bgpi}}(\bgtheta^*)$ is nonsingular. Then, $\tau_{c_n}<\infty$ a.s. $\mathbb{P}_*$ and as $n\to \infty$,
\begin{equation}\label{lim:tau_n}
     \tau(\Gamma_{\bgtheta^*},c_n, \bgtheta^*, {\bgpi})\to \infty,\ \tau_{c_n}\to \infty,\ \text{ and }\frac{\tau_{c_n}}{\tau(\Gamma_{\bgtheta^*},c_n, \bgtheta^*, {\bgpi})}\to 1, \ a.s. \ \mathbb{P}_*.
\end{equation}
\end{lemma}

\begin{proof}[Proof of Lemma \ref{lem:tau_n_prop}]
By Theorem~\eqref{thm:selection bound} and equation \eqref{ineq:Ipi_bound}, we know that there exists $0<\underline{c}<\overline{c}<\infty$ such that 
\[
\underline{c} I_p \preceq \mathcal{I}^{\overline{{\bgpi}}_m}(\widehat\bgtheta_n)\preceq \overline{c} I_p,
\]
{for all $m\geq n_0$.}
By assumption \eqref{lim:Gamma_theta}, there exists $0<v_1<v_2<\infty$ such that for any $n\geq n_0$,
\begin{equation*}
    v_1\leq \Gamma_{\widehat{\bgtheta}_m} ( \{\mathcal{I}^{\overline{{\bgpi}}_m }(\widehat{\bgtheta}_m )\}^{-1}  )\leq v_2.
\end{equation*}
Note that $\tau(\Gamma_{\bgtheta^*},c_n, \bgtheta^*, {\bgpi})=\max\{\lceil \frac{\Gamma_{\bgtheta^*}( \{\mathcal{I}^{{\bgpi}}(\bgtheta^*)\}^{-1}  )}{c_n}  \rceil ,n_0\}\to \infty$, as $n\to \infty$.
Also note that
\[
\Big\{m\geq n_0; \frac{v_2}{m} \leq c_n \Big\}\subset \Big\{m\geq n_0; \frac{1}{m} \Gamma_{\widehat{\bgtheta}_m} ( \{\mathcal{I}^{\overline{{\bgpi}}_m }(\widehat{\bgtheta}_m )\}^{-1}  )\leq c_n \Big\} \subset \Big\{m\geq n_0; \frac{v_1}{m} \leq c_n \Big\} . 
\]
Thus, for any fixed $n$,
\[
\tau_{c_n}\leq \min \Big\{m\geq n_0; \frac{v_2}{m} \leq c_n \Big\} \leq  \Big\lceil  \frac{v_2}{c_n}  \Big\rceil +n_0 <\infty, a.s.\ \mathbb{P}_*,
\]
and as $n\to \infty$,
\begin{equation*}
    \tau_{c_n}\geq \min \Big\{m\geq n_0; \frac{v_1}{m} \leq c_n \Big\} \geq  \Big\lceil  \frac{v_1}{c_n}  \Big\rceil \to \infty.
\end{equation*}
By assumption \eqref{lim:3lim}, we know that
\begin{equation*}
    \lim_{n\to \infty}\{\mathcal{I}^{\overline{{\bgpi}}_n }(\widehat{\bgtheta}_n )\}^{-1}  = \{\mathcal{I}^{{\bgpi}}(\bgtheta^*)\}^{-1}\ a.s. \ \mathbb{P}_*.
\end{equation*}
Combining the compact convergence assumption \eqref{lim:Gamma_theta} and \eqref{ineq:u1u2}, we obtain that
\begin{equation*}
    \lim_{n\to \infty} \Gamma_{\widehat{\bgtheta}_n}\big( \{\mathcal{I}^{\overline{{\bgpi}}_n }(\widehat{\bgtheta}_n )\}^{-1} \big) =\Gamma_{ {\bgtheta}^* } \big( \{\mathcal{I}^{{\bgpi}}(\bgtheta^*)\}^{-1}\big) \ a.s. \ \mathbb{P}_*.
\end{equation*}
That is, with probability 1, for any $\eta >0$, there exits $N_\eta\geq n_0$ such that for any $n\geq N_\eta$, 
\[
\big|\Gamma_{\widehat{\bgtheta}_n}\big( \{\mathcal{I}^{\overline{{\bgpi}}_n }(\widehat{\bgtheta}_n )\}^{-1} \big)-\Gamma_{ {\bgtheta}^* } \big( \{\mathcal{I}^{{\bgpi}}(\bgtheta^*)\}^{-1}\big) \big|<\eta.
\]
Set $N_2=\min\Big\{n\geq N_\eta ; \lceil  \frac{v_1}{c_n}   \rceil \geq N_\eta\Big\}<\infty$.
For any $n\geq N_2$, we obtain that,
\begin{equation*}
\begin{split}
    &\Big\{m\geq n_0; \frac{1}{m} \left\{\Gamma_{ {\bgtheta}^* } \big( \{\mathcal{I}^{{\bgpi}}(\bgtheta^*)\}^{-1}\big) + \eta \right\}\leq c_n \Big\} \\
    \subset &\Big\{m\geq n_0; \frac{1}{m} \Gamma_{\widehat{\bgtheta}_m} ( \{\mathcal{I}^{\overline{{\bgpi}}_m }(\widehat{\bgtheta}_m )\}^{-1}  )\leq c_n \Big\} \\
    \subset &\Big\{m\geq n_0; \frac{1}{m} \left\{\Gamma_{ {\bgtheta}^* } \big( \{\mathcal{I}^{{\bgpi}}(\bgtheta^*)\}^{-1}\big) - \eta \right\}
 \leq c_n \Big\},
\end{split}
\end{equation*}
which implies that
\begin{equation*}
    \Big\lceil \frac{\Gamma_{\bgtheta^*}( \{\mathcal{I}^{{\bgpi}}(\bgtheta^*)\}^{-1}  ) -\eta}{c_n}  \Big\rceil \leq \tau_{c_n} \leq \Big\lceil \frac{\Gamma_{\bgtheta^*}( \{\mathcal{I}^{{\bgpi}}(\bgtheta^*)\}^{-1}  ) +\eta}{c_n}  \Big\rceil.
\end{equation*}
Set $\eta=\xi \cdot \Gamma_{\bgtheta^*}( \{\mathcal{I}^{{\bgpi}}(\bgtheta^*)\}^{-1}  )$ and $ d_n=\Gamma_{\bgtheta^*}( \{\mathcal{I}^{{\bgpi}}(\bgtheta^*)\}^{-1}  )/c_n\to \infty$. Note that
\begin{equation}\label{ineq:tau_xi}
    \frac{ \lceil (1-\xi)d_n \rceil }{\lceil  d_n \rceil}\leq \frac{\tau_{c_n}}{\tau(\Gamma_{\bgtheta^*},c_n, \bgtheta^*, {\bgpi})} \leq \frac{ \lceil (1+\xi)d_n \rceil }{\lceil  d_n \rceil}.
\end{equation}
Taking the infimum limit and supremum limit over both sides of inequalities \eqref{ineq:tau_xi}, we obtain that for any $\xi>0$,
\begin{equation*}
    (1-\xi)\leq \liminf_{n\to \infty}\frac{\tau_{c_n}}{\tau(\Gamma_{\bgtheta^*},c_n, \bgtheta^*, {\bgpi})} \leq \limsup_{n\to \infty}\frac{\tau_{c_n}}{\tau(\Gamma_{\bgtheta^*},c_n, \bgtheta^*, {\bgpi})}  \leq (1+\xi).
\end{equation*}
In conclusion, 
\[
\lim_{n\to \infty}\frac{\tau_{c_n}}{\tau(\Gamma_{\bgtheta^*},c_n, \bgtheta^*, {\bgpi})} = 1, \ a.s. \ \mathbb{P}_*.
\]
\end{proof}

\begin{proof}[Proof of Theorem \ref{thm:Asy_Normal_final_stopping_time}]

Let $v_n=\tau(\Gamma_{\bgtheta^*},c_n,  {\bgtheta}^* ,  {\bgpi}^*  )$. Accordin to Theorem \ref{thm:consistency_final} and Theorem \ref{thm:empirical pi as converge}, the conditions in \eqref{lim:3lim} are satisfied. By Lemma~\ref{lem:tau_n_prop}, we obtain \eqref{lim:tau_n}, which implies that
\begin{equation}\label{lim:I_pa.s.}
    \sqrt{\tau_n} \{ \mathcal{I}^{\overline{{\bgpi}}_n}(\widehat{\bgtheta}_n^{\text{ML}}) \}^{1/2} \frac{1}{\sqrt{v_n}} \{ \mathcal{I}^{ {\bgpi}^* }( \bgtheta^* ) \}^{-1/2} \to I_p,\ a.s. \ \mathbb{P}_*.
\end{equation}
{For the ease of exposition, we write $\tau_n=\tau_{c_n}$.
}
To show the limit result \eqref{lim:stopping_time_lim}, by \eqref{lim:I_pa.s.} and Slutsky's theorem, it suffices to show that
\begin{equation}\label{lim:211_need}
    \sqrt{v_n} \Big\{  \mathcal{I}^{ {\bgpi}^* }(\bgtheta^*)\Big\}^{ 1/2}(\widehat{\bgtheta}_{\tau_n}^{\text{ML}} -\bgtheta^*) \inD N_p\Big(\bm{0}_p,I_p\Big).
\end{equation}
{According to Theorem~\ref{thm:mult_random_clt} with $T_n=\widehat{\bgtheta}_{n}^{\text{ML}}$, $\theta=\bgtheta^*$, $N_n=\tau_n$, $r_n=v_n$, and $W_n = n^{-1/2}\big\{  \mathcal{I}^{ {\bgpi}^* }(\bgtheta^*)\big\}^{ -1/2}$,  \eqref{lim:211_need} we only need to verify the following conditions for  Theorem~\ref{thm:mult_random_clt}: for all $\gamma >0$, $\varepsilon>0$, there exists $0<\delta<1$ such that
\begin{equation*}
    \limsup_{n\to \infty}\mathbb{P}\Big( \max_{|n'-n|\leq \delta n} \norm{\widehat{\bgtheta}_{n'}^{\text{ML}} -\widehat{\bgtheta}_n^{\text{ML}} } \geq  \frac{\varepsilon}{\sqrt{n}} \lambda_{min}(\{ \mathcal{I}^{{\bgpi}^*}(\bgtheta^*) \}^{-1/2})   \Big)\leq \gamma.
\end{equation*}
}

To show the above inequality, it suffices to show that for any $\gamma >0$, $\varepsilon>0$, there exists $0<\delta<1$ such that
\begin{equation*}
    \limsup_{n\to \infty }\mathbb{P}\Big( \max_{n',n'',n''' \in [n, (1+\delta) n] } \sqrt{n'''}\norm{\widehat{\bgtheta}_{n'}^{\text{ML}} -\widehat{\bgtheta}_{n''}^{\text{ML}} } \geq   {\varepsilon}  \lambda_{min}(\{ \mathcal{I}^{{\bgpi}^*}(\bgtheta^*) \}^{-1/2})   \Big)\leq \gamma.
\end{equation*}
Note that
\begin{equation*}
\begin{split}
    &\mathbb{P}\Big( \max_{n',n'',n''' \in [n, (1+\delta) n] } \sqrt{n'''}\norm{\widehat{\bgtheta}_{n'}^{\text{ML}} -\widehat{\bgtheta}_{n''}^{\text{ML}} } \geq   {\varepsilon}  \lambda_{min}(\{ \mathcal{I}^{{\bgpi}^*}(\bgtheta^*) \}^{-1/2})   \Big)\\
    \leq &\mathbb{P}\Big( \max_{n',n''  \in [n, (1+\delta) n] } \sqrt{n}\Big(\norm{\widehat{\bgtheta}_{n'}^{\text{ML}} -\widehat{\bgtheta}_{n}^{\text{ML}} } +\norm{\widehat{\bgtheta}_{n}^{\text{ML}} -\widehat{\bgtheta}_{n''}^{\text{ML}} }  \Big) \geq   \frac{\varepsilon}{\sqrt{1+\delta}}  \lambda_{min}(\{ \mathcal{I}^{{\bgpi}^*}(\bgtheta^*) \}^{-1/2})   \Big)\\
    \leq &2\cdot \mathbb{P}\Big( \max_{n'  \in [n, (1+\delta) n] } \sqrt{n}\norm{\widehat{\bgtheta}_{n'}^{\text{ML}} -\widehat{\bgtheta}_{n}^{\text{ML}} } \geq   \frac{\varepsilon}{2\sqrt{2}}   \lambda_{min}(\{ \mathcal{I}^{{\bgpi}^*}(\bgtheta^*) \}^{-1/2})   \Big).
\end{split}
\end{equation*}
Thus, we only need to show that for all  $\varepsilon>0$,
\begin{equation}\label{lim:215_need}
\lim_{\delta\to 0}\limsup_{n\to \infty}\mathbb{P}\Big( \max_{n'  \in [n, (1+\delta ) n] } \sqrt{n}\norm{\widehat{\bgtheta}_{n'}^{\text{ML}} -\widehat{\bgtheta}_{n}^{\text{ML}} } \geq {\varepsilon} \lambda_{min}(\{ \mathcal{I}^{{\bgpi}^*}(\bgtheta^*) \}^{-1/2})   \Big)=0.
\end{equation}
Let
\begin{equation*}
\begin{split}
    {D}_n=&\Big\{ \nabla_{\bgtheta}l_{n}(\widehat{\bgtheta}_{n}^{\text{ML}}; \va_{n} )=0  \Big\}\\
    &\bigcap\Big\{\frac{1}{n}\sum_{j=1}^n\Psi_2^{a_j}(X_j)
    \norm{\{\nabla^2_{\bgtheta}l_n(\bgtheta^*;\va_n)\}^{-1}}_{op}
    \sup_{n'\geq  n}\psi\Big( \norm{ \widehat{\bgtheta}_n^{\text{ML}}-\widehat{\bgtheta}_{n'}^{\text{ML}} } \Big)\leq \frac{1}{2}  \Big\}\\
    &\bigcap\Big\{\frac{1}{n}\sum_{j=1}^{n}\Psi_2^{a_j}(X_j) \norm{\{\nabla^2_{\bgtheta}l_n(\bgtheta^*;\va_n)\}^{-1}}_{op}\psi\Big( \norm{ \widehat{\bgtheta}_{n}^{\text{ML}}-  \bgtheta^* } \Big)\leq \frac{1}{2}  \Big\}\\
    &\bigcap\Big\{ \norm{ \{\nabla^2_{\bgtheta} l_n(\widehat{\bgtheta}_n^{\text{ML}};\va_n) \}^{-1} }_{op}  \leq \frac{2}{ \lambda_{min}(\mathcal{I}^{{\bgpi}^*}(\bgtheta^*) ) } \Big\}\\
    &\bigcap\Big\{ \norm{ \{\nabla^2_{\bgtheta} l_n(\bgtheta^*;\va_n) \}^{-1} }_{op}  \leq \frac{2}{ \lambda_{min}(\mathcal{I}^{{\bgpi}^*}(\bgtheta^*) ) } \Big\}.
\end{split}
\end{equation*}
By Theorem \ref{thm:consistency_final}, Theorem \ref{thm:selection bound}, Corollary \ref{lemma:MLE score}, Lemma \ref{lem:a.s.bound}, and with probability 1,
\begin{equation*}
    \limsup_{n\to \infty } \norm{ \{\nabla^2_{\bgtheta} l_n(\widehat{\bgtheta}_n^{\text{ML}};\va_n) \}^{-1} }_{op}  \leq \frac{1}{ \lambda_{min}(\mathcal{I}^{{\bgpi}^*}(\bgtheta^*) ) },
\end{equation*}
we know that
\begin{equation*}
    \mathbb{P} \left( \bigcup_{m=1}^\infty\bigcap_{n=m}^\infty    {D}_n \right) = 1.
\end{equation*}
By Lemma \ref{lem:taylor ln}, we obtain that 
\begin{equation*}
    \bigcap_{m\geq n} {D}_m \subset \Big\{ \sup_{n':n\leq n'\leq  (1+\delta)n}\norm{\widehat{\bgtheta}_{n'}^{\text{ML}}-\widehat{\bgtheta}_n^{\text{ML}} } \leq  \frac{4}{ \lambda_{min}(\mathcal{I}^{{\bgpi}^*}(\bgtheta^*) ) }   \sup_{n':n\leq n'\leq  (1+\delta)n} \norm{\nabla_{\bgtheta} l_n( \widehat{\bgtheta}_{n'}^{\text{ML}}; \va_n)} \Big\},
\end{equation*}
and
\begin{equation*}
    {D}_n  \subset \Big\{  \sqrt{n} \norm{\widehat{\bgtheta}_{n}^{\text{ML}}-\bgtheta^* } \leq  \frac{4}{ \lambda_{min}(\mathcal{I}^{{\bgpi}^*}(\bgtheta^*) ) }   \sqrt{n}\norm{\nabla_{\bgtheta} l_n( \bgtheta^*; \va_n)} \Big\}.
\end{equation*}
Thus
\begin{equation}\label{ineq:220_need}
\begin{split}
    &\mathbb{P}\Big( \sup_{n':n\leq n'\leq  (1+\delta)n} \norm{\widehat{\bgtheta}_{n'}^{\text{ML}}-\widehat{\bgtheta}_n^{\text{ML}} } \geq \frac{\varepsilon}{\sqrt{n}} \lambda_{min}(\{ \mathcal{I}^{{\bgpi}^*}(\bgtheta^*) \}^{-1/2})   \Big)\\
    \leq& \mathbb{P} \Big( \Big\{\bigcap_{m\geq  n} {D}_m \Big\}^c\Big)\\
    &+ \mathbb{P}\Big( \Big\{ \sup_{n':n\leq n'\leq  (1+\delta)n} \sqrt{n}\norm{\nabla_{\bgtheta} l_n( \widehat{\bgtheta}_{n'}^{\text{ML}}; \va_n)} \geq  C(\varepsilon) \Big\} \bigcap \bigcap_{m\geq  n} {D}_m   \Big),
\end{split}
\end{equation}
where $C(\varepsilon)=\frac{\varepsilon \lambda_{min}(\{\mathcal{I}^{{\bgpi}^*}(\bgtheta^*) \}^{-1/2})}{4\{\lambda_{min}(\mathcal{I}^{{\bgpi}^*}(\bgtheta^*) )\}^{-1}}$. Let $n'\in [n+1,(1+\delta)n]$. With $\nabla_{\bgtheta} l_{n'}( \widehat{\bgtheta}_{n'}^{\text{ML}}; \va_{n'})=0$, we know that that over $\cap_{m\geq  n} {D}_m $
\begin{equation}\label{ineq:221_need}
\begin{split}
    &\sqrt{n}\norm{\nabla_{\bgtheta} l_n( \widehat{\bgtheta}_{n'}^{\text{ML}}; \va_n)} = \frac{1}{ \sqrt{n} } \norm{\nabla_{\bgtheta} \sum_{j=n+1}^{ n' } \log f_{\widehat{\bgtheta}_{n'}^{\text{ML}} ,a_j}(X_j)}\\
    \leq&\frac{1}{ \sqrt{n} } \norm{ \sum_{j=n+1}^{n' } \nabla_{\bgtheta}\log f_{ \bgtheta^* ,a_j}(X_j)}+\delta \frac{ \sum_{j=n+1}^{n' } \Psi_1^{a_j}(X_j) }{\delta n} \sqrt{n}\norm{  \widehat{\bgtheta}_{n'}^{\text{ML}}-\bgtheta^* }\\
    \leq& \frac{1}{ \sqrt{n} } \norm{ \sum_{j=n+1}^{n' }\nabla_{\bgtheta}\log f_{ \bgtheta^* ,a_j}(X_j)}\\
    &+\frac{4\delta }{ \lambda_{min}(\mathcal{I}^{{\bgpi}^*}(\bgtheta^*) ) } \frac{ \sum_{j=n+1}^{(1+\delta)n } \Psi_1^{a_j}(X_j) }{\delta n}  \frac{1}{\sqrt{n}}\norm{ \sum_{j=1}^{n' } \nabla_{\bgtheta}\log f_{ \bgtheta^* ,a_j}(X_j) }.
\end{split}
\end{equation}
By Markov inequality, we obtain that for any $M>0$
\begin{equation*}
 \mathbb{P}\Big(\frac{  \sum_{n+1\leq j\leq (1+\delta)n )  }\Psi_1^{a_j}(X_j) }{\delta n} \geq M \Big)   \leq\frac{\mu_Y}{M},
\end{equation*}
where 
$\mu_Y=\sum_{a\in \mathcal{A}} \mathbb{E}_{X^a\sim f_{\bgtheta^*,a}}\Psi_1^{a}(X^a)<\infty$. Thus,
\begin{equation*}
    \frac{ \sup_{n': n+1\leq n' \leq (1+\delta) n}\sum_{j =n+1}^{n'}\Psi_1^{a_j}(X_j) }{\delta n}=\frac{  \sum_{j =n+1}^{(1+\delta)n}\Psi_1^{a_j}(X_j) }{\delta n}= O_p(1).
\end{equation*}
Note that $S^n_m=\sum_{j=n+1}^{n+m }\nabla_{\bgtheta}\log f_{ \bgtheta^* ,a_j}(X_j)$ is martingale sequence with respect to filtration $\mathcal{F}^n_m=\mathcal{F}_{n+m}$.

By Assumption~\ref{ass:2}, 
$C_1:=\max_{a\in \mathcal{A} }\max_{\bgtheta\in \bgTheta} \mathbb{E}_{X\sim f_{\bgtheta^*, a}} \{ \norm{\nabla_{\bgtheta} \log f_{\bgtheta, a} (X)}^2 \}<\infty$. Since $\norm{\cdot}$ is convex, by Jensen's inequality, $\norm{S^n_m}$ is a submartingale. Applying Doob's inequality (see Theorem 6.5.d. in \cite{lin2010probability}), we obtain that
\begin{equation*}
    \mathbb{P}\Big( \max_{1\leq m\leq l }\norm{S^n_m}\geq M \Big)\leq  \frac{\mathbb{E}\norm{S^n_l}^2}{M^2}\leq \frac{l\cdot C_1}{M^2}.
\end{equation*}
Hence, we obtain
\begin{equation}\label{ineq:225_need}
    \mathbb{P}\Big(\max_{n+1\leq n'\leq (1+\delta)n} \frac{1}{\sqrt{n \delta }} \norm{ \sum_{j=n+1}^{n' }\nabla_{\bgtheta}\log f_{ \bgtheta^* ,a_j}(X_j)} \geq M \Big)\leq \frac{ C_1}{M^2},
\end{equation}
and 
\begin{equation}\label{ineq:226_need}
     \mathbb{P}\Big(\max_{n+ 1\leq  n'\leq (1+\delta)n} \frac{1}{\sqrt{n}}\norm{ \sum_{j=1}^{n' } \nabla_{\bgtheta}\log f_{ \bgtheta^* ,a_j}(X_j) }  \geq M \Big)\leq \frac{(1+\delta)n\cdot C_1 }{n\cdot M^2}\leq \frac{2C_1}{M^2}.
\end{equation}
Combining \eqref{ineq:221_need}, \eqref{ineq:225_need} and \eqref{ineq:226_need}, we obtain
\begin{equation*}
    \max_{n':n+1\leq n'\leq (1+\delta)n }\sqrt{n}\norm{\nabla_{\bgtheta} l_n( \widehat{\bgtheta}_{n'}^{\text{ML}}; \va_n)}=\sqrt{\delta}  O_p(1)+\delta O_{P}(1).
\end{equation*}
Thus, 
\begin{equation*}
\begin{split}
    &\limsup_{\delta\to 0}\limsup_{n\to \infty}\mathbb{P}\Big( \Big\{ \sup_{n':n\leq n'\leq  (1+\delta)n} \sqrt{n}\norm{\nabla_{\bgtheta} l_n( \widehat{\bgtheta}_{n'}^{\text{ML}}; \va_n)} \geq C(\varepsilon) \Big\} \bigcap \bigcap_{m\geq  n} {D}_m   \Big)\\
    \leq&\limsup_{\delta\to 0}  \mathbb{P}\Big( \sqrt{\delta} O_p(1)\geq C(\varepsilon) \Big) =0.
\end{split}
\end{equation*}
This completes the proof of \eqref{lim:stopping_time_lim}.

We proceed to the proof of the `Furthermore' part of the theorem. Note that
\begin{equation*}
g(\widehat{\bgtheta}_{\tau_n}^{\text{ML}})-g(\bgtheta^*)=\{\nabla g( \widetilde{\bgtheta}_n )\}^{T}(\widehat{\bgtheta}_{\tau_n}^{\text{ML}}- \bgtheta^*),
\end{equation*}
where $\widetilde{\bgtheta}_n \to \bgtheta^*$ and $\nabla g(\widetilde{\bgtheta}_n ) \to \nabla g(\bgtheta^*)$ a.s. $\mathbb{P}_*$ as $n\to \infty$.
Then,
\begin{equation*}
\begin{split}
    &\frac{\sqrt{\tau_n}(g(\widehat{\bgtheta}_{\tau_n}^{\text{ML}}) -g(\bgtheta^*))}{\norm{\Big\{  \mathcal{I}^{ \overline{{\bgpi}}_{\tau_n} }(\widehat{\bgtheta}_{\tau_n}^{\text{ML}})\Big\}^{ -1/2}\nabla g( \widehat{\bgtheta}_{\tau_n}^{\text{ML}} )}} =
    \frac{\sqrt{\tau_n}\{\nabla g( \widetilde{\bgtheta}_n )\}^{T}(\widehat{\bgtheta}_{\tau_n}^{\text{ML}}-\bgtheta^*)}{\norm{\Big\{  \mathcal{I}^{ \overline{{\bgpi}}_{\tau_n} }(\widehat{\bgtheta}_{\tau_n}^{\text{ML}})\Big\}^{ -1/2}\nabla g( \widehat{\bgtheta}_{\tau_n}^{\text{ML}} )}}\\
    =&\frac{
    \Big[\Big\{  \mathcal{I}^{ \overline{{\bgpi}}_{\tau_n} }(\widehat{\bgtheta}_{\tau_n}^{\text{ML}})\Big\}^{ -1/2}  \nabla g( \widetilde{\bgtheta}_n )\Big]^{T}
    }{\norm{\Big\{  \mathcal{I}^{ \overline{{\bgpi}}_{\tau_n} }(\widehat{\bgtheta}_{\tau_n}^{\text{ML}})\Big\}^{ -1/2}\nabla g( \widehat{\bgtheta}_{\tau_n}^{\text{ML}} )}} 
    \sqrt{\tau_n} \Big\{  \mathcal{I}^{ \overline{{\bgpi}}_{\tau_n} }(\widehat{\bgtheta}_{\tau_n}^{\text{ML}})\Big\}^{ 1/2}(\widehat{\bgtheta}_{\tau_n}^{\text{ML}} -\bgtheta^*).
\end{split}
\end{equation*}
By Theorem \ref{thm:selection bound} and Assumption~\ref{ass:6B}, we know that there exists $U>0$ such that for any $n\geq n_0$
\begin{equation*}
    \underline{c}U I_p\preceq \mathcal{I}^{\overline{\bgpi}_n} (\bgtheta) \preceq \overline{c} I_p,\ \forall \bgtheta \in \bgTheta.
\end{equation*}
Thus, the condition number of $\mathcal{I}^{\overline{\bgpi}_n} (\bgtheta)$, $ \kappa\Big(\mathcal{I}^{\overline{\bgpi}_n} (\bgtheta) \Big)\leq \frac{\overline{c}}{\underline{c}U}<\infty$, for any $\bgtheta\in \bgTheta$ and $n\geq n_0$. Moreover, we know that
\begin{equation*}
\begin{split}
    &\norm{\frac{\Big\{  \mathcal{I}^{ \overline{{\bgpi}}_{\tau_n} }(\widehat{\bgtheta}_{\tau_n}^{\text{ML}})\Big\}^{ -1/2}  (\nabla g( \widetilde{\bgtheta}_n )- \nabla g( \widehat{\bgtheta}_{\tau_n}^{\text{ML}} )) }{\norm{\Big\{  \mathcal{I}^{ \overline{{\bgpi}}_{\tau_n} }(\widehat{\bgtheta}_{\tau_n}^{\text{ML}})\Big\}^{ -1/2}\nabla g( \widehat{\bgtheta}_{\tau_n}^{\text{ML}} )}} }\\
    \leq &\kappa\Big(\Big\{\mathcal{I}^{ \overline{{\bgpi}}_{\tau_n} }(\widehat{\bgtheta}_{\tau_n}^{\text{ML}})\Big\}^{-1/2}\Big) \frac{\norm{\nabla g( \widetilde{\bgtheta}_n )- \nabla g( \widehat{\bgtheta}_{\tau_n}^{\text{ML}} )} }{\norm{ \nabla g( \widehat{\bgtheta}_{\tau_n}^{\text{ML}} ) }} =o_p(1).
\end{split}
\end{equation*}
Let $h_n=\frac{\Big\{  \mathcal{I}^{ \overline{{\bgpi}}_{\tau_n} }(\widehat{\bgtheta}_{\tau_n}^{\text{ML}})\Big\}^{ -1/2}\nabla g( \widehat{\bgtheta}_{\tau_n}^{\text{ML}} ) }{\norm{\Big\{  \mathcal{I}^{ \overline{{\bgpi}}_{\tau_n} }(\widehat{\bgtheta}_{\tau_n}^{\text{ML}})\Big\}^{ -1/2}\nabla g( \widehat{\bgtheta}_{\tau_n}^{\text{ML}} )}}$, then $\norm{h_n}=1$. By continuous mapping theorem, 
\begin{equation*}
    h_n\to h: = \frac{\Big\{  \mathcal{I}^{ {\bgpi}^* }( \bgtheta^* )\Big\}^{ -1/2}\nabla g(\bgtheta^* ) }{\norm{\Big\{  \mathcal{I}^{ {\bgpi}^* }( \bgtheta^* )\Big\}^{ -1/2}\nabla g( \bgtheta^* )}}\ a.s. \ \mathbb{P}_*.
\end{equation*}
As $n\to \infty$,
\begin{equation*}
    \frac{\sqrt{\tau_n}(g(\widehat{\bgtheta}_{\tau_n}^{\text{ML}}) -g(\bgtheta^*))}{\norm{\Big\{  \mathcal{I}^{ \overline{{\bgpi}}_{\tau_n} }(\widehat{\bgtheta}_{\tau_n}^{\text{ML}})\Big\}^{ -1/2}\nabla g( \widehat{\bgtheta}_{\tau_n}^{\text{ML}} )}}=h^T \sqrt{\tau_n} \Big\{  \mathcal{I}^{ \overline{{\bgpi}}_{\tau_n} }(\widehat{\bgtheta}_{\tau_n}^{\text{ML}})\Big\}^{ 1/2}(\widehat{\bgtheta}_{\tau_n}^{\text{ML}} -\bgtheta^*)+o_p(1)\inD N(0,1).
\end{equation*} 
This completes the proof of \eqref{lim:stopping_time_lim_g}.

\end{proof}

\subsection{Proof of Corollaries~\ref{cor:app-glm}, \ref{cor:app-cat}, and \ref{cor:app-RA}}\label{appendix:example Verification}
In this section, we will verify the regularity conditions for the applications presented in Sections~\ref{sec:app-glm}, \ref{sec:app-cat}, and \ref{sec:app-RA}, thereby proving Corollaries~\ref{cor:app-glm}, \ref{cor:app-cat}, and \ref{cor:app-RA}. 
First, according to Lemma~\ref{lemma:assumption-AB},  Assumptions~\ref{ass:6A} and \ref{ass:7A} imply Assumptions~\ref{ass:6B} and \ref{ass:7B}. The next lemma is useful for verifying Assumption~\ref{ass:4}.
\begin{lemma}\label{lem:ULLN}
   Let $\mathcal{F}^a=\left\{\log f_{\bgtheta,a}(\cdot): \bgtheta \in \bgTheta\right\},a\in \mathcal{A}$ be collections of measurable functions with a $\mathbb{P}_{*}$ integrable envelope functions. {That is,} $F_a$ satisfies that for all $\bgtheta\in\bgTheta$,
\begin{equation*}
    |\log f_{\bgtheta, a}(x^a)| \leq F_a(x^a) , a.s.\ \mathbb{P}_{*}\text{ and } \mathbb{E}_{X^a\sim f_{\bgtheta^*,a}}F_a(X^a)<\infty.
\end{equation*}
If $\bgTheta$ is compact and mapping $\bgtheta \mapsto \log f_{\bgtheta,a}(x^a)$ is continuous for every $x^a$ and $a\in\mathcal{A}$, then
\begin{equation*}
\mathbb{P}_{*} \left\{
\lim_{n\to \infty} \sup_{\bgtheta \in \bgTheta} 
|l_n(\bgtheta;\va_n)-M(\bgtheta;\overline{\bgpi}_n)|  = 0
\right\} = 1.  
\end{equation*}   
\end{lemma}

\begin{proof}[Proof of Lemma \ref{lem:ULLN}]
   The proof is similar to the proof of Theorem 2.4.1 in \cite{wellner2013weak}.%
   First, we show that the bracketing numbers $N_{[\ ]}\left(\varepsilon, \mathcal{F}^a, L_1(\mathbb{P}_{*})\right)<\infty$, for every $a\in \mathcal{A}$ and $\varepsilon>0$, {where the definition of bracketing number $N_{[\ ]}\left(\varepsilon, \mathcal{F},  \norm{\cdot } \right)$ is given by Definition 2.1.6 in  \cite{wellner2013weak}}.

Let $B(\bgtheta, \delta)=\{ \bgtheta'\in \bgTheta ; \norm{\bgtheta'-\bgtheta}< \delta  \}$. Define 
\begin{equation*}
    u^a_{B(\bgtheta',\delta)}(x^a)=\sup_{\norm{\bgtheta-\bgtheta'} < \delta} \log f_{\bgtheta,a}(x^a), \text{ and, }l^a_{B(\bgtheta',\delta)}(x^a)=\inf_{\norm{\bgtheta-\bgtheta'} < \delta} \log f_{\bgtheta,a}(x^a).
\end{equation*}
Because the envelope function $F_a$ is integrable with respect to $\mathbb{P}_{*}$ {and $\log f_{\bgtheta,a}(\cdot)$ is continuous in $\bgtheta$}, the Dominated Convergence Theorem ensures that for any $\bgtheta'$ and $\varepsilon>0$, there exists $\delta > 0$ such that
\begin{equation*}
    \mathbb{E}_{\bgtheta^*} \left( u^a_{B(\bgtheta',\delta)}(X^a) - l^a_{B(\bgtheta',\delta)}(X^a) \right) < \varepsilon.
\end{equation*} 
By compactness of $\bgTheta$, there exists $(\bgtheta_1,\delta_1),(\bgtheta_2,\delta_2),\cdots,(\bgtheta_{m},\delta_m)$, such that for any $\bgtheta\in \bgTheta$, there exists $1\leq j\leq m$, 
\begin{equation*}
    l_j(x^a)\leq \log f_{\bgtheta,a}(x^a) \leq u_j(x^a),
\end{equation*}
where $u_j= u^a_{B(\bgtheta_j,\delta_j)}$ and $l_j= l^a_{B(\bgtheta_j,\delta_j)}$. Thus, the bracketing numbers $N_{[\ ]}\left(\varepsilon, \mathcal{F}^a, L_1(\mathbb{P}_{*})\right)<\infty$, for all $\varepsilon>0$ and $a\in \mathcal{A}$. That is, we can choose finitely many $\varepsilon-$brackets $[l^a_j,u^a_j]$, whose union contains $\mathcal{F}^a$ and such that $\mathbb{E}_{\bgtheta^*} (u^a_j(X^a)-l^a_j(X^a))<\varepsilon$, %
for every $j$. Hence, for every $\bgtheta\in \bgTheta$, $a\in\mathcal{A}$, there exists $j_a$ such that
\begin{equation*}
    l_{j_a}(x^a)\leq \log f_{\bgtheta,a}(x^a) \leq u_{j_a} (x^a).
\end{equation*}
{The above inequality also implies that
\begin{equation}
  \mathbb{E}_{X\sim f_{\bgtheta,a}} l_{j_a} (X)  \leq \mathbb{E}_{X\sim f_{\bgtheta,a}} \log f_{\bgtheta,a}(x^a)\leq \mathbb{E}_{X\sim f_{\bgtheta,a}} u_{j_a} (X)
\end{equation}}
Note that, {if the functions $f_{\bgtheta,a}(\cdot)$ are inside the brackets $[l^a,u^a]$ for all $a$, then}
{
\begin{equation*}
\begin{split}
        &l_n(\bgtheta;\va_n)-M(\bgtheta;\overline{\bgpi}_n)\\
        \leq & \frac{1}{n}\sum_{i=1}^n(u^{a_i}(X_i) - \mathbb{E}[l^{a_i}(X_i) | \mathcal{F}_{i-1}] )\\
        \leq &\frac{1}{n}\sum_{i=1}^n(u^{a_i}(X_i) - \mathbb{E}[u^{a_i}(X_i) | \mathcal{F}_{i-1}] )+\varepsilon.
\end{split}
\end{equation*}
}
Thus, 
\begin{equation*}
    \sup_{\bgtheta\in \bgTheta}  (l_n(\bgtheta;\va_n)-M(\bgtheta;\overline{\bgpi}_n))\leq \max_{\substack{ a\in \mathcal{A} \\ u^a\in \{ u_j^a; 1\leq j\leq m \}}  } \frac{1}{n}\sum_{i=1}^n(u^{a_i}(X_i) - \mathbb{E}[u^{a_i}(X_i) | \mathcal{F}_{i-1}] )+\varepsilon.
\end{equation*}
By Lemma \ref{lem:SLLN}, 
\begin{equation*}
\frac{1}{n}\sum_{i=1}^n(u^{a_i}(X_i) - \mathbb{E}[u^{a_i}(X_i) | \mathcal{F}_{i-1}] )    \stackrel{\text { a.s. }}{\longrightarrow} 0.
\end{equation*}
Consequently, 
\begin{equation*}
    \limsup_{n\to \infty}\sup_{\bgtheta\in \bgTheta}  (l_n(\bgtheta;\va_n)-M(\bgtheta;\overline{\bgpi}_n)) \leq \varepsilon, a.s. \ \mathbb{P}_{*}. 
\end{equation*}
A similar argument yields that
\begin{equation*}
    \liminf_{n\to \infty}\inf_{\bgtheta\in \bgTheta}  (l_n(\bgtheta;\va_n)-M(\bgtheta;\overline{\bgpi}_n)) \geq -\varepsilon, a.s. \  \mathbb{P}_{*}. 
\end{equation*}
Taking $\varepsilon\to 0$, we obtain that 
\begin{equation*}
\mathbb{P} \left\{
\lim_{n\to \infty} \sup_{\bgtheta \in \bgTheta} 
|l_n(\bgtheta;\va_n)-M(\bgtheta;\overline{\bgpi}_n)|  = 0
\right\} = 1.  
\end{equation*} 
\end{proof}
\begin{remark}\label{rem:ULLN}
    Under Assumptions~\ref{ass:1} and \ref{ass:2}, if we assume $\nabla^2_{\bgtheta}f_{\bgtheta,a}(x^a)$ is continuous in $(\bgtheta,x^a)$ and 
    \[
    L:=\max_{a\in \mathcal{A}}\sup_{\bgtheta\in \bgTheta,x^a\in \operatorname{supp}(f_{\bgtheta,a})}\norm{\nabla^2_{\bgtheta} \log f_{\bgtheta,a}(x^a)  }_{op}<\infty,
    \]
then by the first order Taylor expansion with Lagrange remainder, we can choose the envelop function $F_a$ as 
\[
F_a(x^a)=|\log f_{\bgtheta^*,a}(x^a)|+\norm{\nabla_{\bgtheta} \log f_{\bgtheta^*,a}(x^a)}\cdot \operatorname{diameter}(\bgTheta)+ \frac{L }{2}   \cdot \operatorname{diameter}(\bgTheta)^2.
\]
\end{remark}

\begin{proof}[Proof of Corollary~\ref{cor:app-glm}]

Let $\bgxi_a=\blz^T_a \bgtheta$ and $h_{\bgxi_a,a}(x^a)=\zeta_a(x^a)\exp\{x^a \bgxi_a-B_a(\bgxi_a)\}$. Let $X^a\sim f_{\bgtheta^*,a}$. Note that
\[
\nabla_{\bgxi_a} \log h_{\bgxi_a,a}(X^a)= X^a-B_a'(\bgxi_a) \text{ and } -\nabla^2_{\bgxi_a} \log h_{\bgxi_a,a}(X^a)= B_a''(\bgxi_a)\geq \min_{\bgxi_a=\blz_a^T\bgtheta, \bgtheta\in \bgTheta}B_a''(\bgxi_a)>0.
\]
Assumption~\ref{ass:1}, \( \boldsymbol{\theta}^* \) is in the interior of \( \boldsymbol{\Theta} \), \( \boldsymbol{\xi}_a^* = \boldsymbol{z}_a^T \boldsymbol{\theta}^* \) is in the interior of \( \{ \boldsymbol{z}_a^T \boldsymbol{\theta}; \boldsymbol{\theta} \in \boldsymbol{\Theta} \} \). Applying Theorem 5.8 in \cite{lehmann2006theory}, we know that all moments for \( \nabla_{\boldsymbol{\xi}_a} \log h_{\boldsymbol{\xi}^*_a,a}(X^a) = X^a - B'_a(\boldsymbol{\xi}^*_a) \) exist.  Also note that
 $B''_{a_i}(\blz_{a_i}^T\bgtheta)>0$ and $\mathcal{I}_{\bgxi_a,a}(\bgxi_a)=B''(\bgxi_a)$
is nonsingular. Thus, Assumptions~\ref{ass:2} and \ref{ass:6A} hold. {Note that from the above derivations, $ \norm{\nabla^2_{\bgxi_a} \log f_{\bgtheta,a}(x^a) }_{op}$ does not depend on $x^a$. This, together with  Lemma~\ref{lem:ULLN} and the accompanying Remark~\ref{rem:ULLN}, implies that  Assumption~\ref{ass:4} is  satisfied.
}

Note that
\begin{equation*}
    D_{KL}( h_{\bgxi_a^*,a} \|h_{\bgxi_a,a} )=B_a(\bgxi_a)-B_a(\bgxi^*_a ) +(\bgxi_a^*-\bgxi_a^*)B_a'(\bgxi_a^*)\geq\frac{1}{2}\norm{ \bgxi_a^*-\bgxi_a^*}^2\min_{\widetilde\bgxi_a=\blz_a^T \bgtheta,\bgtheta\in \bgTheta}B_a''(\widetilde\bgxi_a).
\end{equation*}
Thus, Assumption~\ref{ass:7A} is satisfied with $C=\min_{\widetilde\bgxi_a=\blz_a^T \bgtheta,\bgtheta\in \bgTheta}B_a''(\widetilde\bgxi_a)/2$.

Thus, to prove the corollary, it is sufficient to verify Assumption~\ref{ass:3}, which will be the focus of the rest of the proof.

Because the Fisher information is $\mathcal{I}_{\bgxi_a,a}(\bgxi_a)=B''(\bgxi_a)$, the first part of conditions of Assumption~\ref{ass:3} on the smoothness of the Fisher information in $\bgtheta$ holds. We proceed to verify that $\sum_a \mathcal{I}_a(\bgtheta) $ is positive definite.

Note that
\[ \mathcal{I}_{a}(\bgtheta)=B''(\blz_a^T \bgtheta) \blz_a \blz_a^T.
\]
Thus,
\begin{equation*}
   \underline{c} \sum_{ a \in \mathcal{A}}\blz_a \blz_a^T \leq   \sum_{a\in \mathcal{A}} \mathcal{I}_a(\bgtheta)  \leq \overline{c}\sum_{a\in \mathcal{A}} \blz_a \blz_a^T,
\end{equation*}
{where $\underline{c}=\inf_{\bgtheta\in\bgTheta, a\in\mathcal{A}}B''_{a_i}(\blz_a^T \bgtheta)> 0$ and $\overline{c} = \sup_{\bgtheta\in\bgTheta, a\in\mathcal{A}}B''_{a_i}(\blz_a^T \bgtheta)<\infty$.
}

So, it is sufficient to show that $\sum_{ a \in \mathcal{A}}\blz_a \blz_a^T$ is non-singular. In the rest of the proof, we show that $\sum_{ a \in \mathcal{A}}\blz_a \blz_a^T$ is non-singular by proving the following result in linear algebra:
\begin{equation}\label{eq:linear-al}
    \{\blz_a;a\in \mathcal{A}\}^\perp=\operatorname{ker}\Big( \sum_{a\in \mathcal{A}} \blz_a\blz_a^T \Big).
\end{equation}

\paragraph*{Proof of \eqref{eq:linear-al}}
Because $\sum_{a\in \mathcal{A}} \blz_a\blz_a^T$ is positive semidefinite, 
\begin{equation*}
\begin{split}
    &\blz\in \operatorname{ker}\Big( \sum_{a\in \mathcal{A}} \blz_a\blz_a^T \Big)\iff\blz^T \Big( \sum_{a\in \mathcal{A}} \blz_a\blz_a^T \Big)\blz=0\iff \sum_{a\in \mathcal{A}} |\blz^T_a   \blz|^2=0\\
     &\iff \innerpoduct{\blz}{\blz_a}=0, \forall a\in \mathcal{A}\iff \blz\in \{\blz_a;a\in \mathcal{A}\}^\perp.
\end{split}
\end{equation*}
Since $\sum_{a\in \mathcal{A}} \blz_a\blz_a^T$ is symmetric, 
\begin{equation*}
    \mathcal{R}\Big( \sum_{a\in \mathcal{A}} \blz_a\blz_a^T \Big)=\operatorname{ker}\Big( \sum_{a\in \mathcal{A}} \blz_a\blz_a^T \Big)^\perp=\Big(\{\blz_a;a\in \mathcal{A}\}^\perp\Big)^\perp=\operatorname{span}\{\blz_a;a\in \mathcal{A}\}=\mathbb{R}^p.
\end{equation*}
This completes the proof of Corollary~\ref{cor:app-glm}.

\end{proof}
\begin{proof}[Proof of Corollary~\ref{cor:app-cat}]
{Because
\begin{equation*}
    f_{\bgtheta,a}(x_a)=\exp(b_ax_a)\exp\big\{ \blz_a^T \bgtheta-\log(1+\exp(\blz_a^T \bgtheta +b_a)) \big\},x_a\in \{0,1\},
\end{equation*}
the M2PL model described in \eqref{eq:m2pl} is a special case of the GLM described in \eqref{eq:glm}, with $B_a(\bgxi_a) = \log\{1+\exp(\bgxi_a+b_a)\}$, and $\zeta^a(x_a) = \exp(b_ax_a)$. Because the support of $B_a(\cdot)$ is $\mathbb{R}$, conditions of Corollary~\ref{cor:app-glm} are satisfied. As a result, Corollary~\ref{cor:app-cat} follows by directly applying Corollary~\ref{cor:app-cat}.}

\end{proof}

\begin{proof}[Proof of Corollary~\ref{cor:app-RA}]
{Note that the BTL model described in \eqref{eq:btl} is a special case of the M2PL model described in \eqref{eq:m2pl} with the following $\blz_a$ and $b_a$ for $a=(i,j)$, and $0\leq i<j\leq p$,
\begin{equation}
(\blz_{a},b_a)=
\begin{cases}
    (\ble_j-\ble_i, 0) &\text{ if } i\neq 0\\
    (\ble_j, 0) & \text{ if } i=0
\end{cases}.
\end{equation}
Thus, Corollary~\ref{cor:app-RA} is implied by Corollary~\ref{cor:app-cat} as long as we can verify that a connected graph $G$ ensures that $\dim (\operatorname{span}\{ \blz_{a}; a\in \mathcal{A} \})=p$. In the rest of the proof, we prove a slightly stronger result: 
}
for all $\mathcal{G}=\{a^1,\cdots,a^s\}\subset \mathcal{A}$, if the graph $(V, \mathcal{G})$ is connected, where $V=\{0,1,2,\cdots,p\}$, then $\dim (\operatorname{span}\{ \blz_{a}; a\in \mathcal{A} \})=p$.

First of all, if the graph \( G = (V, E) \) is connected, then it implies that \( s \geq p \).
Let $\blZ_{\mathcal{G}}=[\blz_{a^1},\cdots,\blz_{a^s}]$. It suffices to demonstrate that $\operatorname{rank}(\blZ_{\mathcal{G}})=p$. 
{To proceed, we construct a matrix that possesses the same rank as $\mathbf{Z}_{\mathcal{G}}$, as described below. For $a=(i,j)$, $0\leq i<j\leq p$, let
\begin{equation}
     \blz^+_{a}=
    \begin{cases}
    (-1,\blz_a^T)^T &\text{ if } i=0\\
    (0,\blz_a^T)^T &\text{ if } i>0.
    \end{cases}
\end{equation}
Then, define $\blZ^+_{\mathcal{G}}=[\blz^+_{a^1},\cdots,\blz^+_{a^s}]\in\mathbb{R}^{(p+1)\times s}$. Note that $\blz^+_{a} = ( -\blz_a^T\mathbf{1}_p,\blz_a^T)^T$ for all $a$. Consequently, %
$\operatorname{rank}(\blZ_{\mathcal{G}})=\operatorname{rank}(\blZ^+_{\mathcal{G}})$.}

Let $\ble^+_0=\ble_1,\cdots,\ble^+_p=\ble_{p+1}$, where \(\ble_1, \ldots, \ble_{p+1}\) is the standard basis for \(\mathbb{R}^{p+1}\). It is easy to check that $\blz^+_{a}=\ble^+_{a_1}-\ble^+_{a_2}$, where $a=(a_1,a_2)$. Let $v_1=a^1_1$, $v_2=a^1_2$, $S_2=\{v_1,v_2\}$ and $\mathcal{G}_{-1}=\{a\in \mathcal{G}; a\neq a^1 \} $, where $a^1=(a^1_1,a^1_2)$. Set $\widehat{a}_1=a^1$. Now we know that $\operatorname{rank}(\blz^+_{\widehat{a}_1} )=1$

Since $(V,\mathcal{G})$ is a connected graph, there exists $a'\in \mathcal{G}_{-1}$ such that $a'=(a'_1,a'_2)$ $a'_1\in S_2$ and $a'_2\not \in S_2$ (or $a'_2\in S_2$ and $a'_1\not \in S_2$). Set $v_3=a'_2$ (or $v_3=a'_1$), $S_3=S_2\cup \{v_3\}$, $\widehat{a}_2=a'$, and $\mathcal{G}_{-2}=\{a\in \mathcal{G}_{-1};a\not = \widehat{a}_2 \}$. By our construction, we know that $\blz^+_{\widehat{a}_2}=\ble^+_{a'_1}-\ble^+_{a'_2}\not \in \operatorname{span} \{\blz^+_{\widehat{a}_1} \}.$ Thus, $\operatorname{rank}([\blz^+_{\widehat{a}_1} \blz^+_{\widehat{a}_2}] ) > \operatorname{rank}  (\blz^+_{\widehat{a}_1} )$.

Because $(V,\mathcal{G})$ is a connected graph, we can always repeat the above process, until $S_{p+1}=V$. By this process, we obtain a sequence $\widehat{a}_1, \cdots,\widehat{a}_{p}$, such that
\begin{equation*}
     1=\operatorname{rank}( \blz^+_{\widehat{a}_1}  )<\cdots< \operatorname{rank}([\blz^+_{\widehat{a}_1},\cdots, \blz^+_{\widehat{a}_{p-1}}] )< \operatorname{rank}([\blz^+_{\widehat{a}_1},\cdots, \blz^+_{\widehat{a}_{p }}] )=p.
\end{equation*}
Notice that \( p=\operatorname{rank}([\blz^+_{\widehat{a}_1},\cdots, \blz^+_{\widehat{a}_{p }}] )\leq \operatorname{rank}( \blZ^+_{\mathcal{G}} )=\operatorname{rank}( \blZ_{\mathcal{G}} )\leq p\). This implies that $\operatorname{rank}( \blZ_{\mathcal{G}} )= p$.

\end{proof}

\bibliographystyle{apalike}
\bibliography{./bibliography}

\end{document}